\documentclass[11pt,reqno]{amsart}
\usepackage[top=3.5cm,bottom=3.5cm,left=2.5cm,right=2.5cm,heightrounded,bindingoffset=0mm]{geometry}

\usepackage[T1]{fontenc}
\usepackage[utf8]{inputenc}
\usepackage[english]{babel}
\usepackage{microtype}
\usepackage{lmodern}
\usepackage{amsmath,amsthm,amssymb,mathrsfs}
\usepackage{enumerate} 
\usepackage{graphicx}
\usepackage{svg}
\usepackage{color}
\usepackage{graphicx}
\usepackage{tikz}

\newcommand{\bq}{\begin{equation}}
\newcommand{\eq}{\end{equation}}
\newcommand{\bqa}{\begin{eqnarray*}}
\newcommand{\eqa}{\end{eqnarray*}}

\usepackage{hyperref}
\hypersetup{colorlinks={true},linkcolor={blue},citecolor=blue}
\usepackage[nameinlink,noabbrev,capitalise]{cleveref}
\theoremstyle{plain}
\newtheorem{theo}{Theorem}[section]
\newtheorem{prop}[theo]{Proposition}
\newtheorem{lemm}[theo]{Lemma}
\newtheorem{coro}[theo]{Corollary}

\newtheorem{defi}[theo]{Definition}
\theoremstyle{definition}
\newtheorem{rema}[theo]{Remark} 
\newtheorem{nota}[theo]{Notation}
\crefname{lemm}{Lemma}{Lemmas}
\crefname{prop}{Proposition}{Propositions}

\DeclareMathOperator{\supp}{supp}
\DeclareMathOperator{\di}{div}

\DeclareSymbolFont{pletters}{OT1}{cmr}{m}{sl}
\DeclareMathSymbol{s}{\mathalpha}{pletters}{`s}

\def\tt{\theta}
\def\eps{\varepsilon}
\def\na{\nabla}

\def\les{\lesssim}

\def\mez{\frac{1}{2}}
\def\tdm{\frac{3}{2}}

\def\Rr{\mathbb{R}}
\def\Nn{\mathbb{N}}
\def\Zz{\mathbb{Z}}

\def\cG{\mathcal{G}}
\def\cS{\mathcal{S}}

\def\cA{\mathcal{A}}
\def\cB{\mathcal{B}}

\def\cF{\mathcal{F}}

\def\cT{\mathcal{T}}

\def\L1{\mathcal{L}^{(1)}}
\def\L2{\mathcal{L}^{(2)}}
\def\L3{\mathcal{L}^{(3)}}
\def\cU{\mathcal{U}}

\def\ld{\lambda}

\def\p{\partial}
\def\d{\mathrm{d}}
\def\na{\nabla}

\def\wsc{\overset{\ast}{\rightharpoonup}}

\def\ol{\overline}
\def\T{\mathbb{T}}

\def\fa{\mathfrak{a}}
\def\ff{\mathfrak{F}}
\def\fd{\mathfrak{d}}
\def\tM{\mathtt{M}}

\def\vr{\varrho}

\numberwithin{equation}{section}

\title{Well-Posedness of the Free Boundary Incompressible Porous Media Equation}
\date{\today}

\author{Micka\"el Latocca}
\address{
Laboratoire de Mathématiques et de Modélisation d'\'Evry (LaMME)\\
Université d'\'Evry\\
23 Bd François Mitterrand, 91000 Évry-Courcouronnes, France
}
\email[M. Latocca]{mickael.latocca@univ-evry.fr}

\author{Huy Q. Nguyen}
\address{
Department of Mathematics\\
University of Maryland\\
College Park, MD 20742, USA
}
\email[H. Nguyen]{hnguye90@umd.edu}

\makeatletter
\def\l@subsection{\@tocline{2}{0pt}{2.5pc}{5pc}{}}
\makeatother
\makeatletter
\def\@tocline#1#2#3#4#5#6#7{\relax
  \ifnum #1>\c@tocdepth 
  \else
    \par \addpenalty\@secpenalty\addvspace{#2}%
    \begingroup \hyphenpenalty\@M
    \@ifempty{#4}{%
      \@tempdima\csname r@tocindent\number#1\endcsname\relax
    }{%
      \@tempdima#4\relax
    }%
    \parindent\z@ \leftskip#3\relax \advance\leftskip\@tempdima\relax
    \rightskip\@pnumwidth plus4em \parfillskip-\@pnumwidth
    #5\leavevmode\hskip-\@tempdima
      \ifcase #1
       \or\or \hskip 1em \or \hskip 2em \else \hskip 3em \fi%
      #6\nobreak\relax
    \dotfill\hbox to\@pnumwidth{\@tocpagenum{#7}}\par
    \nobreak
    \endgroup
  \fi}
\makeatother

\makeatletter
\renewcommand{\@@and}{\&}
\makeatother

\begin{document}
\newcommand{\huy}[1]{{\color{orange} \textbf{H:} #1}}
\newcommand{\mickael}[1]{{\color{teal} \textbf{M:} #1}}

\begin{abstract}
We consider the free boundary incompressible porous media equation  which describes  the dynamics of  a density  transported by a Darcy flow  in the field of gravity, with a free boundary between the fluid region and the dry region above it. For any  stratified density state, we identify a stability condition for the initial free boundary. Under this condition, we prove that small localized perturbations  of the stratified density  lead to  unique local-in-time solutions in Sobolev spaces. 

\noindent Our proof involves  analytic ingredients that are of independent interest, including   tame fractional Sobolev estimates for operators that map the Dirichlet boundary function and the forcing function of Poisson's equation to its solution in  domains of  Sobolev regularity. 
\end{abstract}



\maketitle

\tableofcontents 

\section{Introduction}
The incompressible prorus media (IPM) equation describes the evolution of a density carried by the flow of a viscous incompressible flow governed by Darcy's law in the field of gravity: 
\bq\label{IPM}
\p_t\rho +u\cdot \na_{x, y} \rho=0,\quad  u+\na_{x, y} p =-\rho e_y,\quad \di u=0,
\eq 
where gravity points downward in the $y$ direction. Here $\rho$ is the density, $u$ is the fluid velocity field,  $p$ is the fluid  pressure, and for the sake of simplicity we have normalized the dynamic viscosity, the 
permeability of the medium, and the gravitational constant in Darcy's law to unity. 

The IPM equation \eqref{IPM} is an active scalar, where  the velocity $u$ has the same level of regularity as the scalar $\rho$. This can be seen most easily in $\Rr^2$, in which case the Biot--Savart law reads   
\bq\label{u:IPM}
u=\na^\perp  (-\Delta)^{-1}\p_x\rho,\quad \na^\perp=(-\p_y, \p_x).
\eq
Thus, IPM is more singular than the 2D Euler equations written in the vorticity formulation, and has the same level of  regularity as the SQG equation. However, compared to 2D Euler and SQG, the  Biot--Savart law of IPM is {\it anisotropic} due to the presence of gravity. 

There have been many recent results on the IPM equation posed on  fixed domains such as $\Rr^2$, $\T^2$, and the channel $\T\times (-1, 1)$. These include local well-posedness \cite{CordobaGancedoRafael, ConstantinVicolWu}, lack of uniqueness of weak solutions \cite{MR2796131, MR3014484}, and  small scale formations \cite{MR4527834}. Any stratified density $\rho(x, y)=\rho(y)$ is a steady state with $u=0$.  Stability and instability results for stratified states were obtained in \cite{MR3665666, MR3916979, MR4527834}. 

On the other hand, the (one-phase)  Muskat problem concerns the  {\it free boundary} dynamics of IPM with {\it constant density} $\rho>0$. Since the density is constant, the Muskat problem can be recast solely in terms of the free boundary, resulting in a {\it quasilinear parabolic} equation.  Local well-posedness has been studied extensively, e.g. in  \cite{MR1384001,  MR2128613, MR2318314, MR2753607,  MR3415681, MR3861893,  NP, MR4541917}.  Stability of the flat free boundary was proven in \cite{MR2070208, MR3595492, MR3869383,  MR3899970, MR4387237, MR4348695}. Large-data global well-posedness of Lipschitz solutions was proven in \cite{MR4655356, DGN3D, AlazardKoch}. We also refer to \cite{MR4690615, MR4797733, N-GC} for the existence and stability of traveling waves when an external pressure is applied to the free boundary.   

Our goal in the present paper is to initiate the analysis of the {\it free boundary incompressible prorous media equation}.  The fluid region lies below the dry region in the porous medium, with a free boundary separating the two regions. This can be viewed as the (one-phase) Muskat problem for inhomogeneous fluids with the density being the fluid's density. We aim to  establish local well-posedness for this  {\it quasilinear  parabolic-hyperbolic} system of the density and the free boundary. Due to the coupling, the equation for the free boundary is not always parabolic. Therefore, one of  the key tasks is to identify the class of initial data  satisfying  suitable stability conditions that ensure the existence of local solutions. 

\subsection{Setting of the problem} The free boundary is assumed to be  the graph of an unknown function $f(x, t)$: 
\bq
    \Sigma_{f(t)}=\{(x, f(x, t)): x\in \Rr^d\}.
\eq
The  fluid domain in the prorous medium can be of  infinite or  finite depth: 
\begin{equation}
    \label{eq.domain} 
    \Omega_{f(t)}=\{(x, y)\in \Rr^d\times \Rr: y<f(x, t)\},
\end{equation}
or 
\bq\label{eq.domain.finite}
    \Omega_{f(t)}=\{(x, y)\in \Rr^d\times \Rr: b(x)<y<f(x, t)\},
\eq
where $b$ defines the fixed bottom $\Sigma_b=\{(x,b(x)), x\in\mathbb{R}^d\}$. This case is depicted in \cref{fig.1}.
\begin{figure}
    \includegraphics[width=.6\textwidth]{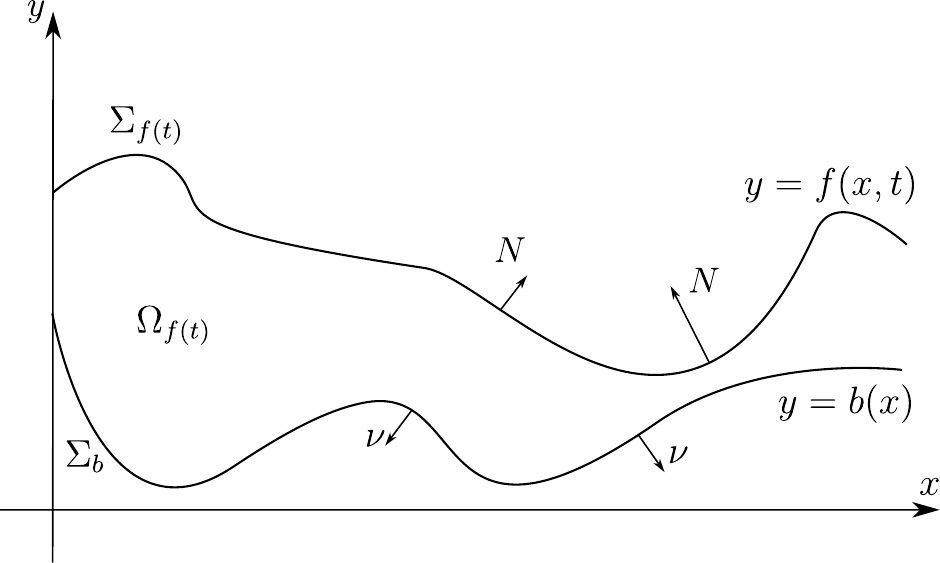}
    \caption{The domain $\Omega_{f(t)}$ in the finite-depth case.\label{fig.1}}
\end{figure}
We recall that the transported density and the fluid obey the equations 
\begin{equation}
    \label[type]{eq.transport-density}
    \p_t \rho + u \cdot \nabla_{x,y}\rho = 0\quad\text{in } \Omega_{f(t)},
\end{equation}
\bq\label{Darcy}
    u(x, y, t)+\na_{x, y}p(x, y, t)=-\rho(x, y, t)e_y, \quad \di_{x, y} u(x, y, t)=0\quad\text{in } \Omega_{f(t)},
\eq
where $e_y=(0,1) \in \Rr^d\times \Rr$. 
 
Assuming that the pressure in the dry region above the fluid is constant,  we can normalize it  to zero, so that  the dynamic boundary condition reads
\begin{equation}
    \label{eq.f0}
    p(x, f(x, t), t)=0.
\end{equation}
The free boundary evolves according to the kinematic boundary condition 
\begin{equation}
    \label{eq.kin} 
    \p_t f(x, t)=u(x, f(x, t), t)\cdot N(x, t),
\end{equation}
where $N(x, t)=(-\na_x f(x, t), 1)$ is the unnormalized exterior normal to the boundary. 

We assume that the fluid  is quiescent  at infinity in the infinite-depth case and does not penetrates the rigid bottom in the finite-depth case:
\bq\label{bc:bottom}
 \begin{cases}
     \lim_{(x,y) \to \infty} u(x,y,t) = 0 \quad\text{for~infinite~depth},\\
    u(x, b(x), t) \cdot \nu = 0  \quad\text{for~finite~depth},
    \end{cases}
\eq
where $\nu=\frac{1}{\sqrt{1+|\na b|^2}}(\na b, -1):=(\nu_x, \nu_y)$ is the normalized exterior normal to $\Sigma_b$. 

We shall refer to the system   \eqref{eq.transport-density}-\eqref{bc:bottom} as the free-boundary IPM equation. Since any stratified state $\rho=\rho(y)$ is a steady state of IPM, we will consider the following ansatz for $\rho$: 
\bq\label{rhoansatz}
    \rho(x, y, t)=\gamma(y)+g(x, y, t),\quad \lim_{(x, y)\to \infty} g(x, y, t)=0,
\eq
where $\gamma$ is any fixed function of $y\in \Rr$. This will allow us to work with the localized part $g$ of $\rho$.  To this end, let $\Gamma$ be the antiderivative of $\gamma$ that vanishes at $0$. We define the modified pressure 
\bq
q(x, y, t):=p(x, y, t)+\Gamma(y),
\eq
so that Darcy's law becomes 
\bq\label{Darcy:q}
u+\na_{x, y} q=-g e_y,\quad \di u=0
\eq
and $q$ solves
\begin{align}
&\Delta_{x, y}q=-\p_yg\quad\text{in } \Omega_{f(t)},\\ 
& q\vert_{\Sigma_{f(t)}}=\Gamma(f(x, t)),\\
&\begin{cases}\label{bc:q}
\lim_{(x, y)\to \infty} \na_{x, y}q(x, y, t)=0 \quad\text{for~infinite~depth},\\
    \p_\nu q\vert_{y=b(x)}=-\nu_yg\vert_{y=b(x)} \quad\text{for~finite~depth}. 
    \end{cases}
\end{align}
We note that \eqref{bc:q} follows from  \eqref{bc:bottom} and \eqref{Darcy:q}. 
 
In order to express $\na_{x, y} q$ in terms of $f$ and $g$, we introduce the operators 
\bq\label{def:GSop}
\cG[f]h= \na_{x, y} \phi^{(1)},\quad \cS[f]k=\na_{x, y} \phi^{(2)}
\eq
for $h: \Rr^d\to \Rr$ and $k: \Omega_{f}\to \Rr$, where $\phi^{(1)}$ and $\phi^{(2)}$ are solutions to the problems
\begin{align}
\label{eq.phi1}
&\Delta_{x, y} \phi^{(1)}=0\quad\text{in } \Omega_{f},\quad \phi^{(1)}\vert_{\Sigma_f}=h,\\ \label{bc:phi1}
&\begin{cases}
\lim_{(x, y)\to \infty} \na_{x, y}   \phi^{(1)}= 0&\text{for infinite depth},\\
\p_\nu  \phi^{(1)}\vert_{y=b(x)}=0&\text{for finite depth,}
\end{cases}
\end{align}
and 
\begin{align}
\label{eq.phi2}
&\Delta_{x, y}  \phi^{(2)}=-\p_y k\quad\text{in } \Omega_{f},\quad \phi^{(2)}\vert_{\Sigma_f}=0\\ \label{bc:phi2}
&\begin{cases}
\lim_{(x, y)\to \infty} \na_{x, y}   \phi^{(2)}= 0&\text{for infinite depth,}\\
 \p_\nu \phi^{(2)}\vert_{y=b(x)}=-\nu_yk\vert_{y=b(x)}&\text{for finite depth}.
\end{cases}
\end{align}
We note that 
\bq
(N\cdot \cG[f]h)\vert_{\Sigma_f}=G[f]h
\eq
is the usual {\it Dirichlet-Neumann operator} (see e.g. \cite{ABZ, NP}). Moreover, since $q=\phi^{(1)}+\phi^{(2)}$ with $h=\Gamma(f)$ and $k=g$, we obtain 
\bq
\na_{x, y} q=\cG[f]\Gamma(f)+\cS[f]g,
\eq
and hence
\bq\label{eq:u}
u=-\cG[f]\Gamma(f)-\cS[f]g-ge_y.
\eq
Therefore,  the  free surface $f$ and the localized part $g$ of the density are solutions of the coupled system 
\begin{align}
    &\p_tf=u\cdot N\vert_{\Sigma_f}\equiv -G[f]\Gamma(f)-(N\cdot\cS[f]g+g)\vert_{\Sigma_{f(t)}}\quad \text{in } \Rr^d,\label{eq:f}\\ 
    &\p_t g+u\cdot\na_{x, y} g+\gamma'(y)u_y=0 \quad\quad\hspace{1.47 cm} \text{in } \Omega_{f(t)}, \label{eq:g} 
\end{align}
where the velocity  $u=(u_x, u_y)$ is given in terms of $f$ and $g$ via  \eqref{eq:u}. 

The one-phase Muskat problem is a special case of the preceding system with $\gamma$  constant and $g=0$. 

\subsection{Main result}
\begin{nota}
\[
J=
\begin{cases}
(-\infty, 0)\quad\text{ in the infinite depth case},\\ 
(-1, 0)\quad\text{ in the finite depth case}.
\end{cases}
\]
\end{nota}
The stratified state $\gamma(y)$ will be taken in the following class. 
\begin{defi}\label{defi:gamma}
Let $\mathfrak{G}$ be the class of $C^\infty$ functions $\gamma: \Rr\to \Rr$ such that for all $a>0$ and $\sigma \ge 0$, there exists a function $C_\gamma: \Rr_+\to \Rr_+$ such that 
\bq\label{cont:gamma}
\| \gamma'(h(x, z)+az)w\|_{ H^\sigma(\Rr^d\times J)}\le C_\gamma(\| h\|_{(L^\infty\cap  H^\sigma)(\Rr^d\times J)})\| w\|_{(L^{\infty}\cap H^\sigma)(\Rr^d\times J)}
\eq
 and 
 \begin{multline}
    \label{cont:diff}
\|\left(\gamma'(h_1(x,z)+az)-\gamma'(h_2(x,z)+az)\right)w\|_{H^\sigma(\Rr^d\times J)}\\
\le C_{\gamma}(\|(h_1, h_2)\|_{(L^\infty \cap  H^\sigma)(\Rr^d\times J)})\|(h_1-h_2)w\|_{(L^{\infty}\cap H^\sigma)(\Rr^d\times J)}.
\end{multline}
\end{defi}
\begin{rema}\label{rema:gamma} We show in  \cref{prop:gamma} that  if $\gamma'  \in C^{\lceil \sigma \rceil+1}_b(\Rr)$  then both  \eqref{cont:gamma} and \eqref{cont:diff} hold. 
\end{rema}
The stability condition will be given in terms of the positivity of the functional 
\bq
\cT(f):=\gamma(f)-G[f]\Gamma(f).
\eq
Our main result can be informally stated as follows: if  the initial free boundary $f_0$ satisfies $\inf_{x\in \Rr^d}\cT(f_0)>0$ and the initial localized part $g_0$ of the initial density is sufficiently small, then the system  \eqref{eq:f}-\eqref{eq:g} -\eqref{eq:u} is locally well-posed. More pricisely,
\begin{theo}\label{thm.main} 
Let $d\geqslant 1$ and $\gamma \in \mathfrak{G}$, defined in \cref{defi:gamma}. Let $s> \frac{3}{2}+\frac{d}{2} $ and $(f_0,g_0) \in H^s(\mathbb{R}^d)\times H^s(\Omega_{f_0})$, where $\Omega_{f_0}$ is defined by either  \eqref{eq.domain} (infinite depth) or \eqref{eq.domain.finite} (finite depth).  Assume that 
\begin{equation}
    \label{eq.RT-initial}
   \inf_{x\in \Rr^d} \cT(f_0)(x)\ge 2\mathfrak{a}>0. 
\end{equation} 
In the finite-depth case we also assume that $b(x)=-H+b_0(x)$ with $H>0$ and $b_0 \in H^{s+\frac{1}{2}}(\mathbb{R}^d)$ such that 
\begin{equation}
    \label{eq.boundary-separation-initial}
   \inf_{x\in \Rr^d}(f_0(x)-b(x))\geqslant 2\mathfrak{d} >0.
\end{equation}
The the following assertions hold.
\begin{enumerate}[(i)]
    \item (Existence and uniqueness) For all $R>0$,  there exist $\varepsilon(R)>0$ and $T=T(R)>0$ such that if $\|f_0\|_{H^s} \le R$ and  $\|g_0\|_{H^s} \le \varepsilon (R)$ then the system  \eqref{eq:f}-\eqref{eq:g}-\eqref{eq:u} has a unique solution  
    \bq
   (f, g)\in \left(C([0, T],H^s(\mathbb{R}^d)) \cap L^2([0, T],H^{s+\frac{1}{2}}(\Rr^d))\right) \times L^{\infty}([0, T],H^s(\Omega_{f(\cdot)})
    \eq
   that satisfies 
       \[
   \forall t\in [0, T],\quad  \inf_{x\in \Rr^d} \cT(f(t))(x) \geqslant \mathfrak{a} 
    \]
    and in the finite depth case,
    \[
     \forall t\in [0, T],\quad  \inf_{x\in \Rr^d}  f(t,x)-b(x) \geqslant \mathfrak{d}.
    \]   
   Moreover, $\| f\|_{C([0, T]; H^s)}\le 2R$ and there exists $L(R)>0$ such that $\| f\|_{L^2([0, T]; H^{s+\mez})}\le L(R)$.
    \item (Continuity of the flow) Let $\ff_f: \Rr^d\times J\to \Omega_f$ be the flattening map defined by \eqref{def:diffeo} and set  $\tilde{g}(t)=g(t) \circ \ff_{f(t)}$, $\tilde{g}_0=g_0 \circ \ff_{f_0}$. Then $\tilde{g}\in C([0, T]; H^s(\Rr^d\times J))$ and the solution map  $(f_0, \tilde{g}_0)\mapsto (f, \tilde{g})$
    \begin{multline*}
    B_{H^s}(0, R) \times B_{H^s}(0, \varepsilon(R)) \longrightarrow \Big(C([0, T(R)]; H^{s'}) \cap L^2([0, T(R)]; H^{s'+\mez})\Big) \times C([0, T(R)]; H^{s'})
    \end{multline*}
    is continuous  for any $s'<s$. Moreover, there exists $\tM=\tM(s, d)>0$ such that 
    \bq\label{est:tildeg:mainthm}
    \| \tilde{g}\|_{C([0, T]; H^s(\Rr^d\times J))}\le 2\tM \eps(R).
    \eq
\end{enumerate}
\end{theo}
Several remarks about \cref{thm.main} are in order. 

(i) \underline{Positivity of the density}. \cref{thm.main}  is a well-posedness statement for the system \eqref{eq:f}-\eqref{eq:g}-\eqref{eq:u}, where the sign of the total density $\rho(x, y)=\gamma(y)+g(x, y)$  need not be positive. On the other hand, it is an immediate consequence of \eqref{est:tildeg:mainthm} that $\rho$ is positive on $[0, T]$ if $\inf_{\Rr}\gamma>0$ and $\eps(R)$ is sufficiently small. Moreover, since $\gamma$ is only evaluated at the height $y$ of points in the fluid domains, it suffices to have $\gamma$ defined on $(-\infty, 2AR]$  in the infinite-depth case and on $(-\| b\|_{L^\infty}, 2AR]$ in the finite-depth case, where $A$ is the norm of the embedding $H^s(\Rr^d)\hookrightarrow L^\infty(\Rr^d)$. In particular, in view of \cref{rema:gamma}, $\gamma$ can be any  $C^\infty(\Rr)$ function in the finite-depth case.

(ii) \underline{ The stability condition \eqref{eq.RT-initial}}. When  $\gamma$ is constant, we have  $\cT(f)=\gamma\big(1-G[f]f\big)$. It was proven in \cite[Proposition 4.3]{NP} that $1-G[f]f\ge c_f>0$ for all $f\in H^s(\Rr^d)$ with $s>1+\frac{d}{2}$. Thus \eqref{eq.RT-initial}  is automatically satisfied when $\gamma$ is constant.  When $\gamma$ is non-constant and  $\inf_{\Rr}\gamma>0$,   \eqref{eq.RT-initial}  is satisfied at least in the following two instances: 
\begin{itemize}
\item  $\| f_0\|_{H^s}$ is sufficiently small compared to $\inf_{\Rr}\gamma$. This sufficient condition depends only on $f_0$ and is a consequence of the estimate $\| G[f]\Gamma(f)\|_{H^{s-1}}\le C(\| f\|_{H^s})\| f\|_{H^s}$, where $C=C(s,  d, \gamma, \fd)$. 
\item  $\gamma'\ge 0$. See  \cref{lemm.initial-taylor}. We note  that for the IPM equation \eqref{IPM} in $\Rr^2$,  linearization around $\rho=\gamma(y)$ using the Biot-Savart law \eqref{u:IPM} gives 
\[
\p_t\rho_{lin}+\gamma'(y)u_2=0,\quad u_2=(-\Delta_{x, y})^{-1}\p_x^2\rho_{lin}.
\]
Since $(-\Delta_{x, y})^{-1}\p_x^2$ is a negative operator, we expect that $\gamma(y)$ is linearly stable when $\gamma'\le -c<0$.  In fact, it was proven in \cite{MR3665666} that  the linear stratified state $\gamma(y)=-cy$ is nonlinearly stable for the IPM equation in $\Rr^2$. The sufficient condition $\gamma'\ge 0$ for \eqref{eq.RT-initial} thus appears to destabilize the density in large time. Nevertheless, since $ (-\Delta_{x, y})^{-1}\p_x^2$ is of  order zero, the condition $\gamma'\ge 0$  would not affect  the short time existence of the density. 
\end{itemize}
(iii) \underline{The smallness condition on  $g_0$}. We shall establish the following paralinearization 
\[
G[f]\Gamma(f)=T_{m} f-T_V\cdot \na f+ l.o.t.
\]
where $m(x, \xi)$ is a first-order symbol, and $V=V(x)$; both $m$ and $V$ depend on $f$.  Under the stability condition $\cT(f)\ge \fa>0$, $m$ is elliptic and  satisfies  $m(x, \xi)\ge \fa |\xi|$. The $f$-equation  \eqref{eq:f} then becomes
\[
\p_tf= -T_{m} f+T_V\cdot \na f-(N\cdot\cS[f]g+g)\vert_{\Sigma_{f(t)}}+l.o.t.
\]
When performing the $H^s$ energy estimate, the parabolic term $-T_m$ yields a gain of $\mez$ derivative, i.e. $f\in L^2_t H^{s+\mez}$, while the transport term $T_V\cdot \na f$ does not induce  any loss of derivative. As a result, all other terms on the right-hand side can be put in $L^2_t H^{s-\mez}$. In particular, since $N$ contains $\na f$  and $\cS[f]g$ is at  linear in $g$, we have 
\[
\| N\cdot (\cS[f]g)\vert_{\Sigma_f}\|_{L^2_t H^{s-\mez}}\le C(\| f\|_{H^s})\| f\|_{L^2_t H^{s+\mez}}\| g\|_{H^s}.
\]
In fact, the estimate \eqref{eq.estimateS1} below shows that the $f$ in $\|(\cS[f]g)\vert_{\Sigma_f}\|_{H^{s-\mez}}$ also contributes to the $\| f\|_{H^{s+\mez}}$ norm on the right-hand side of the preceding estimate. Therefore, the smallness of   $g$ is required in order to close  the energy estimate for $f$  in $L^\infty_t H^s\cap L^2_t H^{s+\mez}$. 

(iv) \underline{The regularity of $f$ and $g$}. In \cref{sec.heuristics} we provide a heuristics assuming the regularity $f\in H^s$ and $g\in H^r$ and explain how $r=s >\frac{d}{2}+\frac{3}{2}$ is a natural condition as far as the energy method is concerned. Here $s>\frac{d}{2}+\tdm=1+\frac{d+1}{2}$ is the smallest Sobolev index needed to ensure that the velocity field $u$ is Lipschitz. On the other hand, since $f\in L^2([0, T];  H^{s+\mez})\subset L^2([0, T];  C^2_b)$, the free boundary has bounded curvature for a.e. $t\in [0, T]$. 

\begin{rema} Given the quasilinear nature of the problem, we expect a stronger statement than (\textit{ii}) to holds, namely that the flow is continuous as a map 
\[
 B_{H^s}(0, R) \times B_{H^s}(0, \varepsilon(R)) \longrightarrow \Big(C([0, T(R)]; H^{s}) \cap L^2([0, T(R)]; H^{s+\mez})\Big) \times C([0, T(R)]; H^{s})
    \]
but not uniformly continuous.  The continuity would follow from an argument based on the strategy used in \cite{N16} or \cite{ABITZ24}. However, we do not pursue this in order to keep the paper of reasonable length.  
\end{rema}
\begin{rema}
Our proof of Theorem \ref{thm.main} can be modified to obtain the same result for  the periodic setting, namely $f$, $b: \T^d\to \Rr$.
\end{rema}
\subsection{Heuristics}\label{sec.heuristics}
Let us consider some initial data $(f_0, g_0)\in H^s_x(\mathbb{R}^d)\cap H^r_{x, y}(\Omega_{f_0})$. Since the graphs of $f_0$ and $b$ constitute the boundary of $\Omega_{f_0}$, we readily require the compatibility condition $r-\mez\le s$, and that  at least $b\in H^{r-\mez}$. 

Assume that $(f, g)$ is a solution of \eqref{eq:f}, \eqref{eq:g} on $[0, T]$. Since \eqref{eq:g} is a transport equation, we expect that $g\in L_T^{\infty} H_{x,y}^r(\Omega_f)$. 
On the other hand, as explained above we expect the parabolicity of the $f$-equation \eqref{eq:f}, provided the stability condition $\inf_{\Rr^d}\cT(f)\ge \fa >0$.  Therefore, our goal would be to close an \textit{a priori} estimate for 
\[
    \| f\|_{ L_T^\infty H_x^s\cap L_T^2 H_x^{s+\mez}}\quad\text{and}\quad \| g\|_{L_T^{\infty} H_{x,y}^r(\Omega_f)}.
\]
From $f \in L_T^2 H_x^{s+\mez}$, for almost all $t\in [0,T]$ we have $f(t)\in H^{s+\mez}$, we expect that $\cS[f]g\in L_T^2 H_{x,y}^r$ provided $s+\mez\geqslant r+\mez$ and $b\in H^{r+\mez}$. Indeed, this follows from a heuristic elliptic regularity for \eqref{eq.phi2} when we expect $\phi^{(2)}\in H^{r+1}_{x, y}(\Omega_f)$ for $k=g\in H^r_{x, y}(\Omega_f)$. On the other hand, from the elliptic regularity for \eqref{eq.phi1}, we expect  $\cG[f]\Gamma(f)\in L_T^\infty H_{x,y}^{s-\mez}\cap L_T^2 H_{x,y}^s$. Because $s\geqslant r$, it follows from \eqref{eq:u} that 
 \[
    u=-\cG[f]\Gamma(f)-\cS[f]g-ge_y\in L^2_T H^r_{x, y}. 
\]
Then the transport equation  \eqref{eq:g}  propagates the $H^r$ regularity of $g$ provided that  $r>1+\frac{d+1}{2}$. To summarize we have imposed that
\bq\label{cd:4}
    s\geqslant r>1+\frac{d+1}{2}.
\eq
Next, closing the $L_T^\infty H_x^s\cap L_T^2 H_x^{s+\mez}$ regularity for \eqref{eq:f} requires  the forcing term $(N\cdot \cS[f]g + g)\vert_{\Sigma_{f(t)}}$ to be in $L^2_TH^{s-\mez}_x$. Since $g \in L^{\infty}_TH^{r}_{x,y}$, we have $g\vert_{\Sigma_{f}} \in L^2_TH^{r-\mez}_x$, thereby demanding $r \geqslant s$. Combining this with  \eqref{cd:4} yields
\bq
s=r>\tdm+\frac{d}{2}
\eq 
and $b_0\in H^{s+\mez}_x$.  This explains heuristically our  regularity conditions in  \cref{thm.main}. All the above regularity claims will be justified. 

\subsection{Method of proof}
 Our work builds in part on the founding paradifferential methods of \cite{ABZ} and can be viewed as an extension of \cite{NP} for the Muskat problem. In particular, we rely  on the paradifferential decomposition of the Dirichlet-Neumann operator  obtained in  \cite{ABZ, NP}. However, since the density is non-constant,   new difficulties  arise, most of which stem from the fact that the dynamics cannot be entirely recast on the free  boundary as in the Muskat problem. 
 
 (i) \underline{A priori estimates}. The main operators to be analyzed in  \eqref{eq:f}-\eqref{eq:g} -\eqref{eq:u} are the linear operators $\cG[f]$ and $\cS[f]$ defined in \eqref{def:GSop}, which are nonlinear and nonlocal with respect to the free boundary $f$. We note that for the Muskat problem \cite{NP} only the normal trace $N\cdot \cG[f]\vert_{\Sigma_f}= G[f]$ appears in the equation. Here the propagation of  the $H^s_{x, y}(\Omega_f)$ regularity of  $g$  requires the full $H^s_{x, y}(\Omega_f)$ regularity of $u$,  and hence of  $\cG[f[\Gamma(f)$ and $\cS[f]g$. For the purpose of establishing closed  a priori estimates, we are thus led to proving Sobolev estimates for  $\cG[f]h$ and $\cS[f]k$ in the entire fluid domain $\Omega_f$. As a consequence of \cref{prop.GSest} below,  we have 
 \begin{equation}
    \label{Gest:intro}
    \|\cG[f]h\|_{H^r(\Omega_f)} \le C(\|f\|_{H^s(\Rr^d)})\left(\|h\|_{H^{r+\mez}(\Rr^d)}+\| f \|_{H^{r+\mez}(\Rr^d)}\| h\|_{H^s(\Rr^d)}\right),
\end{equation}
\begin{equation}
    \label{Sest:intro}
    \|\cS[f]k\|_{H^{r}(\Omega_f)} \le C(\|f\|_{H^s(\Rr^d)})\left(\|k\|_{H^r(\Omega_f)}+\| f \|_{H^{r+\mez}(\Rr^d)}\| k\|_{H^s(\Omega_f)}\right),
\end{equation}
provided $s>\tdm+\frac{d}{2}$, $r\ge s-\mez$ and $\lfloor r\rfloor \le s$.  Keeping $\|f\|_{H^s}$ and $\| h\|_{H^s}$ as the low norms, we observe that both \eqref{Gest:intro} and \eqref{Sest:intro} are linear with respect to the highest norms $\| f\|_{H^{r+\mez}}$ and $\| h\|_{H^{r+\mez}}$.  These estimates will be applied to $h=\Gamma(f)$ with  $r=s-\mez$ and $r=s$. The case $r=s$ (satisfying the condition  $\lfloor r\rfloor \le s$) then produces a linear factor of the dissipation norm $\| f\|_{H^{s+\mez}}$, which is crucial in closing the $L^\infty_t H^s_x\cap L^2_t H^{s+\mez}_x$ energy estimate for the first-order parabolic equation \eqref{eq:f}.  

The estimates  for $\cG[f]$ and $\cS[f]$ in  \cref{prop.GSest} are consequences of sharp elliptic estimates obtained in \cref{prop.GSest-v} for the  problems \eqref{eq.phi1} - \eqref{bc:phi2}. These estimates are {\it tame} with respect to the Sobolev regularity of the domain and are of  independent interest. For the proof of these estimates we establish in \cref{lemm.FaaDiBruno} and \cref{lemm.diffeoSobolev} results on compositions of Sobolev functions  with Sobolev diffeomorphisms on domains. 

(ii) \underline{Contraction estimates}. Since $g$ is defined on the $f$-dependent domain $\Omega_f$, two solutions $(f_1, g_1)$ and $(f_2, g_2)$ cannot be directly compared. By using  diffeomorphisms $\ff_{f_i}$ to pull the $\Omega_{f_i}$ back  to a fixed domain with flat boundary, we then can compare $\tilde{g}_1:=g_1\circ \ff_{f_1}$ and $\tilde{g}_2:=g_2\circ \ff_{f_2}$. This requires contraction estimates for $\na(\phi^{(1)}\circ \ff_f)$  and $\na(\phi^{(2)}\circ \ff_f)$ with respect to $f$, where $\phi^{(j)}$  are solutions of \eqref{eq.phi1} - \eqref{bc:phi2}. These will be proven in \cref{prop.diffEstimateSobolev}. 

(iii) \underline{Construction of solutions}. Since the Muskat problem is recast in terms of $f$ defined on $\Rr^d$, the construction of solutions follows easily  from a priori and contraction estimates via the vanishing viscosity limit $\eps \Delta f$. Here  $g$ satisfies the transport equation \eqref{eq:g} posed in the domain $\Omega_f$. Hence,  a similar regularization for $g$ is inapplicable  due to the lack of boundary conditions. We  will devise an iterative scheme to construct approximate bona fide solutions $(f_n, \tilde{g}_n)$, where $f_n$ solves a nonlinear  Muskat-type equation and  $\tilde{g}_n$ solves a linear transport equation. The transport velocity of the $\tilde{g}_n$-equation is carefully designed so that on one hand it is  tangent to the boundary so as to facilitate  the existence of $\tilde{g}_n$, and on the other hand it decouples the equations for $f_n$ and $\tilde{g}_n$.  This  scheme is intricate, and we refer to the introduction of \cref{sec.lwp} for an exposition. 

\subsection{Organization of the paper}
 In section \cref{sec.prelim} we recall results on the Dirichlet-Neumann operator obtained  in \cite{ABZ, NP}, results on parabolic and transport estimates, and state results on compositions of Sobolev functions,  the proof of which is postponed to \cref{sec.compositionSobolev}. Then, in \cref{sec.DN} we state and prove our main  Sobolev estimates for the operators $\cG[f]$ and $\cS[f]$. Equipped with these estimates, we derive a priori bounds for  solutions in  \cref{sec.apriori}.  \cref{sec.contraction} is devoted to the proof of contraction estimates for $\cG[f]$, $\cS[f]$, and then for solutions.  In \cref{sec.lwp} we devise an iterative scheme to construct  a unique solution, proving  \cref{thm.main}. \cref{sec.paraDiff} provides a quick reminder on  paradifferential tools used in this paper.  \cref{sec.extension} and \cref{sec.interpolation} recall results about extension operators for minimally smooth domains   and some classical interpolation facts. Finally, \cref{appendixE} provides  sufficient conditions for the class  $\mathfrak{G}$ of density stratification profiles considered in \cref{thm.main}.

\subsection*{Acknowledgements}

Part of this work has been completed while the first author was a Novikov postdoc at the University of Maryland. The work of HQN was partially supported by NSF grant DMS-22.

\section{Preliminaries}\label{sec.prelim}

\subsection{Notation}\label{sec.notation}

\begin{itemize}
\item $\Nn=\{0, 1, 2, \dots\}$.
\item The ball of radius $r$ centered at $x$ in a normed space $X$ is denoted by $B_X(x, r)$. If $X$ is the Euclidean space, we simply write $B(x, r)$. 
\item $A \lesssim B$ means  $A \le CB$ for some positive constant  $C$.  When we want to stress the dependence of the constant on a parameter $X$, we write $A \le C(X)B$ as much as possible. When $X\ge 0$ and $C$ is locally bounded,  upon replacing $C(X)$ with $\sup_{0\le Y\le X}C(Y)$, we can assume that $C$ is increasing. 
\item For $s\in \mathbb{R}$ we write $\lfloor s \rfloor$ for the greatest integer smaller or equal to $s$, and $\lceil s \rceil$ for the smallest integer greater than $s$.
\item The Jacobian matrix $\na f$ of $f : \mathbb{R}^N \to \mathbb{R}^N$ is $(\na f)_{ij} = \p_jf_i$, so that $\na (f \circ g)=(\na f \circ g )\na g$ for $g: \Rr^N\to \Rr^N$.
\item We refer to \cref{sec.paraDiff} for the definitions  of  the paradifferential symbol classes $\Gamma^m_\rho$, the  symbol semi-norm  $M^m_{\rho}(a)$, and the paradifferential operator $T_a$. 
\end{itemize}
Next, we fix notation for function spaces. 
\begin{itemize}
\item  We sometimes use the shorthand $L^p_TZ$ for the Bochner space $L^p([0,T],Z)$. \item For $k\in \Nn$,  $C^k_b$ denotes the space of  functions with continuous and bounded derivatives up to order  of $k$.
\item Zygmund spaces are denoted by $C^s_*(\Rr^N)$.
\item  
Let $\Omega\subset \Rr^N$ is an open set, and $p\in [1, \infty]$. For $m\in \Nn$, let $W^{m, p}(\Omega)$ be the usual Sobolev space
\[
W^{m, p}(\Omega)=\left\{u\in L^p(\Omega):~\p^\alpha u\in L^p(\Omega)\quad\forall |\alpha|\le m\right\},
\]
endowed  with the norm 
\[
\| u\|_{W^{m, p}(\Omega)}:=\sum_{|\alpha|\le m}\| \p^\alpha u\|_{L^p(\Omega)}.
\]
For $s=(1-\tt)m_1+\tt m_2$ with $m_j\in \Nn$ and $\tt\in (0, 1)$, we define 
\bq\label{def:Wsp}
W^{s, p}(\Omega):=\Big(W^{m_1, p}(\Omega), W^{m_2, p}(\Omega)\Big)_{s, p}.
\eq
We denote $H^s(\Omega)=W^{s, 2}(\Omega)$. 

If $\Omega$ is a {\it minimally smooth domain}, then we have the the following equivalent norm for $W^{s, p}(\Omega)$ when $s=m+\mu$ with $m\in \Nn$ and $\mu\in (0, 1)$:   
\bq \label{eq.equivalent-norm-sobolev}
\begin{aligned}
& \|u\|_{W^{s ,p}(\Omega)} = \|u\|_{W^{m ,p}(\Omega)} +\sum_{|\alpha|=m}|\p^\alpha u|_{W^{\mu, p}(\Omega)},\\
&|v|_{W^{\mu, p}(\Omega)}:=
 \begin{cases} 
&\left(\iint_{\Omega \times \Omega} \frac{|v(x)-v(y)|^p}{|x-y|^{p\mu+N}} \,\mathrm{d}x\mathrm{d}y\right)^\frac{1}{p}\quad\text{if~} p<\infty,\\
&\text{essup}_{x, y\in \Omega, x\ne y}\frac{|v(x)-v(y)}{|x-y|^\mu}\quad\text{if~} p=\infty.
\end{cases}
\end{aligned}
\eq
We refer to \cref{lem.minimally-smooth-norm} for the justification of this norm equivalence on minimally smooth domains, whose definition is given in \cref{defi:mmdomain}.

When $s\notin \Nn$, $W^{s, p}(\Rr^N)$ coincides with the Besov space $B^s_{p, p}(\Rr^N)$. See \cite[Theorem 2.36]{BCD}. On the other hand, for all $s\ge 0$, $W^{s, 2}(\Rr^N)$ coincides with the Sobolev space 
\bq\label{def:Hs}
H^s(\Rr^N):=\left\{u\in L^2(\Rr^N): \int_{\Rr^N}(1+|\xi|^2)^s|\hat{u}(\xi)|^2<\infty \right\}
\eq
endowed with the equivalent norm 
\[
\| u\|_{H^s(\Rr^N)}:=\left(\int_{\Rr^N}(1+|\xi|^2)^s|\hat{u}(\xi)|^2\right)^\mez.
\]
See \cite[Proposition 1.59]{BCD}.  The definition \eqref{def:Hs} extends to $H^s(\Rr^N)$ for all $s\in \Rr$.
\item For the time dependent fluid domain $\Omega_{f(t)}$, if $g(\cdot, t)$ is defined on $\Omega_{f(t)}$ for a.e. $t\in [0, T]$, we say that $g\in L^\infty([0, T]; W^{s, p}(\Omega_{f(\cdot)}))$ if the form 
\[
\| g\|_{L^\infty([0, T]; W^{s, p}(\Omega_{f(\cdot)}))}:=\text{essup}_{t\in [0, T]} \| g(\cdot, t)\|_{W^{s, p}(\Omega_{f(t)})}
\]
if finite.

\end{itemize}
\subsection{The Dirichlet-Neumann operator}

The Dirichlet-Neumann operator being essential in the analysis of \eqref{eq:f}, we will need the following results on the  boundedness, paralinearization, and   contraction estimates.
\begin{theo}[\cite{ABZ,NP}]\label{theo:DN1}
Let $d\geqslant 1$, $s>1+\frac{d}{2}$ and $\sigma \in [\mez, s]$.  
\begin{enumerate}[(i)]
    \item \cite[Theorem 3.12]{ABZ} If $f\in H^s(\Rr^d)$ and $h\in H^\sigma(\Rr^d)$, then $G[f]h\in H^{\sigma-1}(\Rr^d)$ and 
    \bq\label{est:DN:ABZ000}
    \| G[f]h\|_{H^{\sigma-1}}\leq \mathcal{F}(\| f\|_{H^s})\| h\|_{H^\sigma}
    \eq
    where $\mathcal{F}:\Rr_+\to\Rr_+$ is non-decreasing and depends only on $(d, s, \sigma)$ in the infinite depth case and also on $(\mathfrak{d}, \| \na b\|_{L^\infty})$ in the finite depth case.
    \item \cite[Theorem 3.18]{NP} Let $\delta\in (0, \mez]$ satisfy $\delta<s-\frac{d}{2}-1$. If $f\in H^{s+\mez-\delta}(\Rr^d)$ and $h\in H^\sigma(\Rr^d)$, then we have
    \bq\label{paralin:DN}
       G[f]h= T_{\lambda[f]}(h-T_Bf) - T_V\cdot \nabla f+ R[f]h
    \eq
    where
    \bq\label{def:ld}
    \lambda[f] (x,\xi)=\sqrt{(1+|\nabla f (x)|^2)|\xi|^2 - (\nabla f (x)\cdot \xi)^2},
    \eq
     \bq\label{def:BV}
     V = \nabla h - B\nabla f,\quad B=\frac{\nabla f \cdot \nabla h + G[f]h}{1+|\nabla f|^2},
     \eq
     and the remainder operator $R[f]$ satisfies 
     \bq\label{est:RDN}
     \| R[f]h\|_{H^{\sigma-\mez}}\leq \mathcal{F}(\| f\|_{H^s})(1+\| f\|_{H^{s+\mez-\delta}})\| h\|_{H^\sigma},
     \eq
     where $\mathcal{F}:\Rr_+\to\Rr_+$ is non-decreasing and depends only on $(d,s,\sigma, \delta)$ in the infinite depth case and also on $(\mathfrak{d}, \| \na b\|_{L^\infty})$ in the finite depth case.
\end{enumerate}
\end{theo}

We note that by Cauchy-Schwartz's inequality, $\ld[f]$ is an elliptic symbol: for all $(x, \xi)\in \Rr^d\times \Rr^d$ there holds 
\bq\label{elliptic:ld}
\ld[f](x, \xi)\ge |\xi| .
\eq
To obtain contraction estimates we will also use the following results. 
\begin{theo}\cite{NP}\label{thm.NPdiff} Let $d\geq 1$, $s>1+\frac{d}{2}$ and $\delta \in (0,\frac{1}{2}]$ be such that $\delta < s-1-\frac{d}{2}$. Let $h\in H^s(\Rr^d)$ and also $f_1, f_2 \in H^s(\mathbb{R}^d)$.
\begin{enumerate}[(i)]
\item \cite[Theorem 3.24]{NP} For any $\sigma \in [\mez +\delta,s]$, we have  
\begin{equation}
    \label{eq.diffGf}
    G[f_1]h-G[f_2]h=-T_{\lambda[f_2]B_2}(f_1-f_2)-T_{V_2}\cdot \nabla(f_1-f_2) + R_\sharp[f_1, f_2]h
\end{equation}
where
\begin{equation}
    \label{eq.RsharpEst}
    \|R_\sharp[f_1, f_2]h\|_{H^{\sigma -1}} \le \mathcal{F}(\|(f_1, f_2)\|_{H^{s}})\|f_1-f_2\|_{H^{\sigma - \delta}}\|h\|_{H^s}
\end{equation}
and $\mathcal{F} : \mathbb{R}_+\to\mathbb{R}_+$ is a non-decreasing function depending only on $(d, s, \sigma, \delta)$ in the infinite depth case and also on $(\mathfrak{d}, \| \na b\|_{L^\infty})$ in the finite depth case.
\item \cite[Corollary 3.25]{NP} For any $\sigma\in [\mez, s]$, we have
\bq
    \label{eq.diffGfi}
    \| G[f_1]h-G[f_2]h\|_{H^{\sigma-1}}\le C(\| (f_1, f_2)\|_{H^s})\| f_1-f_2\|_{H^\sigma}\| h\|_{H^s},
\eq
where $\mathcal{F} : \mathbb{R}_+\to\mathbb{R}_+$ is a non-decreasing function depending only on $(d, s, \sigma)$ in the infinite depth case and also on $(\mathfrak{d}, \| \na b\|_{L^\infty})$ in the finite depth case.
\end{enumerate}
\end{theo}

\subsection{Elliptic, parabolic, and transport estimates}

We define the following homogeneous Sobolev spaces 
\bq
\dot H^\sigma(\Rr^d)= \{f\in \mathcal{S}'(\Rr^d)\cap L^2_{\text{loc}}(\Rr^d): |D|^\sigma f\in L^2(\Rr^d)\}/~\Rr,
\eq
\bq\label{def:dotHs}
  H^{1,\sigma}(\Rr^d)=\{f\in \mathcal{S}'(\Rr^d)\cap L^2_{\text{loc}}(\Rr^d): \na f\in H^{\sigma-1}(\Rr^d)\}/~\Rr,
\eq
which are endowed with their natural norms. 

We start with variational estimates for the elliptic problems \eqref{eq.phi1} and \eqref{eq.phi2}.
\begin{lemm}\label{lemm.variationalEstimate} Assume that $f\in W^{1, \infty}(\Rr^d)$ and $b\in W^{1, \infty}(\Rr^d)$. 
\begin{enumerate}[(i)]
\item Set $W= \dot H^\mez(\Rr^d)$ in the infinite depth and $W= H^{1, \mez}(\Rr^d)$ in the finite depth. If $h\in W$, then there exists a unique variational solution $\phi^{(1)} \in \dot H^1(\Omega_{f})$ to \eqref{eq.phi1}-\eqref{bc:phi1} such that 
\bq\label{vari:phi1}
    \| \na_{x, y}\phi^{(1)}\|_{L^2(\Omega_f)}\le \mathcal{F}(\| \na f\|_{L^\infty})\| h\|_W,
\eq 
where $\mathcal{F} : \Rr_+ \to \Rr_+$ is non-decreasing, and depends on $d$, and also on $(\mathfrak{d},\| \na b\|_{L^\infty})$ in the finite depth case. Moreover, in the finite-depth case, if $h\in H^\mez(\Rr^d)$ then 
\bq\label{vari:phi1:Poincare}
    \| \phi^{(1)}\|_{H^1(\Omega_f)}\le \mathcal{F}(\| \na f\|_{L^\infty})\| h\|_{H^\mez(\Rr^d)}. 
\eq
\item If $k\in L^2(\Omega_f)$, then there exists a unique variational solution $\phi^{(2)} \in \dot H^1(\Omega_f)$ to \eqref{eq.phi2}-\eqref{bc:phi2} such that 
\bq\label{vari:phi2}
    \| \na_{x, y}\phi^{(2)}\|_{L^2(\Omega_f)}\le \| k\|_{L^2(\Omega_f)}.
\eq
\end{enumerate}
Moreover, in the finite-depth case there holds 
\bq\label{vari:phi2:Poincare}
\| \phi^{(2)}\|_{L^2(\Omega_f)}\le \| f-b\|_{L^\infty(\Rr^d)}\| \p_y \phi^{(2)}\|_{L^2(\Omega_f)}. 
\eq
\end{lemm}

\begin{proof} 
Part (i) was proven in \cite[Propositions 3.4 and 3.6]{NP}. As for (ii), existence and uniqueness follows from a similar argument, so we only focus on proving \eqref{vari:phi2}. We only provide computations assuming smoothness of $\phi_2$. Let us integrate by parts as follows, using the boundary conditions for $\phi^{(2)}$, so that we find
\begin{align*}
    \int_{\Omega_f}\phi^{(2)}(-\Delta_{x, y}\phi^{(2)})&=\int_{\Omega_f}|\na_{x, y}\phi^{(2)}|^2-\int_{\Rr^d} (\phi^{(2)}\p_\nu\phi^{(2)})(x, b(x))\sqrt{1+|\na b(x)|^2}dx\\
    &=\int_{\Omega_f}|\na_{x, y}\phi^{(2)}|^2-\int_{\Rr^d} (\phi^{(2)}k)(x, b(x))dx.
\end{align*}
On the other hand, using Fubini's theorem we find 
\begin{align*}
    \int_{\Omega_f}\phi^{(2)}\p_y k&=\int_{\Rr^d} \int_{b(x)}^{f(x)}(\phi^{(2)}\p_y k)(x, y)dydx\\
    &=-\int_{\Rr^d} \int_{b(x)}^{f(x)}(\p_y\phi^{(2)}k)(x, y)dydx-\int_{\Rr^d}(\phi^{(2)}k)(x, b(x))dx,
\end{align*}
where we have used $\phi^{(2)}(x,f(x))=0$.
Using $-\Delta_{x, y}\phi^{(2)}=\p_yk$ and the preceding formulas, we deduce 
\[
    \|\na_{x, y}\phi^{(2)}\|_{L^2(\Omega_f)}^2=-\int_{\Omega_f} \p_y\phi^{(2)}k\le \| \p_y\phi^{(2)}\|_{L^2(\Omega_f)}\| k\|_{L^2(\Omega_f)}
\]
which in turn implies \eqref{vari:phi2}. Finally, since $\phi^{(2)}=0$ on $\{y=f(x)\}$, \eqref{vari:phi2:Poincare} is the Poincar\'e inequality for the finite-depth domain $\Omega_f$.
\end{proof}

\begin{rema} Note that as $H^{\mez} \hookrightarrow W$, (\textit{i}) implies 
\begin{equation}
    \label{eq.X12-est}
    \|\nabla_{x,y}\phi_1\|_{L^2(\Omega_f)} \le C(\|\na f\|_{L^{\infty}})\|h\|_{H^{\mez}(\Rr^d)}. 
\end{equation}
\end{rema}

Estimates for first-oder parabolic equations are stated in the following spaces: 
\bq\label{def:XY}
    \begin{aligned}
    &X^{\sigma}(I) = C_z^0(I,H^{\sigma}(\mathbb{R}^d) \cap L^2_z(I, H^{\sigma + \frac{1}{2}}(\Rr^d)),\\ 
    &Y^{\sigma}(I)= L^{1}_z(I,H^{\sigma}(\mathbb{R}^d)) + L^2_z(I, H^{\sigma - \frac{1}{2}}(\Rr^d)),
    \end{aligned}
\eq
endowed with their natural norms. 

\begin{prop}[Proposition 2.18 in \cite{ABZ} and Proposition~3.2 in \cite{WZ17}]\label{prop.pseudoParabolic} Let $0 < \rho < 1$, $\sigma \in \mathbb{R}$, $I=[z_0,z_1] \subset \mathbb{R}$ and $p \in \Gamma^1_{\rho}(I \times \mathbb{R}^d)$ such that there exists $c>0$ such that for all $(z,x,\xi) \in I \times \mathbb{R}^d \times \mathbb{R}^d$ there holds $\operatorname{Re} p (z,x,\xi) \geqslant c |\xi|$. Then for any $f \in Y^\sigma(I)$ and any $w_0 \in H^\sigma$ there exists a unique solution $w \in X^{\sigma}(I)$ to 
\bq\label{eq:para-parabolic}
\begin{array}{rcll}
    \partial_z w + T_pw &=& f &\text{in } I \times \mathbb{R}^d \\ 
    w(0,x)&=&w_0(x) &\text{for } x\in\mathbb{R}^d.
\end{array}
\eq
Moreover, $w$ obeys the bound
\bq\label{parabolicest}
    \|w\|_{{X}^\sigma(I)} \le \cF(\mathcal{M}^1_{\rho}(p)) \left( \|f\|_{{Y}^\sigma(I)} + \|w_0\|_{H^\sigma} + \|w\|_{L^2(I,H^\sigma)}\right),
\eq
where $\cF : \Rr_+ \to \Rr_+$ is a non-decreasing function depending only on $(d,\sigma, \rho)$.  
\end{prop}
\begin{rema}\label{rema:parabolicI}
In order to obtain results for the infinite-depth case, a crucial point for us is that the constant $\cF$ in \eqref{parabolicest} does not depend on the interval $I$. In \cite[Proposition 2.18]{ABZ}, the constant $\cF$ depends on $|I|$, and this dependence was dropped in \cite[Proposition 3.2]{WZ17}, in which the authors prove \eqref{parabolicest} in the stronger setting of Chermin-Lerner spaces.
\end{rema} 

For the proof of \cref{thm.main}, we will reduce \eqref{eq:f} to an equation of the form
\begin{equation}
    \label{eq.parabolic}
    \partial_t f + T_af + T_U\cdot \nabla f = F,
\end{equation}
where $U=U(x, t)$, and $a(x, \xi, t)$ are time-dependent paradifferential symbols. 
Energy estimates for \eqref{eq.parabolic} are based on the following lemma.

\begin{lemm}\label{lemm:parabolicest} Assume that for some  $\delta \in (0, 1)$ and $\mathfrak{a}>0$, we have  $U\in L^\infty([0, T], W^{\delta, \infty}(\mathbb{R}^d))$ and that $a(t, x, \xi)\in L^\infty([0, T], \Gamma^1_\delta)$ is such that for all $x,\xi \in \Rr^d$ there holds $\operatorname{Re}a(x,\xi) \geqslant \mathfrak{a} |\xi| >0$. Assume that $F\in L^2([0, T],H^{s-\mez}(\Rr^d))$ for some $s\in \Rr$ and assume that $f\in C([0, T], H^s(\Rr^d)) \cap L^2([0, T], H^{s+\mez}(\Rr^d))$ is a solution of \eqref{eq.parabolic}. 
Then there exists function $C: \Rr_+^3\to\Rr_+$ depending only on $(\delta, d)$, non-decreasing in each variable and such that for all $t\in(0,T)$ there holds  
\bq\label{parabolic:diff:est}
\begin{aligned}
\mez\frac{d}{dt}\| f(t)\|_{H^s}^2&\le -\frac{1}{C(\|U\|_{L^\infty_TW^{\delta,\infty}}, \| a\|_{L^\infty_TM^1_\delta}, \mathfrak{a})}\| f(t)\|_{H^{s+\mez}}^2\\
&\qquad+C(\|U\|_{L^\infty_TW^{\delta,\infty}}, \| a\|_{L^\infty_TM^1_\delta}, \mathfrak{a})\|f(t)\|_{H^s}^2+\| F(t)\|_{H^{s-\mez}}\| f(t)\|_{H^{s+\mez}}.
\end{aligned}
\eq
\end{lemm}
\begin{proof}  Set  $f_s=\langle D_x\rangle^sf$ and $F_s=\langle D_x\rangle^sF$. By commuting \eqref{eq.parabolic} with $\langle D_x\rangle^s$, using symbolic calculus for paradifferential operators given in \cref{theo:sc}, and performing an $L^2$ estimate, one can show that 
\bq\label{L2est:paraboliceq:1}
    \mez\frac{d}{dt}\| f_s\|_{L^2}^2\le -\frac{1}{C}\| f_s\|_{H^\mez}^2+C\|f_s\|_{H^\mez}\| f_s\|_{H^{\mez-\delta}}+\| f_s\|_{H^\mez}\| F_s\|_{H^{-\mez}},\quad 0<t<T,
\eq
where $C=C(\|U\|_{L^\infty_TW^{\delta,\infty}}, \| a\|_{L^\infty_TM^1_\delta}, \mathfrak{a})$. A detailed proof of \eqref{L2est:paraboliceq:1} can be found in \cite[Section 4.1.3]{NP}. By interpolation, there is $\theta = \theta(\delta) \in (0,1)$ such that 
\[
    \|f_s\|_{H^\mez}\| f_s\|_{H^{\mez-\delta}}\les \|f_s\|^{1+\tt}_{H^\mez}\| f_s\|_{L^2}^{1-\tt}.
\]
Thus applying Young's inequality, we obtain  (with a different constant $C$)
\[
    \mez\frac{d}{dt}\| f_s\|_{L^2}^2\le -\frac{1}{C}\| f_s\|_{H^\mez}^2+C\|f_s\|_{L^2}^2+\| F_s\|_{H^{-\mez}}\| f_s\|_{H^{\mez}},\quad 0<t<T.
\]
This yields \eqref{parabolic:diff:est}. 
\end{proof}
The next theorem concerns well-posedness and regularity estimates for linear transport equations in domains.
\begin{theo}\label{theo:transport}
Let $\cU\subset \Rr^N$ be a minimaly smooth domain with constants $(\iota, K, M)$. Let $\sigma \ge 1$, $h_0\in H^\sigma(\cU)$, $F\in L^1([0, T]; H^\sigma(\cU))$, and $v\in L^p([0, T]; L^2(\cU))$ for some $p>1$, and
\begin{align}
&\na_x v\in L^1([0, T]; H^\frac{N}{2}(\cU) \cap L^\infty(\cU))\quad\text{if } \sigma<1+\frac{N}{2},\\
&\na_x v\in L^1([0, T]; H^{\sigma-1}(\cU))\quad\text{if } \sigma>1+\frac{N}{2}.
\end{align}
Assume that $v\cdot n=0$ on $\p\cU$, where $n$ is the outward unit normal vector to $\p\cU$. Then the transport equation
\bq\label{eq:transport:h}
\p_t h+v\cdot \na h=F\quad\text{in } \cU\times (0, T),\quad h\vert_{t=0}=h_0
\eq
has a solution $h\in C([0, T]; H^\sigma(\cU))$ which satisfies 
\bq\label{transportest}
\| h\|_{C([0, T]; H^\sigma(\cU))}\le \tM\left(\| h_0\|_{H^\sigma(\cU)}+ \| F\|_{L^1([0, T]; H^\sigma(\cU))}\right)\exp(\tM V_T)),
\eq
where $\tM=\tM(\sigma, N, \iota, K, M)$, and 
\bq\label{def:VT}
V_T=\begin{cases}
\| \na_x v\|_{L^1([0, T]; H^\frac{N}{2}(\cU) \cap L^\infty(\cU))}\quad\text{if } \sigma<1+\frac{N}{2}, \\
\| \na_x v\|_{L^1([0, T]; H^{\sigma-1}(\cU))}\quad\text{if } \sigma>1+\frac{N}{2}.
\end{cases}
\eq
Moreover, $h$ is the unique solution in $C([0, T]; H^1(\cU))$.  
\end{theo}
\begin{proof}
Let $\mathfrak{C}$ be the Sobolev extension operator for $\cU$, as given in \cref{th.extension}. Let $\tilde{v}=\mathfrak{C}v$, $\tilde{F}=\mathfrak{C}F$,  $\tilde{h}_0=\mathfrak{C}h_0$, and $k$ be the solution of the transport problem 
\bq\label{transporteq:extension}
\p_t k+\tilde{v}\cdot\na k=\tilde{F},\quad k\vert_{t=0}=\tilde{h}_0,\quad (x, t)\in \Rr^N\times (0, T).
\eq
By virtue of \cite[Theorem 3.19]{BCD} on the solvability of transport equations on $\Rr^N$, \eqref{transporteq:extension} has a unique solution $k\in C([0, T]; H^{\sigma}(\Rr^N))$ which satisfies 
\bq\label{transportest:wholespace}
\| k\|_{L^\infty([0, T]; H^\sigma(\Rr^N))}\le \left(\| \tilde{h}_0\|_{H^\sigma(\Rr^N)}+ \| \tilde F\|_{L^1([0, T]; H^\sigma(\Rr^N))}\right)\exp(\tM\tilde{V}_T)),
\eq
where $\tM$ depends only on $(\sigma, N)$, and  $\tilde{V}_T$ is defined as in \eqref{def:VT} but for $\tilde{v}$ in place of $v$ and $\Rr^N$ in place of $\cU$. Then, $h:=k\vert_{\cU}$ is a solution of \eqref{eq:transport:h}. Moreover, it follows from \eqref{transportest:wholespace} and the continuity of $\mathfrak{C}$ that $h$ satisfies \eqref{transportest} for a new $\tM=\tM(\sigma, N, \iota, K, M)$. Note that $\sigma \ge 0$ suffices for the above proof. 

To prove the uniqueness, assume that $h_1$ and $h_2$ are solutions of \eqref{eq:transport:h} in $C([0, T]; H^1(\cU))$. Then, $f:=h_1-h_2$ satisfies 
\[
    \p_tf+v\cdot \na f=0\quad\text{in } \cU\times (0, T),\quad f\vert_{t=0}=0.
\]
Since $v$ is tangent to $\p \mathcal{U}$ and $f\in C([0, T]; H^1(\cU))$, we can justify the following identity obtained using integrations by parts:
\[
    \mez \frac{d}{dt}\int_{\cU}|f|^2dx=-\mez \int_{\cU }\di v |f|^2dx.
\]
Since $\na v\in L^1([0, T]; L^\infty(\cU))$, Gr\"onwall's lemma implies that $f=0$, and thus $h_1=h_2$.
\end{proof}
\subsection{Compositions and inverses of Sobolev diffeomorphisms}
The following propositions concerning diffeomorphisms in $\Rr^N$ with Sobolev regularity will be applied frequently in order to convert regularity estimates from the fluid domain $\Omega_f$ to its flattened version and vice-versa. These independent results will be proven in \cref{sec.compositionSobolev}. 
\begin{prop}\label{lemm.FaaDiBruno} Let $N \geqslant 2$, $\sigma \geqslant 0$, and $\cU\subset \Rr^N$ be a minimally smooth domain. Assume that $\varphi = \operatorname{id} + \psi: \cU\to \Omega$ is a Lipschitz diffeomorphism such that $|\det \na \varphi|\ge c_0>0$ and $\psi \in H^{\sigma}(\mathcal{U})$. Then for any $F \in H^{\sigma}(\Omega,\mathbb{R})$  with $\na F\in L^\infty(\cU)$,  there holds   
\begin{equation}
    \label{eq.compositionEstimate}
    \|F \circ \varphi\|_{H^{\sigma}(\mathcal{U})} \le C\left(\|\nabla \varphi\|_{L^{\infty}(\mathcal{U})}\right) \left(\|F\|_{H^{\sigma}(\Omega)} + \|\nabla F \|_{L^{\infty}(\Omega)}\|\psi\|_{H^{\sigma}(\mathcal{U})}\right),
\end{equation}
where $C : \Rr_+ \to \Rr_+$ is non-decreasing and depends only on $(N,\sigma, c_0, \iota, K,M)$.
\end{prop}

\begin{prop}\label{lemm.diffeoSobolev} Let $N \geqslant 2, \sigma \geqslant 1$ and $\Omega$, $\mathcal{U} \subset \mathbb{R}^N$ be minimally smooth domains. Let $\varphi = \operatorname{id} + \psi: \mathcal{U} \to \Omega$ be a Lipschitz diffeomorphism such that $|\det \na \varphi|\ge c_0>0$ and $\psi \in H^{\sigma}(\mathcal{U})$. Then there holds 
\begin{equation}
    \label{eq.diffoSobolev}
    \|\varphi^{-1} - \operatorname{id}\|_{H^{\sigma}(\Omega)} \le C(\|\nabla \psi\|_{L^{\infty}(\mathcal{U})})
    \|\psi\|_{H^{\sigma}(\mathcal{U})}, 
\end{equation}
where $C : \Rr_+ \to \Rr_+$ is non-decreasing and depends only on $(N,\sigma, c_0, \iota, K, M)$.
\end{prop}

\subsection{Elliptic regularity in Sobolev domains}
In the finite-depth case, we will utilize the following  Sobolev regularity result for the Neumann problem. The case of integer regularity can be found in \cite[Corollary 3.8]{ChengShkoller}. Here we provide a self-contained proof for non-integer regularity.  
\begin{prop}\label{thm.elliptic-CS} 
Let $\Omega \subset \mathbb{R}^N$ be $C^{\infty}$ bounded domain, $N \ge 2$. We denote by $\nu$  the unit exterior normal to $\p\Omega$. Let $s>\frac{N}{2}$, $A=(A^{jk})_{1 \le j,k \le N} \in \Rr^{N\times N}$, and $a=(a^{jk})_{1 \le j,k \le N}$ a matrix with $H^s(\Omega)$ coefficients. We assume that $A$ is positive definite, i.e. there exists $\lambda >0$ such that 
\[ 
    A^{jk}\xi_j\xi_k \geqslant \lambda |\xi|^2\quad\forall \xi \in \Rr^N.
\] 
Then there exists $\varepsilon_0 = \varepsilon_0(\lambda,\sup_{1\le i,j\le N}|A^{ij}|, s, \Omega)>0$ such that for  $\| a\|_{L^\infty(\Omega)}\le \varepsilon_0$ the following holds. 
For all $(F,g) \in H^{s-1}(\Omega) \times H^{s-\mez}(\p\Omega)$ satisfying  $\int_{\Omega} F + \int_{\p\Omega} g = 0$, the elliptic problem
\begin{equation}
    \label{eq.elliptic-eq}
        \begin{cases}
        -\operatorname{div}((A+a)\nabla\psi)=F \quad \text{in~ } \Omega \\
        ((A+a)\na \psi)\cdot \nu = g \quad \text{on~ } \p\Omega,
        \end{cases}       
\end{equation}
has a unique solution with mean zero which satisfies 
\bq\label{est:elliptic:Aa}
    \|\psi\|_{H^{s+1}(\Omega)} \le C(\lambda, \sup_{1\le i,j\le N}|A^{ij}|, \|a\|_{H^s}, s, \Omega)\left(\|F\|_{H^{s-1}(\Omega)} + \|g\|_{H^{s-\mez}(\p\Omega)}\right). 
\eq  
\end{prop}
    
\begin{proof} 
1. Let $\tilde{F}\in H^{k-1}(\Omega)$ and $\tilde{g}\in H^{k-\mez}(\p\Omega)$ with integer $k\ge 1$ such that  $\int_{\Omega} \tilde{F} + \int_{\p\Omega}\tilde{g} = 0$. Since $A$ is a positive  constant matrix, by the classical regularity for the Neumann problem, the unique mean-zero weak solution to  
\bq\label{eq.elliptic-eq:int}
   \begin{cases}
        -\operatorname{div}(A\nabla\psi)=\tilde{F} \quad \text{in ~} \Omega \\
        (A\na \psi)\cdot \nu = \tilde{g} \quad \text{on ~} \p\Omega
\end{cases}
\eq
satisfies
\[ 
    \|\psi\|_{H^{k+1}(\Omega)} \le C(\lambda, \sup_{1\le i,j \le N}|A^{ij}|, k, \Omega)\left(\|\tilde{F}\|_{H^{k-1}(\Omega)}+\|\tilde{g}\|_{H^{k-\frac{1}{2}}(\p\Omega)}\right).
\] 
Then linear interpolation yields
\begin{equation}\label{eq.Hs-elliptic-bound}
    \|\psi\|_{H^{s+1}(\Omega)} \le C(\lambda, \sup_{1\le i,j \le N}|A^{ij}|, s, \Omega)\left(\|\tilde{F}\|_{H^{s-1}(\Omega)}+\|\tilde{g}\|_{H^{s-\mez}(\p\Omega)}\right),\quad s\ge 1.
\end{equation}
Let us detail this interpolation procedure. Consider  $s+1=(1-\theta)k+\theta (k+1)$ for $\tt\in (0, 1)$ and $\Nn \ni k\ge 2$.  Then the real interpolation space $(H^k(\Omega),H^{k+1}(\Omega))_{\theta, 2}$ is $H^{s+1}(\Omega)$, and simiarly $(H^{k-2}(\Omega) \times H^{k-\frac{3}{2}}(\p\Omega), H^{k-1}(\Omega) \times H^{k-\mez}(\p\Omega))_{\theta, 2} = H^{s-1}(\Omega) \times H^{s-\mez}(\p\Omega)$.
In order to view \eqref{eq.Hs-elliptic-bound} as a continuity bound of the linear operator $\Psi : (\tilde{F},\tilde{g}) \mapsto \psi$ between Sobolev spaces, note that since a necessary and sufficient condition for the existence of $\psi$ is that $\int_{\Omega} \tilde{F} + \int_{\p\Omega} \tilde{g} = 0$, and that uniqueness for $\psi$ holds provided we impose $\int_{\Omega} \psi = 0$. We obtain that the operator $\Psi$ is a well-defined continuous operator 
\[
    (H^{\ell}(\Omega) \times H^{\ell+\mez}(\p\Omega)) \cap \mathcal{E} \to H^{\ell+2}(\Omega) \cap \mathcal{F},
\] 
for $\ell \in\{k-2,k-1\}$, where we have introduced the subspace $\mathcal{E}=\ker(T_1-\operatorname{id})$ with $T_1(\tilde{F},\tilde{g})=(\tilde{F} - |\Omega|^{-1}(\int_{\Omega} \tilde{F} + \int_{\p\Omega} \tilde{g}),\tilde{g})$ and similarly $\cF=\ker(T_2-\operatorname{id})$ with $T_2(\tilde{F})=\tilde{F}-|\Omega|^{-1}\int_{\Omega}\tilde{F}$. Note that $T_j^2=T_j$. It follows that $\Psi$ is continuous 
\[
    \left((H^{k-2}(\Omega) \times H^{k-\frac{3}{2}}(\p\Omega)) \cap \mathcal{E},(H^{k-1}(\Omega) \times H^{k-\mez}(\p\Omega)) \cap \mathcal{E} \right)_{\theta, 2} \to \left(H^{k}(\Omega) \cap \mathcal{F},H^{k+1}\cap \mathcal{F}\right)_{\theta, 2},  
\]
and \eqref{eq.Hs-elliptic-bound} follows from statements of the form $(X_0 \cap \mathcal{F}, X_1 \cap \mathcal{F})_{\theta, 2} = (X_0,X_1)_{\theta, 2} \cap \mathcal{F}$ which are stated in \cref{lemm.interpol-def}. 

2. Suppose that $s>\frac{N}{2}$ and $(F,g) \in H^{s-1}(\Omega) \times H^{s-\mez}(\p\Omega)$ satisfy   $\int_{\Omega} F + \int_{\p\Omega} g = 0$. We consider  $\| a\|_{L^\infty(\Omega)}\le \varepsilon_1$  small enough so that $A+a\ge \frac{\ld}{2}I_N$. By the density of $\mathcal{C}^{\infty}(\bar{\Omega})$ in $H^{s}(\Omega)$, we can choose sequences $\{a_n\}_{n\geq 1}$, $\{F_n\}_{n\geq 1}$ and $\{g_n\}_{n\geq 1}$ of smooth functions such that $a_n \to a$ in $H^s(\Omega)$, $F_n \to F$ in $H^{s-1}(\Omega)$ and $g_n \to g$ in $H^{s-\mez}(\p\Omega)$. Set $m_n= \frac{1}{|\p \Omega|}\left(\int_{\Omega}F_n + \int_{\p\Omega}g_n\right)$. 

For each $n$, let $\psi_n$ be the smooth solution to the Neumann  problem 
   \bq\label{eq:psin}
   \begin{cases}
        -\operatorname{div}((A+a_n)\nabla\psi)=F_n \quad \text{in ~} \Omega \\
        ((A+a_n)\na \psi)\cdot \nu = g_n-m_n \quad \text{on ~} \p\Omega.
\end{cases}
\eq
Clearly $\psi_n= \Psi(\tilde{F}_n, \tilde{g}_n)$ with 
\[
\tilde{F}_n=F_n+\operatorname{div}(a_n\nabla \psi_n),\quad \tilde{g}_n=g_n-m_n-a_n\nabla\psi_n\cdot \nu.
\]
Hence, the estimate \eqref{eq.Hs-elliptic-bound} for $\Psi$ gives
\[
\| \psi_n\|_{H^{s+1}(\Omega)}\le C(\lambda, \sup_{1\le i,j \le N}|A^{ij}|, s, \Omega)\left(\|\tilde{F}_n\|_{H^{s-1}(\Omega)}+\|\tilde{g}_n\|_{H^{s-\mez}(\p\Omega)}\right).
\]
Since $s>\frac{N}{2}$ and $\nu$ is smooth, invoking the trace inequality yields
\bq\label{est:psin:Hs}
\begin{aligned}
   \| \psi_n\|_{H^{s+1}(\Omega)} &\le  C\Big(m_n + \|g_n\|_{H^{s-\mez}(\p\Omega)} +\| F_n\|_{H^{s-1}(\Omega)}+ \| a_n\na \psi_n\|_{H^s(\Omega)} \Big) \\ 
    & \le C\Big(m_n + \|g_n\|_{H^{s-\mez}(\p\Omega)} +\| F_n\|_{H^{s-1}(\Omega)}+ \|a\|_{L^{\infty}(\Omega)}\|\na \psi_n\|_{H^s(\Omega)} + \|a\|_{H^s(\Omega)}\|\na \psi_n\|_{L^{\infty},(\Omega)}\Big)
\end{aligned}
\eq
where $C=C(\lambda, \sup_{1\le i,j \le N}|A^{ij}|, s, \Omega)$. Since $s>\frac{N}{2}$, for some $\delta >0$ small enough there holds 
\begin{align}
    \|\na \psi_n\|_{L^{\infty}(\Omega)} &\les  \|\na \psi_n\|_{H^{s-\delta}(\Omega)} \le \eta \|\na \psi_n\|_{H^s(\Omega)} + C_\eta\|\na \psi_n\|_{L^2(\Omega)},\label{eq.2-gn}
\end{align}
for all $\eta>0$, where $C_\eta$ depends only on $(\eta, \delta, s, d)$.  It follows from \eqref{est:psin:Hs} and \eqref{eq.2-gn} that if $\| a\|_{L^\infty(\Omega)}\le \eps_0:=\min(\eps_1, \frac{1}{2C})$ then 
\[
  \|\psi_n\|_{H^{s+1}(\Omega)} \le C\Big( m_n +\| F_n\|_{H^{s-1}(\Omega)}+ \|g_n\|_{H^{s-\mez}(\p\Omega)} +\|\na \psi_n\|_{L^2(\Omega)}\Big),
  \]
  where $C=C(\lambda, \sup_{1\le i,j \le N}|A^{ij}|, s, \Omega, \|a\|_{H^{s}(\Omega)})$ can be larger than the previous $C$ if needed. Combining this with the variational estimate
  \[
  \| \psi_n\|_{H^1(\Omega)}\le C(\Omega)(\|F_n\|_{L^2(\Omega)}+\| g_n-m_n\|_{H^{-\mez}(\p \Omega)}),   \]
  we obtain 
  \bq\label{uniform:Hspsin}
   \|\psi_n\|_{H^{s+1}(\Omega)} \le C\Big(\|F_n\|_{H^{s-1}(\Omega)}+ \|g_n\|_{H^{s-\mez}(\p\Omega)}\Big).
\eq
Thus the sequence $\{\psi_n\}$ is bounded in $H^{s+1}(\Omega)$. Since $\Omega$ is bounded, upon extracting a subsequence we can assume that $\psi_n$ converges to some  $\psi_*$  weakly in $H^{s+1}(\Omega)$ and strongly  in $H^s(\Omega)$. The strong  convergence in $H^s$ implies that $m_n\to 0$ and allows us to pass to the limit in the weak formulation of \eqref{eq:psin}, proving that $\psi_*$ is a weak solution of \eqref{eq.elliptic-eq}. By uniqueness of weak solutions, we conclude that $\psi=\psi_*\in H^{s+1}(\Omega)$. Finally, the estimate \eqref{est:elliptic:Aa} is obtained by passing to the limit in \eqref{uniform:Hspsin} and using the lower semicontinuity of the weak convergence in $H^{s+1}(\Omega)$.   
\end{proof}

\section{Bounds on the operators \texorpdfstring{$\mathcal{G}$}{G} and \texorpdfstring{$\mathcal{S}$}{S}}\label{sec.DN}

In this section we establish sharp Sobolev estimates for  the operators $\mathcal{S}[f]h$ and $\mathcal{G}[f]h$. The estimates are linear with respect to the highest order derivatives of the boundary function $f$ and the inputs $h$ and $k$. In the finite-depth case, we always assume that 
\bq\label{sepcond}
    \inf_{x\in \Rr^d}\left(f(x)-b(x)\right)\ge \fd>0,
\eq
where
\bq
b(x)=-H+b_0(x),\quad H>0.
\eq
This section is devoted to the following independent result.
\begin{theo}\label{prop.GSest}
Let $d\geqslant 1$ and $s,r \in \Rr$ such that $s>1+\frac{d}{2}$, $r\ge s-\frac{1}{2}  $ and $\lfloor r \rfloor \leq s$.
The following estimates hold provided their right-hand sides are finite:
\begin{equation}
    \label{eq.estimateG1}
    \|\cG[f]h\|_{H^r(\Omega_f)} \le C(\|f\|_{H^s})\left(\|h\|_{H^{r+\mez}}+\| f \|_{H^{r+\mez}}\| h\|_{H^s}\right),
\end{equation}
\begin{equation}
    \label{eq.estimateS1}
    \|\cS[f]k\|_{H^{r}(\Omega_f)} \le C(\|f\|_{H^s})\left(\|k\|_{H^r}+\| f \|_{H^{r+\mez}}\left(\| k\|_{H^{s-\mez}}+\| \na k\|_{L^\infty}\right)\right),
\end{equation}
where $C : \Rr_+ \to \Rr_+$ is a non-decreasing function, depending only on $(d,s,r)$ and in the finite-depth case also on $\mathfrak{d}$ and $\| b_0\|_{H^{s}}$. In the latter case $\| f \|_{H^{r+\mez}}$ should also be replaced with $\| f \|_{H^{r+\mez}} + \| b_0 \|_{H^{r+\mez}}$.
\end{theo}
Using \cref{prop.GSest} and the embedding $H^s(\Omega_f)\hookrightarrow W^{1, \infty}(\Omega_f)$ when $s>\tdm+\frac{d}{2}$, we obtain the following estimates for the fluid velocity.
\begin{coro}
    For $s>\tdm+\frac{d}{2}$ and $u$ given by \eqref{eq:u}, we have 
    \begin{align}\label{estu:Hs}
   & \| u\|_{H^s}\le C(\| f\|_{H^s})\left(\| f\|_{H^{s+\mez}}+(1+\| f\|_{H^{s+\mez}})\| g\|_{H^s}\right),\\ \label{estu:Hs-mez}
    &\| u\|_{H^{s-\mez}}\le C(\| f\|_{H^s})\left(\| f\|_{H^s}+\| g\|_{H^s} \right).
    \end{align}
\end{coro}

In order to prove \cref{prop.GSest}, we follow the classical strategy of straightening the free boundary.

\subsection{Straightening of the free boundary}
Let
\bq
J=\begin{cases} (-\infty, 0)&\text{ for infinite depth},\\
(-1, 0)&\text{ for finite depth}.
\end{cases}
\eq
We straighten  the fluid domain $\Omega_f$ using the change of variables
\bq\label{def:diffeo}
\ff_f: \Rr^d\times J\to\Omega_f,\quad \ff_f(x, z)=(x, \varrho(x, z)), 
\eq
where 
\bq\label{def:vr}
    \varrho(x, z)=
    \begin{cases}
        z+e^{\delta z\langle D_x\rangle}f(x)&\text{ for infinite depth},\\
        (z+1)e^{\delta  z\langle D_x\rangle}f(x)-ze^{-\delta (z+1)\langle D_x\rangle}b_0(x)+zH&\text{ for finite depth}
    \end{cases}
\eq
with $\delta>0$ satisfying the smallness condition \eqref{cd:delta:ff} in the following lemma. We note that $\rho(x, 0)=f(x)$. 
\begin{lemm}\label{lemm:varrho}
\begin{itemize}
\item[(i)] Let $s>1+\frac{d}{2}$, and $f,b_0 \in H^s(\Rr^d)$. There exists a constant $K>0$ depending only on $(d, s)$ such that for 
\bq\label{cd:delta:ff}
\delta \le 
\begin{cases}
(K\|f\|_{H^s}+1)^{-1}&\text{ for infinite depth},\\
\fd\big(K\|f\|_{H^s}+K\|b_0\|_{H^s}+1\big)^{-1}&\text{ for finite depth},
\end{cases}
\eq
 $\ff_f$ is a Lipschitz diffeomorphism, and 
\bq\label{lowerbound:dzvr}
\begin{cases}
\p_z\vr\ge \mez & \quad \text{ for infinite depth},\\
\p_z\varrho \geqslant \frac{\mathfrak{d}}{2} &\quad\text{ for finite depth}.
\end{cases}
\eq
\item[(ii)] For any $\sigma \in \Rr$, there exists $C=C(d, \sigma, \delta)>0$ such that
\begin{align}\label{navarrho:X}
\|(\p_z\varrho-1, \na_x\varrho)\|_{X^\sigma(J)}\le C\| f\|_{H^{\sigma+1}}&\quad\text{for infinite depth},\\ \label{navarrho:X:fd}
\|(\p_z\varrho-H, \na_x\varrho)\|_{X^\sigma(J)}\le C\| f\|_{H^{\sigma+1}}+C\| b_0\|_{H^{\sigma+1}}&\quad\text{for finite depth}.
\end{align}
For any $\sigma \in \Rr$ and $j\ge 2$, there exists $C=C(d, \sigma, j,  \delta)>0$ such that
\begin{align}\label{najvarrho:X}
\|\na_{x, z}^j\varrho \|_{X^\sigma(J)}\le C\| f\|_{H^{\sigma+j+1}}&\quad\text{for infinite depth},\\ \label{najvarrho:X:fd}
\|\na_{x, z}^j\varrho\|_{X^\sigma(J)}\le C\| f\|_{H^{\sigma+j+1}}+C\| b_0\|_{H^{\sigma+1}}&\quad\text{for finite depth}.
\end{align}
\item[(iii)] For any $\sigma \ge 0$, there exists $C=C(d, \sigma, \delta)>0$ such that
\begin{align}\label{varrho:H}
\|\varrho-z\|_{H^\sigma(\Rr^d\times J)}\le C\| f\|_{H^{\sigma-\mez}}&\quad\text{for infinite depth},\\ \label{varrho:H:fd}
\|\varrho-Hz\|_{H^\sigma(\Rr^d\times J)}\le C\| f\|_{H^{\sigma-\mez}}+C\| b_0\|_{H^{\sigma-\mez}}&\quad\text{for finite depth}.
\end{align}
\end{itemize}
\end{lemm}

\begin{proof}
(i) In the infinite-depth case, we have
\bq\label{dz:varrho}
    \partial_z \varrho(x,z)=1 + \delta e^{\delta z\langle D_x \rangle}\langle D_x \rangle f (x).
\eq
For some constant $C=C(d, s)>0$ independent of $\delta \in (0, 1)$ we have 
\[
\| e^{\delta z\langle D_x \rangle}\langle D_x \rangle f \|_{L^\infty(J; H^{s-1})}+\| e^{\delta z\langle D_x \rangle}\na f \|_{L^\infty(J; H^{s-1})}\le C\| f\|_{H^s}.
\]
Then the Sobolev embedding $H^{s-1}(\Rr^d)\subset L^\infty(\Rr^d)$ yields 
\[
\|e^{\delta z\langle D_x \rangle}\langle D_x \rangle f \|_{L^\infty(\Rr^d\times J)}\le \frac{K}{2} \| f\|_{H^s},
\]
 where $K=K(d, s)$. Therefore, if $\delta \le (K\|f\|_{H^s}+1)^{-1}$ then $\| \p_z\varrho-1\|_{L^\infty(\Rr^d\times J)}\le \mez$ and thus $\p_z\varrho \ge \mez$.

 In the finite depth case, we have
\bq\label{dz:varrho:fd}
\begin{aligned}
    \p_z\varrho(x, z)&=\delta(z+1)e^{\delta  z\langle D_x \rangle}\langle D_x \rangle f (x)+\delta ze^{-\delta (z+1)\langle D_x \rangle}\langle D_x \rangle b_0(x)\\
    &\qquad+e^{\delta z\langle D_x\rangle} f (x)-e^{-\delta (z+1)\langle D_x\rangle}b_0(x)+H\\
    &=\delta(z+1)e^{\delta  z\langle D_x \rangle}\langle D_x \rangle f (x) +\delta ze^{-\delta (z+1)\langle D_x\rangle}\langle D_x \rangle b_0(x) +\left(e^{\delta z\langle D_x \rangle} f (x)- f (x)\right) \\
    &\qquad -\left(e^{-\delta (z+1)\langle D_x \rangle}b_0(x)-b_0(x)\right)+f(x)-b_0(x)+H.
\end{aligned}
\eq
Using the mean-value theorem, we obtain 
\[
\begin{aligned}
    \p_z\varrho(x, z)&=\delta(z+1)e^{\delta  z\langle D_x\rangle}\langle D_x \rangle f (x)+\delta ze^{-\delta (z+1)\langle \nabla_x \rangle}\langle D_x \rangle b_0(x) +\delta z\int_0^1e^{ \delta \tau z\langle D_x\rangle}\langle D_x \rangle f (x)d\tau \\
    &\qquad+\delta(z+1)\int_0^1e^{-\delta \tau (z+1)\langle D_x \rangle}\langle D_x \rangle b_0(x)d\tau+f(x)-b(x).
\end{aligned}
\]
Since $f-b\ge \mathfrak{d}$, for $\delta \in (0, 1)$  it follows that 
\[
\p_z\varrho\ge \mathfrak{d}-\delta \frac{K}{2}\big(\| f\|_{H^s}+\| b_0\|_{H^s}\big), \quad K=K(d, s).
\]
Choosing $\delta\le \fd\big(K\|f\|_{H^s}+K\|b_0\|_{H^s}+1\big)^{-1}$, we obtain \eqref{lowerbound:dzvr}. It follows that $\ff_f$ is a Lipschitz local diffeomorphism. Moreover, \eqref{lowerbound:dzvr} shows that $\ff_f$ is one to one, therefore a global diffeomorphism on its image. The claim follows as $\rho(x,z) \underset{z \to -\infty}{\longrightarrow} -\infty$ in the infinite depth case, and $\rho(x,-1)=b(x)$ in the finite depth case.  

(ii) It suffices to use \eqref{dz:varrho}, the first equality in \eqref{dz:varrho:fd}, and the estimate 
\bq
\|e^{\delta z\langle D_x\rangle}h\|_{X^\sigma(J)}\le C(d, \sigma, \delta)\| h\|_{H^\sigma},\quad \sigma \in \Rr.
\eq

(iii) The estimates \eqref{varrho:H} and \eqref{varrho:H:fd} are a consequence of 
\bq\label{Poisson:Sobolev}
\|e^{\delta z\langle D_x\rangle}h\|_{H^\sigma(\Rr^d\times J)}\le C(d, \sigma, \delta)\| h\|_{H^{\sigma-\mez}(\Rr^d)},\quad \sigma \ge 0.
\eq
Clearly  \eqref{Poisson:Sobolev} holds for $\sigma=0$, and by interpolation it suffices to prove it  for  $\sigma\in \Nn$. For any $0\le j\le \sigma$, we have $\p_z^j\na_x^{\sigma-j}e^{\delta z\langle D_x\rangle}h=\delta^je^{\delta z\langle D_x\rangle}\langle D_x\rangle^j\na_x^{\sigma-j}h$, whence 
\[
    \| \p_z^j\na_x^{\sigma-j}e^{\delta z\langle D_x\rangle}h\|_{L^2(\Rr^d\times J)}\les \| \langle D_x\rangle^j\na_x^{\sigma-j}h\|_{H^{-\mez}}\les \| h\|_{H^{\sigma-\mez}}.\qedhere
\]
\end{proof}
We have the following chain rule 
\bq\label{chainrule:ff}
(\na_{x, y} g)\circ \ff_f=\na_{x, z}(g\circ \ff_f)(\na \ff_f)^{-1}= \na_{x, z}(g\circ \ff_f)\begin{bmatrix} 
I_d &  0_{d\times 1}\\
-\frac{\na_x\varrho}{\p_z\varrho} &\frac{1}{\p_z\varrho}
\end{bmatrix},
\eq
where $\na_x\varrho$ is a row matrix. Consequently 
\[
\p_z\varrho(x, z)(\Delta_{x, y}g)\circ \ff_f(x, z)=\di_{x, z} (\cA \na v)(x, z)
\]
with 
\bq
\cA=\begin{bmatrix} 
\p_z\varrho I_d & -\na_x\varrho^T\\
-\na_x\varrho & \frac{1+|\na_x\varrho|^2|}{\p_z\varrho}
\end{bmatrix}.
\eq
Let $\phi$ be either the solution $\phi^{(1)}$ to \eqref{eq.phi1} or $\phi^{(2)}$ to  \eqref{eq.phi2}, and set 
\bq\label{def:v:phij}
v(x, z)=(\phi\circ \ff_f)(x, z).
\eq
 Then $v$ satisfies 
\bq\label{eqphi:div}
\di_{x, z} (\cA \na v)=\begin{cases}
0\quad&\text{ for } \eqref{eq.phi1},\\
-\p_z\varrho(\p_yk)\circ \ff_f\quad &\text{ for } \eqref{eq.phi2}.
\end{cases}
\eq
For sufficiently smooth solutions, we can expand \eqref{eqphi:div} and deduce  that $v$ satisfies 
\begin{equation}
\begin{cases}\label{eq:v}
        (\partial_z^2 + \alpha \Delta_x +\beta \cdot \nabla \partial_z - \gamma \partial_z)v  = F_0,\quad (x,z)\in  \mathbb{R}^d\times I, \\
        v(x, 0)  =  \zeta(x),\quad x\in\mathbb{R}^d\,, 
\end{cases}
\end{equation}
where 
\[
    \alpha = \frac{(\partial_z \varrho)^2}{1+|\nabla _x \varrho|^2}, \quad \beta = -2 \frac{\partial_z \varrho \nabla_x \varrho}{1+|\nabla _x \varrho|^2},\quad 
    \gamma = \frac{1}{\partial_z \varrho} \left(\partial_z^2 \varrho + \alpha \Delta_x \varrho + \beta\cdot\nabla_x\partial_z\varrho\right).
\]
and
\bq
\label{eq.datum}
(\zeta, F_0)=
\begin{cases}
(h, 0)&\text{ for } \eqref{eq.phi1},\\
\left(0, -\frac{(\partial_z \varrho)^2}{1+|\nabla _x \varrho|^2}(\p_y k)\circ \ff_ f\right)&\text{ for } \eqref{eq.phi2}.
\end{cases}
\eq
Estimates for the coefficients $(\alpha, \beta, \gamma)$ in $X^\sigma$ are given in the next lemma.
\begin{lemm}\label{lemm.coefsEstimInfinite} Let $s>\frac{d}{2} + 1$ and $\sigma \geqslant 0$. In the infinite-depth case there holds
\begin{align}
    \label{eq.coefsEstimInfinite}
    &\|\alpha - 1\|_{X^{\sigma}(J)} + \|\beta\|_{X^{\sigma}(J)} \le C(\| f \|_{H^{s}})\|  f \|_{H^{\sigma+1}},\\
    \label{eq.coefsEstimInfinite-gamma}
    &\|\gamma\|_{X^{\sigma-\mez}(J)} \le C(\| f \|_{H^{s}})\|  f \|_{H^{\sigma+\tdm}},
\end{align}
and in the finite-depth case there holds
\begin{align}
    \label{eq.coefsEstimFinite}
   & \|\alpha - H^2\|_{X^{\sigma}(J)} + \|\beta\|_{X^{\sigma}(J)} \le C(\| f \|_{H^{s}})\left(\|  f \|_{H^{\sigma+1}}+\| b_0\|_{H^{\sigma+1}}\right),\\
    \label{eq.coefsEstimFinite-gamma}
   & \|\gamma\|_{X^{\sigma -\mez}(J)} \le C(\| f \|_{H^{s}})\left(\|  f \|_{H^{\sigma+\tdm}}+\| b_0\|_{H^{\sigma+\tdm}}\right),
\end{align}
where $C : \Rr_+ \to \Rr_+$ is a non-decreasing function, only depending on $(d,s,\sigma)$, and also $\mathfrak{d}$ and $\|b_0\|_{H^{s}}$ in the finite-depth case. 
\end{lemm}
\begin{proof}
We will only provide the proof for the infinite-depth case. We have $\alpha - 1 = G(\tilde{\na}\varrho)$, where  $\tilde{\na}\varrho = (\na_x \varrho, \p_z\varrho -1)$ and  $G$ is some smooth function satisfying $G(0)=0$. Then \cref{est:nonl} implies
\begin{equation}
    \label{eq.coefs-estim1}
    \|\alpha - 1\|_{X^{\sigma}(J)} \le C(\|\tilde{\na} \varrho\|_{L^{\infty}(J\times \Rr^d)})\|\tilde{\na}\varrho\|_{X^{\sigma}(J)},
\end{equation}
and it suffices to use \eqref{navarrho:X} and the Sobolev embedding $H^{s-1}(\Rr^d)\hookrightarrow L^\infty(\Rr^d)$ for  $s > 1+ \frac{d}{2}$. The estimate for $\beta$ can be proven along the same lines. 

As for $\gamma$, we use  \eqref{boundpara} and \eqref{Bony1} to have
\begin{align*}
    \| \gamma \|_{X^{\sigma-\mez}(J)}\les &\Big(\| \frac{\alpha}{\p_z\varrho}-1\|_{L^\infty(J; H^{s-1})}+1+\|\frac{\beta}{\p_z\varrho}\|_{L^\infty(J; H^{s-1})}\Big)\| \na_{x, z}^2\varrho\|_{X^{\sigma-\mez}(J)} \\ 
    &+ \Big(\| \frac{\alpha}{\p_z\varrho}-1\|_{X^{\sigma + \mez}(J)}+\| \frac{\beta}{\p_z\varrho}\|_{X^{\sigma + \mez}(J)}\Big)\| \na_{x, z}^2\varrho\|_{L^{\infty}(J;H^{s-2})},
\end{align*}
where we have used that $(s-1)+(\sigma-\mez)\ge s-\tdm>0$. Then we conclude using the estimates 
\[
\| \frac{\alpha}{\p_z\varrho}-1\|_{L^\infty(J; H^{s-1})}+\| \frac{\beta}{\p_z\varrho}\|_{L^\infty(J; H^{s-1})}\le C(\| f\|_{H^s}), \quad \| \na_{x, z}^2\varrho\|_{X^{\sigma-\mez}(J)}\les \| f\|_{H^{\sigma+\tdm}},
\]
\[
\| \frac{\alpha}{\p_z\varrho}-1\|_{X^{\sigma + \mez}(J)}+\| \frac{\beta}{\p_z\varrho}\|_{X^{\sigma + \mez}(J)}\le C(\| f\|_{H^s})\|f\|_{H^{\sigma +\tdm}},
\]
which follow from\cref{est:nonl} and  \cref{lemm:varrho} (ii). 
\end{proof}
Next, we provide estimates for $z$-derivatives of   $(\alpha, \beta, \gamma)$.
\begin{lemm}\label{lem.lowSobolEstCoeffs}  Let $s>1+\frac{d}{2}$. 

(i) Consider  $j\in \Nn$ and $\sigma \ge 0$. In the infinite-depth case, we have 
\begin{equation}
  \label{eq.pzEstim1b}
     \|\partial_z^j(\alpha-1)\|_{ H^{\sigma}(\Rr^d\times J)}+  \|\partial_z^j\beta\|_{ H^{\sigma }(\Rr^d\times J)} \le C(\|f\|_{H^s})\|f\|_{H^{\sigma +j+\frac{1}{2}}},
    \end{equation}
    \begin{equation}
    \label{eq.pzEstim1c}
      \|\partial_z^j\gamma\|_{ H^{\sigma}(\Rr^d\times J)} \le C(\|f\|_{H^s})\|f\|_{H^{\sigma+j +\frac{3}{2}}},
\end{equation}
and if moreover $j\le s$ then 
\begin{equation}
   \label{eq.pzEstim1}
    \|\partial_z^j(\alpha-1)\|_{L^{\infty}(J,H^{\sigma})} + \|\partial_z^j\beta\|_{L^{\infty}(J, H^{\sigma}) } \le C(\|f\|_{H^s})\|f\|_{H^{\sigma +j+1}},
\end{equation}
Here  $C : \Rr_+ \to \Rr_+$ is non-decreasing and depends only on $(d,s,\sigma,  j)$. In the finite-depth case, $C$ also depends on $\mathfrak{d}$ and $\|b\|_{H^{s}}$, $\alpha - 1$ should be replaced with $\alpha - H^2$, and the $\|f\|_{H^{\sigma +\nu}}$ norms on the right-hand sides should be replaced with $\|f\|_{H^{\sigma +\nu}}+\|b_0\|_{H^{\sigma +\nu}}$.

(ii) If $j\in \Nn$ and $\sigma\in \Rr$ satisfy 
\bq\label{dzgamma:L2-cond}
 \sigma+1\ge 0, \quad \sigma+s-1>0,  \quad\text{and~}  j\leq s-1, 
\eq
 then for $q\in\{2, \infty\}$ we have 
\bq\label{dzgamma:L2}
\| \p_z^j\gamma\|_{L^q(J;  H^\sigma)}\le C(\| f\|_{H^s})(\| f\|_{H^{\sigma+j+2-\frac{1}{q}}}+a), 
\eq
where $a=0$ in the infinite-depth case and $a=\| b_0\|_{H^{\sigma+j+2-\frac{1}{q}}}$ in the finite-depth case. 
\end{lemm}

\begin{rema}
Estimates \eqref{eq.coefsEstimInfinite-gamma} and \eqref{eq.coefsEstimFinite-gamma} are special cases of \eqref{dzgamma:L2-cond}.
\end{rema}

\begin{proof} We will only consider  the infinite-depth case. 

\quad (i)  Since $\alpha-1$ and $\beta$ have the same regularity, we will only provide the proof for $\alpha-1$. 

\quad 1. We write as in the proof of \cref{lemm.coefsEstimInfinite} that $\alpha - 1 = G(\tilde{\na}\varrho)$, where   $\tilde{\na}\varrho = (\na_x\varrho, \p_z\varrho - 1)$ and $G$ is a smooth function satisfying $G(0)=0$.  Since $\sigma\ge 0$,  \cref{est:nonl} implies   
\[
    \|\partial_z^j(\alpha-1)\|_{H_{x,z}^{\sigma }} \le \|G(\tilde{\na}\varrho)\|_{H_{x,z}^{\sigma+j}} \le C(\|\tilde{\na}\varrho\|_{L^{\infty}_{x,z}})\|\tilde{\na}\varrho\|_{H_{x,z}^{\sigma+j}}.    
\]
Then we use  \eqref{navarrho:X} together with the Sobolev embedding $H^{s-1}(\Rr^d)\hookrightarrow L^\infty(\Rr^d)$,  and \eqref{varrho:H} to deduce 
\[
\|\tilde{\na}\varrho \|_{L^{\infty}_{x,z}}\les \| f\|_{H^s},\quad \|\tilde{\na}\varrho\|_{H_{x,z}^{\sigma}}\les \| f\|_{H^{\sigma+j+\mez}}.
\] 
It follows that 
 \[
  \|\partial_z^j(\alpha-1)\|_{H_{x,z}^{\sigma }}\le C(\| f\|_{H^s}) \| f\|_{H^{\sigma+j+\mez}}
  \]
as claimed in \eqref{eq.pzEstim1b}.

\quad 2.   We turn to the proof of \eqref{eq.pzEstim1c}. From the definition of $\gamma$, it suffices to estimate the typical term $\p_z^j\big((\frac{\alpha}{\p_z\varrho}-1)\Delta_x\varrho\big)$. The product rule in \cref{thm.sobolevProducts}  gives
\begin{align*}
\| \p_z^j\big(\frac{\alpha}{\p_z\varrho}-1)\Delta_x\varrho\big)\|_{H^{\sigma}_{x, z}}&\le \| (\frac{\alpha}{\p_z\varrho}-1)\Delta_x\varrho\|_{H^{\sigma+j}_{x, z}}\\
&\les \| (\frac{\alpha}{\p_z\varrho}-1)\|_{L^\infty}\|\Delta_x\varrho\|_{H^{\sigma+j}_{x, z}}+ \| \Delta_x\varrho\|_{L^p}\|\frac{\alpha}{\p_z\varrho}-1\|_{W^{\sigma+j, p'}_{x, z}},
\end{align*}
where $\frac{1}{p}+\frac{1}{p'}=\mez$ and $p'<\infty$. Combining the estimates \eqref{navarrho:X}, \eqref{eq.coefsEstimInfinite}, and \eqref{varrho:H}, we deduce 
\[
\| (\frac{\alpha}{\p_z\varrho}-1)\|_{L^\infty}\le  \| (\frac{\alpha}{\p_z\varrho}-1)\|_{L^\infty_z H^{s-1}}\le C(\| f\|_{H^s}),\quad \|\Delta_x\varrho\|_{H^{\sigma+j}_{x, z}}\les \| f\|_{H^{\sigma+j+\tdm}}.
\]
We choose 
\[
\begin{cases}
\frac{1}{p}=\mez-\frac{s-\tdm}{d+1} \quad\text{if~} s-\tdm<\frac{d+1}{2},\\
p=\infty \quad\text{if~} s-\tdm>\frac{d+1}{2},\\
2\ll p<\infty \quad\text{if~}s-\tdm=\frac{d+1}{2}.
\end{cases}
\]
Then, since $s>1+\frac{d}{2}\ge \tdm>0$, we have $H^{s-\tdm}_{x, z}\hookrightarrow L^p_{x, z}$,  $H^{\sigma+j+\frac{d+1}{p}}_{x, z}\hookrightarrow W^{\sigma+j, p'}_{x, z}$, and $\frac{d+1}{p}<1$ in all cases. Hence, invoking \eqref{varrho:H} and \eqref{eq.coefsEstimInfinite} yields 
\[
 \| \Delta_x\varrho\|_{L^p}\|\frac{\alpha}{\p_z\varrho}-1\|_{W^{\sigma+j, p'}_{x, z}} \les \| \Delta_x\varrho\|_{H^{s-\tdm}_{x, z}}\|\frac{\alpha}{\p_z\varrho}-1\|_{H^{\sigma+j+1}_{x, z}}\le C(\| f\|_{H^s})\| f\|_{H^{\sigma+j+\tdm}}.
\]
This completes the proof of  \eqref{eq.pzEstim1c}.

\quad 3. Assuming now that $j\le s$, we  prove \eqref{eq.pzEstim1}. We prove by induction on $j\geq 0$ that for all smooth functions $G$ satisfying $G(0)=0$ and $\sigma\ge 0$  there holds 
\[
    \|\p_z^jG(\na \varrho)\|_{L^{\infty}_zH^{\sigma }} \le C(\| f\|_{H^s}) \|f\|_{H^{\sigma+j +1}}. 
\]
The base case $j=0$ follows from \cref{est:nonl} (i) and \eqref{navarrho:X}. Assume the assertion holds up to $j-1$, $j\geq 1$. We have
\[
    \p_z^jG(\na \varrho) = \p_z^{j-1}\left(\p_z\na \varrho  G'(\na \varrho)\right) =G'(0) \p_z^j\na \varrho + \sum_{k=0}^{j-1}c_{k,j}\p_z^{k+1}\na \varrho \p_z^{j-1-k} (G'(\na \varrho)-G'(0)),    
\]
Let us write $H=G' - G'(0)$, which is a smooth function satisfying $H(0)=0$. Since $j\ge 1$, \eqref{najvarrho:X} implies 
\[
\| \p_z^j\na \varrho \|_{L^\infty_z H^\sigma}\les \| f\|_{H^{\sigma+j+1}}.
\]
 To estimate $\p_z^{k+1}\na \varrho \p_z^{j-1-k}H(\na \varrho)$ for $k\in\{0,\dots, j-1\}$, we use \eqref{boundpara} and \eqref{Bony1}:
\begin{multline*}
    \|\p_z^{k+1}\na \varrho \p_z^{j-1-k} H(\na \varrho)\|_{L^{\infty}_zH^{\sigma }} \les \|\p_z^{k+1}\na \varrho\|_{L^{\infty}_zH^{s-k-2}} \|\p_z^{j-1-k} H(\na \varrho)\|_{L^{\infty}_zH^{\sigma +k+1}} \\
    + \|\p_z^{k+1}\na \varrho\|_{L^{\infty}_zH^{\sigma+j-1-k}} \|\p_z^{j-1-k} H(\na \varrho)\|_{L^{\infty}_zH^{s -j +k}}.
\end{multline*}
This application is justified because 
\[
    \sigma  \leq \sigma  +k+1, \quad \sigma  \leq \sigma+j-1-k, \quad \text{and } s>1+\frac{d}{2}. 
\]
By  \eqref{najvarrho:X},
\[
\|\p_z^{k+1}\na \varrho\|_{L^{\infty}_zH^{s-k-2}}\les \| f\|_{H^s},\quad \|\p_z^{k+1}\na \varrho\|_{L^{\infty}_zH^{\sigma+j-1-k}}\les \| f\|_{H^{\sigma+j+1}}.
\]
Since $j-1-k\leq j-1$, $\sigma +k+1\geq 0$, and $s-j+k\geq s-j\ge 0$,  the induction hypothesis implies 
\[
\|\p_z^{j-1-k} H(\na \varrho)\|_{L^{\infty}_zH^{\sigma +k+1}}\le C(\| f\|_{H^s})\| f\|_{H^{\sigma+j+1}},\quad  \|\p_z^{j-1-k} H(\na \varrho)\|_{L^{\infty}_zH^{s -j +k}} \le C(\| f\|_{H^s}).
\]
Combining the above estimates yields \eqref{eq.pzEstim1}.

\quad (ii)  For the proof of \eqref{dzgamma:L2}, from the definition of $\gamma$ we have
\begin{multline}\label{eq.splitpj-der}
    \|\p^j \gamma\|_{L^q(J;H^{\sigma})} \leq \|\p^j_z\na_{x,z}^2 \varrho\|_{L^q(J;H^{\sigma})} + \left\|\p_z^j\left((\frac{\alpha}{\p_z\varrho}-1) \Delta_x \varrho\right)\right\|_{L^q(J;H^{\sigma})} \\
    + \left\|\p_z^j\left(\frac{\beta}{\p_z\varrho}\cdot \nabla_x\p_z \varrho\right)\right\|_{L^q(J;H^{\sigma})} + \left\|\p_z^j\left((\frac{1}{\p_z \varrho}-1)\p_z^2 \varrho\right)\right\|_{L^q(J;H^{\sigma})}.
\end{multline}
Consider $q\in\{2, \infty\}$. It follows from \cref{lemm:varrho}  (ii) that   $\|\p_z^j\na_{x,z}^2 \varrho\|_{L^q(J;H^{\sigma})} \lesssim \|f\|_{H^{\sigma + j + 2-\frac{1}{q}}}$. Since the last three terms on the right-hand side of \eqref{eq.splitpj-der} can be estimated along the same lines, we will provide the control of the $\alpha$-term
\begin{equation*}
    \left\|\p_z^{k}(\frac{\alpha}{\p_z \varrho} -1) \p_z^{j-k}\Delta_x\varrho\right\|_{L^q(J;H^{\sigma})}, \quad k\in \{0, \dots, j\}. 
\end{equation*}
To this end, we use Bony's decomposition together with \eqref{boundpara} and \eqref{Bony1}: 
\[
\begin{aligned}
\|\p_z^{k}(\frac{\alpha}{\p_z \varrho} -1) \p_z^{j-k}\Delta_x\varrho\|_{L^q_zH^{\sigma}}&\les \|\p_z^{k}(\frac{\alpha}{\p_z \varrho} -1) \|_{L^\infty_zH^{-k+s-1}}\|  \p_z^{j-k}\Delta_x\varrho\|_{L^q _z H^{\sigma+k}}\\
&\qquad+ \|\p_z^{k}(\frac{\alpha}{\p_z \varrho} -1) \|_{L^q_zH^{\sigma+j-k+1}}\|  \p_z^{j-k}\Delta_x\varrho\|_{L^\infty _z H^{k-j+s-2}}
\end{aligned}
\]
The preceding applications are justified because 
\begin{align*}
&\sigma\le \sigma+k,\quad \sigma\le \sigma+j-k+1,~\text{and~} (-k+s-1)+(\sigma+k)=\sigma+s-1>0
\end{align*}
by means of \eqref{dzgamma:L2-cond}. Since $-k+s-1\ge s-1-j\ge 0$ and $\sigma+j-k+1\ge \sigma +1\ge 0$ by \eqref{dzgamma:L2-cond}, it follows from \eqref{eq.pzEstim1} that 
\[
\|\p_z^{k}(\frac{\alpha}{\p_z \varrho} -1) \|_{L^\infty_zH^{-k+s-1}}\le C(\| f\|_{H^s}),\quad \|\p_z^{k}(\frac{\alpha}{\p_z \varrho} -1) \|_{L^q_zH^{\sigma+j-k+1}}\le C(\| f\|_{H^s})\| f\|_{H^{\sigma+j+2-\frac{1}{q}}}.
\]
On the other hand, by \cref{lemm:varrho} (ii),
\[
\|  \p_z^{j-k}\Delta_x\varrho\|_{L^q_z H^{\sigma+k}}\les \| f\|_{H^{\sigma+j+2-\frac{1}{q}}},\quad \|  \p_z^{j-k}\Delta_x\varrho\|_{L^\infty _z H^{k-j+s-2}}\les \| f\|_{H^s}.
\]
Putting the above estimates together we conclude the proof of \eqref{dzgamma:L2}. 
\end{proof}
Finally, the forcing term $F_0$ in \eqref{eq:v} is controlled by means of  following lemma.
\begin{lemm}\label{lemm.F0est} Let $s>1+\frac{d}{2}$ and  $\sigma \ge 1$. Then we have 
\begin{equation}
    \label{eq.F0estBis}
    \|F_0\|_{H^{\sigma-1}(\Rr^d\times J)} \le C(\|f\|_{H^s}) \left(\|f\|_{H^{\sigma + \frac{1}{2}}}(\|k\|_{H^{s-\mez}}+\|\nabla k\|_{L^{\infty}}) + \|k\|_{H^{\sigma}}\right),
\end{equation}
where $C : \Rr_+\to\Rr_+$ is non-decreasing and depends only on $(d,s,\sigma)$ in the infinite depth and also depends on $\mathfrak{d}$ and $\|b_0\|_{H^{s}}$. Also, occurences of $\|f\|_{H^{a}}$ should be replaced with $\|f\|_{H^a} + \|b_0\|_{H^a}$.
\end{lemm}
\begin{proof}
We only consider the infinite-depth case. We have   $F_0=G(\tilde{\nabla} \varrho)\p_z(k\circ \ff_f) + \p_z(k\circ \ff_{f})$, where $G$ is a smooth function of $\tilde{\na} \varrho$ with $G(0)=0$.  \cref{lemm.FaaDiBruno} implies
\bq\label{kfff:Hsigma}
\begin{aligned}
     \|\p_z(k \circ \ff_{f})\|_{H^{\sigma-1}}\le  \|k \circ \ff_{f}\|_{H^{\sigma}} &\le C(\|f\|_{H^s})\left(\| k\|_{H^{\sigma}}+\|\na k\|_{L^\infty}\| \ff_f-\operatorname{id}\|_{H^{\sigma}}\right) \\
    &\le C(\|f\|_{H^s})\left(\| k\|_{H^{\sigma}}+\| \na k\|_{L^\infty}\|f\|_{H^{\sigma -\mez}}\right).
\end{aligned}
\eq
To control the product $G(\tilde{\nabla} \varrho)\p_z(k\circ \ff_f)$, we shall appeal to the following product rule 
\bq\label{pr:200}
\| uv\|_{H^{\sigma-1}_{x, z}}\les \| u\|_{H^\sigma_{x, z}}\| v\|_{H^{s-\tdm}_{x, z}}+ \| u\|_{H^{s-\mez}_{x, z}}\| v\|_{H^{\sigma-1}_{x, z}}
\eq
for $\sigma \ge 1$ and $s>1+\frac{d}{2}$. In the whole space $\Rr^d\times \Rr$, \eqref{pr:200} is a consequence of \eqref{boundpara} and \eqref{Bony1}. For the domain $\Rr^d\times J$, \eqref{pr:200} follows by applying the extension \cref{th.extension}. It follows from \eqref{pr:200} that
\begin{align*}
\| G(\tilde{\nabla} \varrho)\p_z(k\circ \ff_f)\|_{H^{\sigma-1}}&\les  \| G(\tilde{\nabla} \varrho)\|_{H^\sigma}\| \p_z(k\circ \ff_f)\|_{H^{s-\tdm}}+ \| G(\tilde{\nabla} \varrho)\|_{H^{s-\mez}}\| \p_z(k\circ \ff_f)\|_{H^{\sigma-1}}\\
&\les  \| G(\tilde{\nabla} \varrho)\|_{H^\sigma}\| k\circ \ff_f\|_{H^{s-\mez}}+ \| G(\tilde{\nabla} \varrho)\|_{H^{s-\mez}}\| k\circ \ff_f\|_{H^{\sigma}}.
\end{align*}
Then we conclude by using \cref{est:nonl} (i), \cref{lemm.FaaDiBruno}, and \cref{lemm:varrho} (iii). 
\end{proof}

\subsection{Estimates for \texorpdfstring{$\na_{x, z}v$}{nabla v} in \texorpdfstring{$X^\sigma$}{Xsigma}}
We first recall the following result from \cite{ABZ}.
\begin{prop}\cite[Proposition 3.16]{ABZ}\label{prop.ABZ} Let $s > 1+\frac{d}{2}$ and $\sigma \in [-\mez, s-1]$. Fix any $ z_1<z_0$ in $J$. If $f\in H^{\sigma+1}(\Rr^d)$ and $F_0\in Y^\sigma(I)$, then $\na_{x, z}v\in X^\sigma(I)$ and there exists a non-decreasing function $C : \Rr_+ \to \Rr_+$ which depends on $(d,s,\sigma)$ and $z_0-z_1$ such that
\begin{equation}
    \label{eq.ABZ}
    \|\nabla_{x,z}v\|_{X^{\sigma}([z_0, 0])}\le  C(\| f \|_{H^s}) \left(\|\na \zeta\|_{H^{\sigma}}+ \|F_0\|_{Y^{\sigma}([z_1, 0])}+\|\nabla_{x,z}v\|_{X^{-\mez}([z_1, 0])}\right).
\end{equation}
In the infinite depth case, there holds 
\begin{equation}\label{ABZ:infdepthABZ}
    \|\nabla_{x,z}v\|_{X^{\sigma}(J)} \le  C(\| f \|_{H^s}) \left(\|\na\zeta\|_{H^{\sigma }}+ \|F_0\|_{Y^{\sigma}(J)}+\|\nabla_{x,z}v\|_{X^{-\mez}(J)}\right).
\end{equation}
\end{prop} 
\begin{rema} The stated estimate \eqref{eq.ABZ} is slight stronger than the one in \cite[Proposition 3.16]{ABZ}  in that it only requires  $\|\na \zeta\|_{H^{\sigma}}$  in place of  $\|\zeta\|_{H^{\sigma+1}}$ and that  the constant $C$ only depends on $z_0-z_1$ in place of $z_0$ and $z_1$. The first improvement was explained in \cite[Remark 3.13]{NP}, while the second improvement follows from \cref{rema:parabolicI}. Thanks to the second improvement, \eqref{ABZ:infdepthABZ} is deduced from \eqref{eq.ABZ} by choosing $z_0=-n$, $z_1=-n-1$ and letting $n\to \infty$.
\end{rema}

The  estimates \eqref{eq.estimateG1} and \eqref{eq.estimateS1} in \cref{prop.GSest} requires in particular that $\na_{x, z} v\in X^{\sigma}$ for  $\sigma \ge s-1$. To this end, we first generalize \cref{prop.ABZ} to $\sigma \ge s-1$. 

\begin{prop}\label{prop.ABZvariant} Let $s >1+ \frac{d}{2}$ and $\sigma \geqslant -\mez$. Fix any $ z_1<z_0$ in $J$. If $f\in H^{\sigma+1}(\Rr^d)$ and $F_0\in Y^\sigma(I)$, then $\na_{x, z}v\in X^\sigma(I)$ and there exists a non-decreasing function $C : \Rr_+ \to \Rr_+$ which depends on $(d,s,\sigma)$ and $z_0-z_1$ such that 
\begin{multline}
    \label{eq.ABZvariant}
    \|\nabla_{x,z}v\|_{X^{\sigma}([z_0, 0])}\le  C(\| f \|_{H^s}) \Big( \|\na\zeta\|_{H^{\sigma}}+\| f \|_{H^{\sigma +1}}\|\na \zeta\|_{H^{s-1}}  \\
    +\|F_0\|_{Y^{\sigma}([z_1, 0])}+\|\nabla_{x,z}v\|_{X^{-\mez}([z_1, 0])}\Big).
\end{multline}
In the infinite depth case, there holds 
\bq\label{estX:infdepth}
    \|\nabla_{x,z}v\|_{X^{\sigma}(J)} \le  C(\| f \|_{H^s}) \left(\|\na\zeta\|_{H^{\sigma}}+ \|f\|_{H^{\sigma +1}}\|\na \zeta\|_{H^{s-1}}+ \|F_0\|_{Y^{\sigma}(J)}+\|\nabla_{x,z}v\|_{X^{-\mez}(J)}\right).
\eq
In the finite-depth case,  $C$ also depends on $\fd$ and $\|b_0\|_{H^{s}}$, and  $\| f \|_{H^{\sigma +1}}$ is replaced with  $\| f \|_{H^{\sigma +1}} + \| b_0 \|_{H^{\sigma +1}}$.
\end{prop}
We remark  that \eqref{eq.ABZvariant} only provides estimates away from the bottom, while \eqref{estX:infdepth} is a global-in-depth estimate in the infinite depth case. 
The estimates  \eqref{eq.ABZvariant}  and \eqref{estX:infdepth} are linear with respect to the highest order norm $\| f\|_{H^{\sigma+1}}$ of $f$.   On the other hand, \cref{prop.ABZ} implies that for $\sigma>1+\frac{d}{2}$,  
\bq
   \|\nabla_{x,z}v\|_{X^{\sigma}([z_0, 0])} \le  C(\| f \|_{H^{\sigma+1}}) \left(\|\zeta\|_{H^{\sigma +1}} + \|F_0\|_{Y^{\sigma}([z_1, 0])}+\|\nabla_{x,z}v\|_{X^{-\mez}([z_1, 0])}\right),
\eq
which is nonlinear with respect to $\| f\|_{H^{\sigma+1}}$, therefore would \textit{not} yield a closed a priori estimate for $f \in L^{\infty}_TH^s \cap L^2_TH^{s+\mez}$. 

The proof of \cref{prop.ABZvariant}  is given in the   \cref{sec:ABZvariant}. The idea is to use an induction argument on $\sigma$, where the base case $\sigma=s-1$ is provided by \cref{prop.ABZ}. 

Using \cref{prop.ABZvariant} we deduce the following tame estimates for quantities defined as traces on the free boundary $\Sigma_f$. In particular, the estimate \eqref{est:DN:h} for  the Dirichlet-Neumann operator $G[f]h$ generalizes the known estimate  \eqref{est:DN:ABZ000}. 
\begin{coro}
Let $s>1+\frac{d}{2}$ and $\sigma\ge  s-1$. We have 
\begin{align}\label{esttrace:cG}
&\| \cG[f]h\vert_{\Sigma_f}\|_{H^\sigma}\le C(\|f\|_{H^s})\Big(\| h\|_{H^{\sigma+1}}+(\| f\|_{H^{\sigma+1}}+a)\| h\|_{H^s}\Big),\\ \label{est:DN:h}
&\| G[f]h\|_{H^\sigma}\le C(\|f\|_{H^s})\Big(\| h\|_{H^{\sigma+1}}+(\| f\|_{H^{\sigma+1}}+a)\| h\|_{H^s}\Big),
\end{align}
and
\bq\label{esttrace:cS}
\| \cS[f]k\vert_{\Sigma_f}\|_{H^\sigma}\le C(\| f\|_{H^s}) \left(\| k\circ \ff_f\|_{H^{\sigma+\mez}}+(\| f\|_{H^{\sigma+1}}+a)\| k\circ \ff_f\|_{H^{s}} \right),
\eq
where $a=0$ for infinite depth and $a=\| b_0\|_{H^{\sigma+1}}$ for finite depth. 
\end{coro}
\begin{proof}  1.  For the proof of \eqref{esttrace:cG} we first apply the estimate \eqref{eq.ABZvariant}  with $v=\phi^{(1)}\circ \ff_f$, $\zeta=h$,  and $F_0=0$ and invoke \eqref{vari:phi1} to obtain
\bq\label{tracenav(1)}
\| \na_{x, z}v\vert_{z=0}\|_{H^\sigma}\le C(\|f\|_{H^s})\left(\| h\|_{H^{\sigma+1}}+\| f\|_{H^{\sigma+1}}\| h\|_{H^s}\right).
\eq
On the other hand, we have
\[
 \cG[f]h\vert_{\Sigma_f}=\na_{x, y}\phi^{(1)}\vert_{\Sigma_f}=\na_{x, z}v\vert_{z=0}(\na \ff_f)^{-1}\vert_{z=0}=\p_z\varrho\vert_{z=0}\na_{x, z}v\vert_{z=0}A , 
\]
where $(\na \ff_f)^{-1}$ is given by \eqref{chainrule:ff}.  Since $\sigma \ge s-1$, $H^\sigma(\Rr^d)$ is an algebra. Combining  \cref{est:nonl} (i) with the estimates \eqref{navarrho:X} and \eqref{navarrho:X:fd}, we find
\bq\label{est:naff}
\| (\na \ff_f)^{-1}-\tilde{I}_{d+1}\|_{X^\sigma(J)}\le C(\| f\|_{H^s})(\| f\|_{H^{\sigma+1}}+a),
\eq
where $\tilde{I}_{d+1}=I_{d+1}$ in the infinite-depth case and 
\[
\tilde{I}_{d+1}=\begin{bmatrix}
I_{d} & 0_{d\times 1}\\
0_{1\times d} & \frac{1}{H}
\end{bmatrix}
\]
in the finite-depth case.  Then, using the tame product estimate \eqref{tamepr} together with the embedding $H^{s-1}(\Rr^d)\subset L^\infty(\Rr^d)$,  we deduce from the above estimates that
\bq\label{esttrace:cG:10}
\begin{aligned}
\| \cG[f]h\vert_{\Sigma_f}\|_{H_x^\sigma}&\les \left(\| (\na \ff_f)^{-1}\vert_{z=0}-\tilde{I}_{d+1}\|_{H^{s-1}_x}+1\right)\|\na_{x, z}v\vert_{z=0}\|_{H^\sigma_x}\\
&\qquad +\| (\na \ff_f)^{-1}\vert_{z=0}-\tilde{I}_{d+1}\|_{H^{\sigma}_x}\|\na_{x, z}v\vert_{z=0}\|_{H^{s-1}_x}\\
&\le C(\|f\|_{H^s})\left(\| h\|_{H^{\sigma+1}}+\| f\|_{H^{\sigma+1}}\| h\|_{H^s}\right)+ C(\|f\|_{H^s})(\| f\|_{H^{\sigma+1}}+a)\| h\|_{H^s}\\
&\le C(\|f\|_{H^s})\left(\| h\|_{H^{\sigma+1}}+(\| f\|_{H^{\sigma+1}}+a)\| h\|_{H^s}\right).
\end{aligned}
\eq
This concludes the proof of  \eqref{esttrace:cG}. Then, since $G[f]h=N\cdot \cG[f]h\vert_{\Sigma_f}$, \eqref{est:DN:h} follows from \eqref{esttrace:cG} and the tame product estimate
\[
\| N\cdot \cG[f]h\vert_{\Sigma_f}\|_{H_x^\sigma}\les \| \cG[f]h\vert_{\Sigma_f}\|_{H^\sigma}(\| N-e_y\|_{H^{s-1}}+1)+  \| \cG[f]h\vert_{\Sigma_f}\|_{H^{s-1}}\| N-e_y\|_{H^\sigma}.
\]

\quad 2.  For the proof of \eqref{esttrace:cS}, we apply the estimate \eqref{eq.ABZvariant} with $v=\phi^{(2)}\circ \ff_f$, $\zeta=0$ and invoking \eqref{vari:phi2}  to have
\[
\| \na_{x, z}v\vert_{z=0}\|_{H^\sigma_x}\le   C(\| f \|_{H^s})(\|F_0\|_{Y^\sigma([z_1, 0])}+\|k\|_{L^2(\Omega_f)}),\quad z_1\in J,
\]
where 
\[
  F_0= -\frac{(\partial_z \varrho)^2}{1+|\nabla _x \varrho|^2}(\p_y k)\circ \ff_ f= -\frac{\partial_z \varrho}{1+|\nabla _x \varrho|^2}\p_z (k\circ \ff_ f).
\]

Similarly to \eqref{est:naff} we have 
\[
\left\|\frac{\partial_z \varrho}{1+|\nabla _x \varrho|^2}-m\right\|_{X^\sigma(J)}\le C(\| f\|_{H^s})(\| f\|_{H^{\sigma+1}}+a),
\]
where $m=1$ in the infinite-depth case and $m=H$ in the finite-depth case.  Then, arguing as in \eqref{esttrace:cG:10}, we find
 \begin{align*}
  \| F_0\|_{Y^\sigma([z_1, 0])}&\le   \| F_0\|_{L^2([z_1, 0]; H^{\sigma-\mez})}\\ 
  &\les \left(\left\|\frac{\partial_z \varrho}{1+|\nabla _x \varrho|^2}-m\right\|_{L^\infty_z H^{s-1}_x}+m\right)\| \p_z (k\circ \ff_ f)\|_{L^2_zH^{\sigma-\mez}}\\
  &\qquad + \left\|\frac{\partial_z \varrho}{1+|\nabla _x \varrho|^2}-m\right\|_{L^\infty_z H^{\sigma-\mez}_x}\| \p_z (k\circ \ff_ f)\|_{L^2_zH^{s-1}}\\
  &\le C(\| f\|_{H^s})\left(\| k\circ \ff_f\|_{H^{\sigma+\mez}_{x, z}}+(\| f\|_{H^{\sigma+\mez}}+a)\| k\circ \ff_f\|_{H^s_{x, z}}\right).
  \end{align*}
Noting additionally that 
  \[
  \|k\|_{L^2(\Omega_f)}\le C(\| f\|_{H^s})\| k\circ \ff_f\|_{L^2(\Rr^d\times J)},
  \]
we obtain 
\bq\label{tracenav(2)}
\| \na_{x, z}v\vert_{z=0}\|_{H^\sigma_x}\le C(\| f\|_{H^s})\left(\| k\circ \ff_f\|_{H^{\sigma+\mez}_{x, z}}+(\| f\|_{H^{\sigma+\mez}}+a)\| k\circ \ff_f\|_{H^s_{x, z}}\right).
\eq
Since  $\cS[f]k\vert_{\Sigma_f}=\na_{x, z}v\vert_{z=0}(\na \ff_f)^{-1}\vert_{z=0}$, we deduce from the above estimates that 
\begin{align*}
\| \cS[f]k\vert_{\Sigma_f}\|_{H^\sigma_x}
&\les \left(\| (\na \ff_f)^{-1}\vert_{z=0}-\tilde{I}_{d+1}\|_{H^{s-1}_x}+1\right)\|\na_{x, z}v\vert_{z=0}\|_{H^\sigma_x}\\
&\qquad +\| (\na \ff_f)^{-1}\vert_{z=0}-\tilde{I}_{d+1}\|_{H^{\sigma}_x}\|\na_{x, z}v\vert_{z=0}\|_{H^{s-1}_x}\\
&\le C(\| f\|_{H^s}) \left(\| k\circ \ff_f\|_{H^{\sigma+\mez}_{x, z}}+(\| f\|_{H^{\sigma+1}}+a)\| k\circ \ff_f\|_{H^{s}_{x, z}} \right). 
\end{align*}
The proof of \eqref{esttrace:cS} is complete.
\end{proof}
\begin{coro}\label{coro:traceuN}
For $s>1+\frac{d}{2}$, $u$ given by \eqref{eq:u} and $N=(-\na f, 1)$, the estimates  
\begin{align}\label{est:traceuN:l}
&\| u\cdot N\vert_{\Sigma_f}\|_{H^{s-1}}\le C(\| f\|_{H^s})(\| f\|_{H^s}+\| \tilde g\|_{H^s}),\\ \label{est:traceuN:h}
&\| u\cdot N\vert_{\Sigma_f}\|_{H^{s-\mez}}\le C(\| f\|_{H^s})\left(\| f\|_{H^{s+\mez}}+\| \tilde g\|_{H^s}+\| f\|_{H^{s+\mez}}\| \tilde g\|_{H^s}\right)
\end{align}
hold in both the finite-depth and infinite-depth cases.
\end{coro}
\begin{proof}
We recall from \eqref{eq:u} that $u=-\cG[f]\Gamma(f)-\cS[f]g-ge_y$. Since $g\vert_{\Sigma_f}=\tilde{g}(\cdot, 0)$,  the trace theorem for $\Rr^d\times J$ yields $\| g\vert_{\Sigma_f}\|_{H^\sigma}\les \| \tilde{g}\|_{H^{\sigma+\mez}}$. 

For the proof of \eqref{est:traceuN:l}, we apply \eqref{esttrace:cG} and \eqref{esttrace:cS} with $\sigma=s-1$ to have
\[
\| u\vert_{\Sigma_f}\|_{H^{s-1}}\le C(\| f\|_{H^s})(\| f\|_{H^s}+\| \tilde g\|_{H^s}).
\]
Since $H^{s-1}(\Rr^d)$ is an algebra and $\| N-e_y\|_{H^{s-1}}\le C(\| f\|_{H^s})$, the preceding estimate yields \eqref{est:traceuN:l}. 

As for the proof of \eqref{est:traceuN:h}, we apply \eqref{esttrace:cG} and \eqref{esttrace:cS} with $\sigma=s-\mez$ to have 
\[
\| u\vert_{\Sigma_f}\|_{H^{s-\mez}}\le C(\| f\|_{H^s})\left(\| f\|_{H^{s+\mez}}+\| \tilde g\|_{H^s}+\| f\|_{H^{s+\mez}}\| \tilde g\|_{H^s}\right).
\]
To conclude we use the preceding inequality, \eqref{est:traceuN:l}, and the tame product estimate 
\[
\| u\cdot N\vert_{\Sigma_f}\|_{H^{s-\mez}}\les \| u\vert_{\Sigma_f}\|_{H^{s-\mez}}(\| N-e_y\|_{H^{s-1}}+1)+  \| u\vert_{\Sigma_f}\|_{H^{s-1}}\| N-e_y\|_{H^{s-\mez}}.\qedhere
\]
\end{proof}

\subsection{Proof of \cref{prop.ABZvariant}}\label{sec:ABZvariant}
We provide the proof in the infinite-depth case,  modifications needed for  the finite-depth case being straightforward. 

We follow the method in~\cite[Section 3]{ABZ}, which consists in paralinearizing the products  $\alpha \Delta_xv$, $\beta \cdot \nabla \partial_zv$ and then factoring the equation~\eqref{eq:v} into the product of a forward and a backward parabolic paradifferential operators. Precisely, we have
\bq
\begin{aligned}
F_0&=(\partial_z^2 + \alpha \Delta_x +\beta \cdot \nabla_x \partial_z - \gamma \partial_z)v \\
&=(\partial_z^2 + T_\alpha \Delta_x +T_\beta \cdot \nabla_x \partial_z)v+F_1
\end{aligned}
\eq
where
\bq
F_1=(\alpha-T_\alpha)\Delta_x v+(\beta-T_\beta)\cdot \na_x \p_zv-\gamma\p_zv.
\eq
Then, we consider the symbols  
\[
    a(z)=a(z; x,\xi):=\frac{1}{2}\left(-i\beta\cdot \xi - \sqrt{4\alpha |\xi|^2 - (\beta\cdot \xi)^2}\right)
\]
and
\[ 
   A(z)= A(z; x,\xi):=\frac{1}{2}\left(-i\beta\cdot \xi + \sqrt{4\alpha |\xi|^2 - (\beta\cdot \xi)^2}\right)\,
\]
which satisfy $aA=-\alpha|\xi|^2$ and $a+A=-i\beta\cdot \xi$. We have
\[
(\partial_z^2 + T_\alpha \Delta_x +T_\beta \cdot \nabla_x \partial_z)v=(\p_z-T_{a(z)})(\p_z-T_{A(z)})v+ F_2,
\]
where
\bq
F_2=(T_{\alpha}\Delta_x-T_{a(z)}T_{A(z)})v + T_{\partial_zA(z)}v.
\eq
It follows that 
\bq
(\p_z-T_{a(z)})(\p_z-T_{A(z)})v=F_0-F_1-F_2.
\eq
In the following, we fix 
\bq
\eps<\min\{ \mez, s-1-\frac{d}{2}\}
\eq
and $\sigma \ge s-1>\frac{d}{2}$.
\begin{lemm}\label{lemm:F1}
 For any interval $J_0\subset J$, we have
 \bq\label{est:F1}
 \| F_1\|_{Y^{\sigma+\eps}(J_0)}\le C(\| f \|_{H^{s}})\left(\|\p_z v\|_{X^{\sigma}(J_0)}+\|  f \|_{H^{\sigma+1+\eps}} \|\na_{x, z} v\|_{X^{s-1}(J_0)}\right).
 \eq
 \end{lemm}
 \begin{proof}
Using Bony's decomposition we write $(\alpha-T_\alpha)\Delta_x v= T_{\Delta_x v}\alpha + R(\Delta_x v, \alpha)$. Applying \eqref{pp:C-mu}, \eqref{Bony2}, Sobolev's embedding, and \cref{lemm.coefsEstimInfinite}, we obtain
\begin{align*}
    \|T_{\Delta_x v}\alpha + R(\Delta_x v, \alpha)\|_{ L^1(J_0,H^{\sigma + \varepsilon})} & \lesssim  \|\alpha\|_{L^2(J_0,H^{\sigma + \frac{1}{2}})} \|\Delta_x v\|_{L^2(J_0,C_*^{-\frac{1}{2} + \varepsilon})}\\
    & \lesssim  \|\alpha\|_{X^{\sigma}(J_0)}\|\Delta_x v\|_{L^2(J_0,H^{\frac{d}{2}-\frac{1}{2} + \varepsilon})} \\
    & \le C(\| f \|_{H^{s}}) \| f \|_{H^{\sigma+1}}\|\nabla_x v\|_{X^{\frac{d}{2}+\varepsilon}(J_0)}.
\end{align*}
The same argument yields 
\[
\| (\beta-T_\beta)\cdot \na_x \p_zv\|_{\tilde L^1(J_0,H^{\sigma + \varepsilon})}\le    C(\| f \|_{H^{s}}) \|  f \|_{H^{\sigma+1}}\|\p_z v\|_{X^{\frac{d}{2}+\eps}(J_0)}.
\]
Next, we write  $\gamma\p_zv=T_{\gamma}\partial_z v + (T_{\partial_zv}\gamma + R(\partial_zv, \gamma))$ and apply \eqref{pp:C-mu}, \eqref{Bony2}, Sobolev's embedding, and \cref{lemm.coefsEstimInfinite} to have
\begin{align*}
    \|T_{\gamma}\partial_z v\|_{L^2(J_0,H^{\sigma + \varepsilon - \frac{1}{2}})} & \lesssim \|\gamma\|_{L^{\infty}(J_0,C_*^{-1+\varepsilon})}\|\partial_z v\|_{L^{2}(J_0,H^{\sigma + \frac{1}{2}})} \\
    & \lesssim \|\gamma\|_{L^{\infty}(J_0,H^{-1 + \varepsilon + \frac{d}{2}})} \|\partial_z v\|_{X^{\sigma}(J_0)} \\
    & \le C(\| f \|_{H^{s}})\|  f \|_{H^{\frac{d}{2}+1+\varepsilon}} \|\p_z v\|_{X^{\sigma}(J_0)}
\end{align*}
and 
\begin{align*}
    \|T_{\partial_zv}\gamma + R(\partial_zv, \gamma)\|_{L^2(I,H^{\sigma + \varepsilon - \frac{1}{2}})} &\lesssim \ \|\gamma\|_{L^2(J_0,H^{\sigma + \varepsilon - \frac{1}{2}})}\|\partial_z v\|_{L^{\infty}(J_0,L^{\infty})} \\
    & \le C(\| f \|_{H^{s}})\| f \|_{H^{\sigma+1+\varepsilon}}\|\p_z v\|_{X^{s-1}(J_0)}.
\end{align*}
The preceding applications of \cref{lemm.coefsEstimInfinite} are justified because $-1+\eps+\frac{d}{2}>-\mez$ and $\sigma -1 + \varepsilon>\frac{d}{2}-1\ge -\mez$. Finally, noting that $\frac{d}{2}+\eps<s-1$, we conclude the proof.
\end{proof}
It will be important later that in \eqref{est:F1} the highest norm $\| \na f\|_{H^{\sigma+\eps}}$ of $f$ is multiplied by the low norm $\| \na_{x, z}v\|_{X^{s-1}}$ of $v$. 
\begin{lemm}\label{lemm:F2}
 For any interval $J_0\subset J$, we have
 \bq
 \| F_2\|_{ L^2(J_0,  H^{\sigma-\mez+\eps})}\le C(\| f \|_{H^{s}})\|\na_x v\|_{X^{\sigma}(J_0)}.
 \eq
 \end{lemm}
\begin{proof}
In view of the embedding $H^{s-1}\subset C_*^\eps$ and the fact that 
\[
\| \alpha_1\|_{L^\infty(J, H^{s-1})}+\| \beta \|_{L^\infty(J, H^{s-1})}\le C(\| f \|_{H^{s}}),
\]
we deduce
\bq
\sup_{z\in J} M_{\varepsilon}^1(A(z)) +\sup_{z\in J} M_{\varepsilon}^1(a(z)) \le C(\| f \|_{H^{s}}). 
\eq
Since $a(z)A(z)=-\alpha|\xi|^2$, symbolic calculus in \cref{theo:sc} implies that $T_{\alpha}\Delta_x-T_{a(z)}T_{A(z)}$ is an operator of order $2-\eps$ (uniformly in $z$). Similarly, since $\partial_z \alpha$ and $\partial_z \beta$ belong to $L^{\infty}(J, H^{s-2}) \subset L^{\infty}(J, C_*^{\varepsilon - 1})$, we have $\partial_z A(z) \in \Gamma^1_{\varepsilon - 1}([-1,0]\times \mathbb{R}^d)$, and thus $T_{\partial_zA(z)}$ is also an operator of order $2-\varepsilon$ (uniformly in $z$) by virtue of \cref{prop:negOp}. Therefore there holds  
\[
    \|F_2\|_{L^2(J_0,H^{\sigma - \frac{1}{2} + \varepsilon})} \le C(\| f \|_{H^{s}})\|v\|_{L^{2}(J_0, H^{\sigma + \frac{3}{2}})}\,.
\]
The preceding estimate can be improven to 
\[
   \|F_2\|_{L^2(J_0,H^{\sigma - \frac{1}{2} + \varepsilon})} \le C(\| f \|_{H^{s}}) \|\na_xv\|_{L^{2}(J_0, H^{\sigma + \frac{1}{2}})}\le C(\| f \|_{H^s}) \|\na_xv\|_{X^{\sigma}(J_0)}
\]
because $T_{p}=T_{p}\Psi(D)$ where $\Psi$ vanishes in a small neighborhood of $0$, hence $v$ can be replaced by $\Psi(D)v$ in $F_2$. 
\end{proof}
Next, we fix $s>1+\frac{d}{2}$ and $z_1<z_0$ in $J$, and prove by induction on $\sigma\ge s-1$ that 
\begin{align*}
\mathcal{H}_\sigma:~  &\|\nabla_{x,z}v\|_{X^{\sigma}([z_0, 0])} \le  C(\| f \|_{H^{s}})\left(\|\na\zeta\|_{H^{\sigma}}+\|F_0\|_{Y^\sigma([z_1, 0])}\right)\\
& \qquad + C(\| f \|_{H^{s}})\|  f\|_{H^{\sigma+1}} \left(\|\na \zeta\|_{H^{s-1}} + \|F_0\|_{Y^{s-1}([z_1, 0])}+ \|\nabla_{x,z}v\|_{X^{-\mez}([z_1, 0])}\right)
\end{align*}
with $C$ depending only on $(s, \sigma, z_0-z_1)$.

 The based case $\sigma=s-1$ is proven in \cite[Proposition 3.16]{ABZ}, i.e. we have
\bq\label{baseinduction}
\mathcal{H}_{s-1}:~\|\nabla_{x,z}v\|_{X^{s-1}([z_0, 0])} \le C(\| f \|_{H^{s}}) \left(\|\na \zeta\|_{H^{s-1}} + \|F_0\|_{Y^{s-1}([z_1, 0])}+ \|\nabla_{x,z}v\|_{X^{-\mez}([z_1, 0])}\right)
\eq
for all $z_1<z_0$ in  $J$.

 Assuming $\mathcal{H}_\sigma$, we proceed to prove that $\mathcal{H}_{\sigma+\eps}$ holds true. To this end, we fix  a cutoff satisfying 
\[
\chi(z)=1\quad\text{on } [z_0, \infty)\quad\text{and } \chi(z)=0\quad\text{on } (-\infty, z_1].
\]
Then function $w=\chi(z)(\partial_z - T_A)v$ satisfies
\bq\label{eq:w}
 \partial_z w - T_a w =\tilde F:=\chi(z)(F_0-F_1-F_2)+\chi'(z)(\partial_z - T_A)v,\quad w(z_1)=0.
\eq
Note that 
\begin{align*}
\Re A(z)=\Re(-a(z))&=\mez\sqrt{4\alpha |\xi|^2 - (\beta\cdot \xi)^2}\\
&\ge \mez\sqrt{4\alpha-|\beta|^2}|\xi|=\frac{\p_z\varrho}{1+|\na_x\varrho|^2}\ge \frac{c}{1+C\| \na f\|_{H^{s-1}}},
\end{align*}
where $c=1/2$ for infinite depth and $c=\fd/2$ for finite depth. We apply \cref{prop.pseudoParabolic} to the backward parabolic paradifferential equation $(\p_z-T_A)\na_xv=\na_xw$ on $(z_0, 0)$ with final data $\na_xv\vert_{z=0}=\na \zeta$:
\[
\begin{aligned}
\| \na_xv\|_{X^{\sigma+\eps}([z_0, 0])}&\le C(\| f \|_{H^{s}})\left(\| \na \zeta\|_{H^{\sigma+\eps}}+\|\na_xw\|_{Y^{\sigma+\eps}([z_0, 0])} +\| \na_xv\|_{L^2([z_0, 0], H^{\sigma+\eps})}\right)\\
& \le C(\| f \|_{H^{s}})\left(\| \na \zeta\|_{H^{\sigma+\eps}}+\|w\|_{X^{\sigma+\eps}([z_0, 0])} +\| \na_xv\|_{X^{\sigma+\eps-\mez}([z_0, 0])}\right).
\end{aligned}
\]
Since $\p_zv=T_Av+w$ on $[z_0, 0]$, the preceding estimate implies
\bq\label{est:naxv:0}
\begin{aligned}
\| \na_{x, z}v\|_{X^{\sigma+\eps}([z_0, 0])}
& \le C(\| f \|_{H^{s}})\left(\| \na \zeta\|_{H^{\sigma+\eps}}+\|w\|_{ X^{\sigma+\eps}([z_0, 0])} +\| \na_xv\|_{X^{\sigma+\eps-\mez}([z_0, 0])}\right).
\end{aligned}
\eq
 For the forward parabolic $w$ equation \eqref{eq:w}  with initial data $w(z_1)=0$, the estimate in \cref{prop.pseudoParabolic} yields
\bq\label{est:w:0}
\begin{aligned}
\| w\|_{X^{\sigma+\eps}([z_1, 0])}\le C(\| f \|_{H^{s}})\left( \|\tilde F\|_{{Y}^{\sigma+\eps}([z_1, 0])} + \|w\|_{L^2([z_1, 0], H^{\sigma+\eps})}\right).
\end{aligned}
\eq
By the definition of $w$, we have 
\bq\label{est:wlow}
 \|w\|_{L^2([z_1, 0], H^{\sigma+\eps})}\le  C(\| f \|_{H^{s}})\|\na_{x, z}v\|_{L^2([z_1, 0], H^{\sigma+\eps})}\le C(\| \na f\|_{H^{s-1}})\|\na_{x, z}v\|_{X^{\sigma-\mez+\eps}([z_1, 0])}.
\eq
Since
 \[
 \| \chi'(z)(\partial_z - T_A)v\|_{L^2([z_1, 0], H^{\sigma+\mez})}\le \|\chi'\|_{L^\infty}C(\| f \|_{H^{s}})\| \na_{x, z}v\|_{L^2([z_1, 0], H^{\sigma+\mez})},
 \]
it follows from \Cref{lemm:F1,lemm:F2} that 
\bq\label{est:tildeF-bis}
\begin{aligned}
 \|\tilde F\|_{Y^{\sigma+\eps}([z_1, 0])}&\le C(\| f \|_{H^{s}})\left\{\|  f \|_{H^{\sigma+1+\eps}} \|\na_{x, z} v\|_{X^{s-1}([z_1, 0])}\right.\\
 &\qquad\left.+(1+\|\chi'\|_{L^\infty})\|\na_{x, z} v\|_{X^{\sigma}([z_1, 0])}+\|F_0\|_{Y^\sigma([z_1, 0])}\right\}.
\end{aligned}
\eq
As we have remarked after the proof of \cref{lemm:F1}, the highest norm $\| f\|_{H^{\sigma+1+\eps}}$ of $f$ is multiplied by the low norm $\| \na_{x, z}v\|_{X^{s-1}}$ of $v$. 
Since $\eps<1/2$, a combination of \eqref{est:w:0}, \eqref{est:wlow} and \eqref{est:tildeF-bis} yields 
\[
\begin{aligned}
\| w\|_{ X^{\sigma+\eps}([z_1, 0])}&\le C(\| f \|_{H^{s}})\left\{\|  f \|_{H^{\sigma+1+\eps}} \|\na_{x, z} v\|_{X^{s-1}([z_1, 0])}\right.\\
&\qquad\left.+(1+\|\chi'\|_{L^\infty})\|\na_{x, z} v\|_{X^{\sigma}([z_1, 0])}+\|F_0\|_{Y^{\sigma+\eps}([z_1, 0])}\right\}.
\end{aligned}
\]
Then, it follows from \eqref{est:naxv:0} that 
\[
\begin{aligned}
\| \na_{x, z}v\|_{X^{\sigma+\eps}([z_0, 0])}&\le C(\| f \|_{H^{s}})\left\{\| \na \zeta\|_{H^{\sigma+\eps}}+\| f \|_{H^{\sigma+1+\eps}} \|\na_{x, z} v\|_{X^{s-1}([z_1, 0])}\right.\\
 &\qquad\left.+(1+\|\chi'\|_{L^\infty})\|\na_{x, z} v\|_{X^{\sigma}([z_1, 0])}+\|F_0\|_{Y^{\sigma+\eps}([z_1, 0])}\right\}
\end{aligned}
\]
where $\| \chi'\|_{L^\infty}\les (z_0-z_1)^{-1}$. Finally, using the induction hypothesis $\mathcal{H}_\sigma$ and the base property $\mathcal{H}_{s-1}$, we deduce $\mathcal{H}_{\sigma+\eps}$. This finishes the proof of \eqref{eq.ABZvariant}. In the infinite depth, we choose $z_1=z_0-1$ and then let $z_0\to -\infty$ to obtain \eqref{estX:infdepth}.

\subsection{Proof of \cref{prop.GSest}}
We start with the following $X^{-\mez}$ estimate.
\begin{lemm}\label{lemm.variational-v}
With $(\zeta, F_0)$ given by \eqref{eq.datum} and $s>1+\frac{d}{2}$, we have
\bq\label{est:nav:X-mez}
    \|\na_{x,z} v\|_{X^{-\mez}(J)} \le \begin{cases}
    C(\|f\|_{H^s})\|h\|_{H^{\mez}(\Rr^d)}\quad\text{for~}  \eqref{eq.phi1},\\
    C(\|f\|_{H^s}) \|k\|_{H^1(J\times \mathbb{R}^d)} \quad\text{for~}  \eqref{eq.phi2},
    \end{cases}
\eq
where $C : \Rr_+ \to \Rr_+$ is non-decreasing and depends only on $(d,s)$, and also on $\mathfrak{d}$ and $\|b_0\|_{H^{s}}$ in the finite-depth case. 
\end{lemm}
\begin{proof} We provide the proof for the infinite-depth case. We recall from \cite[Lemma 3.11]{NP} that 
\[
    \|\na_{x,z} v\|_{X^{-\mez}(J)} \le C(\|f\|_{H^s})\left(\|\nabla_{x,z}v\|_{L^2_{x,z}(\Rr\times J)}+\|\di_{x, z}(\cA\na v) \|_{L^2(J; H^{-1}(\Rr^d)}\right). 
\]
From the formula 
\[
\nabla_{x,z} v = (\nabla_x \phi \circ \ff_{f}+ \nabla_x \varrho \p_y \phi\circ \ff_{f} , \p_z\varrho\p_y\phi \circ \ff_{f})
\]
 we deduce that  $\|\nabla_{x,z}v\|_{L^2_{x,z}} \le C(\|f\|_{H^s}) \|\na_{x,y}\phi\|_{L^2_{x,y}}$. For \eqref{eq.phi1}, \eqref{est:nav:X-mez} then follows by invoking  \cref{lemm.variationalEstimate} (i). 
As for \eqref{eq.phi2},  we first use  \cref{lemm.variationalEstimate} (ii) to have
\[
\|\nabla_{x,z}v\|_{L^2_{x,z}}\le C(\|f\|_{H^s})\| k\|_{L^2},
\]
and then deduce from the formula \eqref{eqphi:div} that 
\[
\| \di_{x, z}(\cA\na v)\|_{L^2_z H^{-1}_x}\le \| \di_{x, z}(\cA\na v)\|_{L^2_{x, z}}\le C(\|f\|_{H^s})\|\p_yk\|_{L^2_{x, z}}.\qedhere
\] 
\end{proof}
Next, we claim that \cref{prop.GSest} is a consequence of the following result, which will also be useful in the sequel. 
\begin{prop}\label{prop.GSest-v}
Let $d\geqslant 1$, and let $s, r \in \Rr$ such that $s>1+\frac{d}{2}$, $r\ge s-\frac{1}{2}$ and $\lfloor r \rfloor \leq s$. Set $a=0$ in the infinite-depth case and $a=\| b_0\|_{H^{r+\mez}}$ in the finite-depth case. The following estimate holds provided the right-hand side is finite: 
\begin{equation} 
    \label{eq.GS-v}
    \begin{aligned}
    \|\nabla v\|_{H^{r}(\Rr^d\times J)} &\le C(\|f\|_{H^s})\left(\|k\|_{H^r(\Omega_f)}+ \|\zeta\|_{H^{r+\mez}}\right.\\
    &\quad \left. +(\| f \|_{H^{r+\mez}}+a)\left(\| \zeta\|_{H^s}+ \| k\|_{H^{s-\mez}(\Omega_f)}+\| \na k\|_{L^\infty(\Omega_f)}\right)\right),
    \end{aligned}
\end{equation} 
where $C : \Rr_+ \to \Rr_+$ is non-decreasing and depends only on $(d,s,r)$, and also on $\mathfrak{d}$ and $\|b_0\|_{H^{s}}$ in the finite-depth case.
\end{prop}

\begin{proof}[Proof of \cref{prop.GSest}] Let $(m,a) = (1,0)$ for infinite-depth and $(m,a)=(H,\|b\|_{H^{r+\mez}})$ in finite-depth. Again, we only detail the infinite-depth case. In the following, implicit constants only depend on on $ \|f\|_{H^s}$ (and also $\|b_0\|_{H^{s+\mez}}$ in the finite-depth case).

Let $\phi\in \{\phi^{(1)}, \phi^{(2)}\}$ so that $v=\phi\circ\ff_f: \Rr^d \times J \to \Rr$ satisfies \eqref{eq:v} where $(\zeta, F_0)$ is given by \eqref{eq.datum}. An application of \cref{lemm.FaaDiBruno} followed by an application of \cref{lemm.diffeoSobolev} yield
\begin{align*}
    \|\nabla \phi\|_{H^r(\Omega_f)} \le \|\phi\|_{H^{r+1}(\Omega_f)} & \lesssim \|v\|_{H^{r+1}} + \|\nabla v\|_{L^{\infty}}\|\ff_f^{-1} -m\operatorname{id}\|_{H^{r+1}} \\
    &\lesssim \|\nabla v\|_{H^r}  + \|\nabla v\|_{L^{\infty}}(\|f\|_{H^{r+\mez}}+a) + \|\na v\|_{X^{-\mez}(J)},      
\end{align*}
where in the last line we have used $ \|v\|_{H^{r+1}} \lesssim  \|\na v\|_{H^{r}} + \|v\|_{L^2} \lesssim  \|\na v\|_{H^{r}} + \|\na v\|_{X^{-\mez}(J)}$.

Since $s-1>\frac{d}{2}$, Sobolev's embedding and \cref{prop.ABZvariant} applied with $\sigma = s-1$ yield 
\begin{align*}
    \|\nabla v\|_{L^{\infty}}\les \| \na v\|_{X^{s-1}} \lesssim \|\zeta\|_{H^s} + \|F_0\|_{Y^{s-1}(\Rr^d \times J)} + \|\nabla_{x,z}v\|_{X^{-\frac{1}{2}}(J)}.
\end{align*}
\cref{lemm.variational-v} allows us to estimate $\|\nabla v\|_{X^{-\mez}(\Rr^d\times J)}$, so that
\bq\label{nav:Xs-1}
    \|\nabla v\|_{L^{\infty}(J)}\lesssim \|\zeta\|_{H^s} + \|F_0\|_{Y^{s-1}(J)}\,.
\eq 
We recall that from \eqref{eq.datum}, if $\phi=\phi^{(2)}$ then $F_0= \frac{\partial_z \varrho}{1 + |\nabla_x \varrho|^2} \partial_z(k\circ \ff_f)$. Applying \cref{lemm.F0est}  with $\sigma := s-\mez$ gives
\[
    \|F_0\|_{Y^{s-1}(\Rr^d \times J)} \lesssim \|F_0\|_{L^2(J,H^{s-\frac{3}{2}})} \lesssim \| k\|_{H^{s-\mez}}+\| \na k\|_{L^\infty}\| f\|_{H^{s-1}}, 
\]
from which it follows that 
\bq\label{nav:X:s-mez}
    \|\nabla v\|_{L^{\infty}(J)}\les \|\zeta\|_{H^s} + \| k\|_{H^{s-\mez}}+\| \na k\|_{L^\infty}.
\eq
Combining the above estimates, we find
\begin{align}
    &\| \na \phi^{(1)}\|_{H^r(\Omega_f)}\les   \|\nabla v\|_{H^r}  +\| \zeta\|_{H^s}(\| f\|_{H^{r+\mez}}+a), \label{eq.phi1-Sobolev}\\
    &\| \na \phi^{(2)}\|_{H^r(\Omega_f)}\les   \|\nabla v\|_{H^r}  +\left(\| k\|_{H^{s-\mez}}+\| \na k\|_{L^\infty}\right)(\| f\|_{H^{r+\mez}}+a). \label{eq.phi2-Sobolev}
\end{align}
Therefore, \cref{prop.GSest} follows from \cref{prop.GSest-v}.
\end{proof} 

The remainder of this section is devoted to the proof of \cref{prop.GSest-v}. To this end, we will assume that $\na_{x, z} v\in H^r(\Rr^d\times J)$ and  derive \eqref{eq.GS-v}. This a priori estimate can be turned into a regularity result by an approximation argument analogously to the one used in the proof of \cref{thm.elliptic-CS}.

\subsubsection{The infinite-depth case}\label{subsection:Soboleregvinf} We aim to prove   \eqref{eq.GS-v} for $\na_{x, z}v$ in $H^r(I\times \Rr^d)$,  $I=(-z_0,0)\subset J$, with the constant $C(\| f\|_{H^s})$  independent of $z_0\in J$.  In the infinite-depth case, this will yield \eqref{eq.GS-v} for $J=(-\infty, 0)$ by letting $z_0\to -\infty$.  For the finite-depth case, we will establish  near-bottom estimates to complete the proof of  \eqref{eq.GS-v}. 

Let $r\ge s-\mez$ such that $\lfloor r\rfloor \leq s$. We write $r=k+\mu$ with $k=\lfloor r\rfloor \le s$ and $\mu\in [0,1)$. We are going to prove using induction on $\ell \in \{0, \dots, k+1\}$ that 
\begin{align}
    \tag{P${}_{\ell}$}
    \label{eq.Pl}
    \|\partial^{\ell}_z\nabla^{k+1-\ell}_x v\|_{H^{\mu}_{x,z}}  &\les  \|k\|_{H^r}+ \|\zeta\|_{H^{r+\mez}}+(\| f \|_{H^{r+\mez}}+a)\left(\| \zeta\|_{H^s}+ \| k\|_{H^{s-\mez}}+\| \na k\|_{L^\infty}\right) =: \Xi.
\end{align}
The desired estimate \eqref{eq.GS-v} for  $\|\nabla_{x,z}v\|_{H^r(I\times \mathbb{R})}$ then follows \eqref{eq.Pl} and \cref{lemm.variational-v}. When $r=s-\mez$, \eqref{eq.GS-v} reads
\bq\label{eq.GS-v:base}
 \|\na_{x, z}v\|_{H^{s-\mez}}  \lesssim  \| \zeta\|_{H^s}+ \| k\|_{H^{s-\mez}}+\| \na k\|_{L^\infty}. 
\eq
Since the method that we will present for the proof of \eqref{eq.Pl} also yields \eqref{eq.GS-v:base}, we will assume \eqref{eq.GS-v:base} and use it to deduce \eqref{eq.Pl}.
The implicit constants in  \eqref{eq.Pl} and \eqref{eq.GS-v:base} are of the form $C(\|f\|_{H^s})$, which will be neglected in the following estimates. 
\medskip 

\underline{Case $\ell =0$.} We use an interpolation argument  in~\cite[Proposition 4.1]{WZZZ} which is  recalled and proven in \cref{lemm.interpol}. This gives 
\[
    \|\nabla_x^{k+1} v\|_{H_{z,x}^{\mu}} \lesssim \|\na_x^{k+1} v\|_{L^2_zH_x^{\mu}} + \|\partial_z \nabla_x^{k+1}v\|_{L^2H_x^{\mu-1}} \lesssim \|\na_{x,z} v\|_{L^2_zH_x^{k+\mu}}\le \|\na_{x,z} v\|_{X^{k+\mu-\mez}},
\]
where $k+\mu-\mez=r-\mez$. We apply \cref{prop.ABZvariant} with $\sigma = r-\mez\geqslant s-1$  to obtain 
\begin{align}\label{nav:Xr-mez}
 \|\na_{x,z} v\|_{X^{r-\mez}}&\les   \|\zeta\|_{H^{r+\mez}}+\| f \|_{H^{r+\mez}}\| \zeta\|_{H^s}+ \|F_0\|_{Y^{r-\mez}} + \|\nabla_{x,z}v\|_{X^{-\mez}}.
\end{align}
By \cref{lemm.variational-v},
\[
\|\nabla_{x,z}v\|_{X^{-\mez}}\les \| \zeta\|_{H^\mez}+\| k\|_{H^1}.
\]
On the other hand,  \cref{lemm.F0est} applied with  $\sigma =r\ge s-\mez>1$ gives
   \bq\label{est:F0:L2x}
\begin{aligned}
\|F_0\|_{Y^{r-\mez}}& \le \|F_0\|_{L^2_zH_x^{r-1}}  \les \|f\|_{H^{r + \frac{1}{2}}}(\|k\|_{H^{s-\mez}}+\|\nabla k\|_{L^{\infty}}) + \|k\|_{H^r}. 
\end{aligned}
     \eq
Inserting the above estimates into \eqref{nav:Xr-mez} yields
\bq\label{nav:Xr-mez:1}
\|\na_{x,z} v\|_{X^{r-\mez}}\les \Xi
\eq
which implies $(P_0)$. 
\medskip 

\underline{Case $\ell = 1$.} We use again \cref{lemm.interpol} which yields 
\begin{align*}
    \|\partial_z\na_x^k v\|_{H_{z,x}^{\mu}} &\lesssim \|\partial_z\na_x^k v\|_{L_z^2H_x^{\mu}} + \|\partial_z^2\na_x^{k} v\|_{L_z^2H_x^{\mu-1}}  \lesssim \|\partial_zv\|_{L_z^2H_x^{r}} + \|\partial_z^2 v\|_{L_z^2H_x^{k+\mu-1}}.
\end{align*}
The term $\|\partial_zv\|_{L_z^2H_x^{r}}$ is controlled by \eqref{nav:Xr-mez:1}. As for $ \|\partial_z^2 v\|_{L_z^2H_x^{k+\mu-1}}$, we  use \eqref{eq:v} to replace 
\begin{equation}
    \label{eq.2derivatives}
    \partial^2_z v = - \alpha \Delta_x v - \beta \cdot \nabla_x \partial_z v + \gamma \partial_z v + F_0.
\end{equation}
 $\| F_0\|_{L^2_zH^{k+\mu-1}_x}$ is controlled by \eqref{est:F0:L2x}.  We shall estimate the first three terms on  the right-hand side of \eqref{eq.2derivatives} by appealing to the product estimate
\bq\label{productfordz}
    \| u_1u_2\|_{L_z^2H_x^\sigma}\les \| u_1\|_{L_z^\infty L_x^\infty}\| u_2\|_{L_z^2H_x^{\sigma}}+\| u_2\|_{L_z^2 C^{-t}_*}\| u_1\|_{L_z^{\infty}H_x^{\sigma+t}},\quad  \sigma, t> 0,
\eq
which follows from Bony's decomposition, \eqref{pp:Linfty}, \eqref{pp:C-mu} and \eqref{Bony2}. 

For $\alpha\Delta_xv=(\alpha-m)\Delta_xv+m\Delta_xv$ we apply \eqref{productfordz} with $t=\mez$ to $u_1=\alpha-m$ and $u_2=\Delta_xv$ to have
\begin{align*}
    \|\alpha \Delta_x v\|_{L_z^2H_x^{k+\mu-1}}
    &\les \| \alpha-m\|_{L_z^\infty L_x^\infty}\| \Delta_xv\|_{L_z^2H_x^{k+\mu-1}}+ \|\Delta_xv\|_{L_z^2 C_*^{-\mez}}\| \alpha-m\|_{L_z^{\infty}H_x^{k+\mu -\mez}}+\|\Delta_xv\|_{L_z^2H_x^{k+\mu-1}}\\
    &\les \| \na_xv\|_{L_z^2 H_x^{k+\mu}}+\| \na_xv\|_{L_z^2H_x^{s-\mez}}(\| f\|_{H^{k+\mu+\mez}}+a)\\
    &\les   \| \na_xv\|_{X^{r-\mez}}+\| \na_xv\|_{X^{s-1}}(\| f\|_{H^{r+\mez}}+a),
\end{align*}
where we have used \cref{lemm.coefsEstimInfinite} and the embedding $H^{s-\mez}(\Rr^d)\hookrightarrow C^\mez_*(\Rr^d)$ for $s>1+\frac{d}{2}$. Here $a=0$ for infinite depth, and $a=\| b_0\|_{H^{r+\mez}}$ for finite depth.

  Since $\beta$ has the same regularity as $\alpha-m$, we also have
\[
    \| \beta \cdot \nabla_x \partial_z v\|_{L_z^2H_x^{k+\mu-1}}\les \| \na_xv\|_{X^{r-\mez}}+\| \na_xv\|_{X_x^{s-1}}(\| f\|_{H^{r+\mez}}+a).
\]
As for $\gamma\p_zv$, we first apply \eqref{productfordz} with $\sigma=k+\mu-1\ge 1+\mu>0$ and $t=1$ to $u_1=\p_zv$ and $u_2=\gamma$, and then apply \cref{lemm.coefsEstimInfinite} to estimate $\gamma$. We obtain
\begin{align}
    \|\gamma \partial_z v \|_{L_z^2 H_x^{k+\mu-1}}
    &\les \| \p_zv\|_{L_z^{\infty} L_x^\infty}\| \gamma\|_{L_z^{2}H_x^{k+\mu-1}}+\| \gamma\|_{L_z^\infty C_*^{-1}}\| \p_zv\|_{L_z^2H_x^{k+\mu}}\label{eq.PnotH}\\
    &\les\| \p_zv\|_{X^{s-1}}\| \gamma\|_{X^{r-\mez}}+ \| \gamma\|_{X^{s-2}}\| \p_zv\|_{X^{r-\mez}} \notag\\
    & \les\| \p_zv\|_{X^{s-1}}(\| f\|_{H^{r+\mez}}+a)+\| \p_zv\|_{X^{r-\mez}}.
\end{align}

The estimate \eqref{nav:Xr-mez:1} with $r=s-\mez$ gives 
\bq\label{nav:Xs-mez:1}
\| \na_{x, z}v\|_{X^{s-1}}\les \| \zeta\|_{H^s}+ \| k\|_{H^{s-\mez}}+\| \na k\|_{L^\infty}.
\eq
Therefore, putting the above estimates together and apply \eqref{nav:Xr-mez:1}  again, we obtain  $(P_1)$. 
\medskip 

\underline{Case $\ell \in \{2, \dots, k+1\}$.} 
In order to prove \eqref{eq.Pl} for $\ell\in\{2, \dots, k+1\}$, we claim that 
\bq\label{claim:Pl}
\forall \ell\in\{1, \dots, k\},\quad \|\partial_z^{\ell}\nabla_x^{k+1-\ell}v\|_{H^{\mu}_{x,z}}  \le \Xi + A_{\ell},\quad A_{\ell}:= \sum_{j=1}^{\ell} \| \p_z^j v\|_{L^2_{x, z}}
\eq
and 
\bq\label{claim:Pl-k1}
\|\p_z^{k+1}v\|_{H^{\mu}_{x,z}} \le \Xi + A_k + \|\nabla_{x,z}v\|_{H_{x,z}^{k+\mu -1}}.
\eq 
Assume temporarily that \eqref{claim:Pl} and \eqref{claim:Pl-k1} hold. Togther  with the case $\ell=0$, they imply 
\begin{multline*}
    \| \na_{x, z}^{k+1}v\|_{H^\mu_{x,z}}\les \Xi + A_k + \|\nabla_{x,z}v\|_{H^{k+\mu -1}_{x,z}}  \\
    \les \Xi  + \|\p_z v\|_{H^{k-1}_{x,z}} + \|\nabla_{x,z}v\|_{H^{k+\mu -1}_{x,z}}
     \les \Xi + \|\nabla_{x,z}v\|_{H^{k+\mu -1}_{x,z}},
\end{multline*}
whence
\[
    \| \na_{x, z}v\|_{H^{k+\mu}_{x,z}}\les \Xi+ \| \na_{x, z}v\|_{H^{k-1+\mu}_{x,z}}+\| \na_{x, z}v\|_{L^2_{x,z}}.
\]
By interpolating  $\| \na_{x, z}v\|_{H^{k+\mu-1}}$ between $\| \na_{x, z}v\|_{L^2}$ and $\| \na_{x, z}v\|_{H^{k+\mu}}$, we obtain 
\[
    \| \na_{x, z}v\|_{H^{k+\mu}_{x,z}}\les \Xi+\| \na_{x, z}v\|_{L^2_{x,z}}\les \Xi+\| \na_{x, z}v\|_{X^{-\mez}}.
\]
Invoking  \cref{lemm.F0est}, we conclude that $\| \na_{x, z}v\|_{H^{k+\mu}_{x,z}}\les \Xi$ which implies \eqref{eq.Pl} for $\ell\in\{0, \dots, k+1\}$. 

For the proof of \eqref{claim:Pl} and \eqref{claim:Pl-k1}, we will assume $\mu \in (0,1)$. The case $\mu=0$  can be treated analogously except that \cref{lemm.interpol} is not invoked.  
\begin{proof}[Proof of \eqref{claim:Pl}] 
We proceed by induction on $\ell$. The case $\ell=1$ has been obtained above. Assuming that $k\ge 2$ and that \eqref{claim:Pl} holds for all $\ell ' \le \ell$ for some $\ell \in \{1, \dots, k-1\}$, we need to prove that  
\bq\label{claimPl:l}
    \|\p_z^{\ell+1}\na_x^{k-\ell}v\|_{H^\mu_{x, z}}\les \Xi +A_{\ell+1}. 
\eq
An application of~\cref{lemm.interpol} yields
\begin{equation}
    \label{eq.firstSplit}
    \|\partial_z^{\ell+1}\nabla^{k-\ell} v\|_{H^{\mu}} \lesssim \|\partial_z^{\ell+1}\nabla_x^{k-\ell}v\|_{L_z^2H_x^{\mu}} + \|\partial_z^{\ell +2}\nabla_x^{k-\ell}v\|_{L^2H^{\mu -1}}.
\end{equation}
By writing $\p_z^{\ell+1}v=\p_z^{\ell-1}\p_z^{2} v$ and $\p_z^{\ell + 2}v = \p_z^{\ell}\p_z^2v$ we can use \eqref{eq:v} to replace 
\[ 
    \p_z^2 v= F_0 - (\alpha\Delta v + \beta \cdot \nabla_x \p_zv + \gamma\p_zv),
\]
which contains at most one $z$-derivative of $v$. Hence, $\p_z^{\ell + 1}\na_x^{k-\ell}$ contains at most $\ell$ derivatives in $z$ while $\p_z^{\ell +1}\na_x^{k-\ell}v$ can contain up to $\ell +1$ derivatives in $z$. Using the method for controlling the second term that we present below, one obtains 
\begin{equation}\label{claimPl:l0}
    \|\partial_z^{\ell+1}\nabla_x^{k-\ell}v\|_{L_z^2H_x^{\mu}} \les \Xi + A_{\ell}.
\end{equation}
Let us now move to controlling $\|\partial_z^{\ell +2}\nabla_x^{k-\ell}v\|_{L^2_zH^{\mu -1}_x}$:
\begin{align*}
    \|\partial_z^{\ell +2}\na_x^{k-\ell}v\|_{L^2_zH^{\mu -1}_x} &\le \|\partial_z^{\ell}\na_x^{k-\ell}F_0\|_{L_z^2H_x^{\mu -1}} + m^2\|\partial_z^{\ell}\na_x^{k-\ell} \Delta_x v\|_{L^2_zH_x^{\mu -1}} \\
    & \quad  + \|\partial_{z}^{\ell}\na_x^{k-\ell}((\alpha - m^2)\Delta_x v)\|_{L^2_zH_x^{\mu -1}} + \|\p_z^{\ell}\nabla_x^{k-\ell}(\beta\cdot \nabla_x\p_z v)\|_{L^2_zH^{\mu-1}}\\
    &\quad  + \|\partial_z^{\ell}\nabla_x^{k-\ell}(\gamma \partial_zv)\|_{L^2H^{\mu -1}}. 
\end{align*}
Since $k-\ell\ge 1$, an application of \cref{lemm.F0est} with $\sigma=r\ge s-\mez>1$ gives
\[
    \|\partial_z^{\ell}\na_x^{k-\ell}F_0\|_{L_z^2H_x^{\mu -1}} \les     \|\partial_z^{\ell}F_0\|_{L_z^2H_x^{k-\ell+\mu-1 }}\les \|F_0\|_{H^{r-1}}  \les \Xi.
\]
By the induction hypothesis,
\[
    \|\partial_z^{\ell}\na_x^{k-\ell} \Delta_x v\|_{L^2_zH_x^{\mu -1}} \les \|\partial_z^{\ell}\na_x^{k+1-\ell}v\|_{L^2_zH_x^{\mu}}\les \|\partial_z^{\ell}\na_x^{k+1-\ell}v\|_{H_{x, z}^{\mu}} \les \Xi + A_{\ell}.
\]
Combining Leibniz's rule with a product rule  \eqref{productfordz} yields 
\begin{align*}
    \|\partial_z^{\ell}\nabla_x^{k-\ell}((\alpha-m^2)\Delta_x v)\|_{L^2_zH^{\mu -1}} &\les \|\partial_z^{\ell}((\alpha-m^2)\Delta_x v)\|_{L^2_zH^{k+\mu -\ell -1}}\\
     & \lesssim\sum_{0\le j \le \ell}\|\partial_z^j(\alpha-m^2)\|_{L^{\infty}_zC_*^{-j}}\|\partial_z^{\ell-j}\Delta_xv\|_{L^{2}_zH^{k+\mu+j-\ell-1}} \\
    &\quad \quad + \sum_{0\le j \le \ell}\|\partial_z^j(\alpha-m^2)\|_{L_z^{\infty}H^{k+\mu-j-\mez}}\|\partial_z^{\ell-j}\Delta_xv\|_{L^2_zC_*^{j-\ell - \mez}}   
\end{align*}
The preceding application of  \eqref{productfordz} is justified since $k+\mu-\ell-1\ge \mu >0$ and $j -\ell - \mez <0$. Since $s>1+\frac{d}{2}$ and $j\le k-1\le \lfloor  r\rfloor -1\le s-1$,  Sobolev's embedding and the estimate \eqref{eq.pzEstim1} imply
\[
    \|\partial_z^j(\alpha-m^2)\|_{L^{\infty}_zC_*^{-j}} \les \|\p_z^j(\alpha-m^2)\|_{L^\infty H^{s-1-j}}\le  C(\|f\|_{H^s})  
\]
and 
\[
    \|\partial_z^j(\alpha-m^2)\|_{L^\infty_zH^{k+\mu-j -\mez}} \lesssim \|f\|_{H^{k+\mu+\mez}}+a. 
\]
Regarding  $B_j :=\|\partial_z^{\ell-j}\Delta_xv\|_{L_z^2H_x^{k+\mu+j-\ell -1}}$, when $j=\ell$ we have 
\[
    B_\ell= \|\Delta_x v\|_{L^2_zH^{k+\mu-1}_x}\les  \| \na_xv\|_{L^2_zH^{k+\mu}_x}\le \Xi
\]
in view of \eqref{nav:Xr-mez:1}. For $0\le j\le \ell-1$, we have 
\[
    B_j\les  \|\p_z^{\ell-j}  v\|_{L^2_zH^{k+\mu-\ell+1+j}_x}\les \|\p_z^{\ell-j}  v\|_{L^2_{x, z}}+ \|\p_z^{\ell-j} \na_x^{k-\ell+1+j} v\|_{L^2_zH^\mu_x}\les \Xi+A_{\ell}
\]
by the induction hypothesis. 

Since $s>1+\frac{d}{2}$ and $r\ge s-\mez$, it follows from Sobolev's embedding and \eqref{eq.GS-v:base} that   
\[
\begin{aligned}
    \|\partial_z^{\ell -j}\Delta_xv\|_{L^{2}_zC_*^{j-\ell-\mez}}\les  \|\partial_z^{\ell -j}\Delta_xv\|_{L^{2}_zH^{s-\tdm +j - \ell}}\les \| \na_xv\|_{H^{s-\mez}_{x, z}} \lesssim  \| \zeta\|_{H^s}+ \| k\|_{H^{s-\mez}}+\| \na k\|_{L^\infty}.
\end{aligned}
\]
Putting the above estimates together, we find
\[
    \|\partial_z^{\ell}\nabla_x^{k-\ell}((\alpha-m^2)\Delta_x v)\|_{L^2_zH_x^{\mu-1}}\les  \Xi+A_{\ell}. 
\]
The term $\|\partial_z^{\ell}\nabla_x^{k-\ell}(\beta\cdot\na_x\partial_x v)\|_{L_z^2H_x^{\mu-1}}$ can be treated similarly since $\beta$ and $\alpha-m^2$ have the same regularity. We only note that $B_j$ needs to be replaced with 
\[
    B_j':=\|\p_z^{\ell-j} \na_x \p_zv\|_{L^2_zH^{k+\mu-\ell-1+j}_x}\le \|\p_z^{\ell-j+1}  v\|_{L^2_zH^{k+\mu-\ell+j}_x}
\]
which has $\ell+1$ derivatives in $z$ when $j=0$. However, in that case we have
\[
    \|\p_z^{\ell-j+1}  v\|_{L^2_zH^{k+\mu-\ell}_x}=\|\p_z^{\ell+1}  v\|_{L^2_zH^{k+\mu-\ell}_x}\les\|\p_z^{\ell+1}  v\|_{L^2_{x, z}}+\|\p_z^{\ell+1}  \na_x^{k-\ell} v\|_{L^2_zH^\mu_x}\les A_{\ell+1}+ \Xi,
\]
where we have used \eqref{claimPl:l0}. 

As for  the  $\gamma$-term, we need to estimate  $ \|\p_z^j \gamma \p_z^{\ell -j+1}v\|_{L^2_zH^{k+\mu-\ell -1}}$ for all $j\in \{0, \dots, \ell\}$. 
We first consider the case $j\le \ell-1$. Then we have $\mez-\ell+j\le -\mez<0$, so that  the product rule  \eqref{productfordz} gives 
\begin{align*}
    \|\p_z^j\gamma \partial_z^{\ell-j+1}v\|_{L^2_zH^{k+\mu-\ell-1}} & \lesssim \|\p_z^j\gamma\|_{L_z^{\infty}C_*^{-1-j}} \|\partial_z^{\ell-j+1} v\|_{L_z^2H^{k+\mu-\ell+j}} + \|\p_z^j\gamma\|_{L_z^{\infty}H^{k+\mu-j-\frac{3}{2}}} \|\partial_z^{\ell-j+1} v\|_{L_z^{2}C_*^{\frac{1}{2}-\ell+j}} \\
    &\les \|\p_z^j \gamma\|_{L_z^{\infty}H^{s-2-j}} \|\partial_z^{\ell-j+1} v\|_{L_z^2H^{k+\mu-\ell+j}} + \|\p_z^j\gamma\|_{L_z^{\infty}H^{k+\mu-\frac{3}{2}}} \|\partial_z^{\ell-j+1} v\|_{L_z^{2}H^{s-\frac{1}{2}-\ell}}.
\end{align*}
We have  $k+\mu-\tdm=r-\tdm\ge s-2\ge -\mez$, so that the estimates \eqref{eq.coefsEstimInfinite-gamma} and \eqref{eq.coefsEstimFinite-gamma} for $\gamma$ imply 
\[
 \|\gamma\|_{L_z^{\infty}H^{k+\mu-\frac{3}{2}}}\les \|f\|_{H^{r+\mez}}+a.
 \]
Since 
\[
s-2-j\ge s-2-(\ell -1)\ge s-2-(k-2)=s-\lfloor r\rfloor\ge 0,
\]
the estimate \eqref{eq.pzEstim1c} for $\gamma$ is applicable, which yields
\[
\|\p_z^j \gamma\|_{L_z^{\infty}H^{s-2-j}}\le C(\| f\|_{H^s}).
\]
On the other hand, since $s-\mez-\ell \ge \mez >0$, we have that  $\|\partial_z^{\ell-j+1} v\|_{L_z^{2}H^{s-\frac{1}{2}-\ell}} \leq \|\p_zv\|_{H^{s-\mez}_{x,z}}$ is controlled by the \eqref{eq.GS-v:base}.  Lastly, we have
\[
    \|\partial_z^{\ell-j+1} v\|_{L^2_zH^{k+\mu-\ell+j}} \les \|\partial_z^{\ell-j+1}v\|_{L^2_{x, z}}+ \|\partial_z^{\ell-j+1}\na_x^{k-\ell+j}v\|_{L^2_zH^{\mu}} \les A_{\ell+1}+\Xi.
\]
The above estimates imply 
\[
  \|\p_z^j\gamma \partial_z^{\ell-j+1}v\|_{L^2_zH^{k+\mu-\ell-1}}\les \Xi+A_{\ell+1}
\]
for all $0\le j\le \ell-1$.  Regarding the  remaining case $j=\ell$, we use the following substitute for  product rule \eqref{productfordz}: 
\begin{align*}
\|\p_z^\ell \gamma \partial_zv\|_{L^2_zH^{k+\mu-\ell-1}}&\les \|\p_z^\ell \gamma\|_{L^2_z C_*^{-\ell-\mez}} \| \partial_zv\|_{L^\infty_zH^{k+\mu-\mez}}+ \|\p_z^\ell \gamma\| _{L^2_z H^{k+\mu-\ell-1}} \| \partial_zv\|_{L^\infty_{x, z}} \\
&\les \|\p_z^\ell \gamma\|_{L^2_z H^{s-\ell-\tdm}} \| \partial_zv\|_{L^\infty_zH^{k+\mu-\mez}}+ \|\p_z^\ell \gamma\| _{L^2_z H^{k+\mu-\ell-1}} \| \partial_zv\|_{X^{s-1}}.
\end{align*}
Using \cref{lemm:interpolationLM} yields 
\[
    \| \partial_zv\|_{L^\infty_zH^{k+\mu-\mez}} \les \|\p_zv\|_{L^{2}_zH^{k+\mu}} + \|\p_z^2v\|_{L^{2}_zH^{k+\mu-1}} \leq \Xi
\]
by virtue of the estimates obtained in the case $\ell=1$. Also, applying \eqref{dzgamma:L2} with $(j,\sigma)=(\ell, s-\ell-\tdm)$ yields $\|\p_z^\ell \gamma\|_{L^2_z H^{s-\ell-\tdm}} \leq C(\|f\|_{H^s})$. We note that the conditions in \eqref{dzgamma:L2-cond} are satisfied because 
\begin{align*}
& (s-\ell-\tdm)+1 \ge s-(\lfloor r\rfloor-1)-\mez\ge \mez>0,\\
& (s-\ell-\tdm)+(s-1) \ge -\mez +s-1>0,\quad \ell \le \lfloor r\rfloor -1\le s-1.
\end{align*}
 Since $k+\mu-\ell-1 \geq 0$, \eqref{eq.pzEstim1c} implies  
 \[
 \|\p_z^\ell \gamma\| _{L^2_z H^{k+\mu-\ell-1}} \leq C(\|f\|_{H^s})(\|f\|_{H^{k+\mu +\mez}}+a).
 \]
  Finally, $\| \partial_zv\|_{X^{s-1}}$ is controlled by means of \eqref{nav:Xs-mez:1}.

The proof of \eqref{claim:Pl} is now complete.
\end{proof}

\begin{proof}[Proof of \eqref{claim:Pl-k1}] 
We start by using \eqref{eq:v} to replace $\p_z^2v$, which gives
\begin{align*}
    \|\partial_z^{k+1}v\|_{H^{\mu}} &\le \|\partial_z^{k-1}F_0\|_{H^{\mu}} +m^2 \|\partial_z^{k-1}\Delta_xv\|_{H^{\mu}}  + \|\partial_z^{k-1}((\alpha-m^2)\Delta_x v)\|_{H^{\mu}} \\
    & \quad+ \|\partial_z^{k-1}(\beta\cdot\na_x\partial_x v)\|_{H^{\mu}} + \|\partial_z^{k-1}(\gamma\partial_z v)\|_{H^{\mu}},  
\end{align*}
where $\|\partial_z^{k-1}F_0\|_{H^{\mu}}\les \Xi$ in view of \cref{lemm.F0est}. 
An application of \eqref{claim:Pl} with $\ell=k-1$  yields  
\begin{equation*}
    \|\partial_z^{k-1}\Delta_xv\|_{H^{\mu}}\les   \Xi+A_{k-1}.
\end{equation*}
For the $\alpha$-term, we need to estimate 
\begin{multline*}
    \|\p_z^j(\alpha_1-m^2)\p_z^{k-1-j}\Delta_xv\|_{H^{\mu}} \lesssim \|\p_z^j(\alpha_1-m^2)\|_{L^{p_j}}\|\p_z^{k-1-j}\Delta_xv\|_{W^{\mu,p'_j}} \\
    + \|\p_z^j(\alpha_1-m^2)\|_{W^{\mu,q_j}}\|\p_z^{k-1-j}\Delta_xv\|_{L^{q'_j}} 
\end{multline*}
with $j\in\{0, \dots, k-1\}$, $\frac{1}{p_j}+\frac{1}{p'_j}=\frac{1}{q_j}+\frac{1}{q'_j}=\frac{1}{2}$, and $p'_j,q_j < \infty$. We will only explain how to choose $(p_j,p'_j)$, the choice of $(q_j,q'_j)$ being analogous.  We note that  $0\le j\le k-1\le \lfloor r\rfloor -1\le s-1$.
\medskip 

\underline{Case $j=0$.} We choose $p_j=\infty$, so that a combination of  Sobolev's embedding and  \eqref{eq.pzEstim1b} yields 
\[
    \|\alpha_1-m^2\|_{L^{\infty}} \lesssim \|\alpha_1-m^2\|_{H^{\frac{d+1}{2}+\varepsilon}} \les \|\alpha_1-m^2\|_{H^{s-\mez}} \le C(\|f\|_{H^s}).
\] 
Moreover, we have proven that  $\|\partial_z^{k-1}\Delta_xv\|_{H^{\mu}}\les \Xi+A_{k-1}$. 
\medskip 

\underline{Case $1\le j\le s-1$.} In this case, we have $s-\mez-j\ge \mez>0$, so we can choose 
\[
    \begin{cases}
        \frac{1}{p_j}=\frac{1}{2}-\frac{s-\mez -j}{d+1}\quad\text{if~}s-\mez-j< \frac{d+1}{2},\\
        p_j=\infty\quad\text{if~} s-\mez-j> \frac{d+1}{2},\\
        2\ll  p_j<\infty\quad\text{if~} s-\mez-j= \frac{d+1}{2},
    \end{cases}
\]
in such a way that  Sobolev's embedding together with \eqref{eq.pzEstim1b} imply
\[
    \|\partial_z^j(\alpha - m^2)\|_{L^{p_j}} \lesssim \|\partial_z^j (\alpha-m^2)\|_{H^{s-\frac{1}{2}-j}} \le C(\|f\|_{H^s}). 
\] 
Since $s>1+\frac{d}{2}$ and $j\ge 1$, we have  $H^{\mu+j} \hookrightarrow W^{\mu,p'_j}$, whence
\begin{equation*}
    \|\p_z^{k-1-j}\Delta_xv\|_{W^{\mu,p'_j}} \lesssim \|\p_z^{k-1-j}\Delta_xv\|_{H^{\mu +j}}\lesssim \|\p_z^{k-1-j}\Delta_xv\|_{H^{\mu}_{x, z}} + \|\na_{x,z}^{j+1}\p_z^{k-1-j}\na_xv\|_{H^{\mu}}.
\end{equation*}
Since  $\na_{x, z}^{j+1}\p_z^{k-1-j}\na_xv$ contains a total of $k+1$ derivatives with at most $k$ derivatives in $z$, \eqref{claim:Pl}  gives 
\[
     \|\na_{x, z}^{j+1}\p_z^{k-1-j}\na_xv\|_{H^{\mu}_{x, z}}\les \Xi+A_{k}.
\]
On the other hand, we have $\|\p_z^{k-1-j}\Delta_xv\|_{H^{\mu}_{x, z}}\les \| \na_xv\|_{H^{k+\mu-1}}$ since $j\ge 1$. Thus we obtain
\[
    \|\partial_z^{k-1}((\alpha-m^2)\Delta_x v)\|_{H^{\mu}}\les \Xi+A_{k}+\| \na_xv\|_{H^{k+\mu-1}}.
\]
\medskip

The term $ \|\partial_z^{k-1}(\beta\cdot\na_x\partial_x v)\|_{H^{\mu}} $ can be treated similarly except that we now have 
\[
\begin{aligned}
    \|\p_z^{k-1-j}\na_x\p_zv\|_{W^{\mu,p'_j}}& \les \|\p_z^{k-1-j}\na_x\p_zv\|_{H^{\mu +j}_{x, z}}\les  \|\p_z^{k-1-j}\p_zv\|_{H^{\mu+j+1}_{x, z}}\\
    &\les   \|\p_z^{k-1-j}\p_zv\|_{H^\mu_{x, z}}+ \|\na_{x, z}^{j+1}\p_z^{k-1-j}\p_zv\|_{H^{\mu}_{x, z}}\\
    &\les \| \p_zv\|_{H^{k+\mu-1}}+ \Xi+A_k.
\end{aligned}
\]
Hence, 
\[
    \|\partial_z^{k-1}(\beta\cdot\na_x\partial_x v)\|_{H^{\mu}} \les \Xi+A_{k}+\| \na_xv\|_{H^{k+\mu-1}}.
\]

For the $\gamma$ term, we  need to control for any $j\in\{0, \dots, k-1\}$ the quantities $\|\p_z^j\gamma_1\|_{L^{p_j}}$ and $\|\p_z^{k-1-j}\p_zv\|_{W^{\mu,p'_j}}$ where $\frac{1}{p_j}+\frac{1}{p'_j}=\mez$ and $p'_j\neq \infty$. 
\medskip 

\underline{Case $0\leq j\leq s-2$.} In this case, we have $s-\tdm-j\ge \mez>0$, so we can choose 
\[
    \begin{cases}
        \frac{1}{p_j}=\frac{1}{2}-\frac{s-\tdm -j}{d+1}\quad\text{if~}s-\tdm-j< \frac{d+1}{2},\\
        p_j=\infty\quad\text{if~} s-\tdm-j> \frac{d+1}{2},\\
        2\ll  p_j<\infty\quad\text{if~} s-\tdm-j= \frac{d+1}{2},
    \end{cases}
\]
in order to ensure that $H^{s-\frac{3}{2}-j} \hookrightarrow L^{p_j}$ and $H^{\mu +j+1-\varepsilon}\hookrightarrow W^{\mu,p'_j}$ for some $\eps>0$. It follows that $\|\p_z^j\gamma\|_{L^{p_j}} \le C(\|f\|_{H^s})$ thanks to \eqref{dzgamma:L2}, which is applicable because $s-\tdm - j \geq 0$. 
Hence, for all $\nu\in (0, 1)$, 
\begin{align*}
    \|\p_z^{k-1-j}\p_zv\|_{W^{\mu,p'_j}} &\lesssim \|\p_z^{k-1-j}\p_zv\|_{H^{\mu +j+1-\varepsilon}}\\
    &\leq C(\nu)\|\p_z^{k-j}v\|_{H^\mu}+ \nu \|\p_z^{k-j}v\|_{H^{\mu +j+1}} \\  
    & \leq C(\nu)\|\p_z^{k-j}v\|_{H^\mu} + \nu \sum_{m=0}^j\|\p_z^{k-j+m}\na_x^{j+1-m}v\|_{H^{\mu}} + \nu \|\p_z^{k+1} v\|_{H^{\mu}_{x,z}},
\end{align*}
where 
\[
\sum_{m=0}^j\|\p_z^{k-j+m}\na_x^{j+1-m}v\|_{H^{\mu}}\les \Xi+A_k
\]
by \eqref{claim:Pl}. Therefore, we arrive at 
\[
    \|\p_z^{k-1-j}\p_zv\|_{W^{\mu,p'_j}} \le C(\| f\|_{H^s}, \nu)\Big(\Xi+ A_k + \|\na_{x,z}v\|_{H^{k +\mu -1}} \Big) + \nu \|\p_z^{k+1}v\|_{H^{\mu}_{x,z}}\quad\forall \nu\in (0, 1).
\]
\medskip 

\underline{Case $j=k-1$.} This case is the only one possibly not covered by the previous case. Since $\mu \in (0, 1)$, we choose $\frac{1}{p_j}=\frac{1}{2}-\frac{\mu}{d+1}$, so that Sobolev's embedding and \eqref{eq.pzEstim1c} yield 
\[
    \|\p^{k-1}_z\gamma\|_{L^{p_j}}\les  \|\p^{k-1}_z\gamma\|_{H^\mu}\le  C(\|f\|_{H^s})(\|f\|_{H^{k+\mu+\mez}}+a).
\]
We also have $H^{\mu + \frac{d+1}{2}-\mu}\hookrightarrow W^{\mu,p'_j}$, so that 
\[
    \|\p_zv\|_{W^{\mu,p'_j}} \les \|\p_zv\|_{H^{s-\mez}} 
\]
which is controlled by \eqref{nav:Xs-mez:1}. 

From the above two cases, we deduce
\[
    \|\p_z^{k+1}v\|_{H^{\mu}_{x,z}} \le   C(\| f\|_{H^s}, \nu)\Big(\Xi +  A_k + \|\na_{x,z}v\|_{H^{k +\mu -1}}\Big) + \varepsilon \|\p_z^{k+1}v\|_{H^{\mu}_{x,z}}\quad\forall \nu\in (0, 1).
\]
This implies 
\[
    \|\p_z^{k+1}v\|_{H^{\mu}_{x,z}} \les \Xi +  A_k + \|\na_{x,z}v\|_{H^{k +\mu -1}},
\]
thereby finishing the proof of \eqref{claim:Pl-k1}. 
\end{proof}
 
\subsubsection{The finite-depth case.} \label{sec:finite-depth}
In the preceding subsection,  we have proven the estimate \eqref{eq.GS-v} for $\| \na v\|_{H^r}$  on $\Rr^d\times (z_0, 0)$ for any $z_0\in (-1, 0)$. It remains to establish \eqref{eq.GS-v} in an open neighborhood $\Omega_b$ be of $\{y=b(x)\}$. We recall that $b=-H+b_0$, where $b_0\in H^{r+\mez}(\Rr^d)$ with
\[
r\ge s-\mez>\frac{d+1}{2}.
\]
  We  make the following change variables near $\{y=b(x)\}$: 
\[
(x,y)=\ff_b(x,z)=(x,\varrho_b(x,z)):=(x,(z-1)H+e^{-\delta z\langle D_x\rangle}b_0(x)),\quad (x, z)\in \Rr^d\times [0, \tilde \delta]
\]
for small positive constants $\delta$ and $\tilde{\delta}$. Since 
\[
    |\p_z\varrho_b- H|=\delta |e^{-\delta z \langle D_x\rangle}\langle D_x\rangle b_0(x)|\le \delta K\| b_0\|_{H^{r+\mez}},\quad K=K(d, r),
\]
we choose 
\[
\delta \le   \frac{H}{2(1+K\| b_0\|_{H^{r+\mez}})},
\]
so that  $\frac{H}{2}< \p_z\varrho_b< \frac{3H}{2}$.  Hence, $\ff_b$ is injective and  
\begin{equation}
    \label{eq.uniform-separation-b}
    0<\varrho_b(x, \tilde{\delta})-\varrho_b(x, 0)\le \tilde{\delta}\frac{3H}{2}.
\end{equation}
Then we choose $\tilde{\delta}< \frac{2\fd}{3H}$ to ensure $\varrho_b(x, \tilde{\delta})<f(x)$, i.e. $\ff_b(\Rr^d\times (0, \tilde{\delta})) \subset \Omega_f$. 
The above change of variables $\ff_b$ is of the same form as $\ff_f$ defined by \eqref{def:diffeo}--\eqref{def:vr}. In particular, the function $w\in \{w^{(1)}, w^{(2)}\}$, $w^{(j)}=\phi^{(j)}\circ \ff_b$, satisfies 
\[
   \operatorname{div}(\mathcal{A}\na_{x,z}w)=F\quad\text{and}\quad  \p_z^2 w + \alpha \p_x^2 + \beta \cdot \na_x\p_z w + \gamma\p_zw= F_0 ,     
\]
where 
\[
    F= - (\p_z\varrho)\p_yk\circ \ff_b = - \p_z(k\circ \ff_{b}),\quad F_0=-\frac{(\p_z\varrho_b)^2}{1+|\na_x\varrho_b|^2}\p_yk\circ \ff_b
\]
and
\begin{equation}
    \label{def:Arho}
    \mathcal{A}=\begin{bmatrix} \p_z\varrho_b I_d & -\na_x \varrho_b \\ -(\na_x\varrho_b)^T & \frac{1+|\na_x \varrho_b|^2}{\p_z\varrho_b}\end{bmatrix}. 
\end{equation}
For $b_0\in H^{r+\mez}(\Rr^d)$ with $r\ge s-\mez>\frac{1+d}{2}$, we have 
\bq\label{cA-cA0}
\| \cA-\cA_0\|_{H^r_{x, z}} \le C(\| b_0\|_{H^{r+\mez}}),\quad \cA_0:=\begin{bmatrix}
    HI_d& 0 \\ 
    0 & \frac{1}{H}
    \end{bmatrix}.
\eq
The function $w$  satisfies the boundary condition 
\[
    (\mathcal{A}\na w)(x, 0) \cdot e_{d+1} =-\na_{x, y}\phi\vert_{y=b(x)}\cdot (\na_x b, -1)= k(x, b(x))=:g(x).
\]
In the following, we will use the bound $\|\mathcal{A}\|_{L^{\infty}(\mathbb{R}^d \times [0,1])}\le  C_1=C_1(H, \| b_0\|_{H^{s}})$, which follows from \eqref{cA-cA0}.  Let $\varepsilon_0$ be the constant in \cref{thm.elliptic-CS} associated with the domain $V_0:=B(0,1)\times [0,1] \subset \mathbb{R}^{d+1}$ and the positive constant matrix $A=\cA_0$. Let us also fix 
\[
0<\nu < \min\{1, s-\frac{d+1}{2}\}\quad\text{and}\quad 0<\eta <\min\{\frac{\tilde{\delta}}{2}, 1\}.
\]
Then we consider the domains 
\begin{equation}\label{def:Cnr}
    C_{n, \eta}:=B( \eta n,\eta)\times (0, \eta)\subset C'_{n, \eta}:=B(\eta n,2\eta) \times (0,2\eta),\quad n\in\mathbb{Z}^d, 
\end{equation}
the matrices $A_{n, \eta}:=\mathcal{A}(\eta n,\frac{\eta}{2})$ and $a_{n, \eta}(x,z):=\mathcal{A}(x,z)-A_{n, \eta}(x,z)$. We can check that 
\bq\label{verify:smalla}
  \forall (x, z)\in C'_{n, \eta},~  |a_{n, \eta}(x,z)| \le |(x,z)-(rn, \frac{\eta}{2})|^{\nu}[\mathcal{A}]_{C^{\nu}(\mathbb{R}^d)} \le \eta^{\nu}C_2,\quad C_2=C_2(H,  \| b_0\|_{H^{s}}).
    \eq
Now, we further  impose that $\eta < (\varepsilon_0/C_2)^{\frac{1}{\nu}}$, so that 
\bq\label{smallness:aneta}
\| a_{n, \eta}\|_{L^\infty(C'_{n, \eta})}< \eps_0.
\eq

Next, we fix a sequence of nested domains $C_{n,\eta}\equiv  \mathcal{U}_1 \subset \cdots \subset \mathcal{U}_p \subset C'_{n,\eta}$ as depicted on \cref{fig.3}. We choose the domains $\mathcal{U}_i$ in order that the non-flat boundary $\partial\mathcal{U}_{i+1}\setminus \{z=0\}$ is at positive distance from $\partial\mathcal{U}_i\setminus \{z=0\}$. This way, we can fix a smooth cuttof $\chi_i$ such that $\chi_i \equiv 1$ in $\cU_i$ and vanishes on $\partial\mathcal{U}_{i+1}\setminus \{z=0\}$. 
\begin{figure}
    \includegraphics[width=\textwidth]{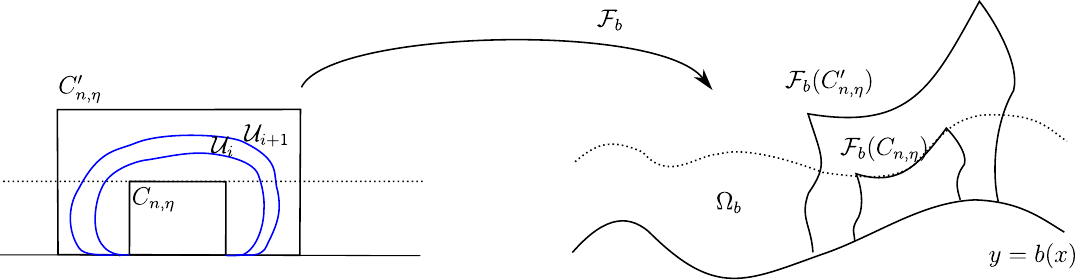}
    \caption{The domains $\mathcal{U}_i$ \label{fig.3}}
\end{figure}
We consider the function  $u=\chi_1w - \int_{C'_{n, \eta}} \chi_1w$ which has mean zero on $C'_{n, \eta}$ and satisfies the equation
\bq\label{eq:u:cU}
    \operatorname{div}((A_{n, \eta}+a_{n, \eta})\nabla u) = \chi_1 F -\na \chi_1 \cdot (\mathcal{A}\na w) - \operatorname{div}(w\mathcal{A}\na \chi_1)=:F_1 \quad \text{in~ } \mathcal{U}_2.
\eq
Moreover, $u$ satisfies the Neumann condition
\bq\label{bc:u:cU}
    (A_{n, \eta}+a_{n, \eta})\nabla u \cdot \nu \vert_{\p \mathcal{U}_2} = \chi_1\vert_{\p \cU_2\cap \{z=0\}} g -w\mathcal{A}\na \chi_1 \cdot e_{d+1}=:g_1, 
\eq
as indeed $\chi_1 \equiv 0$ in a neighborhood of $\p\mathcal{U}_2 \setminus \{z=0\}$. Since  $r>\frac{1+d}{2}$, we can apply the product rule in \cref{coro:productdomain} together with \eqref{cA-cA0} to have
\begin{align*}
 \|\operatorname{div}(w\mathcal{A}\na \chi_1)\|_{H^{r-1}(\cU_2)} &\le  \|w\mathcal{A}\na \chi_1\|_{H^r(\cU_2)} \le  C(\chi_1) \|w\mathcal{A}\na \chi_1\|_{H^r(\cU_2)} \\
 &\le C(\| \na \chi_1\|_{H^r(\cU_2)}, H) \|w\|_{H^r(\cU_2)}(\| \mathcal{A}-\cA_0\|_{H^r(\cU_2)}+1)\\
 &\le  C(\| \na \chi_1\|_{H^r(\cU_2)}, H, \| b_0\|_{H^{r+\mez}}) \|w\|_{H^r(\cU_2)}.
\end{align*}
An analogous argument yields
\[
 \|\na \chi_1 \cdot (\mathcal{A}\na w)\|_{H^{r-1}(\cU_2)}\le C(\| \na \chi_1\|_{H^r(\cU_2)}, H, \| b_0\|_{H^{r+\mez}}) \|w\|_{H^r(\cU_2)}.
 \]
Using the change of variables $(x, z)\mapsto (x-n\eta, z)$, we find that $\| \na \chi_1\|_{H^r(C'_{n, \eta})}\le C(\eta)$ independent of $n$. Thus we obtain 
\bq\label{est:F1g1}
\| F_1\|_{H^{r-1}(\cU_2)}+\| g_1\|_{H^{r-\mez}(\p\cU_2)}\le  C(H, \| b_0\|_{H^{r+\mez}})\Big(\|F\|_{H^{r-1}(\cU_2)}+\|g\|_{H^{r-\mez}(\p\cU_2\cap\{z=0 \})} + \|w\|_{H^r(\cU_2)}\Big),
\eq
where $C$ is independent of $n$.  We note that $\cU_2$ can be replaced by $C'_{n, \eta}$  in \eqref{eq:u:cU}, \eqref{bc:u:cU}, and \eqref{est:F1g1} due to the support of $\chi_1$.  In order to apply  \cref{thm.elliptic-CS} to $u$ on $C'_{n, \eta}$, it remains to verifies the smallness condition on $\| a\|_{L^\infty(C'_{n, \eta})}$.  However, since this smallness condition  depends on the domain $C'_{n, \eta}$, in order to obtain it from \eqref{verify:smalla} in a noncircular and uniform-in-$n$ manner, we perform the following rescaling.  The rescaled functions 
\begin{align*}
&\tilde{\mathcal{A}}(x,z)=\mathcal{A}(\eta(x+n), 2\eta z),\quad\tilde{u}(x,z)=u(\eta(x+n), 2\eta z),\\
&\tilde{F}_1(x, z) =\eta^2F_1(\eta(x+n), 2\eta z),\quad \tilde{g}_1(x)=\eta^2 g_1(\eta(x+n))
\end{align*}
satisfy
\begin{align*}
   & \operatorname{div}(\tilde{\mathcal{A}}\nabla \tilde{u})=\tilde{F}_1\quad \text{in~} V_0=B(0,1)\times (0,1),\\
   & \tilde{\mathcal{A}}\nabla\tilde{u} \cdot \nu = \tilde{g}_1\quad\text{on~}\p V_0. 
\end{align*}
Since  $\|\tilde{a}\|_{L^{\infty}(B(0,1)\times [0,1])} = \|a\|_{L^{\infty}(C'_{n,r})} < \varepsilon_0$ in view of \eqref{smallness:aneta}, we can apply \cref{thm.elliptic-CS} to have 
\[
    \|\tilde{u}\|_{H^{r+1}(V_0)} \le C_3\Big(\|\tilde{F}_1\|_{H^{r-1}(V_0)} + \|\tilde{g}_1\|_{H^{r-\mez}(\partial V_0)}\Big),
\]
where $C_3=C_3(H, \| b_0\|_{H^{r+\mez}})$ is independent of $n\in \Zz$. By undoing the rescaling and using \eqref{est:F1g1}, we obtain
\begin{align*}
     \|u\|_{H^{r+1}(C'_{n, \eta})} &\le  C_3\Big(\|F_1\|_{H^{r-1}(C'_{n,  \eta})} + \|g_1\|_{H^{r-\mez}(\partial C'_{n,  \eta})}\Big) \\
   & \le C_4 \Big(\|F\|_{H^{r-1}(\cU_2)}+\|g\|_{H^{r-\mez}(\p\cU_2\cap\{z=0 \})} + \|w\|_{H^r(\cU_2)}\Big),\quad C_4= C_4(H, \| b_0\|_{H^{r+\mez}}).
\end{align*}
We recall that $u= \chi_1w-\int_{C'_{n, \eta}}\chi_1 w$.  Holder's inequality gives
\[
\left|\int_{C'_{n, \eta}}\chi_1 w\right|=\left|\int_{\cU_2}\chi_1 w\right|\le  \| w\|_{L^2(\cU_2)}\eta^\frac{d+1}{2}\le \| w\|_{L^2(\cU_2)}.
\]
Since $\chi_1=1$ on $\cU_1$, we deduce that 
\bq\label{estw:cU1cU2}
  \|w\|_{H^{r+1}(\cU_1)}\le C_5 \Big(\|F\|_{H^{r-1}(\cU_2)}+\|g\|_{H^{r-\mez}(\p\cU_2\cap\{z=0 \})} + \|w\|_{H^r(\cU_2)}\Big),\quad C_5=C_5(H, \| b_0\|_{H^{r+\mez}}).
  \eq
Choosing $p=\lceil r \rceil+1$ and iterating the preceding inequality through the $\cU_j$, we obtain 
\bq\label{estw:cU1cUp}
    \|w\|_{H^{r+1}(\cU_1)}  \le C_6 \Big(\|F\|_{H^{r-1}(\cU_p)}+\|g\|_{H^{r-\mez}(\p\cU_p\cap \{z=0\})} + \|w\|_{H^1(\cU_p)}\Big),\quad C_6= C_6(H, \| b_0\|_{H^{r+\mez}}).
\eq
Since $\cU_1=C_{n, \eta}$, $\cU_p\subset C'_{n, \eta}$,  and 
\[
    \mathbb{R}^d \times (0,\ \eta) \subset \bigcup_{n\in \mathbb{Z}^d} C_{n,\eta}\quad \text{ and }\quad \forall n\in \Zz^d,~\#\{m\in\mathbb{Z}^d: C'_{m,\eta} \cap C'_{n,\eta} \neq \varnothing\} \le  7^d, 
\]
it follows from \eqref{estw:cU1cUp} that 
\bq\label{estw:Omega-b}
    \|w\|_{H^{r+1}(\mathbb{R}^d \times (0, \eta))} \le C_7 \Big(\|F\|_{H^{r-1}(\mathbb{R}^d \times (0, 2\eta))}+\|g\|_{H^{r-\mez}(\mathbb{R}^d)}+ \|w\|_{H^1(\mathbb{R}^d \times (0, 2\eta))}\Big),
\eq
where $C_7= C_7(H, \| b_0\|_{H^{r+\mez}})$. Combining the trace inequality, \cref{lemm.FaaDiBruno} and the fact that $r>\frac{d+1}{2}$, we find 
\bq\label{est:Fg:bot}
\begin{aligned}
    \|F\|_{H^{r-1}(\mathbb{R}^d \times (0,2\eta))}+\|g\|_{H^{r-\mez}(\mathbb{R}^d \times (0, 2\eta))} & \le   \| k\circ \ff_b\|_{H^r(\mathbb{R}^d \times (0,2\eta))}  \le  C(H, \| b_0\|_{H^{r+\mez}})\|k \|_{H^r(\Omega_f)}.
\end{aligned}
\eq
Next, we recall from \cref{lemm.variationalEstimate} that in the finite-depth case, we have 
\bq\label{variest:wj}
    \begin{aligned}
        &  \| \phi^{(1)}\|_{H^1(\Omega_f)}\le C(\|\na f\|_{L^\infty}, \| \na b\|_{L^\infty})\| f\|_{H^\mez},\\
        & \| \phi^{(2)}\|_{H^1(\Omega_f)}\le \left(1+\| f- b\|_{L^\infty}\right)\| k\|_{L^2(\Omega_f)}.
    \end{aligned}
\eq
Since $r\ge s-\mez >\frac{d+1}{2}$, we have  $\na_{x, z}\varrho_b\in L^\infty(\Rr^d\times (0, \tilde{\delta}))$, and hence \eqref{variest:wj} implies
\bq\label{est:wjbot}
    \begin{aligned}
        &  \|w^{(1)}\|_{H^1(\mathbb{R}^d \times (0, \tilde{\delta}))}\le C(\|f\|_{H^s},\|b_0\|_{H^s})\| f\|_{H^\mez},\\
        & \|w^{(2)}\|_{H^1(\mathbb{R}^d \times (0, \tilde{\delta}))}\le C(\|f\|_{H^s},\|b_0\|_{H^s})\| k\|_{L^2(\Omega_f)}.
    \end{aligned}
\eq
It follows from \eqref{estw:Omega-b}, \eqref{est:Fg:bot}, and \eqref{est:wjbot} that
\bq\label{est:wHr}
\|w\|_{H^{r+1}(\mathbb{R}^d \times (0, \eta))} \le  C(H, \| b_0\|_{H^{r+\mez}}, \|f\|_{H^s})\Big(\| \zeta\|_{H^\mez}+\|k\|_{H^r(\Omega_f)}\Big),
\eq 
Next, we transfer the regularity of $w$ to the regularity of $u=\phi\circ \ff_f$, where $\phi\in \{\phi^{(1)}, \phi^{(2)}\}$ and $\ff_f(x, z)=(x, \varrho(x, z))$, $z\in [-1, 0]$. For sufficiently small $\mu>0$ we have $\varrho(x, -1+\mu)<\varrho_b(x, \eta)$ for all $x\in \Rr^d$, so that 
\[
\cF(\Rr^d\times (-1, -1+\mu)) \subset \ff_b(\Rr^d\times (0, \eta))
\]
and $u=\phi\circ \ff_f=w\circ \ff_b^{-1}\circ \ff_f$ on $\Rr^d\times (-1, -1+\mu)$. Then, we can apply \cref{lemm.FaaDiBruno}, \cref{lemm.diffeoSobolev}, and \eqref{est:wHr} to deduce 
\bq
\begin{aligned}
\| u\|_{H^{r+1}(\Rr^d\times (-1, -1+\mu))}&\le C(\| b_0\|_{H^{r+\mez}}, \|f\|_{H^s})\Big(\| w\|_{H^{r+1}(\Rr^d\times (0, \eta))}   \\
&\qquad  +\| w\|_{H^{s+\mez}(\Rr^d\times (0, \eta))}(\|f\|_{H^{r+\mez}} +\| b_0\|_{H^{r+\mez}})\Big)\\
&\le C(\| b_0\|_{H^{r+\mez}}, \|f\|_{H^s})\Big(\|k\|_{H^r(\Omega_f)}+\|\zeta\|_{H^\mez}  \\
&\qquad +(\|k\|_{H^{s-\mez}(\Omega_f)}+\|\zeta\|_{H^\mez})(\|f\|_{H^{r+\mez}}+\| b_0\|_{H^{r+\mez}})\Big)
\end{aligned}
\eq
which is controlled by the right-hand side of \eqref{eq.GS-v}. This completes the proof of \cref{prop.GSest-v} in the finite-depth case.

\section{A priori estimates}\label{sec.apriori}

In this section, we suppose that $(f, g)$ is a solution of \eqref{eq:f}-\eqref{eq:g} such that 
\bq\label{apriori:reg:0}
\begin{aligned}
f \in L^{\infty}([0, T], H^s(\Rr^d)) \cap L^2([0, T], H^{s+\frac{1}{2}}(\Rr^d)),\quad  g\in L^{\infty}([0, T], H^s(\Omega_{f(\cdot)})
\end{aligned}
\eq
with $s>\tdm+\frac{d}{2}$. Our goal  is to establish \textit{a priori} estimates for $f$ and $\tilde{g}=g\circ \ff_f$, where $\ff_f(x, z)=(x, \varrho(x, z))$ as in \eqref{def:diffeo}-\eqref{def:vr}. As a consequence of \cref{lemm.FaaDiBruno} and \cref{lemm:varrho}, we have
\bq\label{gof:Hs}
\| h\circ \ff_{f}\|_{H^s(\Rr^d\times J)}\le C(\|f\|_{H^{s-\mez}}) \| h\|_{H^s(\Omega_f)}.
\eq
On the other hand, combining \cref{lemm.FaaDiBruno}, \cref{lemm:varrho}, and \cref{lemm.diffeoSobolev}, we obtain 
\bq\label{gofi:Hs}
\| h\|_{H^s(\Omega_f)}\le C(\|f\|_{H^{s-\mez}}) \| h\circ \ff_f\|_{H^s(\Rr^d\times J)}.
\eq
Therefore, $\| \tilde g\|_{H^s(\Rr^d\times J)}$ and $\| g\|_{H^s(\Omega_f)}$ are comparable up to multiplicative constants depending only on $\|f\|_{H^{s-\mez}}$, and we have 
\bq\label{apriori:reg}
 \tilde g\in  L^\infty([0, T]; H^s(\Rr^d\times J)).
\eq
\subsection{Estimates for $g$}\label{sec.apriori-g} 

Since $g(x, y, t)$ solves \eqref{eq:g} in $\Omega_f$, using the chain rule we find that $\tilde g(x, z, t)$ satisfies 
\bq\label{eq:tildeg}
\begin{cases}
\p_t \tilde g+\ol{u}\cdot \na_{x, z}\tilde{g}+\gamma'(\varrho)u_y\circ \ff_f=0, \quad (x, z)\in \Rr^d\times J,\\
\tilde{g}\vert_{t=0}=\tilde{g}_0:=g_0\circ \ff_{f_0},
\end{cases}
\eq
where $u=(u_x, u_y)$ and 
\bq\label{tildeu}
\ol{u}=\left(u_x \circ \ff_f,  \frac{1}{\partial_z \varrho} (u_y \circ \ff_f-\na_x\varrho \cdot u_x\circ \ff_f - \partial_t \varrho)\right)=:(\ol{u}_x, \ol{u}_z). 
\eq
Using \eqref{eq:u}, we can rewrite $\ol{u}_z$ in terms of $\cG$ and $\cS$ as
\bq\label{oluz:GS}
\ol{u}_z=\frac{1}{\p_z\varrho}\Big\{-(-\na_x\varrho, 1)\cdot \big(\cG[f]\Gamma(f)+\cS[f]g\big)\circ \ff_f -g\circ\ff_f- \p_t\rho\Big\}.
\eq
 We have 
\[
\ol{u}_z(\cdot, 0)= \frac{1}{\partial_z \varrho}\vert_{z=0}\left(u_y\vert_{\Sigma_f}-u_x\vert_{\Sigma_f}\cdot \na_x f-\p_tf\right)=0
\]
in view of \eqref{eq.kin}. Moreover, in the finite depth case, we have $J=(-1, 0)$ and 
\[
\ol{u}_z(\cdot, -1)= \frac{1}{\partial_z \varrho}\vert_{z=-1}\left(u_y\vert_{\Sigma_b}-u_x\vert_{\Sigma_b}\cdot \na_x b\right)=0
\]
by virtue of \eqref{bc:bottom} and because $\p_t f \vert_{z=-1}=0$. Therefore, $\ol{u}$ is tangent to the boundary of $\Rr^d\times J$:
\bq\label{tildeu:trace}
\ol{u}_z=0\quad\text{on } \p (\Rr^d\times J).
\eq
In order to apply the transport estimate \eqref{transportest} in \cref{theo:transport}, we need the following regularity for $\ol{u}$. 
\begin{lemm}
For $\ol{u}$ given terms of $f$ and $g$ as in \eqref{tildeu}, we have
\begin{align}\label{esttildeu:h}
& \| \ol{u}\|_{H^s}\le C(\| f\|_{H^s})\left(\| f\|_{H^{s+\mez}}+(1+\| f\|_{H^{s+\mez}})\| \tilde g\|_{H^s}+\| \p_t f\|_{H^{s-\mez}}\right),\\ \label{esttildeu:l}
 &\| \ol{u}\|_{H^{s-\mez}}\le C(\| f\|_{H^s})\left(\| f\|_{H^s}+\| \tilde g\|_{H^s} +\| \p_t f\|_{H^{s-1}}\right).
\end{align}
Consequently, if $f$ satisfies \eqref{eq:f} then
\begin{align}\label{tildeu:h}
& \| \ol{u}\|_{H^s}\le C(\| f\|_{H^s})\left(\| f\|_{H^{s+\mez}}+(1+\| f\|_{H^{s+\mez}})\| \tilde g\|_{H^s}\right),\\ \label{tildeu:l}
 &\| \ol{u}\|_{H^{s-\mez}}\le C(\| f\|_{H^s})\left(\| f\|_{H^s}+\| \tilde g\|_{H^s}\right)
\end{align}
for almost every $t\in [0, T]$. 
\end{lemm}
\begin{proof}
The estimates \eqref{estu:Hs}, \eqref{estu:Hs-mez}, \eqref{gof:Hs}, and \eqref{gofi:Hs} together imply 
\begin{align}\label{uof:h}
& \| u\circ \ff_f\|_{H^s}\le C(\| f\|_{H^s})\left(\| f\|_{H^{s+\mez}}+(1+\| f\|_{H^{s+\mez}})\| \tilde g\|_{H^s}\right),\\ \label{uof:l}
 &\| u\circ\ff_f\|_{H^{s-\mez}}\le C(\| f\|_{H^s})\left(\| f\|_{H^s}+\| \tilde g\|_{H^s}\right).
 \end{align}
 These gives the desired estimates for $\ol{u}_x$. 
 
Next, we recall from \eqref{varrho:H} and \eqref{varrho:H:fd} that $\| (\na_x\varrho, \p_z\varrho-m)\|_{H^\sigma}\les \| f\|_{H^{\sigma+\mez}}+a$ for suitable constants $m$ and $a$. Then, using the tame product rule \eqref{tamepr} and Sobolev emebddding yields 
\begin{align}
&\| \frac{1}{\partial_z \varrho} (u_y \circ \ff_f-\na_x\varrho \cdot u_x\circ \ff_f)\|_{H^s}\le \left(\| f\|_{H^{s+\mez}}+(1+\| f\|_{H^{s+\mez}})\| \tilde g\|_{H^s}\right),\\
& \| \frac{1}{\partial_z \varrho} (u_y \circ \ff_f-\na_x\varrho \cdot u_x\circ \ff_f)\|_{H^{s-\mez}}\le C(\| f\|_{H^s})\left(\| f\|_{H^s}+\| \tilde g\|_{H^s} \right).
\end{align}
From this and the estimate  $\| \p_t \varrho\|_{H^\sigma}\les \| \p_t f\|_{H^{\sigma-\mez}}$, we obtain \eqref{esttildeu:h} and \eqref{esttildeu:l}. 

Finally, if $f$ satisfies \eqref{eq:f} then $\p_tf =u\cdot N\vert_{\Sigma_f}$ and \eqref{tildeu:h} and \eqref{tildeu:l} follow from \eqref{esttildeu:h}, \eqref{esttildeu:l}, and \cref{coro:traceuN}. 
\end{proof}
Adopting  the notation $L^p_TX=L^p([0, T], X)$, we deduce from  \eqref{tildeu:h} and \eqref{tildeu:l} that 
\begin{align}\label{energyest:u}
&\|\ol{u}\|_{L^1_TH^s}\le C(\| f\|_{L^\infty_TH^s})\left(T^\mez\| f\|_{L^2_TH^{s+\mez}}+T^\mez\| f\|_{L^2_TH^{s+\mez}}\| \tilde g\|_{L^\infty_TH^s}+T\| \tilde g\|_{L^\infty_TH^s}\right),\\ \label{energyest:u:l}
&\|\ol{u}\|_{L^\infty_TH^{s-\mez}}\le C(\| f\|_{L^\infty_TH^s})\left(\| f\|_{L^\infty_TH^s}+\| \tilde g\|_{L^\infty_TH^s}\right).
\end{align}
Next, we estimate the forcing term $\gamma'(\varrho)u_y\circ \ff_f$ in \eqref{eq:tildeg}. At this point we use the condition \eqref{cont:gamma} for $\gamma$ with $h=\varrho-az$, where $a=1$ for infinite depth and $a=H$ for finite depth. In conjunction with \eqref{varrho:H} and \eqref{varrho:H:fd}, this yields 
\bq
\| \gamma'(\varrho)\|_{H^s}\le C_\gamma(\| f\|_{H^{s-\mez}}).
\eq
Combining this with  \eqref{uof:h} gives 
\bq\label{estforcing:tildeg}
\|\gamma'(\varrho)u_y\circ \ff_f\|_{H^s}\le C(\| f\|_{H^s})\left(\| f\|_{H^{s+\mez}}+(1+\| f\|_{H^{s+\mez}})\| \tilde g\|_{H^s}\right).
\eq
In view of \eqref{tildeu:trace}, \eqref{energyest:u},  \eqref{energyest:u:l}, \eqref{estforcing:tildeg}, and the assumption $s>1+\frac{d+1}{2}$, we can apply \cref{theo:transport} to deduce that $\tilde{g}$ is the unique solution of \eqref{eq:tildeg} and satisfies 
\bq\label{est:tildeg}
\begin{aligned}
\| \tilde g\|_{L^\infty_TH^s}&\le  \left\{\tM \| \tilde{g}_0\|_{H^s}+ C(\| f\|_{L^\infty_TH^s})\left(T^\mez\| f\|_{L^2_TH^{s+\mez}}+T^\mez\| f\|_{L^2_TH^{s+\mez}}\| \tilde g\|_{L^\infty_TH^s}+T\| \tilde g\|_{L^\infty_TH^s}\right)\right\}\\
&\qquad \cdot \exp\left(C(\| f\|_{L^\infty_TH^s})\left(T^\mez\| f\|_{L^2_TH^{s+\mez}}+T^\mez\| f\|_{L^2_TH^{s+\mez}}\| \tilde g\|_{L^\infty_TH^s}+T\| \tilde g\|_{L^\infty_TH^s}\right)\right),
\end{aligned}
\eq
where $\tM=\tM(d, s)$.  This concludes the {\it a priori} estimate for $\tilde{g}$. 
\subsection{Estimates for $f$}\label{sec.apriori-f}

We recall the equation  \eqref{eq:f} for $f$: 
\bq\label{eq:f:2}
\p_tf=-G[f]\Gamma(f)-(N\cdot\cS[f]g+g)\vert_{\Sigma_{f(t)}}.
\eq
 Using \cref{theo:DN1} (ii) to paralinearize $G[f]\Gamma(f)$, we find
\[
  G[f]\Gamma(f) = T_{\lambda[f]}(\Gamma(f)-T_Bf) - T_V\cdot \nabla f+ R[f]\Gamma(f).
  \]
Here, $\ld[f]$ is given by \eqref{def:ld},
 \bq\label{BV:2}
 V =\gamma(f) \na f - B\nabla f,\quad B=\frac{\gamma(f)|\nabla f|^2+ G[f]\Gamma(f)}{1+|\nabla f|^2},
\eq
and the remainder $R[f]\Gamma(f)$ satisfies
\bq\label{est:RfGf}
\begin{aligned}
 \| R[f]\Gamma(f)\|_{H^{s-\mez}}&\le C(\| f\|_{H^s})(1+\| f\|_{H^{s+\mez-\delta}})\| \Gamma(f)\|_{H^s}\\
 &\le  C(\| f\|_{H^s})(1+\| f\|_{H^{s+\mez-\delta}})\| f\|_{H^s},
 \end{aligned}
 \eq
 where 
 \bq\label{def:delta}
 \delta\in (0, \mez],\quad \delta< s-\frac{d}{2}-1.
 \eq
Applying \cref{paralin:nonl}, we paralinearize the nonlinear term $\Gamma(f)$ as  $\Gamma(f)=T_{\gamma(f)}f+R_1$, where 
\bq\label{est:R1}
\| R_1\|_{H^{s+1+\delta}}\le C(\| f\|_{H^s})\| f\|_{H^s}.
\eq
Thus we obtain
\bq\label{expand:DN:1}
\begin{aligned}
  G[f]\Gamma(f)& = T_{\lambda[f]}(T_{\gamma(f)}f-T_Bf) - T_V\cdot \nabla f+ R[f]\Gamma(f)+ T_{\lambda[f]}R_1\\
  & =T_{\lambda[f] (\gamma(f)-B)}f- T_V\cdot \nabla f+ R[f]\Gamma(f)+ T_{\lambda[f]}R_1\\
  &\qquad+[T_{\lambda[f]}T_{\gamma(f)-B}-T_{\ld[f](\gamma(f)-B)}]f.
  \end{aligned}
 \eq
 Applying \cref{theo:DN1} (i) with $\sigma=s$ gives  
\bq\label{est:BV}
\| (V, B)\|_{H^{s-1}}\le C(\| f\|_{H^s})\| f\|_{H^s}.
\eq
Therefore, $\ld[f]\in \Gamma^1_1$ and $\gamma(f)-B\in \Gamma^0_1$ with seminorms bounded by $C(\| f\|_{H^s})$.  Then, by virtue of \cref{theo:sc} (i) and  \eqref{est:R1}, we have 
 \bq\label{est:lamb}
 \| T_{\lambda[f]}R_1\|_{H^{s+\delta}}\le C(\| f\|_{H^s})\| f\|_{H^s}.
 \eq
 On the other hand, \cref{theo:sc} (ii) implies that $T_{\lambda[f]}T_{\gamma(f)-B}-T_{\ld[f](\gamma(f)-B)}$ is of order $0$ and 
\bq\label{est:commR}
\| [T_{\lambda[f]}T_{\gamma(f)-B}-T_{\ld[f](\gamma(f)-B)}]f\|_{H^s}\le C(\| f\|_{H^s})\| f\|_{H^s}.
\eq 
In view of \eqref{expand:DN:1}, \eqref{est:RfGf}, \eqref{est:R1}, \eqref{est:lamb}, and \eqref{est:commR}, we deduce
\bq\label{expandDBN:2}
G[f]\Gamma(f)=T_{\lambda[f] (\gamma(f)-B)}f-T_V\cdot \nabla f+R_2,
\eq
where
\begin{align}\label{est:R2}
\| R_2\|_{H^{s-\mez}}\le C(\| f\|_{H^s})(1+\| f\|_{H^{s+\mez-\delta}})\| f\|_{H^s}.
\end{align}
Thus \eqref{eq:f:2} rewrites 
\bq\label{eq:f:3}
\p_tf=-T_{\lambda[f] (\gamma(f)-B)}f+T_V\cdot \nabla f-R_2-(N\cdot\cS[f]g+g)\vert_{\Sigma_{f(t)}}.
\eq
In the preceding equation, $T_V\cdot \nabla$ is a transport operator and and $-R_2-(N\cdot\cS[f]g+g)\vert_{\Sigma_{f(t)}}$ will be treated as a forcing term in $L^2_t H^{s-\mez}_x$. Then, we can close the $L^\infty_t H^s_x\cap L^2_t H^{s+\mez}_x$ a priori estimate for $f$ provided $T_{\lambda[f] (\gamma(f)-B)}$ is an elliptic operator, i.e. $\gamma(f)-B>0$. Since $\ld[f]$ is an elliptic symbol  and 
\bq
  \gamma(f)-B=\frac{\gamma(f)-G[f]\Gamma(f)}{1+|\na f|^2},
\eq
$\lambda[f] (\gamma(f)-B)$ is elliptic if and only if
\bq\label{ellipticity}
\inf_{x\in \Rr^d}\cT(f)(x)\equiv \inf_{x\in \Rr^d}\big(\gamma(f)(x)-G[f]\Gamma(f)(x)\big)>0.
\eq
The following lemma provides sufficient conditions  for \eqref{ellipticity}.

\begin{lemm}\label{lemm.initial-taylor}
Assume that $\inf \gamma>0$ and $\gamma'\ge 0$. If $f\in H^\sigma(\Rr^d)$ with $\sigma>1+\frac{d}{2}$, then for all $x\in\Rr^d$ there holds:
\bq\label{upperbound:DN}
  \cT(f)(x)= \gamma(f(x))-G[f]\Gamma(f)(x)>0.
\eq
\end{lemm}
\begin{proof}
Let $\phi$ be the solution of the problem
\bq
\Delta_{x, y}\phi=0\quad\text{in } \Omega_f,\quad \phi\vert_{y=f(x)}=\Gamma(f)
\eq
supplemented with the condition $\lim_{(x, y)\to \infty}\na_{x, y}\phi= 0$ for infinite  depth and $\p_\nu \phi\vert_{y=b(x)}=0$ for finite depth. Then $w:=\phi-\Gamma(y)$ satisfies 
\bq
\Delta_{x, y}w=-\gamma'(y)\le 0\quad\text{in } \Omega_f,\quad w\vert_{y=f(x)}=0
\eq
and $\p_\nu w\vert_{y=b(x)}=(1+|\na b|^2)^{-\mez}\gamma(b(x))>0$ for finite depth, while for  infinite depth we have
\[
\lim_{(x, y)\to \infty}w(x, y)=-\lim_{(x, y)\to \infty}\Gamma(y)=\infty
\]
since $\Gamma(y)\le (\inf \gamma)y$ for  $y<0$. Since $w$ is supperharmonic, $\min w$ is thus attained on the top boundary $\{y=f(x)\}$ on which $w$ is identically zero. Using Hopf's lemma, we deduce that $\p_N w(x, f(x))<0$ for all $x\in \Rr^d$, whence for all $x\in \Rr^d$,
\[
G[f]\Gamma(f)(x)=\p_N \phi(x, f(x))<\gamma(f(x)).\qedhere
\] 
\end{proof}
Fore general $\gamma$, in order to obtain \eqref{ellipticity} for $f(t)$, $t\in [0, T]$, we assume that it holds initially, i.e.
\bq\label{Taylor:data}
\gamma(f_0)-G[f_0]\Gamma(f_0)\ge 2\mathfrak{a}>0,
\eq
and propagate it using  the following lemma.
\begin{lemm}\label{lemm.taylor}
Assume that $f\in C([0, T]; H^{s-\mez})$. There exists $C: \Rr^+\to \Rr^+$ depending only on $(s, d, \gamma, \fd)$ such that for all $t\in [0, T]$, there holds
\bq\label{propagation:Taylor}
\begin{aligned}
&\| \cT(f(t))-\cT(f_0)\|_{H^{s-\tdm}}\\
&\le t^\mez C\big(\|f\|_{L^\infty([0, t]; H^s)\cap C([0, t]; H^{s-\mez})}\big)\left\{\| f \|_{L^2([0, t],H^{s+\mez})}\left(1+\|\tilde g\|_{L^\infty ([0, t]; H^s)}\right)+t^\mez\|\tilde g\|_{L^\infty ([0, t]; H^s)} \right\}.
\end{aligned}
\eq
\end{lemm}
\begin{proof}
We first apply \cref{theo:DN1} (i) and \cref{thm.NPdiff} (ii) with $\sigma=s-\mez$ to have 
\bq\label{TR:proof:1}
\begin{aligned}
\| \cT(f(t))-\cT(f_0)\|_{H^{s-\tdm}} &\le C\big(\|(f(t), f_0) \|_{H^{s-\mez}}\big)\| f(t)-f(0)\|_{H^{s-\mez}}\\
&\le C\big(\|(f(t), f_0) \|_{H^{s-\mez}}\big)t^\mez \| \p_t f\|_{L^2([0, t]; H^{s-\mez})}.
\end{aligned}
\eq
Since  $\p_tf=u\cdot N\vert_{\Sigma_f}$,  the estimate \eqref{est:traceuN:h} implies
\bq\label{eq.fptBound}
\begin{aligned}
&\| \p_t f\|_{L^2([0, t]; H^{s-\mez})}\\
& \le  C(\|f\|_{L^\infty([0, t]; H^s)})\left\{ \| f \|_{L^2([0, t]; H^{s+\mez})}\left(1+\|\tilde g\|_{L^\infty([0, t];  H^s)}\right)+t^\mez\|\tilde g\|_{L^\infty ([0, t]; H^s)}\right\}.
\end{aligned}
\eq
Therefore,  \eqref{propagation:Taylor} follows from \eqref{TR:proof:1} and \eqref{eq.fptBound}. 
\end{proof}

Note that so far $s>1+\frac{d}{2}$ suffices to obtain the above estimates. The main result of this section is the following.

\begin{prop}\label{prop.uniform-f}
Let $s>\tdm+\frac{d}{2}$. Assume that $(f, g)$ is a  solution of \eqref{eq:f}-\eqref{eq:g} on $[0, T]$ such that 
\[
  (f,g) \in \left(L^{\infty}([0, T], H^s) \cap L^2([0, T], H^{s+\frac{1}{2}}) \right)\times L^{\infty}([0, T], H^s(\Omega_{f(\cdot)})
\]
and
\bq\label{RT:apriori}
\inf_{(x, t)\in \Rr^d\times [0, T]}\cT(f(t))(x)\ge \fa>0.
\eq
In the finite depth case we assume additionally that
\bq\label{separate:apriori}
\inf_{x\in \Rr^d}\left(f(x)-b(x)\right)\ge \fd>0.
\eq
Suppose that 
\bq
\label{cd:bound0}
\| f\|_{L^\infty_TH^s}\le B_f,\quad \|\tilde g\|_{L^\infty_TH^s}\le B_g.
\eq
There exist  functions $C_0$, $C_1: (\Rr_+)^2\to \Rr_+$  and $\cF_0$, $\cF:(\Rr_+)^3\to \Rr_+$, depending only on $(s, d, b, \gamma)$ and non-decreasing in each variable, such that if 
\bq\label{cd:MN:1}
C_0(B_f, \fd)B_g\le \frac{1}{4\cF_0(B_f, \fa, \fd)}
\eq
then
\bq\label{aprioriestimate}
\| f\|_{L^\infty_TH^s\cap L^2_TH^{s+\mez}}\le\cF\left(\| f_0\|_{H^s}+T^\mez\cF\left(\| f\|_{L^\infty_TH^s}+\| \tilde g\|_{L^\infty_TH^s} , \fa, \fd\right),  \fa, \fd\right)
\eq
and, for $T\le 1$,
\bq\label{aprioriestimate:g}
\begin{aligned}
\| \tilde g\|_{L^\infty_TH^s}&\le \tM\|  \tilde{g}_0\|_{H^s}\exp\left(C_1(\| f\|_{L^\infty_TH^s}, \fd)T^\mez \cA_T\right)\\
&\qquad+ C_1(\| f\|_{L^\infty_TH^s}, \fd)T^\mez \cA_T\exp\left(C_1(\| f\|_{L^\infty_TH^s}, \fd)T^\mez \cA_T\right),
\end{aligned}
\eq
where 
\bq\label{def:cAT}
 \cA_T:=\| f\|_{L^2_TH^{s+\mez}}+  \| \tilde g\|_{L^\infty_TH^s}+\| f\|_{L^2_TH^{s+\mez}} \| \tilde g\|_{L^\infty_TH^s}
 \eq
 and $\tM=\tM(d, s)$ is the constant in \eqref{transportest}.
\end{prop}
\begin{proof}
We first note that \eqref{eq:f:3} is  of the form \eqref{eq.parabolic} with 
\begin{align*}
a=\lambda[f] (\gamma(f)-B),\quad U=-V, \quad F=-R_2-(N\cdot\cS[f]g+g)\vert_{\Sigma_{f(t)}},
\end{align*}
where $a=\ld[f]\cT(f)\ge C\fa |\xi|$ by \eqref{elliptic:ld} and the assumption \eqref{RT:apriori}. We recall that 
\bq\label{norm:Va}
\| V\|_{L^\infty_TW^{\delta, \infty}}\les \| V\|_{L^\infty_TH^{s-1}}\le C(\| f\|_{H^s})\| f\|_{H^s},\quad \| a\|_{L^\infty_T\Gamma^1_\delta}\le C(\| f\|_{L^\infty_TH^s}, \fd),
\eq
where  $\delta$ is as in \eqref{def:delta}. Clearly 
\[
\| g\vert_{\Sigma_f}\|_{H^{s-\mez}_x}=\| \tilde{g}\vert_{z=0}\|_{H^{s-\mez}_x}\les \| \tilde g\|_{H^s}.
\]
 by the trace inequality for $\Rr^d\times J$.  On the other hand,  using the tame product estimate and the trace inequality again, we find
  \[
  \begin{aligned}
  \| N\cdot (\cS[f]g)\vert_{\Sigma_{f}}\|_{H^{s-\mez}}&\les (\| N-1\|_{H^{s-\mez}}+1)\| (\cS[f]g)\vert_{\Sigma_f}\|_{H^{s-1}_x}+  (\| N-1\|_{H^{s-1}}+1)\| (\cS[f]g)\vert_{\Sigma_f}\|_{H^{s-\mez}_x}\\
  &\le C(\| f\|_{H^s})\left(\| f\|_{H^{s+\mez}}\| (\cS[f]g)\vert_{\Sigma_f}\|_{H^{s-1}_x}+  \| (\cS[f]g)\vert_{\Sigma_f}\|_{H^{s-\mez}_x}\right).
  \end{aligned}
  \]
 Then we apply the estimate   \eqref{esttrace:cS}  for $\cS[f]g$ with $\sigma=s-1$ and $\sigma=s-\mez$ to obtain 
 \[
 \| N\cdot (\cS[f]g)\vert_{\Sigma_{f}}\|_{H^{s-\mez}} \le C(\| f\|_{H^s})(1+\| f \|_{H^{s+\mez}})\|\tilde g\|_{H^s}.
 \]
 Combining this with  the estimate \eqref{est:R2} for $R_2$  yields
\bq\label{est:F:Hs-mez}
\| F\|_{H^{s-\mez}}\le C_0(\| f\|_{H^s}, \fd)\left\{ (1+\| f\|_{H^{s+\mez-\delta}})\| f\|_{H^s}+(1+\| f \|_{H^{s+\mez}})\|\tilde g\|_{H^s}\right\}.
\eq
Therefore, applying  \cref{lemm:parabolicest} gives
\bq\label{energyest:f:1}
\begin{aligned}
\mez\frac{d}{dt}\| f(t)\|_{H^s}^2&\le -\frac{1}{\cF_0(\|f\|_{L^\infty_TH^s}, \mathfrak{a}, \fd)}\| f(t)\|_{H^{s+\mez}}^2\\
&\qquad+\cF_0(\| f\|_{L^\infty_TH^s}, \mathfrak{a}, \fd)\|f(t)\|_{H^s}^2+\| F(t)\|_{H^{s-\mez}}\| f(t)\|_{H^{s+\mez}},
\end{aligned}
\eq
for some increasing function $\cF_0: (\Rr_+)^3\to \Rr_+$. It follows from \eqref{est:F:Hs-mez} that 
\bq\label{est:F:Hs-mez:2}
\begin{aligned}
\| f(t)\|_{H^{s+\mez}}\| F(t)\|_{H^{s-\mez}}&\le C_0(\| f(t)\|_{H^s}, \fd)\left\{ \| f(t)\|_{H^{s+\mez}}\| f(t)\|_{H^{s+\mez-\delta}}+\| f(t)\|_{H^{s+\mez}}\| f(t)\|_{H^s}\right.\\
&\qquad\left.+\| f(t) \|_{H^{s+\mez}}^2\|\tilde g(t)\|_{H^s}+\| f(t)\|_{H^{s+\mez}}\|\tilde g(t)\|_{H^s}\right\}.
\end{aligned}
\eq
Since 
\[
 \| f\|_{H^{s+\mez}}\| f\|_{H^{s+\mez-\delta}}\les  \| f\|^{1+\tt}_{H^{s+\mez}}\| f\|_{H^s}^{1-\tt}
 \]
 for some $\tt\in (0, 1)$, we can apply Young's inequality  to deduce from \eqref{est:F:Hs-mez:2} that
 \bq\label{est:F:Hs-mez:3}
\begin{aligned}
\| f(t)\|_{H^{s+\mez}}\| F(t)\|_{H^{s-\mez}}&\le \frac{1}{4\cF_0(\|f\|_{L^\infty_TH^s}, \mathfrak{a}, \fd)}\| f(t)\|_{H^{s+\mez}}^2 + C_0(\| f\|_{L^\infty_TH^s}, \fd)\| f(t)\|_{H^{s+\mez}}^2 \|\tilde g(t)\|_{H^s}\\\
&\quad+\cF_1(\|f\|_{L^\infty_TH^s}, \mathfrak{a}, \fd)\left(\| f(t)\|^2_{H^s}+\|\tilde g(t)\|_{H^s}^2\right).
\end{aligned}
\eq
Under the assumption \eqref{cd:MN:1}, we have 
\begin{equation}
    \label{es:F:Hs-mez:4}
    C_0(\| f\|_{L^\infty_TH^s}, \fd)\| f(t)\|_{H^{s+\mez}}^2 \|\tilde g(t)\|_{H^s}\le \frac{1}{4\cF_0(\|f\|_{L^\infty_TH^s}, \mathfrak{a}, \fd)}\| f(t)\|_{H^{s+\mez}}^2.
\end{equation}
Thus it follows from \eqref{energyest:f:1}, \eqref{est:F:Hs-mez:3} and \eqref{es:F:Hs-mez:4} that 
\bq\label{energyest:f:1b}
\begin{aligned}
\mez\frac{d}{dt}\| f(t)\|_{H^s}^2&\le -\frac{1}{2\cF_0(\|f\|_{L^\infty_TH^s}, \mathfrak{a}, \fd)}\| f(t)\|_{H^{s+\mez}}^2\\
&\qquad+\cF_2(\|f\|_{L^\infty_TH^s}, \mathfrak{a}, \fd)\left(\| f(t)\|^2_{H^s}+\|\tilde g(t)\|_{H^s}^2\right).
\end{aligned}
\eq
We first discard the first term on the right-hand side and apply Gr\"onwall's lemma to have
\bq\label{energyest:f:2}
\| f\|_{L^\infty_T H^s}^2\le \left(\| f_0\|_{H^s}^2+2\cF_2T\| \tilde g\|^2_{L^\infty_TH^s}\right)\exp(2T\cF_2).
\eq
Then, we return to \eqref{energyest:f:1b}, integrate both sides, and use \eqref{energyest:f:2} to obtain
\bq\label{energyest:f:3}
\begin{aligned}
\| f\|^2_{L^2_TH^{s+\mez}}&\le \cF_0(\|f\|_{L^\infty_TH^s}, \mathfrak{a}, \fd) \left(\| f_0\|_{H^s}^2+2\cF_2T\| \tilde g\|^2_{L^\infty_TH^s}\right)\left(1+2T\cF_2\exp(2T\cF_2)\right).
\end{aligned}
\eq
 We bound  $\| f\|_{L^\infty_TH^s}$ in  $\cF_0(\|f\|_{L^\infty_TH^s}, \mathfrak{a}, \fd)$ by \eqref{energyest:f:2}, so that for some $\cF: (\Rr_+)^3\to \Rr_+$, 
\bq\label{energyest:f:4}
\begin{aligned}
\| f\|_{L^2_TH^{s+\mez}}&\le\cF\left(\| f_0\|_{H^s}+T^\mez\cF\left(\| f\|_{L^\infty_TH^s}+\| \tilde g\|_{L^\infty_TH^s} , \fa, \fd\right),  \fa, \fd\right).
\end{aligned}
\eq
Noting that the right-hand side of \eqref{energyest:f:2} is also bounded by that of \eqref{energyest:f:4}, we conclude the proof of \eqref{aprioriestimate}.

Finally, for $T\le 1$, \eqref{aprioriestimate:g} follows from  \eqref{est:tildeg}.
\end{proof}

\begin{rema}\label{rema:apriori:variant} In \cref{sec.uniform-bounds} we will need the following variant of \eqref{aprioriestimate}. Suppose that 
\[
f\in L^\infty([0, T]; H^s)\cap L^2([0, T]; H^{s+\mez+\frac{\mu}{2}})
\]
for some $\mu>0$, $s>1+\frac{d}{2}$, and $f$ satisfies \eqref{RT:apriori} and \eqref{separate:apriori} on $[0, T]$ and for all $(x, t)\in \Rr^d\times (0, T)$,
\bq\label{eq:f:reg}
    \p_t f + G[f]\Gamma(f) = R -\nu |D|^{1+\mu} f,\quad\nu\ge 0.
\eq
Then there exists $\cF: (\Rr_+)^3\to \Rr_+$ depending only on $(s, d, b, \gamma)$ (but not $\nu$) and non-decreasing in each variable, such that for $T\leq 1$ we have
\begin{equation}
    \label{eq.apriori-fR}
    \begin{aligned}
  &  \|f\|_{L^{\infty}_TH^s} + \frac{1}{\cF(\|f\|_{L^{\infty}_TH^s},\mathfrak{a},\mathfrak{d})}\|f\|_{L^2_TH^{s+\mez}} \\
  &\hspace{1cm} \le \exp\left(T\cF(\| f\|_{L^\infty_TH^s}, \mathfrak{a}, \fd)\right)\left(\|f_0\|_{H^s} + \cF(\| f\|_{L^\infty_TH^s}, \mathfrak{a}, \fd)\|R\|_{L^2_TH^{s-\mez}}\right).
  \end{aligned}
\end{equation}
We note that \eqref{eq.apriori-fR} is uniform in $\nu$ because  $-\nu|D|^{1+\mu}$ produces the good term $-\nu \| |D|^{\mez+\frac{\mu}{2}}f\|^2_{H^s}\le 0$ in the $H^s$ energy estimate for $f$. 
\end{rema}

\section{Contraction estimates}\label{sec.contraction}

We begin with the following contraction estimates for the operators $\cG[f]$ and $\cS[f]$, which will be utilized to establish contraction estimates for  solutions of \eqref{eq:f}-\eqref{eq:g}.

\begin{prop}\label{prop.diffEstimateSobolev} Consider $f_1,~f_2\in H^s(\Rr^d)$ and $k_1,~k_2\in H^s(\Omega_{f_i})$, where $f_i$ satisfy 
\[
\inf_{x\in \Rr^d}(f_i(x)-b(x))\ge \fd>0
\]
 in the finite-depth case. Let  $v_i$ be either $v_i^{(1)}:=\phi_i^{(1)}\circ \ff_{f_i}$ or  $v_i^{(2)}:=\phi_i^{(2)}\circ \ff_{f_i}$, where $\phi_i^{(1)}$ solves \eqref{eq.phi1} with data $h_i\in H^s(\Rr^d)$ and $\phi_i^{(2)}$ solves \eqref{eq.phi2} in $\Omega_{f_i}$ with data $k_i\in H^s(\Omega_{f_i})$. Set $v^{(j)}:=v_1^{(j)}-v_2^{(j)}$. 

(i) If $s>1+\frac{d}{2}$, then 
\begin{equation}\label{eq.lowNormEst-v1}
    \|\nabla_{x,z}v^{(1)}\|_{H^{s-1}} \le C(\|(f_1, f_2)\|_{H^s})\|h_1-h_2\|_{H^{s-\mez}} + C(\|(f_1,f_2,h_1,h_2)\|_{H^s})\|f_1-f_2\|_{H^{s-\mez}}.
\end{equation} 
(ii) If $s>\tdm+\frac{d}{2}$, then
\begin{multline}\label{eq.lowNormEst-v2}
    \|\nabla_{x,z}v^{(2)}\|_{H^{s-1}}  \le C(\|(f_1, f_2)\|_{H^s})\|k_2 \circ \ff_{f_2} - k_1 \circ \ff_{f_1}\|_{H^{s-1}}\\
    +C(\|(f_1, f_2)\|_{H^s})\|f_1-f_2\|_{H^{s-\mez}}\|(k_1, k_2)\|_{H^{s}}.
\end{multline}
In both \eqref{eq.lowNormEst-v1} and \eqref{eq.lowNormEst-v2}, $C$ is a non-decreasing function which depends only on $(d, s)$, and also on $\mathfrak{d}$ and $\|b_0\|_{H^{s}}$ in the finite-depth case. 
\end{prop}

The proof of \cref{prop.diffEstimateSobolev} is postponed to \cref{subsection:contrav}. 

Suppose now that $(f_1, g_1)$ and $(f_2, g_2)$ are solutions of \eqref{eq:f}-\eqref{eq:g} with the regularity \eqref{apriori:reg:0}, i.e.
\bq\label{reg:fjgj}
f_i \in L^{\infty}([0, T], H^s(\Rr^d)) \cap L^2([0, T], H^{s+\frac{1}{2}}(\Rr^d)),\quad  g_i\in L^{\infty}([0, T], H^s(\Omega_{f_i(\cdot)}),
\eq
where $s>\tdm+\frac{d}{2}$ unless  stated otherwise. We denote $\tilde{g}_j:=g_j\circ \ff_{f_j}$. Our goal in the next two  sections  is to establish contraction estimates for 
\bq
f:=f_1-f_2\quad\text{and } \tilde{g}:=\tilde{g}_1-\tilde{g}_2
\eq
in $L^\infty_t H^{s-1}\cap L^2_t H^{s-\mez}$ and $\tilde{g}$ in $L^\infty_t H^{s-1}$ respectively.  
\subsection{Estimates for $f$}\label{sec.contraction-f}

We recall from~\eqref{eq:f} that for each $i\in\{1,2\}$ we have 
\[
    \partial_tf_i + G[f_i]\Gamma(f_i)+ (N_i\cdot \cS[f_i]g_i+g_i)\vert_{\Sigma_{f_i}}=0,    
\]
so that $f:=f_1-f_2$ satisfies the  equation
\begin{equation}
    \label{eq.difference-f}
    \partial_t f + \left(G[f_1]\Gamma(f_1) - G[f_2]\Gamma(f_2)\right) = \underbrace{(N_2\cdot \cS[f_2]g_2 +g_2)\vert_{\Sigma_{f_2}}-(N_1\cdot \cS[f_1]g_1 +g_1)\vert_{\Sigma_{f_1}}}_{=:R_0},
\end{equation}
with initial data $f(0)=f_0:=f_{0,1}-f_{0,2}$. We write $R_0=R_0(f_1,f_2,g_1,g_2)$ whenever there is a need to recall the dependence of $R_0$. 

As in \cref{sec.apriori}, we start with the paralinearization of 
\[
    G[f_1]\Gamma(f_1)-G[f_2]\Gamma(f_2) = \left(G[f_1]\Gamma(f_1)-G[f_2]\Gamma(f_1)\right) +  G[f_2](\Gamma(f_1)-\Gamma(f_2)).  
\]
By Theorem  \cref{thm.NPdiff}~(i), we have
\[
    G[f_1]\Gamma(f_1)-G[f_2]\Gamma(f_1)= -T_{\lambda[f_2]B_{2,1}}(f_1-f_2) - T_{V_{2,1}}\cdot \na (f_1-f_2) + R_{\sharp}(f_1,f_2)\Gamma(f_1),
\]
where 
\bq
B_{2,i} = \frac{\na f_2 \cdot \Gamma(f_i) + G[f_2]\Gamma(f_i)}{1+|\na f_2|^2},\quad V_{2,i}=\na(\Gamma(f_i)) - B_{2,i}\na f_2.
\eq
On the other hand,  \cref{theo:DN1}~(ii) gives
\[
    G[f_2](\Gamma(f_1)-\Gamma(f_2)) = T_{\lambda[f_2]}(\Gamma(f_1)-\Gamma(f_2)) - T_{\lambda[f_2]}T_{B_{2,1}-B_{2,2}}f_2 - T_{V_{2,1}-V_{2,2}}\cdot \na f_2 + R[f_2](\Gamma(f_1)-\Gamma(f_2)). 
\]
Then, introducing the remainders  
\bq
R_\gamma(f_i):=\Gamma(f_i)-T_{\gamma(f_i)}f_i,
\eq
 we obtain
\begin{align*}
    G[f_1]\Gamma(f_1)-G[f_2]\Gamma(f_2) &= -T_{\lambda[f_2]B_{2,1}}f + T_{\lambda[f_2]}(T_{\gamma(f_1)}f_1-T_{\gamma(f_2)}f_2) - T_{\lambda[f_2]}T_{B_{2,1}-B_{2,2}}f_2 \\
    & \quad- T_{V_{2,1}-V_{2,2}}\cdot \na f_2 - T_{V_{2,1}}\cdot \na f  \\
    & \quad + R_{\sharp}(f_1,f_2)\Gamma(f_1) + T_{\lambda[f_2]}(R_{\gamma}(f_1)-R_{\gamma}(f_2)) + R[f_2](\Gamma(f_1)-\Gamma(f_2)). 
\end{align*}
Rewriting with commutators, we have 
\bq\label{paralin:G:contra}
    G[f_1]\Gamma(f_1)-G[f_2]\Gamma(f_2) = T_{\lambda[f_2](\gamma(f_2)-B_2)}f - T_{V_2}\cdot \nabla f - R', 
\eq
where $V_2:=V_{2,2}$, $B_2:=B_{2,2}$ and 
\bq
    \label{eq.R:def}
    \begin{aligned}
    R' &= -R_{\sharp}(f_1,f_2)\Gamma(f_1) - R[f_2](\Gamma(f_1)-\Gamma(f_2))- T_{\lambda[f_2]}(R_\gamma(f_1)-R_\gamma(f_2)) \\
    &\quad- T_{\lambda[f_2]}T_{\gamma(f_1)-\gamma(f_2)}f_1-(T_{\lambda[f_2]}T_{\gamma(f_2)}-T_{\lambda[f_2]\gamma(f_2)})f  \\
    &\quad   +T_{\lambda[f_2]}T_{B_{2,1}-B_{2,2}}f_2 + T_{V_{2,1}-V_{2,2}}\cdot\na f_1  + T_{\lambda[f_2](B_{2,1}-B_{2,2})}f. 
    \end{aligned}
\eq
Therefore, $f$ satisfies 
\begin{equation}
    \label{eq.transport-diference-final}
    \partial_t f + T_{\lambda[f_2](\gamma(f_2)-B_2)}f - T_{V_2}\cdot \nabla f = R_\sharp,\quad R_\sharp:=R_0+R'.
\end{equation}
Our next task is to estimate $R_\sharp$ in $H^{s-\tdm}$. We first consider $R'$, for which $s>1+\frac{d}{2}$ suffices. 
\begin{lemm}\label{lemm:est:R'}
Let $s>1+\frac{d}{2}$ and fix  $\delta$ as in \eqref{def:delta}. We have 
\bq\label{est:R':contra}
\begin{aligned}
\| R'\|_{H^{s-\tdm}}\le C(\|(f_1,f_2)\|_{H^s})\left\{\|f_1-f_2\|_{H^{s-\mez -\delta}}+\big(1+\|(f_1, f_2)\|_{H^{s+\mez-\delta}}\big)\|f\|_{H^{s-1}}\right\},
\end{aligned}
\eq
where $C : \Rr_+ \to \Rr_+$ is a non-decreasing function which depends only on $(d,s)$, and also on $\mathfrak{d}$ and $\|b_0\|_{H^{s}}$ in the finite-depth case. 
\end{lemm}
 \begin{proof}
First, applying \eqref{est:RDN}  and \eqref{eq.RsharpEst} with $\sigma=s-\mez$ gives 
\[
\begin{aligned}
    \|R[f_2](\Gamma(f_2)-\Gamma(f_1))\|_{H^{s-\frac{3}{2}}} &\le C(\|f_2\|_{H^s})(1+\|f_2\|_{H^{s+\mez-\delta}})\|\Gamma(f_2)-\Gamma(f_1)\|_{H^{s-1}}\notag \\
    & \le C(\|(f_1,f_2)\|_{H^s})(1+\|f_2\|_{H^{s+\mez-\delta}})\|f_1-f_2\|_{H^{s-1}}
\end{aligned}
\]
and
\[
\begin{aligned}
    \|R_{\sharp}(f_1,f_2)\Gamma(f_1)\|_{H^{s-\frac{3}{2}}} &\le C(\|(f_1,f_2)\|_{H^s})\|f_1-f_2\|_{H^{s-\mez-\delta}}\|\Gamma(f_1)\|_{H^s} \notag \\
    &\le C(\|(f_1,f_2)\|_{H^s})\|f_1-f_2\|_{H^{s-\mez -\delta}}.
\end{aligned}
\]
Next, we have the following estimates  
\begin{align*}
&M^1_0(\ld[f_i])+\| \gamma(f_i)\|_{L^\infty}\le C(\|f_i\|_{H^s}), \\
& \| \gamma(f_1)-\gamma(f_2)\|_{H^{s-1}}\les \| f_1-f_2\|_{H^{s-1}},\\
&\| V_{2, i}\|_{H^{s-1}}+\| B_{2, i}\|_{H^{s-1}}\le C(\| f_i\|_{H^s}),\\
& \| V_{2, 1}-V_{2, 2}\|_{H^{s-2}}+\| B_{2, 1}-B_{2, 2}\|_{H^{s-2}}\le C(\|(f_1, f_2)\|_{H^s})\| f_1-f_2\|_{H^{s-1}}.
\end{align*}
Note also that the embedding $H^{s-1-\delta}(\Rr^d)\subset L^\infty(\Rr^d)$ holds for $s>1+\frac{d}{2}+\delta$. Then, using \cref{theo:sc} (i) and (ii) yields 
\[
\begin{aligned}
\| T_{\lambda[f_2]}T_{\gamma(f_1)-\gamma(f_2)}f_1\|_{H^{s-\tdm}}&\le C(\| f_2\|_{H^s})\| \gamma(f_1)-\gamma(f_2)\|_{L^\infty} \| f_1\|_{H^{s-\mez}}\\
&\le C(\|(f_1,  f_2)\|_{H^s})\| f_1-f_2\|_{H^{s-1}}
\end{aligned}
\]
and $T_{\lambda[f_2]}T_{\gamma(f_2)}-T_{\lambda[f_2]\gamma(f_2)}$ is of order $0$ so that 
\[
    \|(T_{\lambda[f_2]}T_{\gamma(f_2)}-T_{\lambda[f_2]\gamma(f_2)})f\|_{H^{s-\frac{3}{2}}} \le C(\|f_2\|_{H^s}) \|f\|_{H^{s-\frac{3}{2}}}.
\]
Now. combining \cref{theo:sc} (i) and the paraproduct estimate \eqref{boundpara} yields 
\[
\begin{aligned}
\| T_{\lambda[f_2]}T_{B_{2,1}-B_{2,2}}f_2\|_{H^{s-\tdm}}&\le C(\| f_2\|_{H^s})\| T_{B_{2,1}-B_{2,2}}f_2\|_{H^{s-\mez}}\le C(\| f_2\|_{H^s}) \| B_{2,1}-B_{2,2}\|_{H^{s-2}}\| f_2\|_{H^{s+\mez-\delta}}\\
&\le  C(\|(f_1,  f_2)\|_{H^s})\| f_1-f_2\|_{H^{s-1}}\| f_2\|_{H^{s+\mez-\delta}}
\end{aligned}
\]
and
\[
\begin{aligned}
\| T_{V_{2,1}-V_{2,2}}\cdot\na f_1\|_{H^{s-\tdm}}&\les \| V_{2, 1}-V_{2, 2}\|_{H^{s-2}}\| \na f_1\|_{H^{s-\mez-\delta}}\\
& \le C(\|(f_1,  f_2)\|_{H^s})\| f_1-f_2\|_{H^{s-1}}\| f_2\|_{H^{s+\mez-\delta}}. 
\end{aligned}
\]
As for $T_{\lambda[f_2](B_{2,1}-B_{2,2})}f$ we first write 
\[
T_{\lambda[f_2](B_{2,1}-B_{2,2})}f=\left(T_{\lambda[f_2](B_{2,1}-B_{2,2})}f-T_{\lambda[f_2]}T_{B_{2,1}-B_{2,2}}f\right)+T_{\lambda[f_2]}T_{B_{2,1}-B_{2,2}}f=:I+II.
\]
Since $\| B_{2, 1}-B_{2, 2}\|_{M^0_\delta}\le C(\|(f_1,  f_2)\|_{H^s})\| f_1-f_2\|_{H^s}$, by \cref{theo:sc} (ii), the operator in $I$ is of order $1-\delta$ and 
\[
\| I\|_{H^{s-\tdm}}\le C(\|(f_1,  f_2)\|_{H^s})\| f_1-f_2\|_{H^s}\| f_1-f_2\|_{H^{s-\mez-\delta}}.
\]
Using interpolation of Sobolev norms, we infer 
\[
\| I\|_{H^{s-\tdm}}\le C(\|(f_1,  f_2)\|_{H^s})\| f_1-f_2\|_{H^{s-1}}\| f_1-f_2\|_{H^{s+\mez-\delta}}.
\]
On the other hand, combining \cref{theo:sc} (i) with \eqref{boundpara} gives 
\begin{multline*}
\| II\|_{H^{s-\tdm}}\le C(\| f_2\|_{H^s})\| T_{B_{2,1}-B_{2,2}}f\|_{H^{s-\mez}}\le C(\| f_2\|_{H^s})\|B_{2,1}-B_{2,2} \|_{H^{s-2}}\| f\|_{H^{s+\mez-\delta}}\\
\le  C(\| f_2\|_{H^s})\|f_1-f_2\|_{H^{s-1}}(\| f_1\|_{H^{s+\mez-\delta}}+\| f_2\|_{H^{s+\mez-\delta}}).
\end{multline*}
It follows that
\[
\| T_{\lambda[f_2](B_{2,1}-B_{2,2})}f\|_{H^{s-\tdm}}\le C(\|(f_1,  f_2)\|_{H^s})(1+\|(f_1, f_2)\|_{H^{s+\mez-\delta}})\|f_1-f_2\|_{H^{s-1}}.
\]
As for the last term $T_{\lambda[f_2]}(R_\gamma(f_1)-R_\gamma(f_2))$, we first use \cref{theo:sc} (i) to have 
\begin{align*}
    \|T_{\lambda[f_2]}(R_{\gamma}(f_1)-R_{\gamma}(f_2))\|_{H^{s-\frac{3}{2}}}
    & \le C(\|(f_1,f_2)\|_{H^s})\|R_\gamma(f_1)-R_\gamma(f_2)\|_{H^{s-\frac{1}{2}}}.
\end{align*}
Then, we apply \eqref{Rparalin:contra} with $s_0=s-\mez$, $s_1=s-1>\frac{d}{2}$, $s_2=s+\mez-\delta$, which satisfy \eqref{cd:sj:contra}. This yields 
\[
\|R_\gamma(f_1)-R_\gamma(f_2)\|_{H^{s-\frac{1}{2}}}\le C(\|(f_1,  f_2)\|_{H^s})(1+\|(f_1, f_2)\|_{H^{s+\mez-\delta}})\|f_1-f_2\|_{H^{s-1}}
\]
which completes the control of the last term. 

Putting together the above estimates, we conclude the proof of \eqref{est:R':contra}. 
\end{proof}
The $H^{s-\tdm}$ estimate for $R_0$ is given in the next lemma. We recall \eqref{eq.difference-f} for the definition of $R_0$. 
\begin{lemm}\label{lemm.R0est} For  $s>\frac{3}{2}+\frac{d}{2}$, there holds 
    \bq\label{est:R0:contra}
        \|R_0(f_1,f_2,g_1,g_2)\|_{H^{s-\frac{3}{2}}} \le C(\|(f_1, f_2)\|_{H^s})\left\{\|\tilde{g}\|_{H^{s-1}}+\|f \|_{H^{s-\mez}}(\|\tilde{g}_1\|_{H^{s-1}}+\|\tilde{g}_2 \|_{H^{s-1}})\right\},
    \eq
    where  $C : \Rr_+ \to \Rr_+$ is a non-decreasing function only depending on $(d,s)$, and also $\fd$ and $\|b_0\|_{H^{s}}$ in the finite-depth case.
\end{lemm}

\begin{proof} Since $g_i\vert_{\Sigma_{f_i}}(x)=\tilde{g}_i \vert_{z=0}$, the trace theorem for $\Rr^d\times J$ gives 
\[
    \|g_2\vert_{\Sigma_{f_2}}-g_1\vert_{\Sigma_{f_1}}\|_{H^{s-\frac{3}{2}}} = \|\tilde{g}_2 \vert_{z=0}-\tilde{g}_1 \vert_{z=0}\|_{H^{s-\frac{3}{2}}} \lesssim \|\tilde{g}\|_{H^{s-1}}.    
\]
By the tame product estimate \eqref{tamepr} and the embedding $H^{s-1}(\Rr^d)\subset L^\infty(\Rr^d)$, we have
\bq
\begin{aligned}
&\| N_2\cdot\left(\cS[f_2]g_2\right)\vert_{\Sigma_{f_2}} - N_1\cdot\left(\cS[f_1]g_1\right)\vert_{\Sigma_{f_1}}\|_{H^{s-\tdm}}\\
&\quad \les  \| N_2-N_1\|_{H^{s-\tdm}}\|\left(\cS[f_1]g_1\right)\vert_{\Sigma_{f_1}}\|_{H^{s-1}}+ (1+\| N_2 - e_y\|_{L^\infty}\|) \left(\cS[f_2]g_2\right)\vert_{\Sigma_{f_2}} - \left(\cS[f_1]g_1\right)\vert_{\Sigma_{f_1}}\|_{H^{s-\tdm}}\\
&\quad =:I+II,
\end{aligned}
\eq
where 
\[
\| N_2-N_1\|_{H^{s-\tdm}}\le C(\|(f_1, f_2)\|_{H^s})\|f\|_{H^{s-\mez}},\quad \| N_2-e_y\|_{L^\infty}\le C(\|f_2\|_{H^s}).
\]
Using the estimate \eqref{esttrace:cS} with $\sigma=s-1$, we obtain 
\[
I\le C(\|(f_1, f_2)\|_{H^s})\|f\|_{H^{s-\mez}}\| \tilde{g}_1\|_{H^s}.
\]
For $II$ we first recall that $\cS[f_i]g_i\vert_{\Sigma_{f_i}}=\na_{x, z}v_i^{(2)}\vert_{z=0}(\na \ff_{f_i})^{-1}\vert_{z=0}$, where $(\na \ff_{f_i})^{-1}$ is given in terms of $\na_{x, z}\varrho_i$ as in \eqref{chainrule:ff}. Then, with $v^{(2)}=v_1^{(2)}-v_2^{(2)}$ we have
\[
II\les \| \na_{x, z}v^{(2)}\vert_{z=0}\|_{H^{s-\tdm}}\| (\na \ff_{f_1})^{-1}\vert_{z=0}\|_{L^\infty}+  \| \na_{x, z}v_2^{(2)}\vert_{z=0}\|_{H^{s-1}}\| (\na \ff_{f_1})^{-1}\vert_{z=0}- (\na \ff_{f_2})^{-1}\vert_{z=0}\|_{H^{s-\tdm}}.
\]
It is readily seen that 
\begin{align*}
&\| (\na \ff_{f_i})^{-1}\vert_{z=0}\|_{L^\infty}\le C(\| f_i\|_{H^s}),\\
&\| (\na \ff_{f_1})^{-1}\vert_{z=0}- (\na \ff_{f_2})^{-1}\vert_{z=0}\|_{H^{s-\tdm}}\le C(\|(f_1, f_2)\|_{H^s})\| f\|_{H^{s-\mez}}.
\end{align*}
The estimate \eqref{tracenav(2)} implies
\[
    \| \na_{x, z}v_2^{(2)}\vert_{z=0}\|_{H^{s-1}}\le C(\| f_2\|_{H^s})\| \tilde{g}_2\|_{H^s}.
\]
We also have by the trace theorem for $\Rr^d\times J$,
\[
    \| (\na \ff_{f_1})^{-1}\vert_{z=0}- (\na \ff_{f_2})^{-1}\vert_{z=0}\|_{H^{s-\tdm}}\le C(\|(f_1, f_2)\|_{H^s})\|f \|_{H^{s-\mez}}.
\]
Combining the above estimate and the difference estimate \eqref{eq.lowNormEst-v2} as well, we obtain
\bq
    II\le C(\|(f_1, f_2)\|_{H^s})\left\{\|\tilde{g}\|_{H^{s-1}}+\|f \|_{H^{s-\mez}}(\|\tilde{g}_1\|_{H^{s-1}}+\|\tilde{g}_2 \|_{H^{s-1}})\right\},
\eq
thereby concluding the proof of \eqref{est:R0:contra}.
\end{proof}
We are now ready to state the following \textit{a priori} contraction estimate for $f$.
\begin{prop}\label{prop.contractionEstimates} Let $s>\frac{3}{2}+\frac{d}{2}$ and consider $f_1, f_2 \in L^{\infty}([0,T],H^s) \cap L^2([0,T],H^{s+\mez})$ solutions to \eqref{eq:f} on the time interval $[0,T]$, $T\le 1$, with initial conditions $f_{0,1}$ and $f_{0,2}$. Assume furthermore that \eqref{RT:apriori} holds for $f_1$ and $f_2$ and that in the finite-depth case, \eqref{separate:apriori} holds for $f_1$ and $f_2$. Assume also that 
\[
    \|(f_1,f_2)\|_{L^{\infty}_TH^s} \le B_f \text{ and } \|(\tilde{g}_1, \tilde{g}_2)\|_{L^{\infty}_TH^{s}} \le B_g.   
\] 
Then, there exist $C_0 : \Rr_+^2 \to \Rr_+$ and $\mathcal{F}_0, \mathcal{F} : \Rr_+^3 \to \Rr_+$ depending only on $(s,d,b,\gamma)$ and non-decreasing in each variable and such that if
\begin{equation}
    \label{eq.assump-smallness}
    C_0(B_f,\mathfrak{d})B_g \le \frac{1}{4\cF_0(B_f,\mathfrak{a},\mathfrak{d})}, 
\end{equation}
then 
\begin{equation}
    \label{eq.contractionEstimates}
    \|f\|_{L^{\infty}_TH^{s-1} \cap L^2_TH^{s-\mez}}\le \mathcal{F}\left(\|(f_1,f_2)\|_{L^{\infty}_TH^s \cap L^2_TH^{s+\mez}},\mathfrak{a},\mathfrak{d}\right) \left(\|f_0\|_{H^{s-1}} + T^{\frac{1}{2}}\|\tilde{g}\|_{L^{\infty}_TH^{s-1}}\right),
\end{equation}
where  $f_0:=f_{0,1}-f_{0,2}$. Here, $C_0$, $\cF_0$, and $\cF$ can be chosen to be the same functions as in \cref{prop.uniform-f}.
\end{prop}

\begin{proof} We first note that  $f$ satisfies \eqref{eq.transport-diference-final} which is of the form  \eqref{eq.parabolic} with 
\[ 
a=\lambda[f_2] (\gamma(f_2)-B_2)\equiv \lambda[f_2] \cT(f_2),\quad U=-V_2,\quad F=R_\sharp:=R_0+R',
\]
where $a\ge C\fa|\xi|$ by \eqref{elliptic:ld} and the assumption \eqref{RT:apriori}.  Recalling the norm bounds in \eqref{norm:Va} for $V_2$ and $a$, we can  apply \cref{lemm:parabolicest} to have
\begin{align}\label{eq.gronwall-completecompletediff-initial}
    \mez\frac{d}{dt}\|f(t)\|_{H^{s-1}}^2&\le -\frac{1}{\cF_0(\|f_2\|_{L^{\infty}_TH^s}, \mathfrak{a}, \mathfrak{d})}\| f(t)\|_{H^{s-\mez}}^2 \notag\\
    &\qquad+\cF_0(\|f_2\|_{L^{\infty}_TH^s}, \mathfrak{a}, \mathfrak{d})\|f(t)\|_{H^{s-1}}^2+\| R_\sharp(t)\|_{H^{s-\frac{3}{2}}}\|f(t)\|_{H^{s-\mez}}.
\end{align}
By virtue of \eqref{est:R':contra} and \eqref{est:R0:contra}, there holds
\begin{align*}
    \|R_\sharp\|_{H^{s-\frac{3}{2}}} &\le C_0(\|(f_1,f_2)\|_{H^s},\mathfrak{d})\left\{\|f\|_{H^{s-\mez -\delta}}+(1+\|(f_1, f_2)\|_{H^{s+\mez-\delta}})\|f\|_{H^{s-1}}\right.\\
    &\qquad \left.+ \|\tilde{g}\|_{H^{s-1}} + \|f\|_{H^{s-\mez}}\|(\tilde{g}_1, \tilde{g}_2)\|_{H^s}\right\}.
\end{align*}
By interpolation, we have 
\bq\label{interpolate:contraf}
\|f\|_{H^{s-\mez -\delta}}\|f\|_{H^{s-\mez}}\le \|f\|_{H^{s-\mez}}^{1+\theta}\|f\|_{H^{s-1}}^{1-\tt} 
\eq
for some $\tt\in (0, 1)$. Thus, we obtain after applying Young's inequality that 
\begin{align*}
    \|R_{\sharp}(t)\|_{H^{s-\frac{3}{2}}}\|f(t)\|_{H^{s-\mez}} &\le \frac{1}{4\cF_0(\|(f_1,f_2)\|_{L^{\infty}_TH^s},\mathfrak{a},\mathfrak{d})}\|f\|_{H^{s-\mez}}^2 \\
    &\quad+ \cF_1(\|(f_1,f_2)\|_{L^{\infty}_TH^s},\mathfrak{a},\mathfrak{d})\left\{(1+\|(f_1, f_2)\|^2_{H^{s+\mez-\delta}})\|f\|_{H^{s-1}}^2+ \|\tilde{g}\|_{H^{s-1}}^2\right\} \\
    &\quad+C_0(\|(f_1,f_2)\|_{L^{\infty}_TH^s},\mathfrak{d})\|f\|_{H^{s-\mez}}^2\|(\tilde{g}_1, \tilde{g}_2)\|_{H^s},\quad t\le T.
\end{align*}

In view of the condition \eqref{eq.assump-smallness}, it follows that 
\begin{align}\label{eq.gronwall-diff-initial}
    \mez\frac{d}{dt}\|f(t)\|_{H^{s-1}}^2&\le -\frac{1}{2\cF_0(\|f_2\|_{L^{\infty}_TH^s}, \mathfrak{a}, \mathfrak{d})}\| f(t)\|_{H^{s-\mez}}^2 \notag\\
&\qquad+\cF_2(\|(f_1,f_2)\|_{L^{\infty}_TH^s},\mathfrak{a},\mathfrak{d})\left\{(1+\|(f_1, f_2)\|^2_{H^{s+\mez-\delta}})\|f\|_{H^{s-1}}^2+ \|\tilde{g}\|_{H^{s-1}}^2\right\}.
\end{align}
Discarding the first term on the right-hand side and applying Grönwall's inequality yields 
\bq\label{contra:f:Linfty}
\|f\|_{L^\infty_TH^{s-1}}^2\le \big(\|f_0\|_{H^{s-1}}^2+2\cF_2T\|\tilde{g}\|_{L^\infty_TH^{s-1}}^2\big)\exp\left(\big(T+\|(f_1, f_2)\|^2_{L^2_TH^{s+\mez-\delta}}\big)\cF_2\right),
\eq
where $\cF_2= \cF_2(\|(f_1,f_2)\|_{L^{\infty}_TH^s},\mathfrak{a},\mathfrak{d})$. Returning to \eqref{eq.gronwall-diff-initial} and integrating in time, we find
\bq\label{contra:f2}
\| f\|_{L^2_TH^{s-\mez}}^2\le \cF_0\| f_0\|_{H^{s-1}}^2+2\cF_0\cF_2\left\{\big(T+\|(f_1, f_2)\|^2_{L^2_TH^{s+\mez-\delta}}\big)\| f\|_{L^\infty_TH^{s-1}}^2+T\|\tilde{g}\|_{L^\infty_TH^{s-1}}^2\right\}
\eq
For $T\le 1$, combining \eqref{contra:f:Linfty} and \eqref{contra:f:Linfty} yields \eqref{eq.contractionEstimates}.
\end{proof}

\begin{rema}\label{rema:contraction:variant}In \cref{sec.uniform-bounds} we will need contraction estimates in the case where 
\[
    \p_t f + T_{\lambda[f_2](\gamma(f_2)-B_2)}f - T_{V_2}\cdot \nabla f =F,  
\]
 With similar estimates, we obtain  
\begin{align}
    \label{eq.contractionEstimates-R}
    \|f\|_{L^{\infty}_TH^{s-1}} &+ \frac{1}{\cF(\|(f_1,f_2)\|_{L^{\infty}_TH^s},\mathfrak{a},\mathfrak{d})} \|f\|_{L^2_TH^{s+\mez}} \le \|f_{0}\|_{H^{s-1}} \exp\left(T\cF(\|(f_1,f_2)\|_{L^{\infty}_TH^s},\mathfrak{a},\mathfrak{d})\right) \notag \\
    & + \cF(\|(f_1,f_2)\|_{L^{\infty}_TH^s},\mathfrak{a},\mathfrak{d})\|F\|_{L^2_TH^{s-\frac{3}{2}}}\exp\left(T\cF(\|(f_1,f_2)\|_{L^{\infty}_TH^s},\mathfrak{a},\mathfrak{d})\right)
\end{align}
\end{rema}

\subsection{Estimates for $g$}
\
We recall from \eqref{eq:g} that $g_i$ satisfies 
\[
    \partial_t g_i + u_i\cdot \nabla_{x,y}g_i + \gamma '(y)u_{i,y} = 0,\quad u_i=(u_{i, x}, u_{i, y}).     
\]
Then, as in \eqref{eq:tildeg} and \eqref{tildeu}, $\tilde{g}_i=g_i\circ \ff_{f_i}$ satisfies 
\begin{equation}
    \label{eq.newCoord_g_i}
    \partial_t \tilde{g}_i + \ol{u}_i(x, z, t)\cdot \nabla_{x,z}\tilde{g}_i+ \gamma '(\varrho_i)u_{i,y}\circ \ff_{f_i} = 0,\quad (x, z)\in \Rr^d\times J,
\end{equation}
where
\bq\label{tildeui}
\ol{u}_i= \begin{bmatrix} u_{i,x} \circ \ff_{f_i} \\ \frac{1}{\partial_z \varrho_i} (u_{i,y} \circ \ff_{f_i}-\na_x\varrho_i \cdot u_{i,x}\circ \ff_{f_i} - \partial_t \varrho_i)\end{bmatrix}.
\eq 
We also recall that $u_i$ can be expressed in terms of $\cG[f_i]$ and $\cS[f_i]$ as in \eqref{eq:u}. Setting $\ol{u} = \ol{u}_1 -\ol{u}_2$, we find that  $\tilde{g} = \tilde{g}_1-\tilde{g}_2$ satisfies 
\begin{equation}
    \label{eq.gEquationDifference}
    \begin{aligned}
  &  \partial_t \tilde{g} +\ol{u}_1 \cdot \nabla_{x,z} \tilde{g} = \tilde{F},\\
  &\tilde{g}\vert_{t=0}=\tilde{g}_0:=g_{0, 1}\circ\ff_{0, 1}-g_{0, 2}\circ\ff_{0, 2},
  \end{aligned}
\end{equation}
 where 
\begin{equation}
    \label{eq.defFtilde}
    \tilde{F}:= \gamma '(\varrho_2)u_{2,y} \circ \ff_{f_2} - \gamma '(\varrho_1)u_{1,y} \circ \ff_{f_1} - \ol{u} \cdot \nabla_{x,z}\tilde{g}_2.  
\end{equation}
We recall from \eqref{tildeu:trace} that the $\tilde{u}_i$ are tangential to $\p(\Rr^d\times J)$. Thus, applying \cref{theo:transport} with $\sigma=s-1\ge 1$ gives 
\[
    \| \tilde{g}\|_{L^\infty([0, T]; H^{s-1})}\le \tM\left(\| \tilde{g}_0\|_{H^{s-1}}+ \| \tilde{F}\|_{L^1([0, T]; H^{s-1})}\right)\exp(\tM V_T)),
\]
where $\tM=\tM(d, s)$ and 
\bq
    V_T=\begin{cases}
    \| \na_x \ol{u}_1\|_{L^1([0, T]; H^\frac{d+1}{2} \cap L^\infty)}\quad\text{if } s<2+\frac{d+1}{2}, \\
    \| \na_x  \ol{u}_1\|_{L^1([0, T]; H^{s-2})}\quad\text{if } s>2+\frac{d+1}{2}.
\end{cases}
\eq
Since $H^{s-1}(\Rr^d\times J)\subset H^\frac{d+1}{2} \cap L^\infty$ for $s>\tdm+\frac{d}{2}$, it follows that 
\bq\label{transport:g:contra}
\| \tilde{g}\|_{L^\infty([0, T]; H^{s-1})}\le \tM\left(\| \tilde{g}_0\|_{H^{s-1}}+ \| \tilde{F}\|_{L^1([0, T]; H^{s-1})}\right)\exp\big(\tM \| \na_x  \ol{u}_1\|_{L^1([0, T]; H^{s-1})}\big)
\eq
for all $s \in (\frac{d}{2}+\frac{3}{2},\infty) \setminus \{\frac{5}{2}+\frac{d}{2}\}$. For $s=\frac{5}{2}+\frac{d}{2}$ we fix $s'\in (\tdm+\frac{d}{2}, s)$ and obtain  \eqref{transport:g:contra} with $s$ replaced by $s'$. Thus we only consider $s\ne \frac{5}{2}+\frac{d}{2}$ in the following. 

We recall from \eqref{energyest:u} that 
\bq\label{energyest:u1}
\|\ol{u}_1\|_{L^1_TH^s}\le C(\| f_1\|_{L^\infty_TH^s})\left(T^\mez\| f_1\|_{L^2_TH^{s+\mez}}+T^\mez\| f_1\|_{L^2_TH^{s+\mez}}\| \tilde g_1\|_{L^\infty_TH^s}+T\| \tilde g_1\|_{L^\infty_TH^s}\right).
\eq
Our next task is to estimate $\| \tilde{F}\|_{L^1_TH^{s-1}}$.
\begin{lemm}
$\tilde{F}$, given by \eqref{eq.defFtilde}, satisfies 
\bq\label{est:tildeF}
\| \tilde{F}\|_{H^{s-1}}\le  C(\|(f_1, f_2)\|_{H^s})\left\{\|\tilde{g} \|_{H^{s-1}}+\|f\|_{H^{s-\mez}}\big(1+\|\tilde{g}_1\|_{H^{s}} + \|\tilde{g}_2\|_{H^s}\big)\right\}.
\eq
Consequently 
\bq\label{est:tildeF:L1}
\| \tilde{F}\|_{L^1_TH^{s-1}}\le  C(\|(f_1, f_2)\|_{L^\infty_TH^s})\left\{T\|\tilde{g} \|_{L^\infty_TH^{s-1}}+T^\mez\|f\|_{L^2_TH^{s-\mez}}\big(1+\|\tilde{g}_1\|_{L^\infty_TH^{s}} + \|\tilde{g}_2\|_{L^\infty_TH^s}\big)\right\}.
\eq
\end{lemm}
\begin{proof}
  We start by writing
\begin{multline*}
    \gamma '(\varrho_1)u_{1,y} \circ \ff_{f_1} - \gamma '(\varrho_2)u_{2,y} \circ \ff_{f_2} = (\gamma '(\varrho_1)-\gamma'(\varrho_2))u_{1,y} \circ \ff_{f_1}
     + \gamma'(\varrho_2)(u_{1,y}\circ \ff_{f_1}-u_{2,y}\circ \ff_{f_2}).
\end{multline*}  
Applying \cref{prop:gamma} on  multiplier properties of $\gamma'$, we obtain 
\bq\label{est:tildeF:1}
\begin{aligned}
    \|\tilde{F}\|_{H^{s-1}} & \lesssim C(\|\varrho_1-m,  \varrho_2-m)\|_{H^{s-1}})\|(\varrho_1-m)- (\varrho_2-m)\|_{H^{s-1}}\|u_{1,y} \circ \ff_{f_1}\|_{H^{s-1}} \\
    & \quad + C(\| \varrho_2-m\|_{H^{s-1}})\|u_{1,y} \circ \ff_{f_1} - u_{2,y}\circ \ff_{f_2}\|_{H^{s-1}} + \|\ol{u}\|_{H^{s-1}}\|\nabla \tilde{g}_2\|_{H^{s-1}}\\
    & \lesssim C(\|(f_1, f_2)\|_{H^{s-\tdm}})\|f_1-f_2\|_{H^{s-\tdm}}\|u_{1,y} \circ \ff_{f_1}\|_{H^{s-1}} \\
     & \quad + C(\| f_2\|_{H^{s-\tdm}})\|u_{1,y} \circ \ff_{f_1} - u_{2,y}\circ \ff_{f_2}\|_{H^{s-1}} + \|\ol{u}\|_{H^{s-1}}\|\nabla \tilde{g}_2\|_{H^{s-1}},
\end{aligned}
\eq
where we have used the embedding $H^{s-1}(\Rr^d\times J)\subset L^\infty(\Rr^d\times J)$. Recalling the definition \eqref{eq:u} and the notation in \cref{prop.diffEstimateSobolev}, we have $u_{i, y}=-\p_y\phi_i^{(1)}-\p_y\phi_i^{(2)}-g_i$ and 
\[
u_{i,y} \circ \ff_{f_i}=-\frac{1}{\p_z\varrho_i}\p_zv_i^{(1)}-\frac{1}{\p_z\varrho_i}\p_zv_i^{(2)}-\tilde {g}_i.
\]
Consequently 
\[
u_{1,y} \circ \ff_{f_1}-u_{2,y} \circ \ff_{f_2}=-\frac{1}{\p_z \varrho_1}(\partial_z v^{(1)} + \partial_z v^{(2)})  
     + \left(\frac{1}{\p_z \varrho_2}-\frac{1}{\p_z \varrho_1}\right)(\p_z v^{(1)}_2 +\p_z v^{(2)}_2) -\tilde{g}, 
\]
where $v^{(j)}=v_1^{(j)}-v_2^{(j)}$. From this we deduce 
\bq\label{est:tildeF:2}
\begin{aligned}
\| u_{1,y} \circ \ff_{f_1}-u_{2,y} \circ \ff_{f_2}\|_{H^{s-1}}&\le C(\| f_1\|_{H^{s-\mez}})(\| \partial_z v^{(1)}\|_{H^{s-1}}+\|\partial_z v^{(2)}\|_{H^{s-1}})+\| \tilde{g}\|_{H^{s-1}}\\
&\quad +C(\| (f_1, f_2)\|_{H^{s-\mez}})\| f\|_{H^{s-\mez}}(\|\p_z v^{(1)}_2\|_{H^{s-1}}+\|\p_z v^{(2)}_2\|_{H^{s-1}}).
\end{aligned}
\eq
By virtue of  the estimate \eqref{eq.GS-v} with $r=s-1$ and $(\zeta, k)=(\Gamma(f_2), 0)$ or $(\zeta, k)=(0, g_2)$, we obtain 
\bq\label{est:tildeF:3}
\begin{aligned}
&  \|\nabla v_2^{(1)}\|_{H^{s-1}} \le C(\|f_2\|_{H^s})\|f_2\|_{H^{s-\mez}},\quad \|\nabla v_2^{(2)}\|_{H^{s-1}} \le C(\|f_2\|_{H^s})\|\tilde{g}_2\|_{H^{s-1}},
\end{aligned}
\eq
On the other hand, the estimates \eqref{eq.lowNormEst-v1} and \eqref{eq.lowNormEst-v2} give 
\bq\label{est:tildeF:4}
\begin{aligned}
    &\|\nabla_{x,z}v^{(1)}\|_{H^{s-1}} \le C(\|(f_1, f_2)\|_{H^s})\|f\|_{H^{s-\mez}},\\
    &\|\nabla_{x,z}v^{(2)}\|_{H^{s-1}} \le C(\|(f_1, f_2)\|_{H^s})\left\{\|\tilde{g} \|_{H^{s-1}}+\|f\|_{H^{s-\mez}}\big(\|\tilde{g}_1\|_{H^{s}} + \|\tilde{g}_2\|_{H^s}\big)\right\}.
\end{aligned}
\eq
Combining \eqref{est:tildeF:2}, \eqref{est:tildeF:3}, and \eqref{est:tildeF:4} yields
\bq\label{uof:contra}
\begin{aligned}
\| u_{1,y} \circ \ff_{f_1}-u_{2,y} \circ \ff_{f_2}\|_{H^{s-1}}\le C(\|(f_1, f_2)\|_{H^s})&\left\{\|\tilde{g} \|_{H^{s-1}}+\|f\|_{H^{s-\mez}}\big(1+\|\tilde{g}_1\|_{H^{s}} + \|\tilde{g}_2\|_{H^s}\big)\right\}.
\end{aligned}
\eq
Therefore, it remains to show that $\| \ol{u}\|_{H^{s-1}}$ is bounded by the right-hand side of \eqref{est:tildeF}. 
In view of \eqref{tildeui} and estimate similar as above, we have 
\bq\label{est:olu:contra}
\begin{aligned}
\| \ol{u}\|_{H^{s-1}}\le C(\|(f_1, f_2)\|_{H^s})&\left\{\|f\|_{H^{s-\mez}}\Big(\| u_1\circ \ff_{f_1}\|_{H^{s-1}}+ \| u_2\circ \ff_{f_1}\|_{H^{s-1}}+\| \p_t\varrho_1\|_{H^{s-1}}+\| \p_t\varrho_2\|_{H^{s-1}}\Big)\right. \\
&\quad\left.+\| u_1\circ \ff_{f_1}- u_2\circ \ff_{f_2}\|_{H^{s-1}}+\| \p_t\varrho_1-\p_t\varrho_2\|_{H^{s-1}}\right\}.
\end{aligned}
\eq
The term $\| u_1\circ \ff_{f_1}- u_2\circ \ff_{f_2}\|_{H^{s-1}}$ can be estimated using \eqref{uof:contra}. We recall from  \eqref{uof:l} that 
\[
\| u_i\circ \ff_{f_i}\|_{H^{s-\mez}}\le C(\| f_i\|_{H^s})(\| f_i\|_{H^s}+\| \tilde{g}_i\|_{H^s}).
\]
Using the equation satisfied by $f_i$, \eqref{est:traceuN:l} and \cref{theo:DN1} (i) with $\sigma = s-\mez$ provides us with  
\[
    \| \p_t\varrho_i\|_{H^{s-1}}\les \| \p_t f_i\|_{H^{s-\tdm}}\le C(\| f_i\|_{H^s})(\| f_i\|_{H^s}+\| \tilde{g}_i\|_{H^s}).
\]
Finally, since $f=f_1-f_2$ solves \eqref{eq.difference-f} for $\p_t f$, we find
\[
    \| \p_t\varrho_1-\p_t\varrho_2\|_{H^{s-1}}\les \| \p_tf\|_{H^{s-\tdm}}\les \| G[f_1]\Gamma(f_1)-G[f_2]\Gamma(f_2)\|_{H^{s-\tdm}}+\| R_0\|_{H^{s-\tdm}}.
\]
Then, invoking  \cref{thm.NPdiff} (ii) with $\sigma = s-\mez$ and \cref{lemm.R0est}, we can bound $\| \p_t\varrho_1-\p_t\varrho_2\|_{H^{s-1}}$ by the right-hand side of \eqref{est:tildeF}. 
Plugging the above estimates in \eqref{est:olu:contra}, we conclude that $\| \ol{u}\|_{H^{s-1}}$ is also controlled by the right-hand side of \eqref{est:tildeF} as claimed and this completes the proof of \eqref{est:tildeF}.
\end{proof}

Plugging the estimates \eqref{energyest:u1} and \eqref{est:tildeF:L1} in \eqref{transport:g:contra}, we arrive at the following contraction estimate for $g$. 

\begin{prop}\label{prop.contraction-gn} 
Let $s>\frac{d}{2}+\frac{3}{2}$. Assume that $(f_1, g_1)$ and $(f_2, g_2)$ are solutions of \eqref{eq:f}-\eqref{eq:g} with the regularity given by \eqref{reg:fjgj} and such that \eqref{separate:apriori} holds for $f=f_1,~ f_2$ in the finite-depth case.  If $s\ne \frac52+\frac{d}{2}$, then there exists a function $C : \Rr_+^2\ \to \Rr^+$ only depending on $(s,d,b,\gamma)$, non-decreasing in each variable such that for $T\le 1$, there holds
\bq    \label{eq.apriori-difference-g}
\begin{aligned}
    \|\tilde{g}\|_{L^{\infty}_TH^{s-1}} &\le \tM\left(\|\tilde{g}_{0}\|_{H^{s-1}}+T^\mez C\left(\|(f_1, f_2)\|_{L^{\infty}_TH^s}+\|(\tilde{g}_1, \tilde{g}_2)\|_{L^{\infty}_TH^s}, \fd\right)\big(\|\tilde{g} \|_{L^\infty_TH^{s-1}}+\|f\|_{L^2_TH^{s-\mez}}\big) \right)\\
    &\quad\cdot \exp\left(T^\mez C(\|f_1\|_{L^\infty_TH^s}, \fd)\Big(\|f_1\|_{L^2_TH^{s+\mez}}+\| \tilde g_1\|_{L^\infty_TH^s}+\|f_1\|_{L^2_TH^{s+\mez}}\| \tilde g_1\|_{L^\infty_TH^s}\Big)\right),
    \end{aligned}
    \eq
where we recall that $f=f_1-f_2$, $\tilde{g}=\tilde{g}_1-\tilde{g}_2$, and $\tilde{g}_0=\tilde{g}_{0,1}-\tilde{g}_{0,2}$. If $s=\frac52+\frac{d}{2}$, then \eqref{eq.apriori-difference-g} holds with $s$ replaced by any $s'\in (\tdm+\frac{d}{2}, \frac52+\frac{d}{2})$.
\end{prop}
\subsection{Proof of \cref{prop.diffEstimateSobolev}}\label{subsection:contrav} 
We will prove \eqref{eq.lowNormEst-v1} and \eqref{eq.lowNormEst-v2} at once. To this end, we will understand that $v_i=\phi\circ \ff_{f_i}$, $i\in\{1, 2\}$ where $\phi_i$ solves 
\begin{align*}
&\Delta_{x, y}  \phi_i=-\p_y k_i\quad\text{in } \Omega_{f_i},\quad \phi^{(2)}\vert_{\Sigma_f}=h_i,\\
&\begin{cases}
\lim_{(x, y)\to \infty} \na_{x, y}   \phi_i= 0&\text{for infinite depth,}\\
 \p_\nu \phi_i\vert_{y=b(x)}=-\nu_yk_i\vert_{y=b(x)}&\text{for finite depth}.
\end{cases}
\end{align*}
This means that $v_i^{(1)}$ corresponds to $k_i=0$ and $v_i^{(2)}$ corresponds to $h_i=0$. The condition $s>1+\frac{d}{2}$ is assumed throughout, while the stronger condition $s>\tdm+\frac{d}{2}$ is only needed when $k_i\ne 0$. 

We start by recalling from \eqref{eq:v}  that for $i \in \{1,2\}$, the function $v_i$ satisfies
\[
    (\partial_z^2 + \alpha_i\Delta_x + \beta_i\cdot \nabla_x \partial_z - \gamma_i\partial_z)v_i = F_{0,i},
\] 
where $F_{0, i}$ is given by \eqref{eq.datum}, and $v_i(x, 0)=\zeta_i(x)$. In this notation, we only need the stronger condition $s>\tdm +\frac{d}{2}$ when $F_{0, i}\ne 0$.  Then $v=v_1-v_2$ satisfies
\begin{align}
    \label{eq.v-diff}
    &(\partial_z^2 + \alpha_1\Delta_x + \beta_1\cdot \nabla_x \partial_z - \gamma_1\partial_z)v = F_0,\quad v(x, 0)=\zeta(x):=\zeta_1(x)-\zeta_2(x),
\end{align}
where  
\bq\label{def:F0:contra}
    F_0 := (F_{0,1}-F_{0,2}) -(\alpha_1-\alpha_2)\Delta_x v_2 - (\beta_1-\beta_2)\cdot \nabla_x\partial_z v_2 +(\gamma_1-\gamma_2)\partial_zv_2.    
\eq
We assume that $f_i\in H^s(\Rr^d)$ and $k_i\in H^s(\Omega_{f_i})$. We have $\na_{x, z} v_i\in H^{s-\mez}(\Rr^d\times J)$ by virtue of \cref{prop.GSest-v}. Our goal is to prove the contraction estimate for $ \na_{x, z}v_1-\na_{x, z}v_2$ in the lower norm $H^{s-\tdm}(\Rr^d\times J)$. 
\subsubsection{Estimates of $ \|\nabla_{x,z}v\|_{X^{s-\tdm}}$} We fix arbitrary numbers $z_1< z_0$ in $J$. Applying   \cref{prop.ABZ} to $v$ solving~\eqref{eq.v-diff} with data $(\zeta, F_0)$, we obtain 
\bq\label{estv:Xs-tdm}
    \|\nabla_{x,z}v\|_{X^{s-\tdm}([z_0, 0])} \le  C(\| f_1 \|_{H^s}) \left( \|\zeta\|_{H^{s-\mez}}+ \|F_0\|_{Y^{s-\tdm}([z_1, 0])}+\|\nabla_{x,z}v\|_{X^{-\mez}([z_1, 0])}\right),
\eq
 where $C$ on $(z_0, z_1)$ is only through $z_1-z_0$. In the infinite-depth case, we can fix $z_1-z_0=1$ and let $z_0$ tend to $-\infty$ to obtain the estimate for $\na_{x, z}v$ in $X^{s-\tdm}((-\infty, 0])$. On the other hand, in the finite-depth case, we will need to prove separately estimates for $v$ near $z=-1$. 
 
 The control of  $\| \na_{x, z}v\|_{X^{-\mez}(J)}$ in \eqref{estv:Xs-tdm} is given in the next lemma. 
\begin{lemm}\label{lemm.FirstStepIteration} For $s>1+\frac{d}{2}$, there holds 
    \begin{equation}
        \label{eq.variational-v1}
        \|\nabla_{x,z}v^{(1)}\|_{X^{-\mez}(J)} \le C(\|(f_1,f_2)\|_{H^s})\left(\|h_1-h_2\|_{H^{\mez}} + (\|h_1\|_{H^s}+\|h_2\|_{H^s})\|f_1-f_2\|_{H^{\mez}}\right),
    \end{equation}
    and 
    \begin{multline}
        \label{eq.variational-v2}
        \|\nabla_{x,z}v^{(2)}\|_{X^{-\mez}(J)} \le C(\|(f_1,f_2)\|_{H^s})\left\{(\|k_1\|_{H^s}+\|k_2\|_{H^s})\|f_1-f_2\|_{H^{\mez}}\right.\\
     \left. +\|\tilde{k}_1-\tilde{k}_2\|_{L^2}+\|\p_z\tilde{k}_1-\p_z\tilde{k}_2\|_{L^2_z H^{-\mez}_x}\right\}.
        \end{multline}
In both cases, $C : \Rr_+ \to \Rr_+$ is a non-decreasing function only depending on $(d,s)$, and also on $\fd$ and $\|b_0\|_{H^s}$ in the finite-depth case. In the finite-depth case, $\|v^{(i)}\|_{L^2(\Rr^d\times J)}$, $i\in\{1,2\}$ is also bounded by \eqref{eq.variational-v1} and \eqref{eq.variational-v2}. 
\end{lemm}
\begin{proof} Let us start with~\eqref{eq.variational-v1}. Denoting $v_i^{(1)}=v_i^{(1)}(h_i)$  to stress the linear dependence of $v_i^{(1)}$ on $h_i$, we  write 
\begin{equation}
    \label{eq.decomp-v1}
    v^{(1)}=v_1^{(1)}-v_2^{(1)} = v_1^{(1)}(h_1-h_2) + v_1^{(1)}(h_2)-v_2^{(1)}(h_2).
\end{equation}
 By \cref{lemm.variational-v}, we have 
 \[
\| v_1^{(1)}(h_1-h_2)\|_{X^{-\mez}(J)}\le C(\| f_1\|_{H^s})\| h_1-h_2\|_{H^\mez}. 
 \]
On the other hand, \cite[Lemma 3.27]{NP} gives 
\[
\| v_1^{(1)}(h_2)-v_2^{(1)}(h_2)\|_{X^{-\mez}(J)}\le C(\|(f_1,f_2)\|_{H^s})\|h_2\|_{H^s})\|f_1-f_2\|_{H^{\mez}}.
\]
In order to prove \eqref{eq.variational-v2} we denote $\tilde{k}_i=k_i\circ \ff_{f_i}$ and  write $v_i^{(2)}=v_i^{(2)}(\tilde{k}_i)$ to emphasize the linear dependence on $k_i$ in the relation $\di_{x, z} (\cA_i \na_{x, z}v_i^{(2)})=-\frac{\p_z\varrho_i} {1+|\na_x\varrho_i|^2}\p_z\tilde{k}_i$. Then
\[
    v^{(2)}=v_1^{(2)}-v_2^{(2)} = v_1^{(2)}(\tilde{k}_1-\tilde{k}_2) + v_1^{(2)}(\tilde{k}_2)-v_2^{(2)}(\tilde{k}_2).     
\]
The estimate of $v_1^{(2)}(F_{0,1}-F_{0,2})$ follows from \cref{lemm.variational-v}: 
\[
\begin{aligned}
\| v_1^{(2)}(F_{0,1}-F_{0,2})\|_{X^{-\mez}(J)}&\le C(\| f_1\|_{H^s})\Big(\|\tilde{k}_1-\tilde{k}_2\|_{L^2}+\| \frac{\p_z\varrho_1} {1+|\na_x\varrho_1|^2}(\p_z\tilde{k}_1-\p_z\tilde{k}_2)\|_{L^2_z H^{-1}_x}\Big)\\
&\le C(\| f_1\|_{H^s})\Big(\|\tilde{k}_1-\tilde{k}_2\|_{L^2}+\|\p_z\tilde{k}_1-\p_z\tilde{k}_2\|_{L^2_z H^{-\mez}_x}\Big),
\end{aligned}
\]
where we have used \eqref{pr} with $(s_0, s_1, s_2)=(-1, s-1, -\mez)$. Finally, the estimate of $v_1^{(2)}(F_{0,2})-v_2^{(2)}(F_{0,2})$ can be obtained using a similar procedure as the one used in \cite[Lemma 3.27]{NP}, which we omit. In the finite depth case, $\| v\|_{L^2}$ is controlled by $\| \na_{x, z}v\|_{L^2}$ via Poincar\'e's inequality. 
\end{proof}
To prepare for the estimation of $F_0$, we recall the following. 
\begin{lemm}\cite[Lemma 3.28]{NP}\label{lemm.estABC} Let $s>1+\frac{d}{2}$ and $\sigma \in [-\mez,s-1]$, then there holds 
    \begin{equation}
        \label{eq.alphaBetaDiff}
        \|\alpha_1-\alpha_2\|_{X^{\sigma}(J)} + \|\beta_1 - \beta_2\|_{X^{\sigma}(J)} \le C(\|(f_1,f_2)\|_{H^s})\|f_1-f_2\|_{H^{\sigma +1}}
    \end{equation}
    and 
    \begin{equation}
        \label{eq.gammaDiff}
        \|\gamma_1-\gamma_2\|_{Y^{\sigma}(J)} \le C(\|(f_1, f_2)\|_{H^s})\|f_1-f_2\|_{H^{\sigma +1}},
    \end{equation}
where $C : \Rr_+ \to \Rr_+$ is a non-decreasing function depending only on $(d,s)$ and also on $\fd$ and $\|b_0\|_{H^{s}}$ in the finite-depth case. 
\end{lemm}
Let us proceed with the estimation of $ \|F_0\|_{Y^{s-\tdm}([z_1, 0])}$, where $F_0$ is given by \eqref{def:F0:contra}. We first apply \cref{prop.ABZ} to $v_2$ to have 
\bq\label{est:nav2:s-1}
\|\na_{x, z} v_2\|_{X^{s-1}([z_1, 0])} \le C(\|f_2\|_{H^s})\left( \|\zeta_2\|_{H^{s}}+ \|F_{0,2}\|_{Y^{s-1}(J)}+\|\nabla_{x,z}v_2\|_{X^{-\mez}(J)}\right).
\eq
Using \eqref{est:nav2:s-1}, \cref{lemm.estABC}, and \eqref{eq.estXYY} with $s_0=s_1=s-\tdm$ and $s_2=s-1$, we obtain
\begin{multline}\label{subeq.boundF}
    \|(\gamma_1-\gamma_2)\partial_zv_2\|_{Y^{s-\tdm}([z_1, 0])} \lesssim \|\gamma_1-\gamma_2\|_{Y^{s-\tdm}([z_1, 0])}\|\partial_z v_2\|_{X^{s-1}([z_1, 0])} \\
    \le C(\|(f_1,f_2)\|_{H^s})\|f_1-f_2\|_{H^{s-\mez}}\left( \|\zeta_2\|_{H^{s}}+ \|F_{0,2}\|_{Y^{s-1}(J)}+\|\nabla_{x,z}v_2\|_{X^{-\mez}(J)}\right).
\end{multline}
 On the other hand, by  \eqref{est:nav2:s-1}, \cref{lemm.estABC}, and \eqref{eq.estXXY}  with $s_0= s_1=s-\tdm$ and $s_2= s-2$,  we have  
\begin{align*}
    \|(\alpha_1-\alpha_2)\Delta_xv_2\|_{Y^{s-\tdm}([z_1, 0])}&\les \| \alpha_1-\alpha_2\|_{X^{s-\tdm}([z_1, 0])}\| \Delta_x v_2\|_{X^{s-1}([z_1, 0])}\\
    &\le    C(\|(f_1, f_2)\|_{H^s})\|f_1-f_2\|_{H^{s-\mez}} \|\nabla_{x,z}v_2\|_{X^{s-1}}.
\end{align*}
The same argument yields 
\begin{align*}
    \|(\beta_1 - \beta_2)\cdot\nabla_x\partial_zv_2\|_{Y^{s-\tdm}([z_1, 0])}
    &\le  C(\|(f_1, f_2)\|_{H^s})\|f_1-f_2\|_{H^{s-\mez}} \|\nabla_{x,z}v_2\|_{X^{s-1}}. 
\end{align*}
Hence, $\|(\alpha_1-\alpha_2)\Delta_xv_2\|_{Y^{s-\tdm}}$ and  $\|(\beta_1 - \beta_2)\cdot\nabla_x\partial_zv_2\|_{Y^{s-\tdm}}$ are also bounded by~\eqref{subeq.boundF}. We note that $s>1+\frac{d}{2}$ suffices for the above estimates. 

Next, we assume $s>\tdm+\frac{d}{2}$ to estimate quantities involving $k_i$. Since  $\p_z\varrho_i\p_y k_i \circ \ff_{f_i} = \p_z(k_i\circ \ff_{f_i})$ with $m=1$ in infinite depth and $m=H$ in finite depth, we have  
\begin{align*}
    F_{0,1}-F_{0,2} &= \left(m + \left(\frac{\partial_z\varrho_2}{1+|\nabla_x\varrho_2|^2}-m\right)\right)\p_z(k_2\circ \ff_{f_2}-k_1\circ \ff_{f_1}) \\
    &  \quad + \left(\frac{\partial_z\varrho_2}{1+|\nabla_x\varrho_2|^2} - \frac{\partial_z\varrho_1}{1+|\nabla_x\varrho_1|^2}\right)\p_z(k_1\circ \ff_{f_1}) =: A + B.    
\end{align*}
Set $G_j:=\frac{\partial_z\varrho_j}{1+|\nabla_x\varrho_j|^2}-m$. We have 
\[
\| G_j\|_{X^{s-1}(J)}\le C(\| f_j\|_{H^s}),\quad \| G_1-G_2\|_{X^{s-\tdm}(J)}\le C(\| (f_1, f_2)\|_{H^s})\| f_1-f_2\|_{H^{s-\mez}}.
\]
Combining these with the product rule \eqref{pr}, we deduce
\[
\begin{aligned}
\| A\|_{L^2_z H^{s-2}_x}&\les (1+\| G_2\|_{L^\infty_z H^s_x})\| \p_z(k_2\circ \ff_{f_2}-k_1\circ \ff_{f_1})\|_{L^2_zH^{s-2}_x}\\
&\le C(\| f_2\|_{H^s})\| k_2\circ \ff_{f_2}-k_1\circ \ff_{f_1}\|_{H^{s-1}}
\end{aligned}
\]
and 
\[
\begin{aligned}
\| B\|_{L^2_z H^{s-2}_x}&\le C(\| (f_1, f_2)\|_{H^s})\| f_1-f_2\|_{H^{s-\mez}}\|\p_z(k_1\circ \ff_{f_1}) \|_{L^2_z H^{s-\tdm}_x}.
\end{aligned}
\]
Using  \cref{lemm.FaaDiBruno} and the condition  $s>1+\frac{1+d}{2}$ gives
\begin{multline*}
\|\p_z(k_1\circ \ff_{f_1}) \|_{L^2_z H^{s-\tdm}_x}\les \|k_1\circ \ff_{f_1} \|_{H^{s-\mez}} \\
\le C(\| f_1\|_{H^s})(\| k_1\|_{H^{s-\mez}}+ \| \na k_1\|_{L^\infty})\le C(\| f_1\|_{H^s})\| k_1\|_{H^s}.
\end{multline*}
We have proven that 
\bq\label{est:F01-F02}
\begin{aligned}
\|F_{0,1}-F_{0,2}\|_{L^2_zH^{s - 2}_x} 
    & \le C(\|(f_1,f_2)\|_{H^s})\left( \|k_2\circ \ff_{f_2}-k_1\circ \ff_{f_1}\|_{H^{s-1}} +\|f_1-f_2\|_{H^{s-\mez}}\|k_1\|_{H^s}\right). 
    \end{aligned}
    \eq
It remains to estimate $\|F_{0,2}\|_{Y^{s-1}(J)}+\|\nabla_{x,z}v_2\|_{X^{-\mez}(J)}$ in \eqref{subeq.boundF}. 

Another application of \cref{lemm.FaaDiBruno} and the condition $s>\tdm+\frac{d}{2}$ gives 
\begin{align*}
    \|F_{0,2}\|_{Y^{s-1}} \le \|F_{0,2}\|_{L^2_zH^{s-\frac{3}{2}}_x} \le C(\|f_2\|_{H^{s}})\| k_2\circ \ff_{f_2}\|_{H^{s-\mez}}
    \le C(\|f_2\|_{H^{s}})\|k_2\|_{H^{s}}. 
\end{align*}
It follows from \cref{lemm.FirstStepIteration} that
    \begin{multline}\label{estv:diff:low}
    \|\nabla_{x,z}v\|_{X^{-\mez}(J)}\le C(\|(f_1,f_2)\|_{H^s})\left\{ \|h_1-h_2\|_{H^{\mez}} + \|k_2\circ \ff_{f_2}-k_1\circ \ff_{f_1}\|_{H^{s-1}}  \right.\\
    \left. +\Big(\|(h_1, h_2)\|_{H^s}+\|(k_1, k_2) \|_{H^s}\Big)\|f_1-f_2\|_{H^{s-\mez}} \right\},
    \end{multline}
where the condition $s>\tdm+\frac{d}{2}\ge 2$ is only needed when $k_i\ne 0$.
    
  Gathering the above estimates yields
 \begin{multline}
 \label{est:F0:Y}
     \|F_0\|_{Y^{s-\tdm}}  \le C(\|(f_1,f_2)\|_{H^s})\left\{ \|k_2\circ \ff_{f_2}-k_1\circ \ff_{f_1}\|_{H^{s-1}} \right.\\
 \left.+\|f_1-f_2\|_{H^{s-\mez}}\Big(\|(h_1, h_2)\|_{H^s}+\|(k_1, k_2)\|_{H^s}\Big)\right\}. 
 \end{multline}
In view of   \eqref{estv:Xs-tdm},  \eqref{estv:diff:low}, and \eqref{est:F0:Y},  we obtain
 \begin{multline}   \label{eq.basicEstDiff}
    \|\nabla_{x,z}v\|_{X^{s-\frac{3}{2}}([z_0, 0])} \le C(\|(f_1, f_2)\|_{H^s})\left\{\|h_1-h_2\|_{H^{s-\mez}} + \|k_2 \circ \ff_{f_2} - k_1 \circ \ff_{f_1}\|_{H^{s-1}} \right.  \\
 \left. +\Big(\| (h_1, h_2)\|_{H^s}+\|(k_1, k_2)\|_{H^s}\Big)\|f_1-f_2\|_{H^{s-\mez}} \right\}=:\Xi. 
 \end{multline}
Our next task is to upgrade \eqref{eq.basicEstDiff} to an estimate for $\na_{x, z}v$ in $H^{s-1}(\Rr^d\times (z_0, 0))$. 
\subsubsection{Estimates of $\|\na_{x, z}v\|_{H^{s-1}(\Rr^d\times (z_0, 0))}$} We write $s-1=k+\mu$ with $k\ge 0$ and $\mu\in[0,1)$, and claim that \begin{align}
    \label{eq.iterateEstDiff}
 \forall \ell \in \{0, \dots, k+1\},\quad   \|\partial_z^{\ell}\nabla_x^{k+1-\ell}v\|_{H^{\mu}_{x,z}(\Rr^d\times (z_0, 0))} & \le \Xi.
\end{align}
\medskip 

\underline{Case $\ell =0$.} By virtue of the interpolation \cref{lemm.interpol}, 
\[
    \|\nabla_x^{k+1}v\|_{H^{\mu}_{x,z}} \lesssim \|\nabla_x^{k+1}v\|_{L^2_zH^{\mu}_x} + \|\partial_z\nabla_x^{k+1}v\|_{L^2_zH^{\mu-1}_x}.
\]
 Since  both terms on the right-hand side are controlled by $ \|\nabla_{x,z}v\|_{X^{s-\frac{3}{2}}([z_0, 0])}$, \eqref{eq.iterateEstDiff}  follows from \eqref{eq.basicEstDiff}.
\medskip 

\underline{Case $\ell = 1$.}  Another application of \cref{lemm.interpol} gives
\[
    \|\partial_z\nabla_x^{k}v\|_{H^{\mu}_{x,z}} \lesssim \|\partial_z\nabla_x^{k}v\|_{L^2_zH^{\mu}_x} + \|\partial^2_z\nabla_x^{k}v\|_{L^2_zH^{\mu-1}_x}.
\]
The term $\|\partial_z\nabla_x^{k}v\|_{L^2_zH^{\mu}}$ is controlled using \eqref{eq.basicEstDiff}. As for $\|\partial^2_z\nabla^{k}v\|_{L^2_zH^{\mu-1}}$, we use the equation \eqref{eq.v-diff} for $\p_z^2v$, the product rule \eqref{pr}, the condition $s>1+\frac{d}{2}$, and \cref{lemm.coefsEstimInfinite}  to have
\begin{align*}
    \|\partial^2_z\nabla_x^{k}v\|_{L^2_zH^{\mu-1}_x} &\le \|\partial^2_zv\|_{L^2_zH^{k+\mu-1}_x} =\|\partial^2_zv\|_{L^2_zH^{s-2}_x}\\
    &\lesssim \|\Delta_x v\|_{L^2_zH^{s-2}_x} + \|(\alpha_2-m^2)\Delta_x v\|_{L^2_zH^{s-2}_x} + \|\beta_2\cdot \nabla_x\partial_zv\|_{L^2_zH^{s-2}_x} \\
    & \quad+ \|\gamma_2\p_zv\|_{L^2_zH^{s-2}_x} + \|F_0\|_{L^2_zH^{s-2}_x}\\
    &\les C(\| f_2\|_{H^s})\| \na_{x, z}v\|_{L^2_zH^{s-1}_x}+ \|F_0\|_{L^2_zH^{s-2}_x}\\
      &\les C(\| f_2\|_{H^s})\| \na_{x, z}v\|_{X^{s-\tdm}_x}+ \|F_0\|_{Y^{s-\tdm}}.
\end{align*}
Invoking \eqref{est:F0:Y} and \eqref{eq.basicEstDiff}, we conclude the proof of \eqref{eq.iterateEstDiff} for $\ell=1$.
\medskip 

\underline{Case $\ell\in \{2, \dots, k+1\}$}. This is only the case when $k\ge 1$, i.e. $s\ge 2$. In order to prove \eqref{eq.iterateEstDiff} for $\ell\in\{2, \dots, k+1\}$, we claim that 
\bq\label{claim:nav:inter}
\forall \ell\in\{1, \dots, k\},\quad \|\partial_z^{\ell}\nabla_x^{k+1-\ell}v\|_{H^{\mu}_{x,z}}  \les \Xi +A_{\ell},\quad A_{\ell}:= \sum_{j=1}^{\ell} \| \p_z^j v\|_{L^2_{x, z}}
\eq
and 
\bq\label{claim:nav:inter:2}
\|\partial_z^{k+1}v\|_{H^{\mu}_{x,z}}  \les \Xi +A_{k}+\| \na_{x, z}v\|_{H^{s-2}_{x, z}}.
\eq
Assume temporarily that \eqref{claim:nav:inter} and \eqref{claim:nav:inter:2} hold. In conjunction with the case $\ell=0$, they imply
\[
\| \na_{x, z}^{k+1}v\|_{H^\mu_{x,z}}\les \Xi+ A_k+\| \na_{x, z}v\|_{H^{s-2}_{x, z}}\les \Xi+\| \p_zv\|_{H^{k-1}_{x, z}}+\| \na_{x, z}v\|_{H^{s-2}_{x, z}}\les  \Xi+\| \na_{x, z}v\|_{H^{k-1+\mu}_{x, z}}
\]
since $s-2=k-1+\mu$. It follows that 
\[
\| \na_{x, z}v\|_{H^{k+\mu}_{x,z}}\les \Xi+ \| \na_{x, z}v\|_{H^{k-1+\mu}_{x,z}}+\| \na_{x, z}v\|_{L^2_{x,z}}.
\]
By interpolating  $\| \na_{x, z}v\|_{H^{k-1+\mu}}$ between $\| \na_{x, z}v\|_{L^2}$ and $\| \na_{x, z}v\|_{H^{k+\mu}}$, we obtain 
\[
\| \na_{x, z}v\|_{H^{k+\mu}_{x,z}}\les \Xi+\| \na_{x, z}v\|_{L^2_{x,z}}\les \Xi+\| \na_{x, z}v\|_{X^{-\mez}}.
\]
Recalling \cref{lemm.FirstStepIteration}, we conclude that $\| \na_{x, z}v\|_{H^{k+\mu}_{x,z}}\les \Xi$ which implies \eqref{eq.iterateEstDiff} for $\ell\in\{0, \dots, k+1\}$. 

\begin{proof}[Proof of \eqref{claim:nav:inter}] 
We proceed by induction on $\ell$. We note that  the left-hand side of \eqref{claim:nav:inter} has a total of $k+1$ derivatives with at most $k$ $z$-derivatives.  The  case $\ell=1$ has been obtained above. Assuming that $k\ge 2$ and that \eqref{claim:nav:inter} holds for some $\ell\in\{1, \dots, k-1\}$, we need to prove 
\bq\label{induction:dzellv:1}
\|\p_z^{\ell+1}\na_x^{k-\ell}v\|_{H^\mu_{x, z}}\les \Xi +A_{\ell+1}. 
\eq
We also assume $\mu \in (0,1)$. The case $\mu=0$  can be treated analogously  except that \cref{lemm.interpol} is not invoked.

By \cref{lemm.interpol},
\bq\label{induction:dzellv}
    \|\partial_z^{\ell+1}\nabla_x^{k-\ell}v\|_{H^{\mu}} \lesssim \|\p_z^{\ell+1}\na_x^{k-\ell}v\|_{L^2_zH^\mu_x} + \|\p_z^{\ell+2}\na_x^{k-\ell}v\|_{L^2_zH^{\mu-1}_x}. 
\eq
We write $\p_z^{\ell+1}=\p_z^{\ell-1}\p_z^2v$, $\p_z^{\ell+2}\na_x^{k-\ell}v=\p_z^\ell\p_z^2v$ and use equation \eqref{eq.v-diff} to  substitute 
\[
\p_z^2v=F_0-(\alpha_1\Delta_x + \beta_1\cdot \nabla_x \partial_z - \gamma_1\partial_z)v,
\]
which contains at most one $z$-derivative of $v$. Hence, $\p_z^{\ell+1}\na_x^{k-\ell}v$ contains at most $\ell$ derivatives in $z$ while $\p_z^{\ell+2}\na_x^{k-\ell}v$ contains at most $\ell+1$ derivatives in $z$ making it the more difficult term. Using the method for controlling  the second term that we will present below, one obtains 
\bq\label{prop5.1:term1}
\|\p_z^{\ell+1}\na_x^{k-\ell}v\|_{L^2_zH^\mu_x} \les \Xi+A_{\ell}.
\eq
This  will be utilized in controlling  the second term  $\|\p_z^{\ell+2}\na_x^{k-\ell}v\|_{L^2_zH^{\mu-1}_x}$.  We have 
\begin{align*}
    \|\partial_z^{\ell+2}\nabla_x^{k-\ell}v\|_{H_{x,z}^{\mu - 1}} &\le  \|\partial_z^{\ell}\nabla_x^{k-\ell}F_0\|_{L^2_zH_x^{\mu-1}} + m^2\|\p_z^{\ell}\na_x^{k-\ell}\Delta_xv\|_{L^2_zH_x^{\mu-1}} \\ 
    & \quad + \|\partial_z^{\ell}\nabla_x^{k-\ell}((\alpha_1-m^2)\Delta_x v)\|_{L^2_zH_x^{\mu-1}} + \|\partial_z^{\ell}\nabla_x^{k-\ell}(\beta_1\cdot\na_x\p_z v)\|_{L_z^2H_x^{\mu-1}} \\ 
    & \quad \quad + \|\partial_z^{\ell}\nabla_x^{k-\ell}(\gamma_1\partial_z v)\|_{L_z^2H_x^{\mu-1}}.    
\end{align*}
The control of $F_0$ follows from the following, whose proof is postponed to the end of this section. 
\begin{lemm}\label{lemm.F0Estimates} 
If $s>1+\frac{d}{2}$ and $s\ge 2$, then 
\begin{equation}
    \label{eq.bound-F0b}
   \begin{aligned}
    \| F_0\|_{H^{s-2}} &\le C(\|(f_1, f_2)\|_{H^s})\left\{\|k_1\circ \ff_{f_1}-k_2\circ \ff_{f_2}\|_{H^{s-1}} \right.\\
    &\qquad\left.+ \|f_1-f_2\|_{H^{s-\mez}}(\|(k_1, k_2)\|_{H^s} + \|(h_1, h_2)\|_{H^s})\right\},
    \end{aligned}
\end{equation}
where we assume furthermore that $s>\tdm+\frac{d}{2}$ when $k_i\ne 0$. 
\end{lemm}
Since we are proving \eqref{induction:dzellv:1} under the assumption that $k\ge 2$, we have $s=1+k+\mu\ge 3$. Hence, \cref{lemm.F0Estimates}  is applicable, giving  
 \[
 \|\partial_z^{\ell}\nabla_x^{k-\ell}F_0\|_{L^2_zH^{\mu-1}_x}\les  \|\partial_z^{\ell}F_0\|_{L^2_zH^{k-\ell+\mu-1}_x} \les   \|F_0\|_{H^{s-2}}\les \Xi,
 \]   
 where we have used that $k-\ell\ge 1$.  On the other hand,   the induction hypothesis implies
\[
\|\p_z^{\ell}\na_x^{k-\ell}\Delta_xv\|_{L^2_zH_x^{\mu-1}}\les \|\p_z^{\ell}\na_x^{k+1-\ell}v\|_{L^2_zH_x^{\mu}}\les \|\p_z^{\ell}\na_x^{k+1-\ell}v\|_{H_{x, z}^{\mu}} \les \Xi +A_{\ell}.
\]
Combining Leibniz's rule with the product rule in \cref{thm.sobolevProducts} yields 
\begin{align*}
    \|\partial_z^{\ell}\nabla_x^{k-\ell}((\alpha_1-m^2)\Delta_x v)\|_{L^2_zH^{\mu -1}} &\les \|\partial_z^{\ell}((\alpha_1-m^2)\Delta_x v)\|_{L^2_zH^{k+\mu -\ell -1}_x}\\
     &\lesssim \sum_{0\le j \le \ell}\|\partial_z^j(\alpha_1-m^2)\|_{L^{\infty}_zL^{p_j}_x}\|\partial_z^{\ell-j}\Delta_xv\|_{L^{2}_zW^{k+\mu-\ell-1,p'_j}_x}  \\
    &\quad + \sum_{0\le j \le \ell}\|\partial_z^j(\alpha_1-m^2)\|_{L_z^{\infty}W^{k+\mu-\ell-1,q_j}_x}\|\partial_z^{\ell-j}\Delta_xv\|_{L^2_zL^{q'_j}_x}   
\end{align*}
with $\frac{1}{p_j}+\frac{1}{p'_j}=\frac{1}{q_j}+\frac{1}{q'_j}=\frac{1}{2}$ and $p'_j, q_j \neq \infty$. 

Note that $s-1-j\ge 1$ for $j\le \ell\le k-1$. We choose
\[
\begin{cases}
\frac{1}{p_j}=\mez - \frac{s-1-j}{d}\quad\text{if~}s-1-j< \frac{d}{2}, \\
p_j=\infty \quad\text{if~}s-1-j= \frac{d}{2}, \\
2\ll p_j<\infty \quad\text{if~}s-1-j= \frac{d}{2}, 
\end{cases}
\]
so that $H^{s-1-j}_x \hookrightarrow L^{p_j}_x$. Then, since $s>1+\frac{d}{2}$, we have $H^{k+\mu-\ell-1+j}_x  \hookrightarrow  W^{k+\mu-\ell-1,p'_j}_x$ in all cases. Consequently
\[
    \|\partial_z^j(\alpha_1-m^2)\|_{L^{\infty}_zL^{p_j}_x} \lesssim \|\p_z^j(\alpha_1-m^2)\|_{L^{\infty}_zH^{s-j-1}_x} \le C(\|f_1\|_{H^s})
\]
by \eqref{eq.pzEstim1}, and 
\[
    \|\partial_z^{\ell-j}\Delta_xv\|_{L^{2}_zW^{k+\mu-\ell-1,p'_j}_x}   \lesssim \|\p_z^{\ell-j} \Delta_x v\|_{L^2_zH^{k+\mu-\ell-1+j}_x}=:B_j.
    \]
    When $j=\ell$, 
    \[
    B_\ell= \|\Delta_x v\|_{L^2_zH^{k+\mu-1}_x}\les  \| \na_xv\|_{L^2_zH^{k+\mu}_x}= \| \na_xv\|_{L^2_zH^{s-1}_x}\le \| \na_xv\|_{X^{s-\tdm}}\le \Xi
    \]
  in view of \eqref{eq.basicEstDiff}. For $0\le j\le k-1$, we have 
  \[
  B_j\les  \|\p_z^{\ell-j}  v\|_{L^2_zH^{k+\mu-\ell+1+j}_x}\les \|\p_z^{\ell-j}  v\|_{L^2_{x, z}}+ \|\p_z^{\ell-j} \na_x^{k-\ell+1+j} v\|_{L^2_zH^\mu_x}\les \Xi+A_{\ell}
  \]
  by the induction hypothesis. The pair $(q_j,q'_j)$ can be chosen using an analogous argument. We omit further details. Thus 
  \[
  \|\partial_z^{\ell}\nabla_x^{k-\ell}((\alpha_1-m^2)\Delta_x v)\|_{L^2_zH_x^{\mu-1}}\les  \Xi+A_{\ell}. 
  \]
The term $\|\partial_z^{\ell}\nabla_x^{k-\ell}(\beta_1\cdot\na_x\partial_x v)\|_{L_z^2H_x^{\mu-1}}$ can be treated similarly since $\beta_1$ and $\alpha_1-m^2$ have the same regularity. We only note that $B_j$ needs to be replaced with 
\[
B_j':=\|\p_z^{\ell-j} \na_x \p_zv\|_{L^2_zH^{k+\mu-\ell-1+j}_x}\le \|\p_z^{\ell-j+1}  v\|_{L^2_zH^{k+\mu-\ell+j}_x}
\]
which has $\ell+1$ derivatives in $z$ when $j=0$. But when $j=0$, we have 
 \[
 \|\p_z^{\ell-j+1}  v\|_{L^2_zH^{k+\mu-\ell}_x}=\|\p_z^{\ell+1}  v\|_{L^2_zH^{k+\mu-\ell}_x}\les\|\p_z^{\ell+1}  v\|_{L^2_{x, z}}+\|\p_z^{\ell+1}  \na_x^{k-\ell} v\|_{L^2_zH^\mu_x}\les A_{\ell+1}+ \Xi
 \]
by invoking \eqref{prop5.1:term1}. 

Next, we consider  the $\gamma$ term. It suffices to estimate 
\[
   \|\partial_z^j\gamma_1\|_{L^{\infty}L^{p_j}} \text{ and } \|\partial_z^{\ell-j+1}v\|_{L^2_zW^{k+\mu-\ell-1,p'_j}}    
\]
for  $j\in\{0,\dots,\ell\}$, where $\frac{1}{p_j}+\frac{1}{p'_j}=\frac{1}{2}$ and $p'_j\neq \infty$. We recall that $s-2-j\ge \mu >0$. Note that if we had included $\ell=k+1$ in \eqref{claim:nav:inter}, then when $j=\ell=k$ in the induction hypothesis we would have had $s-2-j=\mu-1<0$. We choose 
\[
\begin{cases}
\frac{1}{p_j}=\mez - \frac{s-2-j}{d}\quad\text{if~}s-2-j< \frac{d}{2}, \\
p_j=\infty \quad\text{if~}s-2-j= \frac{d}{2}, \\
2\ll p_j<\infty \quad\text{if~}s-2-j= \frac{d}{2}, 
\end{cases}
\]
so that $H^{s-2-j}_x\hookrightarrow L^{p_j}_x$ and $H^{k+\mu-\ell+j}_x \hookrightarrow  W^{k+\mu-\ell-1,p'_j}_x$ for $s>1+\frac{d}{2}$. Then we have  
\[
    \|\partial_z^j\gamma_1\|_{L^{\infty}_zL^{p_j}_x} \lesssim \|\partial_z^j\gamma_1\|_{L^\infty_z H^{s-2-j}_x} \le  C(\|f_1\|_{H^s})
\]  
thanks to \eqref{dzgamma:L2}, and,  since $1\le \ell-j+1\le \ell+1$, 
\begin{align*}
    \|\partial_z^{\ell-j+1}v\|_{L^2_zW^{k+\mu-\ell-1,p'_j}_x}& \lesssim \|\partial_z^{\ell-j+1}v\|_{L^2_zH^{k+\mu-\ell+j}_x}.
\end{align*} 
When $j\ge 1$, the induction hypothesis gives 
\[
\|\partial_z^{\ell-j+1}v\|_{L^2_zH^{k+\mu-\ell+j}_x}\les \|\partial_z^{\ell-j+1}v\|_{L^2_{x,z}}+\|\partial_z^{\ell-j+1}\na_x^{k-\ell+j}v\|_{L^2_zH^\mu_x} \les \Xi+A_{\ell}. 
\]
When  $j=0$, we use  \eqref{prop5.1:term1} to obtain
\[
\|\partial_z^{\ell-j+1}v\|_{L^2_zH^{k+\mu-\ell+j}_x}=\|\partial_z^{\ell+1}v\|_{L^2_zH^{k+\mu-\ell}_x}\les \|\partial_z^{\ell+1}v\|_{L^2_{x, z}}+ \|\partial_z^{\ell+1}\na_x^{k-\ell}v\|_{L^2_zH^{\mu}_x} \les A_{\ell+1}+\Xi.
\]
Thus 
\[
 \|\partial_z^{\ell}\nabla_x^{k-\ell}(\gamma_1\partial_z v)\|_{L_z^2H_x^{\mu-1}}\les \Xi+A_{\ell+1}.
\]
The proof of \eqref{claim:nav:inter} is complete.
\end{proof}

\begin{proof}[Proof of \eqref{claim:nav:inter:2}]    We first use \eqref{eq.v-diff} for $\p_z^2v$ to have
\begin{align*}
    \|\partial_z^{k+1}v\|_{H^{\mu}} &\le \|\partial_z^{k-1}F_0\|_{H^{\mu}} +m^2 \|\partial_z^{k-1}\Delta_xv\|_{H^{\mu}}  + \|\partial_z^{k-1}((\alpha_2-m^2)\Delta_x v)\|_{H^{\mu}} \\
    & \quad+ \|\partial_z^{k-1}(\beta_2\cdot\na_x\partial_x v)\|_{H^{\mu}} + \|\partial_z^{k-1}(\gamma_2\partial_z v)\|_{H^{\mu}}, 
\end{align*}
where  $\|\partial_z^{k-1}F_0\|_{H^{\mu}}\les  \Xi$ in view of  \eqref{eq.bound-F0b}.  By  \eqref{claim:nav:inter} with $\ell=k-1$, 
\begin{align*}
 \|\partial_z^{k-1}\Delta_xv\|_{H^{\mu}}
  &\les   \Xi+A_{k-1}.
 \end{align*}
    As  for the $\alpha$-term, we need to estimate 
\begin{align*}
    \|\p_z^j(\alpha_1-m^2)\p_z^{k-1-j}\Delta_xv\|_{H^{\mu}} &\lesssim \|\p_z^j(\alpha_1-m^2)\|_{L^{p_j}}\|\p_z^{k-1-j}\Delta_xv\|_{W^{\mu,p'_j}} \\
    &+ \|\p_z^j(\alpha_1-m^2)\|_{W^{\mu,q_j}}\|\p_z^{k-1-j}\Delta_xv\|_{L^{q'_j}} 
\end{align*}
for $j\in\{0, \dots, k-1\}$, where $\frac{1}{p_j}+\frac{1}{p'_j}=\frac{1}{q_j}+\frac{1}{q'_j}=\frac{1}{2}$ and $p'_j,q_j \neq \infty$.  Note that $s-\mez-j \ge \tdm+\mu>0$. We choose 
\[
\begin{cases}
\frac{1}{p_j}=\mez - \frac{s-\mez-j}{d}\quad\text{if~}s-\mez-j< \frac{d+1}{2}, \\
p_j=\infty \quad\text{if~}s-\mez-j= \frac{d+1}{2}, \\
2\ll p_j<\infty \quad\text{if~}s-\mez-j= \frac{d+1}{2}, 
\end{cases}
\]
 so that  $H^{s-\mez-j}_{x, z}\hookrightarrow L^{p_j}_{x, z}$ and $H^{\mu+j}_{x, z}\hookrightarrow W^{\mu,p'_j}_{x, z}$ for $s>1+\frac{d}{2}$. Hence, we have 
 \[
    \|\p_z^j(\alpha_1-m^2)\|_{L^{p_j}_{x, z}} \lesssim \|\p_z^j(\alpha_1-m^2)\|_{H^{s-\mez-j}_{x, z}} \le C(\|f_1\|_{H^s})
\] 
by  \eqref{eq.pzEstim1b}, and 
\begin{align*}
    \|\p_z^{k-1-j}\Delta_xv\|_{W^{\mu,p'_j}_{x, z}} &\lesssim \|\p_z^{k-1-j}\Delta_xv\|_{H^{\mu +j}_{x, z}}\les \|\p_z^{k-1-j}\na_xv\|_{H^{\mu +j+1}_{x, z}} \\
    &\les  \|\p_z^{k-1-j}\na_xv\|_{H^{\mu}_{x, z}} + \|\na_{x, z}^{j+1}\p_z^{k-1-j}\na_xv\|_{H^{\mu}_{x, z}}. 
\end{align*}
Since  $\na_{x, z}^{j+1}\p_z^{k-1-j}\na_xv$ contains a total of $k+1$ derivatives with at most $k$ $z$-derivatives, using \eqref{claim:nav:inter}  with $\ell\le k-1$ gives 
    \[
     \|\na_{x, z}^{j+1}\p_z^{k-1-j}\na_xv\|_{H^{\mu}_{x, z}}\les \Xi+A_{k}.
     \]
     On the other hand, $ \|\p_z^{k-1-j}\na_xv\|_{H^{\mu}_{x, z}}\les \| \na_xv\|_{H^{s-2}}$ since $k-1-j+\mu\le s-2$. 
Thus we obtain
 \[
 \|\partial_z^{k-1}((\alpha_1-m^2)\Delta_x v)\|_{H^{\mu}}\les \Xi+A_{k}+\| \na_xv\|_{H^{s-2}}.
 \]
 The $\beta$-term can be treated similarly except that we now have 
\[
\begin{aligned}
 \|\p_z^{k-1-j}\na_x\p_zv\|_{W^{\mu,p'_j}}& \les \|\p_z^{k-1-j}\na_x\p_zv\|_{H^{\mu +j}_{x, z}}\les  \|\p_z^{k-1-j}\p_zv\|_{H^{\mu+j+1}_{x, z}}\\
 &\les   \|\p_z^{k-1-j}\p_zv\|_{H^\mu_{x, z}}+ \|\na_{x, z}^{j+1}\p_z^{k-1-j}\p_zv\|_{H^{\mu}_{x, z}}\\
 &\les \| \p_zv\|_{H^{s-2}}+ \Xi+A_k.
 \end{aligned}
   \]
Hence, 
 \[
\|\partial_z^{k-1}(\beta_1\cdot\na_x\partial_x v)\|_{H^{\mu}} \les \Xi+A_{k}+\| \na_xv\|_{H^{s-2}}.
 \]
For the $\gamma$ term, we  need to control for any $j\in\{0, \dots, k-1\}$ the quantities $\|\p_z^j\gamma_1\|_{L^{p_j}}$ and $\|\p_z^{k-1-j}\p_zv\|_{W^{\mu,p'_j}}$ where $\frac{1}{p_j}+\frac{1}{p'_j}=\mez$ and $p'_j\neq \infty$. We recall that $s-\tdm-j\ge \mez+\mu>0$. We choose 
\[
\begin{cases}
\frac{1}{p_j}=\mez - \frac{s-\tdm-j}{d+1}\quad\text{if~}s-\tdm-j< \frac{d+1}{2}, \\
p_j=\infty \quad\text{if~}s-\tdm-j= \frac{d+1}{2}, \\
2\ll p_j<\infty \quad\text{if~}s-\tdm-j= \frac{d+1}{2}, 
\end{cases}
\]
so that  $H^{s-\frac{3}{2}-j} \hookrightarrow L^{p_j}$ and $H^{\mu +1 +j -\varepsilon}\hookrightarrow W^{\mu,p'_j}$ for any $0<\eps<\min\{s -(1+\frac{d}{2}), \mez\}$. 
It follows that $\|\p_z^j\gamma_1\|_{L^{p_j}} \le C(\|f_1\|_{H^s})$ thanks to \eqref{dzgamma:L2}, and for any $\nu>0$, 
\begin{align*}
    \|\p_z^{k-1-j}\p_zv\|_{W^{\mu,p'_j}} &\les \|\p_z^{k-1-j}\p_zv\|_{H^{\mu +j+1-\varepsilon}} \leq C(\nu)\|\p_z^{k-1-j}\p_zv\|_{H^{\mu}} + \nu \|\p_z^{k-1-j}\p_zv\|_{H^{\mu +j+1}}\\
    &\leq  C(\nu)\|\p_z^{k-j-1}\p_zv\|_{H^\mu}+ \nu\|\na_{x, z}^{j+1}\p_z^{k-j}v\|_{H^{\mu}}\\ 
    & \les \| \p_zv\|_{H^{s-2}}+\Xi+A_k + \nu \|\p_z^kv\|_{H^{\mu}}. 
\end{align*} 
Gathering the above estimates leads to 
\[
    \|\p_z^kv\|_{H^{\mu}} \les  \| \na_{x,z}v\|_{H^{s-2}}+\Xi+A_k + \nu \|\p_z^kv\|_{H^{\mu}}\quad\forall \nu>0.
\]
This in turn implies the desired estimate \eqref{claim:nav:inter:2}. 
\end{proof}

\subsubsection{The finite-depth case}

In order to complete the proof of  \cref{prop.diffEstimateSobolev} in the finite-depth case, it remains to prove \eqref{eq.lowNormEst-v1} and \eqref{eq.lowNormEst-v2}  on a strip $\mathbb{R}^d \times (-1,-1+\eta)$ for some small $\eta>0$ that depends  only on $\|b_0\|_{H^{s+\mez}}$ and $\|f_i\|_{H^s}$, $i\in\{1,2\}$. We have
\begin{align*}
    \operatorname{div}(\mathcal{A}_i\na_{x,z}v_i^{(j)})&=-\p_z(k_i\circ \ff_{f_i}(x,z))=:F_i^{(j)}(x,z),\quad (x, z)\in \Rr^d\times (-1, 0),\\
    \mathcal{A}_j\na_{x,z}v_i^{(j)}(x,0)\cdot e_{d+1} &= k_i(x,b(x))=:g_i^{(j)}(x),\quad x \in \mathbb{R}^d,
\end{align*}
where $\mathcal{A}_i$ is given by \eqref{def:Arho} with $\varrho_i$ in place of $\varrho_b$, and $k_i=0$ if $j=1$. Hence, $v^{(j)}=v_1^{(j)}-v_2^{(j)}$ satisfies
\begin{align*}
    \operatorname{div}(\mathcal{A}_1\na_{x,z}v^{(j)}) &= F_1^{(j)}-F_2^{(j)} + \operatorname{div}((\mathcal{A}_2-\mathcal{A}_1)\na_{x,z}v_2^{(j)})=: F^{(j)}, \\
    \mathcal{A}_1\na_{x,z}v^{(j)}(x,0) \cdot e_{d+1} &= g_1^{(j)}(x) - g_2^{(j)}(x) + (\mathcal{A}_2-\mathcal{A}_1)\na_{x,z}v_2^{(j)}(x,0)\cdot e_{d+1} =:g^{(j)}(x). 
\end{align*}
We have  $\|\mathcal{A}_i-\mathcal{A}_0\|_{H^{s-\mez}(\Rr^d\times (-1, 0))}\le C(\| b_0\|_{H^{s}}, \|f_i\|_{H^{s}})$, where $\cA_0$ is given by \eqref{cA-cA0}.  For $s>1+\frac{d}{2}$ and $0<\nu<\min\{1, s-1-\frac{d}{2}\}$, since $H^{s-\mez}_{x, z} \subset C^\nu_{x, z}$, we can choose $\eta$ sufficiently small so that \eqref{verify:smalla} holds with $\cA=\cA_j$. 

The argument leading to  \eqref{estw:Omega-b} yields
\begin{align*}
    \|v^{(j)}\|_{H^{s}(\Rr^d\times (-1, -1+\eta))}&\le C(\|b_0\|_{H^{s}},\|f_1\|_{H^{s}})\Big(\|F^{(j)}\|_{H^{s-2}(\Rr^d\times (-1, -1+ 2\eta))} \\
    &\qquad + \|g^{(j)}\|_{H^{s-\tdm}(\Rr^d)} + \|w\|_{H^1(\Rr^d\times (-1, -1+2\eta))}\Big). 
\end{align*}

Since $s-\mez>\frac{d+1}{2}$, the product rule in \cref{coro:productdomain} implies  
\begin{align*}
    \|F^{(j)}\|_{H^{s-2}}&\le \|F_1^{(j)} - F_2^{(j)}\|_{H^{s-2}} + \|(\mathcal{A}_2-\mathcal{A}_1)\na_{x,z}v_2^{(j)}\|_{H^{s-1}}\\
   & \lesssim \|F_1^{(j)} - F_2^{(j)}\|_{H^{s-2}} + \|\mathcal{A}_2-\mathcal{A}_1\|_{H^{s-1}}\|\na_{x,z}v_2^{(i)}\|_{H^{s-\mez}},
\end{align*}
where 
\[
    \|F_1^{(2)} - F_2^{(2)}\|_{H^{s-2}}\le \|k_1 \circ \ff_{f_1} - k_2\circ \ff_{f_2}\|_{H^{s-1}}.
\]
On the other hand, it follows from \cref{lemm.alphaGammaEst} that
\[
    \|\mathcal{A}_2-\mathcal{A}_1\|_{H^{s-1}}\le C(\|(f_1,f_2)\|_{H^s},\|b_0\|_{H^{s+\mez}})\|f_1-f_2\|_{H^{s-\mez}}.
\]
Using \cref{prop.GSest-v} with $r=s-\mez$, we find 
\bq\label{nav2:Hs-mez}
     \|\na_{x,z}v_2^{(j)}\|_{H^{s-\mez}}\le C(\| f_2\|_{H^s}, \| b_0\|_{H^{s}})(\| h_2\|_{H^s}+\| k_2\|_{H^{s-\mez}}+\| \na k_2\|_{L^\infty}).
\eq
When $k_i\ne 0$, we use the assumption that $s>\tdm+\frac{d}{2}$ to bound $\| \na k_2\|_{L^\infty}\les \| k_2\|_{H^s}$. 

Thus $\|F^{(i)}\|_{H^{s-2}}$ is controlled  by the right-hand side of \eqref{eq.lowNormEst-v1}--\eqref{eq.lowNormEst-v2}. Since $g^{(2)}_1-g^{(2)}_2=(k_1\circ \ff_{f_1}-k_2\circ \ff_{f_2})\vert_{z=-1}$,   the trace inequality implies 
\[
\|g^{(j)}\|_{H^{s-\tdm}_x}\le C\| k_1\circ \ff_{f_1}-k_2\circ \ff_{f_2}\|_{H^{s-1}_{x, z}}.
\]
Finally, owing to \cref{lemm.FirstStepIteration}, $\|v^{(j)}\|_{H^1}$ is also controlled by the right-hand side of \eqref{eq.lowNormEst-v1}--\eqref{eq.lowNormEst-v2}. Therefore, the proof of \cref{prop.diffEstimateSobolev} in the finite-depth case is complete.

\subsubsection{Proof of \cref{lemm.F0Estimates}} 
We will use the following estimates. 

\begin{lemm}\label{lemm.alphaGammaEst} Consider $s>1+\frac{d}{2}$ satisfying $s\ge 2$. Then we have 
    \[
        \| \alpha_1-\alpha_2\|_{H^{s-1}_{x, z}} + \|\beta_1 - \beta_2\|_{H^{s-1}_{x, z}}+\|\gamma_1-\gamma_2\|_{H^{s-2}_{x, z}}  \le C(\|(f_1, f_2)\|_{H^s})\|f_1-f_2\|_{H^{s-\mez}},
        \]
       where $C : \Rr_+ \to \Rr_+$ is a non-decreasing function and depends only on $(d,s)$, and also on $\fd$ and $\|b_0\|_{H^{s}}$ in the finite-depth case. 
\end{lemm}

\begin{proof} 
We only consider the infinite-depth case.  We recall that  $\alpha_i-1= \frac{(\partial_z\varrho)^2}{1+|\nabla_x\varrho|^2}-1 = G(\tilde{\nabla}\varrho_i)$ for some smooth function  $G$ of  $\tilde{\na}\varrho=(\na_x\varrho, \p_z\varrho - 1)$. Since $s-2\ge 0$, applying \cref{est:nonl} (ii) gives 
\begin{align*}
    \|\alpha_1 - \alpha_2\|_{H^{s-2}} 
    & \le C(\|(\nabla \varrho_1, \nabla \varrho_2)\|_{L^{\infty}})\|\na\varrho_1-\na\varrho_2\|_{H^{s-1}} \\
    & \le C(\|f_1\|_{H^s}, \|f_2\|_{H^s}) \|f_1-f_2\|_{H^{s-\mez}},
\end{align*}
where we have used \cref{lemm:varrho} (ii) -- (iii) together with the condition $s>1+\frac{d}{2}$.

The control of $ \|\beta_1 - \beta_2\|_{H^{s-1}_{x, z}}$ can be obtained along the same lines. Regarding $\gamma$, let us estimate the typical term $\frac{\alpha}{\p_z\varrho}\Delta_x\varrho$. For $s>1+\frac{d}{2}$ and $s\ge 2$, the product rule in \cref{coro:productdomain} is applicable with $(s_0, s_1, s_2)=(s-2, s-\mez, s-2)$, yielding
\[
\| \frac{\alpha}{\p_z\varrho}\Delta_x\varrho\|_{H^{s-2}}\les \big(\|\frac{\alpha}{\p_z\varrho}-1\|_{H^{s-\mez}}+1\big)\| \Delta_x\varrho\|_{H^{s-2}}.
\]
Then we conclude by using \eqref{eq.pzEstim1b} and \cref{lemm:varrho} (iii).
\end{proof}

For the proof of \cref{lemm.F0Estimates}, we assume that $s>1+\frac{d}{2}$ and $s\ge 2$. When $k_i\ne 0$, we assume furthermore that $s>\tdm+\frac{d}{2}$.  We recall from \eqref{def:F0:contra} that 
 \[
    F_0= (F_{0,1}-F_{0,2}) -(\alpha_1-\alpha_2)\Delta_x v_2 - (\beta_1-\beta_2)\cdot \nabla_x\partial_z v_2 +(\gamma_1-\gamma_2)\partial_zv_2.    
\]
To fix idea we shall only consider the infinite-depth case. 
For $s>1+\frac{d}{2}=\mez+\frac{d+1}{2}$ satisfying $s\ge 2$, the product rule in \cref{coro:productdomain} can be applied with $(s_0, s_1, s_2)=(s-2, s-1, s-\tdm)$ and $(s_0, s_1, s_2)=(s-2, s-2, s-\mez)$, yielding
\begin{align*}
&\| (\alpha_1-\alpha_2)\Delta_x v_2\|_{H^{s-2}}\les \| \alpha_1-\alpha_2\|_{H^{s-1}}\| \Delta_x v_2\|_{H^{s-\tdm }},\\
&\| (\beta_1-\beta_2)\cdot \nabla_x\partial_z v_2\|_{H^{s-2}}\les \| \beta_1-\beta_2\|_{H^{s-1}}\| \nabla_x\partial_z v_2\|_{H^{s-\tdm }},\\
&\| (\gamma_1-\gamma_2)\partial_zv_2\|_{H^{s-2}}\les \| \gamma_1-\gamma_2\|_{H^{s-2}}\| \partial_zv_2\|_{H^{s-\mez}}.
\end{align*}
In conjunction with the estimate \eqref{nav2:Hs-mez} and \cref{lemm.alphaGammaEst}, it follows that 
\bq
\begin{aligned} 
&\| (\alpha_1-\alpha_2)\Delta_x v_2 - (\beta_1-\beta_2)\cdot \nabla_x\partial_z v_2 +(\gamma_1-\gamma_2)\partial_zv_2\|_{H^{s-2}}\\
&\le C(\| (f_1, f_2)\|_{H^s})\| f_1-f_2\|_{H^{s-\mez}}(\| h_2\|_{H^s}+\| k_2\|_{H^{s-\mez}}+\| \na k_2\|_{L^\infty}).
\end{aligned}
\eq
When $k_i\ne 0$,  we write 
\[
    F_{0,1}-F_{0,2} = (1+G(\nabla \varrho_2))\p_z(k_2\circ\ff_{f_2} - k_1\circ\ff_{f_1}) - (G(\nabla \varrho_2)-G(\nabla \varrho_1))\p_z(k_1\circ \ff_{f_1}),
\]
where $G(\nabla \varrho) = \frac{\partial_z \varrho}{1+|\nabla_x\varrho|^2}-1$ is a smooth function of $\tilde{\nabla} \varrho = (\na_x\varrho, \p_z\varrho - 1)$. For $s>1+\frac{d}{2}$ satisfying $s\ge 2$, we can apply the product rule in \cref{coro:productdomain}  with $(s_0, s_1, s_2)=(s-2, s-\mez, s-2)$ and $(s_0, s_1, s_2)=(s-2, s-1, s-\tdm)$ to obtain 
\begin{align*}
\| F_{0,1}-F_{0,2} \|_{H^{s-2}}&\les  \| 1+G(\nabla \varrho_2)\|_{H^{s-\mez}}\| \p_z(k_2\circ\ff_{f_2} - k_1\circ\ff_{f_1})\|_{H^{s-2}}\\
&\qquad +\| G(\nabla \varrho_2)-G(\nabla \varrho_1)\|_{H^{s-1}}\| \p_z(k_1\circ \ff_{f_1})\|_{H^{s-\tdm}}\\
&\le  C(\| (f_1, f_2)\|_{H^s})\Big(\| k_2\circ\ff_{f_2} - k_1\circ\ff_{f_1}\|_{H^{s-1}}+\| f_1-f_2\|_{H^{s-\mez}}\| k_1\circ \ff_{f_1}\|_{H^{s-\mez}}\Big).
\end{align*}
By virtue of \cref{lemm.FaaDiBruno}, 
\[
\| k_1\circ \ff_{f_1}\|_{H^{s-\mez}}\le C(\| f_1\|_{H^s})(\| k_1\|_{H^{s-\mez}}+\|\na  k_1\|_{L^\infty}).
\]
When $k_i\ne 0$, we use the stronger  assumption that $s>\tdm+\frac{d}{2}$ to bound $\| \na k_i\|_{L^\infty}\les \| k_i\|_{H^s}$. Gathering the above estimates yields the desired estimate \eqref{eq.bound-F0b}.

\section{Local Well-Posedness: proof of the main theorem}\label{sec.lwp}

In this section it is assumed that $d \geq 1$ and $s>\tdm+\frac{d}{2}$. We recall the estimates \eqref{gof:Hs} and \eqref{gofi:Hs} for compositions with $\ff_f$ and $\ff_f^{-1}$, as they will be used often. 

Heuristically, our main \cref{thm.main} follows from the a priori and contraction estimates proven in \cref{sec.apriori} and \cref{sec.contraction}. This is however not straightforward since $g$ satisfies a transport equation in the moving domain $\Omega_f$. The vanishing viscosity method used for the Muskat problem \cite{NP} is thus not applicable.  We shall devise an iterative scheme to construct approximate solutions $(f_n,g_n)$ whose existence is easier to establish. In fact, we will construct  $\tilde{g}(t): \Rr^d\times J \to \Rr$ and define $g(t):= \tilde{g}(t) \circ \ff^{-1}_{f(t)}$. We recall from \eqref{eq:tildeg} that $\tilde{g}$ satisfies 
\[
\p_t \tilde g+\ol{u}\cdot \na_{x, z}\tilde{g}+\gamma'(\varrho)u_y\circ \ff_f=0, \quad \text{in~} \Rr^d\times J,
\]
where $\ol{u}=\ol{u}(f, g)$ is given by \eqref{tildeu} in terms of $u\circ \ff_f$, $\na_{x, z}\varrho$ and $\p_t \varrho$.  The usual strategy is to linearize the preceding equation at step $n$ using the approximate solution at step $n-1$, namely to construct $\tilde{g}_n$ solving  the linear transport equation
\bq\label{eq.gn-fake}
    \partial_t\tilde{g}_n + \ol{u}_{n-1}\cdot \nabla \tilde{g}_n + \gamma'(\varrho_{n-1})u_{n-1, y}\circ \ff_{f_{n-1}}=0\quad\text{in } \Rr^d\times J.
\eq
If we simply choose $\ol{u}_n=\ol{u}(f_{n-1}, g_{n-1})$, then $\ol{u}_n$ is not tangent to $\Rr^d\times \p J$ unless $f_{n-1}$ solves the exact nonlinear equation \eqref{eq:f}, i.e., 
\bq\label{eq:f:n-1}
\p_t f_{n-1}=-G[f_{n-1}]\Gamma(f_{n-1})-\Big(N_{n-1}\cdot\cS[f_{n-1}]g_{n-1}+g_{n-1}\Big)\vert_{\Sigma_{f_{n-1}(t)}}.
\eq
However, this would produce a {\it coupled} system for   $f_{n-1}$ and $\tilde{g}_{n-1}$, whose solvability is not simple.   We will make a different choice of $\ol{u}_n$  in \eqref{eq.gn-fake} such that on one hand it is tangent to $\Rr^d\times \p J$, and on the other hand  it {\it decouples} the equations for $f_n$ and $\tilde{g}_n$ so as to facilitate their  solvability. The $\ol{u}_n$ is given by \eqref{def:olum} in consistence with the nonlinear  $f_n$-equation given by \eqref{eq.fn} and \eqref{def:Rn-1}. $f_n$ will then be solved using a fractional regularization $\nu |D|^{1+\nu}$ for small $\nu>0$.  

\subsection{Approximate solutions and uniform estimates}\label{sec.uniform-bounds}

Let $R>0$ and $(f_0,g_0) \in H^s(\mathbb{R}^d) \times H^s(\Omega_{f_0})$ such that  
\begin{align}
    \label{eq.assump-f0g0}
    \|f_0\|_{H^s} \le R, \quad \|g_0 \circ \ff_{f_0}\|_{H^s}\le \eps=\varepsilon(R) \ll 1,\\ \label{RT:f0}
    \inf_{x\in \Rr^d}\left(\gamma(f_0) - G[f_0]\Gamma(f_0)\right)(x)\geqslant 2\mathfrak{a}>0,\\ \label{separation:f0} 
    \inf_{x \in \mathbb{R}^d} \left(f_0(x) - b(x)\right) \geqslant  2\mathfrak{d} >0  \text{ in the finite-depth case,}
\end{align}
where $\eps=\varepsilon(R)$ will be chosen later. By virtue of \eqref{gof:Hs}, $\|g_0 \circ \ff_{f_0}\|_{H^s}$ can be arbitrarily small provided $\|g_0\|_{H^s(\Omega_{f_0})}$ is sufficiently small. We impose firther that $\varepsilon (R)$ is chosen such that 
\begin{equation}
    \label{eq.assump-smallness2}
    C_0(2R,\mathfrak{d})2\varepsilon(R) \le \frac{1}{4\cF_0(2R,\mathfrak{a},\mathfrak{d})}, 
\end{equation}
where $C_0, \cF_0$ are the functions appearing in \eqref{eq.assump-smallness}.

We shall construct a local-in-time solution $(f,g)\in (L^{\infty}_{T}H^s\cap L^2_{T}H^{s+\mez}) \times L^{\infty}_{T}H^s$ to \eqref{eq:f} and \eqref{eq:g} with initial data $(f_0, g_0)$ via the following iterative scheme. We choose
\bq\label{f1g1}
(f_1, \tilde{g}_1)\in H^{s+1}(\Rr^d)\times H^{s+1}(\Rr^{d+1})
\eq
such that 
\bq\label{cd:f1}
\| f_1-f_0\|_{H^s(\Rr^d)}+\| \tilde{g}_1-g_0\circ \ff_0\|_{H^s(\Rr^d\times J)}< \delta,
\eq
where $\delta$ is small enough so that 
\bq\label{cd:f1g1}
\begin{aligned}
    \|f_1\|_{H^s} \le \tdm R, \quad \|\tilde{g}_1\|_{H^s}\le \tdm\tM\eps,\\
    \inf_{x\in \Rr^d}\left(\gamma(f_1) - G[f_1]\Gamma(f_1)\right)(x)\geqslant \tdm\mathfrak{a}>0,\\
    \inf_{x \in \mathbb{R}^d} \left(f_1(x) - b(x)\right) \geqslant  \tdm\mathfrak{d} >0 \text{ in the finite-depth case.}
\end{aligned}
\eq

For $n\ge 2$, we let $(f_n, \tilde{g}_n)$ solve 

\bq  \label{eq.fn}
\begin{cases}
\partial_tf_n + G[f_n]\Gamma(f_n)=R_{n-1}\quad\text{in } \Rr^d\times (0, \infty), \\
f_n\vert_{t=0}=f_0,
\end{cases}
\eq
and
\bq\label{eq.gn}
\begin{cases} 
    \partial_t\tilde{g}_n + \ol{u}_{n-1}\cdot \nabla \tilde{g}_n + \gamma'(\varrho_{n-1})u_{n-1, y}\circ \ff_{f_{n-1}}=0\quad\text{in } (\Rr^d\times J)\times (0, \infty),\\
    \tilde{g}_n\vert_{t=0}=\tilde{g}_0:=g_0\circ  \ff_{f_0},
\end{cases}
\eq
where $N_n:=(-\na_xf_n, 1)$, 
\bq\label{def:Rn-1}
R_{n-1}:=
-(N_{n-1}\cdot\cS[f_{n-1}]g_{n-1}+g_{n-1})\vert_{\Sigma_{f_{n-1}}},
\eq
and $g_m:=\tilde{g}_m\circ \ff_{f_m}^{-1}$. In \eqref{eq.gn} we have introduced $u_m=(u_{m, x}, u_{m, y})$ and $\ol{u}_m=(\ol{u}_{m, x}, \ol{u}_{m, z})$ which are defined by 
\bq\label{def:um}
u_m:=-\cG[f_m]\Gamma(f_m)-\cS[f_m]g_m-g_me_y,
\eq
\bq\label{def:olum}
\begin{aligned}
\ol{u}_{m, x}&=u_{m, x} \circ \ff_{f_m},\\ 
 \ol{u}_{m, z}&=\frac{1}{\p_z\varrho_m}\Big\{-(-\na_x\varrho_m, 1)\cdot (\cG[f_m]\Gamma(f_m))\circ \ff_{f_m}\\
 &\quad -(-\na_x\varrho_{m-1}, 1)\cdot (\cS[f_{m-1}]g_{m-1})\circ \ff_{f_{m-1}}-g_{m-1}\circ \ff_{m-1} - \partial_t \varrho_m\Big\}.
\end{aligned}
\eq
Observe that \eqref{eq.fn} is a nonlinear equation for $f_n$, while \eqref{eq.gn} is a linear equation for $\tilde{g}_n$. 
For the actual problem \eqref{eq:f}-\eqref{eq:g}, we have defined  $\ol{u}$ in terms of $u$ via \eqref{tildeu}. On the other hand, for the approximate problems, $\ol{u}_{m, y}$ and $u_m$ are independently defined  as in \eqref{def:um} and \eqref{def:olum}, where \eqref{def:um} is precisely the original relation \eqref{eq:u}. The formula of $\ol{u}_{m, z}$ was designed so as to satisfy 
\bq\label{bc:olum}
\ol{u}_{m, z}=0\quad\text{on~}\Rr^d\times \p J.
\eq
In order to see \eqref{bc:olum}, we recall that $\varrho_m(x, 0)=f_m(x)$, $\cG[f_m]\Gamma(f_m)=\na_{x, y}\phi^{(1)}$, and $\cS[f_{m-1}]g_{m-1}=\na_{x, y}\phi^{(2)}$, where $\phi^{(1)}$ solves \eqref{eq.phi1} with $(f,h)=(f_m,\Gamma(f_m))$, and $\phi^{(2)}$ solves \eqref{eq.phi2} with $(f,k)=(f_{m-1},g_{m-1})$. At $z=0$ we have
\[
\ol{u}_{m, z}\vert_{z=0}=\frac{1}{\p_z\varrho_m}\vert_{z=0}\Big\{-G[f_m]\Gamma(f_m)-(N_{m-1}\cdot\cS[f_{m-1}]g_{m-1}+g_{m-1})\vert_{\Sigma_{f_{m-1}}}-\p_t f_m\Big\}=0
\]
in view of \eqref{def:Rn-1} and \eqref{eq.fn}. In the finite depth case, we have $\rho_m(x, -1)=b(x)$, $\p_t \rho_m(x, -1)=0$, and the Neumann conditions in \eqref{eq.phi1} and \eqref{eq.phi2} imply that $\ol{u}_{m, z}\vert_{z=-1}=0$. This completes the proof of \eqref{bc:olum}. 

Approximate solutions exist on a uniform interval and obey uniform bounds, as given in the following proposition. 
\begin{prop}\label{prop.iterative-bounds} Let $s>\frac{d}{2}+\frac{3}{2}$ and $R>0$. There exist $\varepsilon = \varepsilon(R) \in (0, 1]$, $T_*=T_*(R) \in (0, 1]$ both non-increasing in $R$, and $L=L(R)>0$ non-decreasing in $R$ such that if $(f_0,g_0)$ satisfy \eqref{eq.assump-f0g0}, \eqref{RT:f0} and \eqref{separation:f0} in the finite-depth case, then for $n\ge 2$ there exists a solution $(f_n,\tilde{g}_n)$ to \eqref{eq.fn}-\eqref{eq.gn} in the space
\bq\label{space:fngn}
Z_{T_*}^s:= \left(C([0, T_*], H^s)\cap L^2([0, T_*], H^{s+\mez})\right)\times C([0, T_*], H^s),
\eq
and we have for all $n\ge 1$ that 
\begin{align*}
    &\|f_n\|_{L^{\infty}_{T_*}H^s} \le 2R \label{eq.uniform-fn1}\tag{I\({}_n\)} \\
   & \|f_n\|_{L^2_{T_*}H^{s+\mez}} \le L \label{eq.uniform-fn2}\tag{II\({}_n\)}\\
    &\|\tilde{g}_n\|_{L^{\infty}_{T_*}H^s} \le 2\tM\varepsilon \label{eq.uniform-gn}\tag{III\({}_n\)}\\
    & \inf_{(x, t) \in \mathbb{R}^d\times [0, T_*]} \left(f_n(x, t) - b(x)\right) \ge  \mathfrak{d} \tag{IV\({}_n\)} \text{ in the finite-depth case,}\label{eq.uniform-bottom}\\
& \inf_{(x, t)\in \Rr^d\times [0, T_*]}\cT(f_n(t))(x) \ge \mathfrak{a}\tag{V\({}_n\)}\label{eq.uniform-taylor} \\
  &  \|f_n-f_{n-1}\|_{L^{\infty}_{T_*}H^{s-1}\cap L^2_{T_*}H^{s-\mez}} \le 2^{-n} \label{eq.uniform-contraction-fn}\tag{VI\({}_n\)}\\
  &  \|\tilde{g}_n-\tilde{g}_{n-1}\|_{L^{\infty}_{T_*}H^{s-1}} \le 2^{-n} \label{eq.uniform-contraction-gn}\tag{VII\({}_n\)},
\end{align*}
where  $C_0, \cF_0$ are the functions appearing in \cref{prop.uniform-f},  $\tM=\tM(s, d)$, and with the exception that \eqref{eq.uniform-fn2} is valid for $n\ge 2$.
\end{prop}
It follows from \eqref{eq.uniform-gn}, \eqref{eq.uniform-fn1}, and \eqref{gof:Hs} that 
\begin{equation}
    \label{eq.transfer-gn-gTilde}
    \|g_n\|_{L^{\infty}_{T_*}H^s} \le C(\|f_n\|_{L^{\infty}_{T_*}H^s})\|\tilde{g}_n\|_{L^{\infty}_{T_*}H^s} \le 2C(2R)\tM\varepsilon(R).
\end{equation}
The proof of \cref{prop.iterative-bounds} proceeds by induction, and we shall construct $\varepsilon(R)$, $T(R)$ and $L(R)$ in the proof. We may diminish $T(R)$ during the induction step, but this will be done independently of $n$.  
We note also that the constants $C(\cdot)$ appearing in the remaining of this section may change from line to line, but are always non-decreasing functions. Moreover, there are only finitely many $C$'s, Thus, all of the following estimates hold with the largest $C$. In what follows we assume $s\in (\tdm+\frac{d}{2}, \infty)\setminus\{\frac52+\frac{d}{2}\}$, so that we can apply  the transport estimate \eqref{transport:g:contra} to $\| \tilde{g}_n-\tilde{g}_{n-1}\|_{H^{s-1}}$. For  $s=\frac{5}{2}+\frac{d}{2}$, one only need to replace $s$  by any $s'\in (\tdm+\frac{d}{2}, s)$ to obtain the contractions estimates for $\| f_n-f_{n-1}\|_{L^\infty_{T_*}H^{s'-1}\cap L^2_{T_*}H^{s'-\mez}}$ and $\| \tilde{g}_n-\tilde{g}_{n-1}\|_{L^\infty_{T_*}H^{s'-1}}$, from which \eqref{eq.uniform-contraction-fn} and \eqref{eq.uniform-contraction-gn} follow by interpolation with  \eqref{eq.uniform-fn1} and \eqref{eq.uniform-gn}. 

From \eqref{cd:f1} we obtain \eqref{eq.uniform-fn1}, \eqref{eq.uniform-gn}, \eqref{eq.uniform-bottom}, and \eqref{eq.uniform-taylor}  for $n=1$.  Moreover, we can choose $\delta$ smaller in \eqref{cd:f1} if necessary to have \eqref{eq.uniform-contraction-fn} and \eqref{eq.uniform-contraction-gn} for $n=1$.

Assume that for some $n\geqslant 2$ and all $1\le i\le n-1$, the couple $(f_i,\tilde{g}_i)$ has been constructed on $[0, T_*]$ independent of $n$ such that (I${}_{i}$) -- (VII${}_{i}$) hold. 
We proceed to  prove that $(f_n, \tilde{g}_n)$ exists  and satisfies  (I${}_{n}$) -- (VII${}_{n}$) . We always consider $T_*\le 1$ and $\eps\in (0, 1]$ in what follows.   

{\it Step 1.} We first prove the existence and estimate  for $\tilde{g}_n$. 
\begin{lemm}\label{lemm:gn}
Let  $n\ge 2$. Assume that $(f_{n-1}, \tilde{g}_{n-1})\in Z^s_T$, $T\in (0, 1]$, and $f_{n-1}$ satisfies (IV${}_{n-1}$) in the finite-depth case. Then \eqref{eq.gn} has a unique solution $\tilde{g}_n\in C([0, T]; H^s(\Rr^d\times J))$. Moreover, $\tilde{g}_n$ satisfies 
\begin{multline}
    \label{eq.control-gTilde-n}
    \|\tilde{g}_n\|_{L^{\infty}_{T}H^s} \le \left(\tM\| \tilde{g}_0\|_{H^s}+C(\|f_{n-1}\|_{L^{\infty}_{T}H^s})T^{\mez}\cA_{T, n-1}\right) \exp\left(T^{\mez}C(\|f_{n-1}\|_{L^{\infty}_{T}H^s})\cA_{T, n-1}\right),
\end{multline}
where $\tM=\tM(s, d)$, $C: \Rr_+\to \Rr_+$ depends only on $(s, d, b, \gamma, \fd)$, and 
\bq
\cA_{T, n-1}:=\| f_{n-1}\|_{L^2_TH^{s+\mez}}+\|\tilde{g}_{n-1}\|_{L^\infty_TH^s}+\| f_{n-1}\|_{L^2_TH^{s+\mez}}\| \tilde{g}_{n-1}\|_{L^\infty_TH^s}.
\eq
\end{lemm}
\begin{proof}
We recall from \eqref{bc:olum} that $\ol{u}_n$ is tangent to $\Rr^d\times \p J$. Therefore, by appealing to \cref{theo:transport} with $\cU=\Rr^d\times J$,  Lemma \ref{lemm:gn} can be proven similarly to the a priori estimate \eqref{aprioriestimate:g}.  We note that \cref{theo:transport} yields both the existence of $\tilde{g}_n$ and the desired estimate. 
\end{proof}

We recall that  $f_1=f_1(x)\in H^{s+1}\subset Z^s_1$. Then we can use the induction hypothesis and apply \cref{lemm:gn} $n-1$ times with $T=T_*$ to  deduce that $\tilde{g}_n$ exists and satisfies  \eqref{eq.control-gTilde-n}. In other words,
\[
    \|\tilde{g}_n\|_{L^{\infty}_{T_*}H^s} \le \left(\tM \eps +C(2R)T_*^\mez \cB\right)\exp\left(C(2R)T_*^\mez \cB\right),
\]
where 
\bq
\cA_{T_*, {n-1}}\le L+2\tM\eps+L2\tM\eps\le \cB:=L+2\tM+2\tM L.
\eq 
Therefore, choosing  
\bq\label{chooseT:1}
    T_*^\mez \le  T_{*, 1}^\mez :=\min\left\{\frac{\ln\tdm}{C(2R)\cB}, \frac{\tM\eps }{3C(2R)\cB}\right\},
\eq
we obtain $\|\tilde{g}_n\|_{L^\infty_{T_*}H^s} \le 2\tM \eps$, thereby proving  \eqref{eq.uniform-gn}.

In Steps 2, 3, and 4 below, we will solve  \eqref{eq.fn} using the  fractional regularization 
\bq\label{eq.fn-nu}
\begin{cases}
    \p_t f_{n,\nu} + G[f_{n,\nu}]\Gamma(f_{n,\nu}) + \nu |D|^{1+\mu}f_{n,\nu} = R_{n-1}\quad\text{in } (\Rr^d\times J) \times (0,T_{n,\nu}), \\
f_{n,\nu}\vert_{t=0}=f_0,
\end{cases}
\eq 
where $\mu \in (0, \mez)$ is a small fixed number.  We impose that
 \bq\label{epsR:1}
\tM\eps \le R.
\eq

{\it Step 2.} We first construct $f_{n,\nu}$ on a time interval depending on $n$ and $\nu$:

\begin{lemm}\label{lemm:fn}
There exists $T\in (0,  T_*]$ such that \eqref{eq.fn-nu} admits a unique solution $f_{n,\nu}$ in the space 
\bq
\begin{aligned}
    V^s_{T}&:=\left\{f\in C([0, T], H^s)\cap L^2([0, T],H^{s+\mez + \frac{\mu}{2}}):~\inf_{(x, t)\in \Rr^d\times [0, T]}(f(x, t)-b(x))\ge \fd \right\}.
\end{aligned}
\eq
Let $T_{n,\nu}^*$ be the maximal existence time of $f_{n,\nu}$ in $V^s_T$. If $T_{n,\nu}^*<T^*$, then either 
\bq\label{bc:fn:1}
    \limsup_{t \to T_{n,\nu}^*}\|f_{n,\nu}(t)\|_{H^s}=\infty\quad\text{or} 
\eq
\bq\label{bc:fn:2}
    \inf_{(x, t)\in \Rr^d\times [0, T_{n,\nu}^*)}(f_{n,\nu}(x, t)-b(x))=\fd.
\eq
\end{lemm}

\begin{proof}
We first  solve \eqref{eq.fn-nu} for $f_{n,\nu}$ by a fixed-point argument in the complete metric space $V_T^s$ endowed with the norm $\|f\|_{V_T^s}:= \|f\|_{L^{\infty}_TH^s} + \nu^\mez\|f\|_{L^2_TH^{s+\frac{1}{2}+\frac{\mu}{2}}}$, where $T\le T_*$. To that end, we rewrite \eqref{eq.fn-nu} in the integral form
\[
    f_{n,\nu}(t)=e^{-\nu t|D|^{1+\mu}}f_0 + \int_0^te^{-\nu(t-t')|D|^{1+\mu}}F_{n-1}(f_{n, \nu})(t')\,\mathrm{d}t' =: \Phi_{n-1}(f_{n,\nu})(t),   
\]
where $F_{n-1}(f)=-G[f]\Gamma(f)+R_{n-1}$ and $\|f_0\|_{H^s}\le R$. One can verify that for all $f\in V_T^s$, we have $\Phi_{n-1}(f)\in C([0, T]; H^s)$.  We will prove that $\Phi_{n-1}$ maps the ball $B(0, 2R)$ of  $Z_T^{s+\mez+\frac{\mu}{2}}$  into itself for some $T = T_{n,\nu}$ sufficiently small. The contraction of $\Phi_{n-1}$  on $B(0, 2R)$ can be proven analogously.

Standard energy estimates on the fractional heat kernel $e^{-\nu t|D|^{1+\mu}}$ give
\bq\label{est:Phi:0}
    \|\Phi_{n-1}(f)\|_{V^s_T} \le \|f(0)\|_{H^s} + \nu^{-\mez}\|F_{n-1}\|_{L^2_TH^{s-\mez-\frac{\mu}{2}}}.      
\eq
We now estimate $F_{n-1}$ for $n\ge 2$.  First, since $s-\mez-\frac{\mu}{2}\ge s-1$, we can use \eqref{est:DN:h} to have 
\begin{equation}\label{eq.estG-mu}
    \|G[f]\Gamma(f)\|_{H^{s-\mez - \frac{\mu}{2}}}\le C(\| f\|_{H^s})\| f\|_{H^{s+\mez-\frac{\mu}{2}}},
\end{equation}
where $C$ depends on $\fd$. On the other hand, the trace theorem for $\Rr^d\times J$ implies
\bq\label{trace:gn-1:mu}
    \|g_{n-1}\vert_{\Sigma_{f_{n-1}}}\|_{H^{s-\mez - \frac{\mu}{2}}}=\| \tilde{g}_{n-1}\vert_{z=0}\|_{H^{s-\mez - \frac{\mu}{2}}}\le C(\|f_{n-1}\|_{H^s}) \| \tilde g_{n-1}\|_{H^{s}}.
\eq
Next, using the tame product estimate and the Sobolev embedding $H^{s-1}(\Rr^d) \hookrightarrow L^{\infty}$, we obtain 
\begin{align*}
    & \|N_{n-1}\cdot \cS[f_{n-1}]g_{n-1}\vert_{\Sigma_{f_{n-1}}}\|_{H^{s-\mez - \frac{\mu}{2}}}\\
    & \les \left(\|N_{n-1}-e_y\|_{H^{s-1}}+1\right)\|\cS[f_{n-1}]g_{n-1}\vert_{\Sigma_{f_{n-1}}}\|_{H^{s-\mez - \frac{\mu}{2}}}  + \|N_{n-1}-e_y\|_{H^{s-\mez - \frac{\mu}{2}}}\|\cS[f_{n-1}]g_{n-1}\vert_{\Sigma_{f_{n-1}}}\|_{H^{s-1}}  \\
    & \le C(\|f_{n-1}\|_{H^{s}})\left\{\|\cS[f_{n-1}]g_{n-1}\vert_{\Sigma_{f_{n-1}}}\|_{H^{s-\mez - \frac{\mu}{2}}} + \|f_{n-1}\|_{H^{s+\mez - \frac{\mu}{2}}}\|\cS[f_{n-1}]g_{n-1}\vert_{\Sigma_{f_{n-1}}}\|_{H^{s-1}}\right\}.
\end{align*}
It follows from \eqref{esttrace:cS} with $\sigma=s-1$ that 
\begin{equation}
    \label{eq.low-S-estim}
    \|\cS[f_{n-1}]g_{n-1}\vert_{\Sigma_{f_{n-1}}}\|_{H^{s-1}} \le  C(\|f_{n-1}\|_{H^{s}})\|\tilde{g}_{n-1}\|_{H^s},
\end{equation}
and from \eqref{esttrace:cS} with $\sigma = s-\mez - \frac{\mu}{2}$ that 
\[
    \|\cS[f_{n-1}]g_{n-1}\vert_{\Sigma_{f_{n-1}}}\|_{H^{s-\mez - \frac{\mu}{2}}} \le C(\|f_{n-1}\|_{H^{s}})\|\tilde{g}_{n-1}\|_{H^s}(1+\|f_{n-1}\|_{H^{s+\mez - \frac{\mu}{2}}}).
\]
Consequently, we obtain 
\bq\label{NS:n-1:mu}
    \|N_{n-1}\cdot \cS[f_{n-1}]g_{n-1}\vert_{\Sigma_{f_{n-1}}}\|_{H^{s-\mez - \frac{\mu}{2}}}\le C(\|f_{n-1}\|_{H^{s}})\|\tilde{g}_{n-1}\|_{H^s}(1+\|f_{n-1}\|_{H^{s+\mez - \frac{\mu}{2}}}).
\eq
Combining \eqref{trace:gn-1:mu} and \eqref{NS:n-1:mu}, we find 
 \bq\label{est:Rn-1:fixedpoint}
 \|R_{n-1}\|_{H^{s-\mez - \frac{\mu}{2}}}\le C(\|f_{n-1}\|_{H^{s}})\|\tilde{g}_{n-1}\|_{H^s}(1+\|f_{n-1}\|_{H^{s+\mez - \frac{\mu}{2}}}).
\eq
By interpolation of $H^{s+\mez - \frac{\mu}{2}}$ between $H^s$ and $H^{s+\mez}$, there holds
\[
    \|h\|_{L^2_TH^{s+\mez - \frac{\mu}{2}}}\le \| h\|_{L^\infty_T H^s}^\mu T^{\frac{\mu}{2}} \| h\|^{1-\mu}_{L^2_TH^{s+\mez}}.
\]
Consequently, in view of \eqref{eq.estG-mu} and \eqref{est:Rn-1:fixedpoint}, we deduce 
\begin{equation}\label{est:Fn-1:mu}
\begin{aligned}
    \|F_{n-1}\|_{L^2_TH^{s-\mez - \frac{\mu}{2}}} &\le C(\|f\|_{L^{\infty}_TH^{s}})T^{\frac{\mu}{2}}\|f\|^{1-\mu}_{L^2_TH^{s+\mez}}\\
    &\qquad+C(\|f_{n-1}\|_{L^{\infty}_TH^{s}})\|\tilde{g}_{n-1}\|_{L^{\infty}_TH^s}(T^{\mez}+T^{\frac{\mu}{2}}\|f_{n-1}\|^{1-\mu}_{L^2_TH^{s+\mez}}). 
    \end{aligned}
\end{equation}
Using this, the induction hypothesis, and \eqref{epsR:1},  it follows from \eqref{est:Phi:0} that 
\bq
    \| \Phi_{n-1}(f)\|_{V_T^s}\le R+ \nu^{-(1-\frac{\mu}{2})}\left\{C(2R)T^\mu (2R)^{1-\mu}+C(2R)2R(T^\mez+ T^\mu L)\right\}
\eq 
for $f\in B(0, 2R)\subset V^s_T$. Thus, for some $T_{\nu, 1}=T_{\nu, 1}(\nu, R, L, \fd)>0$,  $\| \Phi_{n-1}(f)\|_{V_T^s}\le 2R$ whenever $f\in B(0, 2R)\subset V^s_T$ and $T\le T_{\nu, 1}$.

Next, for $\mu$ sufficiently small and $t\le T$, we have 
\begin{align*}
    \|(\Phi_{n-1}(f)(t) - b) - (f_0-b)\|_{L^{\infty}}& \le C\|\Phi_{n-1}(f)(t) - f_0\|_{H^{s-1-\mu}} \\
    &\le C\|e^{-\nu t|D|^{1+\mu}}f_0-f_0\|_{H^{s-1-\mu}} + tC\|F_{n-1}\|_{L^{\infty}_TH^{s-1}}\\
    &\le C\nu t \| f_0\|_{H^s} + tC\|F_{n-1}\|_{L^{\infty}_TH^{s-1}}
\end{align*}
  The estimate \eqref{eq.low-S-estim}, the trace inequality,  and \cref{theo:DN1} (i) imply 
\begin{align*}
    \|F_{n-1}(f)\|_{L^{\infty}_TH^{s-1}}& \le C(\|f\|_{L^{\infty}_TH^s})\|f\|_{L^{\infty}_TH^s} +   C(\|f_{n-1}\|_{L^{\infty}_TH^{s}})\|\tilde{g}_{n-1}\|_{L^{\infty}_TH^s}\\
    &\le C(\|f\|_{L^{\infty}_TH^s})\|f\|_{L^{\infty}_TH^s}+C(2R)2R,
     \end{align*}
     where we have used the induction hypothesis and \eqref{epsR:1}.  Therefore, since $f_0-b\ge 2\fd$, for  $f\in B(0, 2R)\subset V^s_T$ we have 
     \[
     \Phi_{n-1}(f)(x, t)-b(x)\ge 2\fd- C\nu T R -TC(2R)4R\ge \fd
     \]
provided $t\le T\le T_{\nu, 2}=T_{\nu, 2}(\nu, R, \fd)>0$. 

We have proven that $\Phi_{n-1}$ maps the ball $B(0, 2R)\subset V^s_T$ into itself if $T\le \min\{T_{\nu, 1}, T_{\nu, 2}\}$. Moreover, $\Phi_{n-1}$ is a contraction  if $T$ is small enough. The Banach fixed point theorem implies that \eqref{eq.fn-nu} has a unique solution $f_{n, \nu}\in V^s_T$, where $T$ depends on $(\nu,  R, L, \| f_0\|_{H^s})$. Therefore, if $T_{n, \nu}^*\le T_*$ is the maximal existence time, then $T_{n, \nu}^*<T_*$ if and only if either \eqref{bc:fn:1} or \eqref{bc:fn:2} holds.
\end{proof}

{\it Step 3.} Our next task is to prove that  $T_{n,\nu}^*= T_*$.  By virtue of the continuation criteria \eqref{bc:fn:1} and \eqref{bc:fn:2} in \cref{lemm:fn}, it suffices to prove that 
\bq\label{bc1:proof}
\sup_{t\in [0, T^*_{n,\nu})} \| f_{n,\nu}(t)\|_{H^s}<\infty\quad\text{and }
\eq
\bq\label{bc2:proof}
 \inf_{(x, t)\in \Rr^d\times [0, T^*_{n,\nu})}(f_{n,\nu}(x, t)-b(x))>\fd.
\eq
Note that in the following we use that $f_{n,\nu} \in C^0([0,T_{n,\nu}^*),H^s)$.  From the definition of $T_{n,\nu}^*$ we have

\bq\label{separation:fn}
    \inf_{(x, t)\in \Rr^d\times [0, T_{n}^*)}(f_{n,\nu}(x, t)-b(x))\ge \fd.
\eq
\begin{lemm} For all $t<T^*_{n,\nu}$, we have
    \bq\label{lowerRT:fn}
    \begin{aligned}
    &\inf_{x\in \Rr^d} \cT(f_{n,\nu}(t))(x)\ge 2\fa\\
    &\qquad -t^\frac13C(\|f_{n,\nu}\|_{C([0, t]; H^s)})\left(\| f_{n,\nu}\|_{L^\infty([0, t]; H^s)} + C(\|f_{n-1}\|_{L^\infty([0, t]; H^s)})\| \tilde{g}_{n-1}\|_{L^\infty([0, t]; H^s)}\right)^\frac{1}{3}
\end{aligned}
\eq
and 
\bq\label{lowersep:fn}
\begin{aligned}
&\inf_{x\in \Rr^d}(f_{n,\nu}(x, t)-b(x))\ge 2\fd\\
&\qquad -tC_s\left(C(\|f_{n,\nu}\|_{L^\infty([0, t]; H^s)})\| f_{n,\nu}\|_{L^\infty([0, t]; H^s)} + C(\|f_{n-1}\|_{L^\infty([0, t]; H^s)})\| \tilde{g}_{n-1}\|_{L^\infty([0, t]; H^s)}\right),
\end{aligned}
\eq
where $C_s$ depends only on $(s, d)$. 
\end{lemm}

\begin{proof} We fix $t<T^*_{n,\nu}$. Since $s>\tdm +\frac{d}{2}\ge 2$, we can apply \cref{thm.NPdiff} (ii) with $\sigma=s-\tdm>\mez$ to have 
  \begin{align*}
  \| \cT(f_{n,\nu}(t))-\cT(f_0)\|_{H^{s-\frac52}}&\le C(\|f_{n,\nu}(t)\|_{H^s}, \| f_0\|_{H^s})\| f_{n,\nu}(t)-f_0\|_{H^{s-\tdm}}\\
  & \le C(\| f_{n,\nu}\|_{C([0, t]; H^s)})t\| \p_t f_{n,\nu}\|_{L^\infty([0, t]; H^{s-\tdm})}.
  \end{align*}
  On the other hand, \cref{theo:DN1} (i) again implies
  \[
   \| \cT(f_{n,\nu}(t))-\cT(f_0)\|_{H^{s-1}}\le  \| \cT(f_{n,\nu}(t))\|_{H^{s-1}}+\|\cT(f_0)\|_{H^{s-1}}\le C(\|f_{n,\nu}(t)\|_{H^s}, \| f_0\|_{H^s}).
  \]
  By interpolation, the two preceding estimates  yield
  \[
    \| \cT(f_{n,\nu}(t))-\cT(f_0)\|_{H^{s-\tdm}}\le C(\| f_{n,\nu}\|_{C([0, t]; H^s)})t^\frac13\| \p_t f_{n,\nu}\|_{L^\infty([0, t]; H^{s-\tdm})}^\frac13.
  \]
  Since $H^{s-\tdm}(\Rr^d)\subset L^\infty(\Rr^d)$, it follows that 
\bq\label{cTfn:1}
\forall x\in \Rr^d,~\cT(f_{n,\nu}(t))(x)\ge 2\fa-t^\frac13C(\| f_{n,\nu}\|_{C([0, t]; H^s)})\| \p_t f_{n,\nu}\|_{L^\infty([0, t]; H^{s-\tdm})}^\frac13.
\eq
We recall from \eqref{eq.fn} that $\p_tf_{n,\nu}=-G[f_{n,\nu}]\Gamma(f_{n,\nu})+R_{n-1}-\nu|D|^{1+\mu}f_{n,\nu}$. Using \cref{theo:DN1} (i), the estimate \eqref{esttrace:cS}, and the trace inequality for $\Rr^d\times J$, we obtain
\bq\label{dtfn:l}
\begin{aligned}
\| \p_t f_{n,\nu}\|_{L^\infty([0, t];  H^{s-\tdm})}&\le \| G[f_{n,\nu}]\Gamma(f_{n,\nu})\|_{L^\infty([0, t]; H^{s-1})}+\| R_{n-1}\|_{L^\infty([0, t];  H^{s-1})} \\
& \qquad + \nu \||D|^{1+\mu} f_{n,\nu}\|_{L^\infty([0, t];  H^{s-\tdm})} \\
&\mkern-54mu \le C(\|f_{n,\nu}\|_{L^\infty([0, t]; H^s)})\| f_{n,\nu}\|_{L^\infty([0, t]; H^s)} + C(\|f_{n-1}\|_{L^\infty([0, t]; H^s)})\| \tilde{g}_{n-1}\|_{L^\infty([0, t]; H^s)}.
\end{aligned}
\eq
Combining \eqref{cTfn:1} and \eqref{dtfn:l} yields \eqref{lowerRT:fn}. 

Since 
\[
\inf_{x\in \Rr^d}(f_{n,\nu}(x, t)-b(x))\ge \inf_{x\in \Rr^d}(f_0(x)-b(x))-tC_s\| \p_tf_{n,\nu}\|_{L^\infty([0, t]; H^{s-\tdm})}
\]
for some constant $C_s=C_s(s, d)>0$, \eqref{lowersep:fn} follows from \eqref{dtfn:l}.
\end{proof}
Now we consider any $T<T_{n,\nu}^*$ such that 
\bq\label{lowerRT:fn:0}
\inf_{(x, t)\in \Rr^d\times [0, T]}\cT(f_{n,\nu}(t))(x)\ge \fa.
\eq
The estimate \eqref{lowerRT:fn} implies that \eqref{lowerRT:fn:0} holds at least for small  $T$. In view of \eqref{separation:fn} and \eqref{lowerRT:fn:0}, we can use \cref{rema:apriori:variant} and deduce from the estimate \eqref{eq.apriori-fR} that
\bq\label{est:fn:energy}
 \begin{aligned}
    &\|f_{n,\nu}\|_{C_TH^{s}}+ \frac{1}{C(\|f_{n,\nu}\|_{C_TH^s})}\|f_{n,\nu}\|_{L^2_TH^{s+\mez}} \\
    & \le \exp(TC(\|f_{n,\nu}\|_{C_TH^s)})\left(\|f_0\|_{H^s} +C(\|f_{n,\nu}\|_{C_TH^s}) \|R_{n-1}\|_{L^2_{T}H^{s-\mez}}\right),
    \end{aligned}
    \eq
    where $C_TH^s\equiv C([0, T]; H^s)$. Using  \eqref{esttrace:cS}, the induction hypothesis, and the fact that $T^*\le 1$ , we have 
\bq    \label{eq.Rn-1-bound}
\begin{aligned}
    \|R_{n-1}\|_{L^2_{T}H^{s-\mez}}& \le C(\|f_{n-1}\|_{L^\infty_{T}H^s})(1+\| f_{n-1}\|_{L^2_{T}H^{s+\mez}})\| \tilde{g}_{n-1}\|_{L^\infty_{T}H^s}\\
    & \le 2C(2R)(1+L(R))\tM\varepsilon.
\end{aligned}
\eq  
 Consequently, by increasing $C$ and imposing that 
\bq\label{cd:eps:2}
\tM\eps \le \frac{R}{4C^2(2R)(1+L(R))},
\eq
 we obtain from \eqref{est:fn:energy}, \eqref{eq.Rn-1-bound}, and  \eqref{lowerRT:fn}  that for any $T<T_{n,\nu}^*$ satisfying \eqref{lowerRT:fn:0} we have
\begin{align}\label{apriori:fn:1}
  &  \|f_{n,\nu}\|_{C_TH^{s}}\le R\exp(TC(\|f_{n,\nu}\|_{C_TH^s}))\left(1+ \frac{C(\|f_{n,\nu}\|_{C_TH^s})}{2C(2R)}\right), \\ \label{apriori:fn:2}
    & \|f_{n,\nu}\|_{L^2_TH^{s+\mez}}\le R C(\|f_{n,\nu}\|_{C_TH^s})\exp(TC(\|f_{n,\nu}\|_{C_TH^s}))\left(1+ \frac{C(\|f_{n,\nu}\|_{C_TH^s})}{2C(2R)}\right),
\end{align}
\bq\label{apriori:RT:fn}
\begin{aligned}
&\inf_{(x, t)\in \Rr^d\times [0, T]} \cT(f_{n,\nu}(t))(x)\ge 2\fa\\
&\qquad -T^\frac13C(\|f_{n,\nu}\|_{C_TH^s})\left(R\exp(TC(\|f_{n,\nu}\|_{C_TH^s}))\left(1+ \frac{C(\|f_{n,\nu}\|_{C_TH^s})}{2C(2R)}\right)+2C(2R)\tM\eps\right)^\frac13,
\end{aligned}
\eq
and 
\bq\label{lowersep:fn:2}
\begin{aligned}
&\inf_{(x, t)\in \in \Rr^d\times [0, T]}(f_{n,\nu}(x, t)-b(x))\ge 2\fd-T\left(C(\|f_{n,\nu}\|_{L^\infty([0, T]; H^s)})\| f_{n,\nu}\|_{L^\infty([0, T]; H^s)} + C(2R)2\tM\eps\right).
\end{aligned}
\eq
Set 
\begin{align*}
T_{n,\nu}'&=\sup\left\{T \in (0, T_{n,\nu}^*):\quad \| f_{n,\nu}\|_{C_TH^s}\le 2R,~\inf_{(x, t)\in \Rr^d\times [0, T]} \cT(f_{n,\nu}(t))(x)\ge \fa,\right.\\
&\qquad\qquad \left. \inf_{(x, t)\in \Rr^d\times [0, T]}(f_{n,\nu}(x, t)-b(x))\ge \tdm \fd \right\}>0.
\end{align*}
There are two cases: 

Case 1:  $T_{n,\nu}'=T_{n,\nu}^*$. Then the blowup criteria \eqref{bc:fn:1} and \eqref{bc:fn:2} imply that $T_{n,\nu}^*=T_*$. 

Case 2:  $T_{n,\nu}'<T^*_{n,\nu}$. Then  \eqref{apriori:fn:1}, \eqref{apriori:fn:2}, and \eqref{apriori:RT:fn} hold with $T=T_{n,\nu}'$ and one of the following possibilities must occur. 

\quad Possibility 1: $\| f_{n,\nu}\|_{C_{T_{n,\nu}'}H^s}=2R$ . Then \eqref{apriori:fn:1} with $T=T_{n,\nu}'$  implies   
\[
2R=\| f_{n,\nu}\|_{C_{T_{n,\nu}'}H^s}\le R\exp(T_{n,\nu}'C(2R))\tdm.
\]
Consequently 
\bq
T_{n,\nu}'\ge T_{*, 2}:=\frac{1}{C(2R)}\ln \frac43.
\eq  
\quad Possibility 2:  $\inf_{(x, t)\in \Rr^d\times [0, T]} \cT(f_{n,\nu}(t))(x)= \fa$. Then \eqref{apriori:RT:fn} with $T=T'_{n,\nu}<T_{n,\nu}^*< 1$  implies  
\begin{align*}
\fa&\ge 2\fa- (T_{n,\nu}')^\frac13C(2R)R\left(\exp(C(2R))\tdm+2C(2R)\right)^\frac13,
\end{align*}
and hence 
\bq
T_{n,\nu}'\ge T_{*, 3}:= \frac{\fa^3}{C^3(2R)R^3\left(\tdm \exp(C(2R))+2C(2R)\right)}.
\eq
\quad Possibility 3: $ \inf_{(x, t)\in \Rr^d\times [0, T]}(f_{n,\nu}(x, t)-b(x))= \tdm \fd$. Then using \eqref{lowersep:fn:2} with $T=T_{n,\nu}'$, we deduce 
\bq
T_{n,\nu}'\ge T_{*, 4}:=\frac{\fd}{8RC(2R)}.
\eq
Therefore, by choosing 
\bq\label{chooseT:2}
T_*< \min\{T_{*, 2}, T_{*, 3}, T_{*, 4}\},
\eq
we always have $T_{n,\nu}'= T_*$. In particular, the definition of $T_{n,\nu}'$ implies that 
 \begin{align*}
    &\|f_{n,\nu}\|_{L^{\infty}_{T_*}H^s} \le 2R \label{eq.uniform-fn1-nu}\tag{I\({}_{n,\nu}\)}, \\
    & \inf_{(x, t) \in \mathbb{R}^d\times [0, T_*]} \left(f_{n,\nu}(x, t) - b(x)\right) \ge  \mathfrak{d} \tag{IV\({}_{n,\nu}\)} \label{eq.uniform-bottom-nu},\\
& \inf_{(x, t)\in \Rr^d\times [0, T_*]}\cT(f_{n,\nu}(t))(x) \ge \mathfrak{a}\tag{V\({}_{n,\nu}\)}\label{eq.uniform-taylor-nu}
\end{align*} 
hold. 
Then, inserting \eqref{eq.uniform-fn1-nu} in \eqref{apriori:fn:2}, we obtain 
\begin{equation}\label{eq.uniform-fn2-nu}
    \|f_{n,\nu}\|_{L^2_{T_*}H^{s+\mez}} \le L=L(R) \tag{II\({}_{n,\nu}\)}
\end{equation}
for 
\bq\label{cd:L}
 L(R)\ge \tdm RC(2R)\exp(C(2R)).
 \eq

{\it Step 4.} Fix $n\ge 2$. Our goal in this step we pass to the limit $\nu \to 0$ in $f_{n, \nu}$ to obtain $f_n$. Set $\ol{f}_{n,m}=f_{n,\nu_m}-f_{n,\nu_{m-1}}$,  where $\nu_m=\frac{1}{m}$ for $m\ge 2$.   From \eqref{eq.fn-nu} we have
\[
    \p_t \ol{f}_{n,m}+G[f_{n,\nu_m}]\Gamma(f_{n,\nu_m})-G[f_{n,\nu_{m-1}}]\Gamma(f_{n,\nu_{m-1}})=-\nu_m|D|^{1+\mu} \ol{f}_{n,m} + (\nu_m-\nu_{m-1})|D|^{1+\mu}f_{n,\nu_{m-1}}
\]
on $(0, T_*)$. Using the paralinearization \eqref{paralin:G:contra} for $G[f_{n,\nu_m}]\Gamma(f_{n,\nu_m})-G[f_{n,\nu_{m-1}}]\Gamma(f_{n,\nu_{m-1}})$, we find 
\bq\label{eq:olfn-nu}
    \p_t\bar{f}_{n,m} + T_{\lambda[f_{n,\nu_{m-1}}](\gamma(f_{n,\nu_{m-1}})-B_{n,\nu_{m-1}})}\bar{f}_{n,m}-T_{V_{n,\nu_{m-1}}}\cdot \nabla\bar{f}_{n,m} + \nu_m|D|^{1+\mu} \ol{f}_{n,m}= R'_{n,m}+F_{n,m},
\eq
where
\bq
F_{n,m}=(\nu_m-\nu_{m-1})|D|^{1+\mu}f_{n,\nu_{m-1}}
\eq
and $R'_{n,m}$ is given by \eqref{eq.R:def} with $(f_1,f_2)$  replaced by $(f_{n,\nu_m},f_{n,\nu_{m-1}})$. Thus combining \cref{lemm:est:R'} with \eqref{eq.uniform-fn1-nu} and \eqref{eq.uniform-fn2-nu} we find
\begin{multline}\label{R'n-nu}
    \| R'_{n,m}\|_{H^{s-\tdm}}\le C(\| (f_{n,\nu_m}, f_{n,\nu_{m-1}})\|_{H^s})\left\{\|\ol{f}_{n,m}\|_{H^{s-\mez -\delta}} \right. \\ 
    + \left.\big(1+\|(f_{n,\nu_m}, f_{n,\nu_{m-1}})\|_{H^{s+\mez-\delta}}\big)\|\ol{f}_{n,m}\|_{H^{s-1}}\right\}, 
\end{multline}
and also 
\bq \label{est:Fn-nu}
    \| F_{n,m}\|_{H^{s-\tdm}} \le |\nu_m-\nu_{m-1}|\|f_{n,\nu_{m-1}}\|_{H^{s-\frac{1}{2}+\mu}} \le  |\nu_m-\nu_{m-1}|\|f_{n,\nu_{m-1}}\|_{H^{s+\frac{1}{2}}}.
\eq
By virtue of (V${}_{n,\nu_{m-1}}$), we have $\gamma(f_{n,\nu_{m-1}})-B_{n,\nu_{m-1}}\ge \fa$, so that $\lambda[f_{n,\nu_{m-1}}](\gamma(f_{n,\nu_{m-1}})-B_{n,\nu_{m-1}})$ is a first-order elliptic symbol (uniformly in $m$). Therefore, in view of \eqref{R'n-nu} and \eqref{est:Fn-nu}, equation \eqref{eq:olfn-nu} is of the same form as \eqref{eq.transport-diference-final} except for the additional dissipation $-\nu_m|D|^{1+\mu} \ol{f}_{n,m}$. Since 
\[
    (-\nu_m|D|^{1+\mu} \ol{f}_{n,m}, \ol{f}_{n,m})_{H^{s-1}}\le 0, 
\] 
the $H^{s-1}$ energy estimate as in \cref{lemm:parabolicest} gives
\begin{align}\label{energy:olfn:1-nu}
    \mez\frac{d}{dt}\|\ol{f}_{n,m}(t)\|_{H^{s-1}}^2&\le -\frac{1}{\cF_0(\|f_{n,\nu_{m-1}}\|_{L^{\infty}_TH^s}, \mathfrak{a}, \mathfrak{d})}\| \ol{f}_{n,m}(t)\|_{H^{s-\mez}}^2+\cF_0(\|f_{n,\nu_{m-1}}\|_{L^{\infty}_TH^s}, \mathfrak{a}, \mathfrak{d})\|\ol{f}_{n,m}(t)\|_{H^{s-1}}^2\notag\\
&\qquad+\big(\| R_{n,m}'(t)\|_{H^{s-\frac{3}{2}}}+\| F_{n,m}(t)\|_{H^{s-\tdm}}\big)\|\ol{f}_{n,m}(t)\|_{H^{s-\mez}},\quad t\le T\le T_*.
\end{align}
By Young's inequality, we have
\begin{multline*}
    \| F_{n,m}(t)\|_{H^{s-\tdm}}\|\ol{f}_{n,m}(t)\|_{H^{s-\mez}}\le \frac{1}{4\cF_0(\|f_{n,\nu_{m-1}}\|_{L^{\infty}_TH^s}, \mathfrak{a}, \mathfrak{d})}\| \ol{f}_{n,m}(t)\|_{H^{s-\mez}}^2 \\ 
    +\cF_0(\|f_{n,\nu_{m-1}}\|_{L^{\infty}_TH^s}, \fa, \fd)\|F_{n,m}(t)\|_{H^{s-\tdm}}^2.
\end{multline*}
On the other hand, using the interpolation inequality \eqref{interpolate:contraf}, Young's inequality, and \eqref{R'n-nu}, we find 
\begin{multline*}
\| R'_{n,m}(t)\|_{H^{s-\tdm}}\|\ol{f}_{n,m}(t)\|_{H^{s-\mez}}\le\frac{1}{4\cF_0(\|f_{n,\nu_{m-1}}\|_{L^{\infty}_TH^s}, \mathfrak{a}, \mathfrak{d})}\| \ol{f}_{n,m}(t)\|_{H^{s-\mez}}^2\\
+\cF_1(\|(f_{n,\nu_m}, f_{n,\nu_{m-1}})\|_{L^{\infty}_TH^s}, \mathfrak{a}, \mathfrak{d})\big(1+\|(f_{n,\nu_m}, f_{n,\nu_{m-1}})\|^2_{H^{s+\mez-\delta}}\big)\| \ol{f}_{n,m}\|^2_{H^{s-1}}.
\end{multline*}

Plugging the two preceding estimates in \eqref{energy:olfn:1-nu} yields
\begin{align}\label{energy:olfn-nu:1}
    \mez\frac{d}{dt}\|\ol{f}_{n,m}(t)\|_{H^{s-1}}^2&\le -\frac{1}{2\cF_0(\|f_{n,\nu_{m-1}}\|_{L^{\infty}_TH^s}, \mathfrak{a}, \mathfrak{d})}\| \ol{f}_{n,m}(t)\|_{H^{s-\mez}}^2\notag\\
    &\qquad+\cF_2(\|(f_{n,\nu_m}, f_{n,\nu_{m-1}})\|_{L^{\infty}_TH^s}, \mathfrak{a}, \mathfrak{d})\big(1+\|(f_{n,\nu_m}, f_{n,\nu_{m-1}})\|^2_{H^{s+\mez-\delta}}\big)\| \ol{f}_{n,m}\|^2_{H^{s-1}}\notag\\
    &\qquad +\cF_0(\|f_{n,\nu_{m-1}}\|_{L^{\infty}_TH^s}), \fa, \fd\| F_{n,m}(t)\|_{H^{s-\tdm}}^2,\quad t\le T\le T_*.
\end{align}
Since $\ol{f}_{n,m}(0)=0$, Gr\"onwall's lemma implies 
\bq
    \|\ol{f}_{n,m}\|_{L^\infty_TH^{s-1}\cap L^2_TH^{s-\mez}}\le \| F_{n,m}\|_{L^2_TH^{s-\tdm}}\cF_3\left(\|(f_{n,\nu_m}, f_{n,\nu_{m-1}})\|_{L^{\infty}_TH^s\cap L^2_TH^{s+\mez}}, \mathfrak{a}, \mathfrak{d}\right),\quad T\le T_*.
\eq
Invoking the estimate \eqref{est:Fn-nu} and the bounds \eqref{eq.uniform-fn1-nu}, \eqref{eq.uniform-fn2-nu} we find
\[
    \|\ol{f}_{n,m}\|_{L^\infty_TH^{s-1}\cap L^2_TH^{s-\mez}} \le   L |\nu_m-\nu_{m-1}|\cF_3(4R+2L, \mathfrak{a}, \mathfrak{d}),\quad T\le T_*. 
\]
Therefore,  $(f_{n,\nu_m})_{m\geq 0}$ is a Cauchy sequence in $C_{T_*}H^{s-1}\cap L^2_{T_*}H^{s-\mez}$, thereby converging to some $f_n$ in $C_{T_*}H^{s-1}\cap L^2_{T_*}H^{s-\mez}$ as $m\to \infty$.  

In view of the uniform bounds \eqref{eq.uniform-fn1-nu} and \eqref{eq.uniform-fn2-nu} and up to extracting a subsequence $\nu_m \to 0$, we may also assume that $f_{n,\nu_m} \underset{m \to \infty}{\wsc} f_n$ weakly in $L^{\infty}_{T_*}H^s \cap L^2_{T_*}H^{s+\mez}$. In particular, $f_n$ satisfies \eqref{eq.uniform-fn1} and \eqref{eq.uniform-fn2} by lower semi-continuity of the weak limits. We claim that $f_n$ solves \eqref{eq.fn}. 
In order to see this, it suffices to explain how to pass to the limit in the nonlinear term $G[f_{n,\nu}]\Gamma(f_{n,\nu_m})$. This follows from the  convergence $f_{n,\nu_m} \underset{m\to\infty}{\longrightarrow} f_n$ in $L^\infty_{T^*}H^{s-1}$ and the estimate 
\begin{align*}
    \|G[f_{n,\nu_m}]\Gamma(f_{n,\nu_m})-G[f_n]\Gamma(f_n)\|_{L^{\infty}_{T_*}H^{s-1}} &\le \|G[f_n](\Gamma(f_{n,\nu_m})-\Gamma(f_n))\|_{L^{\infty}_{T_*}H^{s-1}} \\
    & \qquad\qquad + \|(G[f_{n,\nu_m}]-G[f_n])\Gamma(f_{n,\nu_m})\|_{L^{\infty}_{T_*}H^{s-1}} \\
    & \le C(\|f_{n,\nu_m},f_n\|_{L^{\infty}_{T_*}H^s})\|f_{n,\nu_m}-f_n\|_{L^{\infty}_{T_*}H^{s}}.
\end{align*}
The preceding estimate also allows one to pass to the limit $m \to \infty$ in  \eqref{eq.uniform-taylor-nu}, so that $f_n$ satisfies \eqref{eq.uniform-taylor}. Moreover, the above strong convergence of $f_{n, \nu_m}$ to $f_n$ implies pointwise convergence, and hence \eqref{eq.uniform-bottom} follows from \eqref{eq.uniform-bottom-nu}. 

{\it Step 5:} Proof of \eqref{eq.uniform-contraction-fn}. 

Set $\ol{f_{n}}=f_{n}-f_{n-1}$ for $n\ge 2$. From \eqref{eq.fn} we have 
\[
    \p_t \ol{f_n}+G[f_n]\Gamma(f_n)-G[f_{n-1}]\Gamma(f_{n-1})=R_{n-1}-R_{n-2}
\]
Using the paralinearization \eqref{paralin:G:contra} for $G[f_n]\Gamma(f_n)-G[f_{n-1}]\Gamma(f_{n-1})$, we find 
\bq\label{eq:olfn}
    \p_t\bar{f}
    _n + T_{\lambda[f_{n-1}](\gamma(f_{n-1})-B_{n-1})}\bar{f}_{n}-T_{V_{n-1}}\cdot \nabla\bar{f}_{n} = R'_n+F_n,
\eq
where
\bq
F_n=R_{n-1}-R_{n-2}
\eq
and $R'_n$ is given by \eqref{eq.R:def} with $(f_1,f_2)$  replaced by $(f_{n},f_{n-1})$. Thus combining \cref{lemm:est:R'} with \eqref{eq.uniform-fn1}-\eqref{eq.uniform-gn} obtained in the previous steps, we find
\bq\label{R'n}
\| R'_n\|_{H^{s-\tdm}}\le C(\| (f_n, f_{n-1})\|_{H^s})\left\{\|\ol{f}_n\|_{H^{s-\mez -\delta}}+\big(1+\|(f_n, f_{n-1})\|_{H^{s+\mez-\delta}}\big)\|\ol{f}_n\|_{H^{s-1}}\right\}.
\eq
When $n=2$, we have $F_2=R_1$  and \eqref{esttrace:cS} (with $\sigma=s-1$) gives
\bq\label{est:F2}
\|F_2\|_{H^{s-1}}\le C(\| f_1\|_{H^s})(\| f_1\|_{H^{s+1}}+\| \tilde g_1\|_{H^s}).
\eq
 On other other hand, for $n\ge 3$, we have
 \[
 R_{n-1}-R_{n-2}=R_0(f_{n-1},f_{n-2},g_{n-1},g_{n-2})
 \]
as given in \eqref{eq.difference-f}. Thus, applying \cref{lemm.R0est} yields
\bq\label{est:Fn}
    \|F_n\|_{H^{s-\frac{3}{2}}} \le C(\|(f_{n-1}, f_{n-2})\|_{H^s})\left\{\|\ol{g}_{n-1}\|_{H^{s-1}}+\|\ol{f}_{n-1} \|_{H^{s-\mez}}(\|\tilde{g}_{n-1}\|_{H^{s-1}}+\|\tilde{g}_{n-2} \|_{H^{s-1}})\right\},
\eq
where $\ol{g}_{n-1}:=\tilde{g}_{n-1}-\tilde{g}_{n-2}$. By virtue of (V${}_{n-1}$), we have $\gamma(f_{n-1})-B_{n-1}\ge \fa$, so that $\lambda[f_{n-1}](\gamma(f_{n-1})-B_{n-1})$ is a first-order elliptic symbol. Therefore, in view of \eqref{R'n} and \eqref{est:Fn}, equation \eqref{eq:olfn} is of the same form as \eqref{eq.transport-diference-final}.
The $H^{s-1}$ energy estimate as in \cref{lemm:parabolicest}  gives
\begin{align}\label{energy:olfn:1}
    \mez\frac{d}{dt}\|\ol{f}_n(t)\|_{H^{s-1}}^2&\le -\frac{1}{\cF_0(\|f_{n-1}\|_{L^{\infty}_TH^s}, \mathfrak{a}, \mathfrak{d})}\| \ol{f}_n(t)\|_{H^{s-\mez}}^2+\cF_0(\|f_{n-1}\|_{L^{\infty}_TH^s}, \mathfrak{a}, \mathfrak{d})\|\ol{f}_n(t)\|_{H^{s-1}}^2+ \notag\\
&\qquad+\big(\| R_n'(t)\|_{H^{s-\frac{3}{2}}}+\| F_n(t)\|_{H^{s-\tdm}}\big)\|\ol{f}_n(t)\|_{H^{s-\mez}},\quad t\le T\le T_*.
\end{align}
By Young's inequality, we have
\[
\| F_n(t)\|_{H^{s-\tdm}}\|\ol{f}_n(t)\|_{H^{s-\mez}}\le \frac{1}{4\cF_0(\|f_{n-1}\|_{L^{\infty}_TH^s}, \mathfrak{a}, \mathfrak{d})}\| \ol{f}_n(t)\|_{H^{s-\mez}}^2+\cF_0(\|f_{n-1}\|_{L^{\infty}_TH^s}), \fa, \fd\| F_n(t)\|_{H^{s-\tdm}}^2.
\]
On the other hand, using the interpolation inequality \eqref{interpolate:contraf}, Young's inequality, and \eqref{R'n}, we find 
\[
\begin{aligned}
\| R'_n(t)\|_{H^{s-\tdm}}\|\ol{f}_n(t)\|_{H^{s-\mez}}&\le\frac{1}{4\cF_0(\|f_{n-1}\|_{L^{\infty}_TH^s}, \mathfrak{a}, \mathfrak{d})}\| \ol{f}_n(t)\|_{H^{s-\mez}}^2\\
&\quad+\cF_1(\|(f_{n}, f_{n-1})\|_{L^{\infty}_TH^s}, \mathfrak{a}, \mathfrak{d})\big(1+\|(f_n, f_{n-1})\|^2_{H^{s+\mez-\delta}}\big)\| \ol{f}_n\|^2_{H^{s-1}}.
\end{aligned}
\]
Plugging the two preceding estimates in \eqref{energy:olfn:1} yields 
\begin{align}\label{energy:olfn:1b}
    \mez\frac{d}{dt}\|\ol{f}_n(t)\|_{H^{s-1}}^2&\le -\frac{1}{2\cF_0(\|f_{n-1}\|_{L^{\infty}_TH^s}, \mathfrak{a}, \mathfrak{d})}\| \ol{f}_n(t)\|_{H^{s-\mez}}^2\notag\\
&\qquad+\cF_2(\|(f_{n}, f_{n-1})\|_{L^{\infty}_TH^s}, \mathfrak{a}, \mathfrak{d})\big(1+\|(f_n, f_{n-1})\|^2_{H^{s+\mez-\delta}}\big)\| \ol{f}_n\|^2_{H^{s-1}}\notag\\
&\qquad +\cF_0(\|f_{n-1}\|_{L^{\infty}_TH^s}), \fa, \fd\| F_n(t)\|_{H^{s-\tdm}}^2,\quad t\le T\le T_*.
\end{align}
Since $\ol{f}_n(0)=0$, Gr\"onwall's lemma implies 
\bq
\|\ol{f}_n\|_{L^\infty_TH^{s-1}\cap L^2_TH^{s-\mez}}\le \| F_n\|_{L^2_TH^{s-\tdm}}\cF_3\left(\|(f_{n}, f_{n-1})\|_{L^{\infty}_TH^s\cap L^2_TH^{s+\mez}}, \mathfrak{a}, \mathfrak{d}\right),\quad T\le T_*.
\eq
Invoking the estimate \eqref{est:Fn} and the induction hypothesis, we deduce
\bq
\begin{aligned}
&\|\ol{f}_n\|_{L^\infty_TH^{s-1}\cap L^2_TH^{s-\mez}}\\
& \le \left\{T^\mez\|\ol{g}_{n-1}\|_{H^{s-1}}+\|\ol{f}_{n-1} \|_{L^2_TH^{s-\mez}}\big(\|\tilde{g}_{n-1}\|_{L^\infty_TH^{s-1}}+\|\tilde{g}_{n-2} \|_{L^\infty_TH^{s-1}}\big)\right\}\\
&\qquad \times \cF_4\left(\|(f_n, f_{n-1})\|_{L^{\infty}_TH^s\cap L^2_TH^{s+\mez}}+\| f_{n-2}\|_{L^\infty_TH^s}, \mathfrak{a}, \mathfrak{d}\right)\\
&\le  \left\{T^\mez2^{-(n-1)}+2^{-(n-1)}4\tM\eps\right\}\cF_4(6R+2L, \fa, \fd),\quad T\le T_*.
\end{aligned}
\eq
Therefore, choosing 
\begin{align}\label{chooseT:3}
&T_*^\mez\le \frac{1}{8\cF_4(6R+2L, \fa, \fd)}, \\ 
&4\tM\eps\le  \frac{1}{8\cF_4(6R+2L, \fa, \fd)}, \label{cd:eps:3}\\
&L\le  \frac{1}{4\cF_4(6R+2L, \fa, \fd)},
\end{align}
we obtain $\|\ol{f}_n\|_{L^\infty_{T_*}H^{s-1}\cap L^2_{T_*}H^{s-\mez}}\le 2^{-n}$, thereby proving \eqref{eq.uniform-contraction-fn}.

{\it Step 6:} Proof of \eqref{eq.uniform-contraction-fn}.

It follows from  \eqref{eq.gn} that $\bar{g}_n:=\tilde{g}_n-\tilde{g}_{n-1}$ satisfies
\[
    \partial_t\bar{g}_n+ \ol{u}_{n-1}\cdot\nabla\bar{g}_n = \tilde{F}_n,
\]
where  $\tilde{F}_n$ is defined by \eqref{eq.defFtilde} with subscripts $i$ replaced with $n-i$, and $\ol{u}_m$ is given by \eqref{def:olum}.  Since $s>\tdm+\frac{d}{2}$, in the same spirit as  \eqref{esttildeu:h} we obtain
\begin{align*}
 \| \ol{u}_{n-1}\|_{H^{s}}&\le C(\| (f_{n-1}, f_{n-2})\|_{H^{s}})\left\{\| (f_{n-1}, f_{n-2})\|_{H^{s+\mez}}\right.\\
 &\qquad\left.+(1+\| (f_{n-1}, f_{n-2})\|_{H^{s+\mez}})\| (\tilde g_{n-1}, \tilde{g}_{n-2})\|_{H^{s}}+\| \p_t f_{n-1}\|_{H^{s-\mez}}\right\}.
\end{align*}
It follows from the  $f_n$-equation \eqref{eq.fn}, \eqref{eq.uniform-fn1} and \eqref{eq.Rn-1-bound} that
\[
    \| \p_t f_{n-1}\|_{H^{s-\mez}}\le \| G[f_{n-1}]\Gamma(f_{n-1})\|_{H^{s-\mez}}+\| R_{n-2}\|_{H^{s-\mez}}\le C(2R)L+ 2C(2R)(1+L)\tM\varepsilon,
\]
where we have bounded $\| G[f_{n-1}]\Gamma(f_{n-1})\|_{H^{s-\mez}}$ with the estimate \eqref{est:DN:h}.   
Combining the  two preceding estimates with the induction hypothesis yields 
\bq\label{eq.ol-un}
 \| \ol{u}_{n-1}\|_{L^1_TH^{s}}\le C(4R)(3L+(1+2L)6\tM),\quad T\le T_*\le 1,
\eq
for a larger $C$. Next, applying \eqref{est:tildeF:L1} and the induction hypothesis on $\|\ol{g}_{n-1} \|_{L^\infty_TH^{s-1}}$, we find 
\bq\label{eq.ol-fn}
\begin{aligned}
\| \tilde{F}_n\|_{L^1_TH^{s-1}}&\le  C(\|(f_{n-1}, f_{n-2})\|_{L^\infty_TH^{s}})\left\{T\|\ol{g}_{n-1} \|_{L^\infty_TH^{s-1}}\right.\\
&\qquad\left.+T^\mez\|\ol{f}_{n-1}\|_{L^2_TH^{s-\mez}}\big(1+\|\tilde{g}_{n-1}\|_{L^\infty_TH^{s}} + \|\tilde{g}_{n-2}\|_{L^\infty_TH^{s}}\big)\right\}\\
&\le C(4R)\left\{T^\mez 2^{-(n-1)}+T^\mez 2^{-(n-1)}(1+4\tM)\right\},\quad T\le T_*\le 1.
\end{aligned}
\eq
Combining \eqref{eq.ol-un}, \eqref{eq.ol-fn} and $\ol{g}_n \vert_{t=0}=0$ in \eqref{transport:g:contra} yields 
\[
    \|\ol{g}_n\|_{L^\infty_TH^{s-1}} \leq \tM C(4R)\left\{T^\mez 2^{-(n-1)}+T^\mez 2^{-(n-1)}(1+4\tM)\right\}\exp\Big(C(4R)\big(3L+(1+2L)6\tM\big)\Big).    
\]
We can conclude that \eqref{eq.uniform-contraction-gn} holds provided 
\[
    T_* \le \left(\frac{1}{8\tM C(4R)(1+2\tM)\exp\Big(C(4R)\big(3L+(1+2L)6\tM\big)\Big)}\right)^2.    
\]

\subsection{Passing to the limit}\label{sec.limit}

From \eqref{eq.uniform-fn1}, \eqref{eq.uniform-fn2} and \eqref{eq.uniform-gn} we infer that $(f_n,\tilde{g}_n)_{n\geqslant 1}$ is bounded in $(L^{\infty}_TH^s\cap L^2_TH^{s+\mez} )\times L^{\infty}_TH^s$. Therefore up to extracting along a subsequence, we can assume that 
\bq\label{wc:fngn}
    (f_n,\tilde{g}_n) \wsc (f,\tilde{g}) \in (L^{\infty}_TH^s\cap L^2_TH^{s+\mez} )\times L^{\infty}_TH^s.
\eq
In particular, $f$ and $\tilde{g}$ inherit the bounds \eqref{eq.uniform-fn1}, \eqref{eq.uniform-fn2}, and \eqref{eq.uniform-gn}. By virtue of  \eqref{eq.uniform-contraction-fn} and \eqref{eq.uniform-contraction-gn},  we also have  
\bq\label{sc:fngn}
    (f_n,\tilde{g}_n) \to (f,\tilde{g}) \text{ in } \Big(C([0, T]; H^{s-1})\cap L^2_TH^{s-\mez}\Big)\times C([0, T]; H^{s-1}).
\eq
Interpolation with the uniform bounds \eqref{eq.uniform-fn1}, \eqref{eq.uniform-fn2} and \eqref{eq.uniform-gn} yield the convergence 
\bq\label{sc:fngn:2}
    (f_n,\tilde{g}_n) \to (f,\tilde{g}) \text{ in } \Big(C([0, T]; H^{s-\varepsilon})\cap L^2_TH^{s+\mez-\varepsilon}\Big)\times C([0, T]; H^{s-\varepsilon})\quad\forall \eps \in (0, 1].
\eq
In particular, $f_n\to f$ uniformly on $\Rr^d\times [0, T]$, and hence $f$ satisfies \eqref{eq.uniform-bottom}.

Next, we pass to the limit in \eqref{eq.fn}-\eqref{eq.gn} to show that $(f, g)$ solves \eqref{eq:f}-\eqref{eq:g}, where $g:=\ff_f^{-1}\tilde g$.  Since \eqref{wc:fngn} implies that  $(\partial_tf_n, \p_t\tilde{g}_n) \to  (\p_t f, \p_t\tilde{g})$ in the sense of distributions, it remains to explain how one can pass to the limit in the nonlinear terms.  We begin with \eqref{eq.fn}. Combining  \cref{theo:DN1} (i) and \cref{thm.NPdiff} (ii), we find
\begin{align*}
    \|G[f_n]\Gamma(f_n)-G[f]\Gamma(f)\|_{L^{\infty}_TH^{s-1-\varepsilon}} 
    & \le C(\|f,f_n\|_{L^{\infty}_TH^s})\|f_n-f\|_{L^{\infty}_TH^{s-\varepsilon}},
\end{align*}
which tends to zero thanks to \eqref{eq.uniform-fn1} and \eqref{sc:fngn:2}. In particular, $G[f_n]\Gamma(f_n)\to G[f]\Gamma(f)$ uniformly on $\Rr^d\times [0, T]$, and hence $f$ satisfies \eqref{eq.uniform-taylor}. Next, we pass to the limit in $R_{n-1}$. Clearly $g_{n-1}\vert_{\Sigma_{f_{n-1}}}=\tilde{g}_{n-1}\vert_{z=0}\to \tilde{g}\vert_{z=0}$ in $L^\infty_t H^{s-\mez-\eps}_x$ by the trace theorem. We use the estimate \eqref{eq.lowNormEst-v2} in \cref{prop.diffEstimateSobolev} for the difference  $v^{(2)}_n- v^{(2)}$, where $v^{(2)}_n$ is associated to $(g_n, f_n)$ and $v^{(2)}$ to $(g, f)$, and where $k_{1} \circ \ff_{f_n}=\tilde{g}_n$, $k_{2} \circ \ff_{f}=\tilde{g}$.  We obtain
\[
    \|\nabla_{x, z} (v^{(2)}_n- v^{(2)})\|_{H^{s-1}} \le C(\|(f_n,f,g_n,g)\|_{H^s})\left(\|\tilde{g}_n-\tilde{g}\|_{H^{s-1}} + \|f_n-f\|_{H^{s-\mez}}\right),
\]
Combining this with \eqref{sc:fngn:2} gives $ \|\nabla_{x, z} (v^{(2)}_n- v^{(2)})\|_{L^\infty_T H^{s-1}} \to 0$. Note that
\[
 (\cS[f_n]g_n )\vert_{\Sigma_{f_n}}=\nabla_{x, z}v_n^{(2)}\vert_{z=0}(\na \ff_{f_n})^{-1}\vert_{z=0}.
\]
Since $(\na \ff_{f_n})^{-1}\vert_{z=0} - (\na \ff_{f})^{-1}\vert_{z=0} \to 0$ and $N_n - N\to 0$ in $L^{\infty}_TH^{s-1-\varepsilon}$,  $\nabla_{x, z} v^{(2)}_n\vert_{z=0}\to  \nabla_{x, z} v^{(2)}\vert_{z=0}$ in  $L^\infty_T H^{s-\tdm}$, and that $H^{s-\tdm}(\Rr^d)$ is an algebra for $s>\tdm+\frac{d}{2}$, it follows that $N_n\cdot (\cS[f_n]g_n )\vert_{\Sigma_{f_n}}\to N\cdot (\cS[f]g )\vert_{\Sigma_{f}}$ in $L^\infty_TH^{s-\tdm}$. Thus $f$ satsfies \eqref{eq:f}. Moreover, we have shown the strong convergence 
\bq\label{cv:dtfn}
\p_t f_n\to \p_t f\quad\text{in } L^\infty_TH^{s-\tdm}(\Rr^d).
\eq
Next, we pass to the limit in \eqref{eq.gn}. We have
\begin{align*}
u_{n-1}\circ \ff_{f_{n-1}} &=-\left(\cG[f_{n-1}]\Gamma(f_{n-1})-\cS[f_{n-1}]g_{n-1}\right)\circ \ff_{f_{n-1}} - \tilde{g}_{n-1}e_y\\
&=- \na v^{(1)}_{n-1}(\na \ff_{f_{n-1}})^{-1} - \na v^{(2)}_{n-1}(\na \ff_{f_{n-1}})^{-1}- \tilde{g}_{n-1}e_y,
\end{align*}
where $v^{(j)}_n$ is associated to  $(g_n, f_n)$ and $v^{(j)}$ to $(g, f)$ as in \cref{prop.diffEstimateSobolev}. Applying \cref{prop.diffEstimateSobolev} and the fact that $H^{s-1}(\Rr^d\times J)$ is an algebra, we deduce that  $u_{n-1}\circ \ff_{f_{n-1}}\to u\circ \ff_{f}$ in $L^\infty_TH^{s-1}$. Then invoking  \eqref{cont:diff} yields 
\[
    \gamma'(\varrho_{n-1})u_{n-1,y}\circ \ff_{f_{n-1}} \to  \gamma'(\varrho)u_y \circ \ff_f \quad \text{in } L^{\infty}_TH^{s-1}.    
\]
We recall \eqref{tildeu} and \eqref{def:olum} for the definitions  of $\ol{u}$ and  $\ol{u}_m$. By \cref{prop.diffEstimateSobolev} and the convergence $\p_t \varrho_{n-1}\to \p_t \varrho$ in $L^\infty_TH^{s-1}$, which follows from \eqref{cv:dtfn}, we find 
\[
\ol{u}_{n-1, x}= u_{n-1,x} \circ \ff_{f_{n-1}}\to u_{x} \circ \ff_{f}=\ol{u}_x\quad\text{in } L^\infty_TH^{s-1}
\]
and 
\begin{multline*}
  \ol{u}_{n-1, z} =\frac{1}{\p_z \varrho_{n-1}}\Big\{-(-\na_x\varrho_{n-1}, 1)\cdot\cG[f]\Gamma(f_{n-1}) \circ \ff_{f_{n-1}}- (-\na_x\varrho_{n-2}, 1)\cdot \cS[f_{n-2}]g_{n-2} \circ \ff_{n-2}\\
  -g_{n-2} \circ \ff_{f_{n-2}} - \p_t\varrho_{n-1})\Big\} \\
\to \frac{1}{\p_z\varrho}\Big\{-(-\na_x\varrho, 1)\cdot \big(\cG[f]\Gamma(f)+\cS[f]g\big)\circ \ff_f -g\circ\ff_f- \p_t\rho\Big\} =\ol{u}_z\quad\text{in } L^\infty_TH^{s-1}.
\end{multline*}
Here $\ol{u}=(\ol{u}_x, \ol{u}_z)$ is given by \eqref{tildeu} and \eqref{oluz:GS}. Thus $\ol{u}_{n-1}\to \ol{u}$ in $L^\infty_TH^{s-1}$. But \eqref{sc:fngn:2} implies that $\na \tilde{g}_n\to \na \tilde{g}$ in $L^\infty_TH^{s-1-\eps}$. Choosing $0<\eps<s-\tdm-\frac{d}{2}$ so that $H^{s-1-\eps}(\Rr^d\times J)$ is an algebra, we obtain that $\ol{u}_{n-1}\cdot \nabla \tilde{g}_n\to \ol{u}\cdot \nabla \tilde{g}$ in $L^\infty_TH^{s-1-\eps}$. Therefore, $\tilde{g}$ satisfies \eqref{eq:tildeg}. Finally, by undoing the change of variables $\cF_f$, we conclude that $g$ satisfies \eqref{eq:g}.

\subsection{Continuity in time, uniqueness, and continuity of the flow} \label{sec.uniqueness}

Let us first explain the continuity in time statement $(f, \tilde{g})\in C([0,T],H^s(\Rr^d))\times C([0, T]; H^s(\Rr^d\times J))$. The continuity of $f$ follows from \cref{lemm:interpolationLM}: since we have already proven that $f$ belongs to $L^{\infty}([0,T],H^s) \cap L^2([0,T],H^{s+\mez})$, it remains to check that 
\[
    \p_t f = - G[f]f - (N\cdot \cS[f]g + g)\vert_{\Sigma_[f(t)]} \in L^{2}_TH^{s-\mez}.     
\]
That $G[f]f \in L^2_TH^{s-\mez}$ is a direct consequence of \eqref{est:DN:h}. Then, as $g \in L^{\infty}_TH^s$, \eqref{esttrace:cS} and product estimates in Sobolev spaces imply that $N\cdot \cS[f]g\vert_{\Sigma_f} \in L^{2}_TH^{s-\mez}$. As for the continuity of $\tilde{g}$, we recall from \eqref{sc:fngn:2} that $\tilde{g}$ is a $L^\infty([0, T]; H^s)\cap C([0, T]; H^1)$ solution to the transport problem \eqref{eq:tildeg}. We also recall from \eqref{energyest:u} and \eqref{energyest:u:l} that $\ol{u}\in L^1([0, T]; H^s)\cap L^\infty([0, T]; H^{s-\mez})$, and from \eqref{estforcing:tildeg} that $\gamma'(\varrho)u_y\circ \ff_f\in L^2([0, T]; H^s)$. Therefore, the existence and uniqueness statements in \cref{theo:transport} together imply that $\tilde{g}\in C([0, T]; H^s)$. 

Uniqueness and stability in lower   norms follow from estimates provided in \cref{sec.contraction-f}: let $f_{0,1}$ and $f_{0,2}$ be $H^s(\mathbb{R}^d)$ functions satisfying \eqref{eq.RT-initial} and \eqref{eq.boundary-separation-initial}, and let $R=\max\{\|f_{0,1}\|_{H^s},\|f_{0,2}\|_{H^s}\}$. Let $\eps(R)$ be the constant given in \cref{prop.iterative-bounds} and consider $g_{0,i} \in H^{s}(\Rr^d \times J)$ with $\|\tilde{g}_{0,1}\|_{H^s}, \|\tilde{g}_{0,2}\|_{H^s} \le \varepsilon(R)$. Let  $(f_1,\tilde{g}_1)$ and $(f_2,\tilde{g}_2)$ be  the corresponding  solutions to \eqref{eq:f}-\eqref{eq:g} constructed in the previous subsection. Because $\varepsilon(R)$ has been chosen such that \eqref{eq.assump-smallness2} holds, it follows that $(f,\tilde{g})$ satisfy \eqref{eq.assump-smallness} and therefore \cref{prop.contractionEstimates} yields  
\begin{align*}
    \|f_1-f_2\|_{X_T^{s-1}} &\le C(\|(f_1,f_2)\|_{X^s_T})\left(\|f_{0,1}-f_{0,2}\|_{H^{s-1}} + T^{\mez}\|\tilde{g}_1-\tilde{g}_2\|_{L^{\infty}_TH^{s-1}}\right),
\end{align*}
where $X_T^{s-1}=L^{\infty}_TH^s\cap L^2_TH^{s-\mez}$. Simiarly \cref{prop.contraction-gn} yields 
\begin{align*}
    \|\tilde{g}_1-\tilde{g}_2\|_{L^{\infty}_TH^{s-1}} & \le \tM\Big\{\|\tilde{g}_{0,1}-\tilde{g}_{0,2}\|_{H^{s-1}} \\
    &\quad + T^{\mez}C\Big(\|f_1,f_2\|_{X^s_T},\|(\tilde{g}_1,\tilde{g}_2)\|_{L^{\infty}_TH^s}\Big)\Big(\|\tilde{g}_1-\tilde{g}_2\|_{L^\infty_TH^{s-1}} + \|f_1-f_2\|_{L^2_TH^{s-\mez}}\Big)\Big\}\\ 
    &\quad \cdot \exp\Big(T^\mez C(\|f_1,f_2\|_{X^s_T},\|(\tilde{g}_1,\tilde{g}_2)\|_{L^{\infty}_TH^s})\Big),
\end{align*}
provided $s\ne \frac{5}{2}+\frac{d}{2}$.  

Adding the two preceding estimates and choosing of $T=T(R)$ small enough such that 
\[
(1+ \tM)T^{\mez}C\Big(\|f_1,f_2\|_{X^s_T},\|(\tilde{g}_1,\tilde{g}_2)\|_{L^{\infty}_TH^s}\Big)\exp\Big(T^\mez C(\|f_1,f_2\|_{X^s_T},\|(\tilde{g}_1,\tilde{g}_2)\|_{L^{\infty}_TH^s})\Big)\le \mez,
\]
 we obtain
\begin{align*}
    \|f_1-f_2\|_{X^{s-1}_T} +  \|\tilde{g}_1-\tilde{g}_2\|_{L^{\infty}_TH^{s-1}} &\le \frac{1}{2} \|f_1-f_2\|_{X^{s-1}_T} + \frac{1}{2}\|\tilde{g}_1-\tilde{g}_2\|_{L^\infty_TH^{s-1}} \\
    &\quad +  \mez (\|f_{0,1}-f_{0,2}\|_{H^{s-1}} + \|\tilde{g}_{0,1}-\tilde{g}_{0,2}\|_{H^{s-1}}).
\end{align*}
It follows that 
\[
    \|f_1-f_2\|_{X^{s-1}_T} +  \|\tilde{g}_1-\tilde{g}_2\|_{L^{\infty}_TH^{s-1}} \le  \|f_{0,1}-f_{0,2}\|_{H^{s-1}} + \|\tilde{g}_{0,1}-\tilde{g}_{0,2}\|_{H^{s-1}}
\]
which proves both the uniquness and the Lipschitz coninuity of the flow map $H^s \times H^s \to H^{s-1} \times H^{s-1}$ on $B(0, R) \times B(0, \varepsilon(R))$. As the flow is also bounded on $B(0, R) \times B(0, \varepsilon(R))$ as a map $H^s \times H^s \to H^{s} \times H^{s}$, interpolation implies that it is also continuous as a map $H^s \times H^s \to H^{s'} \times H^{s'}$ for any $s'<s$.

\appendix

\section{Tools in paradifferential calculus}\label{sec.paraDiff}
We first  recall the Bony decomposition  
\[
    au=T_au+T_ua+R(a,u),
\]
where $T_au$ denotes the usual paraproduct. See  \cite[Chapter 2]{BCD}. 

 Useful product and paraproduct rules are gathered in the following theorem.
\begin{theo}\cite[Section 2.8.1]{BCD}\label{theo.productEstimates}
Let $s_0$, $s_1$, $s_2\in \Rr$ and consider functions $a, u, u_1, u_2$ on $\Rr^d$. 
\begin{enumerate}
\item For any $s\in \Rr$, 
\bq\label{pp:Linfty}
\| T_a u\|_{H^s}\le C\| a\|_{L^\infty}\| u\|_{H^s}.
\eq
\item If $s\in \Rr$ and $\mu>0$ then 
\bq\label{pp:C-mu}
\| T_a u\|_{H^{s-\mu}}\le C\| a\|_{C_*^{-\mu}}\| u\|_{H^s}.
\eq
\item If
$s_0\le s_2$ and~$s_0 < s_1 +s_2 -\frac{d}{2}$, 
then
\begin{equation}\label{boundpara}
\Vert T_a u\Vert_{H^{s_0}}\le C \Vert a\Vert_{H^{s_1}}\Vert u\Vert_{H^{s_2}}.
\end{equation}
\item  If~$s_1+s_2>0$ then
\begin{align}
&\Vert R(a,u) \Vert _{H^{s_1 + s_2-\frac{d}{2}}} \leq C \Vert a \Vert _{H^{s_1}}\Vert u\Vert _{H^{s_2}},\label{Bony1}\\
&\Vert R(a,u) \Vert _{H^{s_1+s_2}} \leq C \Vert a \Vert _{C^{s_1}_*}\Vert u\Vert _{H^{s_2}}.\label{Bony2}
\end{align}
\item If $s_1+s_2> 0$, $s_0\le s_1$ and $s_0< s_1+s_2-\frac{d}{2}$ then  
\bq\label{paralin:product}
\Vert au - T_a u\Vert_{H^{s_0}}\le C \Vert a\Vert_{H^{s_1}}\Vert u\Vert_{H^{s_2}}.
\eq
\item If $s>0$ then 
\bq\label{tamepr}
\| u_1u_2\|_{H^s}\le C\| u_1\|_{L^\infty}\| u_2\|_{H^s}+C\| u_1\|_{H^s}\| u_2\|_{L^\infty}.
\eq
\item \label{it.1} If $s_1+s_2> 0$, $s_0\le s_1$, $s_0\le s_2$ and  $s_0< s_1+s_2-\frac d2
$ then 
\begin{equation}\label{pr}
\Vert u_1 u_2 \Vert_{H^{s_0}}\le C \Vert u_1\Vert_{H^{s_1}}\Vert u_2\Vert_{H^{s_2}}.
\end{equation}
\end{enumerate}
\end{theo}
In \cref{sec.DN} we use the following product estimates in the spaces $X^s$ and $Y^s$  defined in \eqref{def:XY}. 
\begin{lemm}\cite[Lemma 3.10]{NP}\label{lemm.estXandY} Let $s_0, s_1, s_2 \in \mathbb{R}$ and $J\subset \mathbb{R}$. 
\begin{enumerate}
    \item[(i)] If $s_0 \le \min\{s_1 +1, s_2+1\}$ and $s_1+s_2+1>\max\{s_0+\frac{d}{2},0\}$,  then 
    \begin{equation}
        \label{eq.estXXY}
        \|u_1u_2\|_{Y^{s_0}(J)} \lesssim \|u_1\|_{X^{s_1}(J)}\|u_2\|_{X^{s_2}(J)}.
    \end{equation}
    \item[(ii)] If $s_0 \le \min\{s_1, s_2\}$ and $s_1+s_2>\max\{s_0+\frac{d}{2},0\}$,  then 
    \begin{equation}
        \label{eq.estXYY}
        \|u_1u_2\|_{Y^{s_0}(J)} \lesssim \|u_1\|_{Y^{s_1}(J)}\|u_2\|_{X^{s_2}(J)}.
    \end{equation}
\end{enumerate}    
\end{lemm}
We have the following estimates for the left composition of a Sobolev function by a smooth function. 
\begin{theo}\cite[Theorem 2.89]{BCD}\label{est:nonl}
Let  $\cU\subset \Rr^N$ be a minimally smooth domain, and let $F\in C^\infty(\Rr^K; \Rr)$. For any $s\ge 0$, there exists a non-decreasing function~$\mathcal{F}\colon\Rr_+\rightarrow\Rr_+$ such that the following assertions hold.
\begin{enumerate}
    \item[(i)]  For any $u\in H^s(\cU; \Rr^K) \cap L^{\infty}(\cU; \Rr^K)$, there holds
            \bq\label{est:Fu}
                \| F(u)-F(0)\|_{H^s(\cU)}\le \cF(\| u\|_{L^\infty(\cU)})\| u\|_{H^s(\cU)}.
            \eq
    \item[(ii)]   For any  $u, v \in L^{\infty}(\cU; \Rr^K)\cap H^s(\cU; \Rr^K)$, there holds 
            \bq
            \begin{aligned}
            \label{eq.diff-sobol-reg}
                \|F(u)-F(v)\|_{H^s(\cU)} &\le  \cF(\|(u, v)\|_{L^\infty(\cU)})\left(\| (u, v)\|_{H^s(\cU)}\|u-v\|_{L^\infty(\cU)}+\|(u, v)\|_{L^\infty(\cU)}\|u-v\|_{H^s(\cU)}\right)\\
                &\qquad+|F'(0)|\| u-v\|_{H^s(\cU)}.
            \end{aligned}
            \eq
\end{enumerate}
\end{theo}
\begin{rema}\label{rem:diff-sobol-bound} 
The estimate \eqref{est:Fu} for $\cU=\Rr^N$ is classical and can be found in Theorem 2.89,  \cite{BCD}. For a minimally smooth domain $\cU$, \eqref{est:Fu} follows  from the case $\cU=\Rr^N$ and the Sobolev extension operator in \cref{th.extension}. On the other hand, the contraction estimate \eqref{eq.diff-sobol-reg} is a consequence of \eqref{est:Fu} since
\[
F(u)-F(v)=\int_0^1\big(F'(v+t(u-v))-F'(0)\big)\d t(u-v)+F'(0)(u-v).
\]
\end{rema}

Nonlinear functions can be paralinearized using the following theorem.
\begin{theo}\cite[Theorem 5.2.4]{MePise}\label{paralin:nonl}
Let $F\in C^\infty(\Rr)$ with $F(0)=0$ If $u\in H^s(\Rr^d)$ with $\mu=s-\frac{d}{2}>0$, then 
\bq\label{paralin:Fu}
\| F(u)-T_{F'(u)}u\|_{H^{s+\mu}}\le C(\| u\|_{L^\infty})\| u\|_{C^\mu_*}\| u\|_{H^s}.
\eq
\end{theo}

The next proposition provides a contraction estimate for the remainder in the paralinearization \eqref{paralin:Fu}. 
\begin{prop}
Let $F\in C^\infty(\Rr)$ and consider $s_0, s_1, s_2 \in \mathbb{R}$ satisfying 
\bq\label{cd:sj:contra}
s_0\le s_2,\quad 0\le s_1\le s_2,\quad s_2>\frac{d}{2},\quad \text{and } s_0<s_1+s_2-\frac{d}{2}.
\eq
Then there exist $C: \Rr_+\to \Rr_+$ and $C'>0$  such that the inequality 
\begin{multline}\label{Rparalin:contra}
 \| (F(u)-T_{F'(u)}u)-(F(v)-T_{F'(v)}v)\|_{H^{s_0}}\le  C'|F'(0)| \| u-v\|_{H^{s_1}} + C'|F''(0)|\| u-v\|_{H^{s_1}}\| u-v\|_{H^{s_2}} \\ 
 + \left(C(\| (u, v)\|_{L^\infty\cap H^{s_1}}) + |F''(0)|\right)\| u-v\|_{H^{s_1}\cap L^{\infty}}(\| u\|_{H^{s_2}}+\| v\|_{H^{s_2}})
\end{multline}
holds for all $u$, $v\in H^{s_2}$. 
\end{prop}
\begin{proof}
We denote 
\[
A_F=(F(u)-T_{F'(u)}u)-(F(v)-T_{F'(v)}v).
\]
If we set $\tilde{F}(x)=F(x)-F'(0)x$ then 
\[
A_F=A_{\tilde{F}}+F'(0)(u-v)-T_{F'(0)}(u-v)=A_{\tilde{F}}+F'(0)(1-\Psi(D))(u-v),
\]
where $\Psi$ is given by \eqref{cond.psi}. We have 
\[
\| F'(0)(1-\Psi(D))(u-v)\|_{H^{s_0}}\le C'|F'(0)|\| u-v\|_{H^{s_1}}. 
\]
 Note that $\tilde{F}'(0)=0$ and $\tilde{F}''=F''$. Thus, to prove \eqref{Rparalin:contra} we can assume $F'(0)=0$ in what follows. 

Using the mean value theorem, we have 
\[
\begin{aligned}
A_F=(F(u)-T_{F'(u)}u)-(F(v)-T_{F'(v)}v)&=(F(u)-F(v))-T_{F'(u)-F'(v)}u-T_{F'(v)}(u-v)\\
&=w(u-v)-T_{F'(u)-F'(v)}u-T_{F'(v)}(u-v),
\end{aligned}
\]
where $w=\int_0^1F'(v+\tau(u-v))d\tau$. By paralinearizing the product $w(u-v)$ and using the fact that
\[
w-F'(v)=\big(\tilde w+\mez F''(0)\big)(u-v),\quad \tilde{w}=\int_0^1\int_0^1\Big(F''(u+\tau'\tau(u-v))-F''(0)\Big)\tau d\tau' d\tau,
\]
we obtain 
\bq
A_F=T_{u-v}w+R(u-v, w)-T_{F'(u)-F'(v)}u+T_{(\tilde{w}+\mez F''(0))(u-v)}(u-v).
\eq
We now show that each term in $A_F$ is controlled by the right-hand side of \eqref{Rparalin:contra}.

Since $s_2\ge 0$ and $F'(0)=0$, it follows from \cref{est:nonl} that
\[
\| w\|_{H^{s_2}}\le C(\| (u, v)\|_{L^\infty})(\| u\|_{H^{s_2}}+\| v\|_{H^{s_2}}).
\]
Under the conditions in \eqref{cd:sj:contra}, we have $s_0\le s_2$, $s_0<s_1+s_2-\frac{d}{2}$, $s_1+s_2>0$. Thus we can apply  \eqref{boundpara}, \eqref{Bony1}, and \eqref{est:Fu} to have 
\[
\begin{aligned}
\| T_{u-v}w\|_{H^{s_0}}+\| R(u-v, w)\|_{H^{s_0}}&\les \| u-v\|_{H^{s_1}}\| w\|_{H^{s_2}}\\
&\le C(\| (u, v)\|_{L^\infty})\| u-v\|_{H^{s_1}}(\| u\|_{H^{s_2}}+\| v\|_{H^{s_2}}).
\end{aligned}
\]
On the other hand, combining \eqref{boundpara} and \eqref{eq.diff-sobol-reg} (for $s_1\ge 0$) yields 
\[
\begin{aligned}
\|T_{F'(u)-F'(v)}u\|_{H^{s_0}}&\les \|F'(u)-F'(v)\|_{H^{s_1}}\| u\|_{H^{s_2}}\\
& \le C(\| (u, v)\|_{L^{\infty}\cap H^{s_1}})\| u-v\|_{L^{\infty}\cap H^{s_1}}\| u\|_{H^{s_2}}+|F''(0)|\|u-v\|_{H^{s_1}}\|u\|_{H^{s_2}}.
\end{aligned}
\]
As for the last term, we first apply \eqref{boundpara} to have 
\[
\| T_{(\tilde{w}+\mez F''(0))(u-v)}(u-v)\|_{H^{s_0}}\les \|(\tilde{w}+\mez F''(0))(u-v)\|_{H^{s_1}}\| u-v\|_{H^{s_2}}.
\]
Since  $s_1+s_2>0$, $s_1\le s_2$, and $s_2>\frac{d}{2}$, \eqref{pr} implies 
\[
\| \tilde{w}(u-v)\|_{H^{s_1}}\les \| \tilde{w}\|_{H^{s_2}}\| u-v\|_{H^{s_1}}.
\]
Invoking \eqref{est:Fu} again, we have
\[
 \| \tilde{w}\|_{H^{s_2}}\le C(\| (u, v)\|_{L^\infty})(\| u\|_{H^{s_2}}+\| v\|_{H^{s_2}}),
 \]
 thereby deducing 
 \begin{align*}
 \| T_{(\tilde{w}+\mez F''(0))(u-v)}(u-v)\|_{H^{s_0}}&\le C(\| (u, v)\|_{L^\infty})\| u-v\|_{H^{s_1}}(\| u\|_{H^{s_2}}+\| v\|_{H^{s_2}})\\
 &\qquad+C'|F''(0)|\| u-v\|_{H^{s_1}}\| u-v\|_{H^{s_2}}.
\end{align*}
Combining the above estimates, we arrive at \eqref{Rparalin:contra}. 
\end{proof}

Next, we review basic facts about Bony's paradifferential calculus (see e.g. \cite{Bony, Hormander, MePise, ABZ}). 
\begin{defi}\label{defi:para}
\begin{enumerate}[(1)]
\item (Symbols) Given~$\varrho\in [0, \infty)$ and~$m\in\Rr$,~$\Gamma_{\varrho}^{m}(\Rr^d)$ denotes the space of
locally bounded functions~$a(x,\xi)$
on~$\Rr^d\times(\Rr^d\setminus 0)$,
which are~$C^\infty$ with respect to~$\xi$ for~$\xi\neq 0$ and
such that, for all~$\alpha\in\Nn^d$ and all~$\xi\neq 0$, the function
$x\mapsto \partial_\xi^\alpha a(x,\xi)$ belongs to~$W^{\varrho,\infty}(\Rr^d)$ and there exists a constant
$C_\alpha$ such that for all $|\xi|\ge \mez$ there holds
\begin{equation*}
\Vert \partial_\xi^\alpha a(\cdot,\xi)\Vert_{W^{\varrho,\infty}(\Rr^d)}\le C_\alpha
(1+|\xi|)^{m-|\alpha|}.
\end{equation*}
Let $a\in \Gamma_{\varrho}^{m}(\Rr^d)$, we define the semi-norm
\begin{equation}\label{defi:norms}
M_{\varrho}^{m}(a)= 
\sup_{|\alpha|\le 2(d+2) +\varrho ~}\sup_{|\xi| \ge \mez~}
\Vert (1+|\xi|)^{|\alpha|-m}\partial_\xi^\alpha a(\cdot,\xi)\Vert_{W^{\varrho,\infty}(\Rr^d)}.
\end{equation}
\item (Paradifferential operators) Given a symbol~$a$, we define
the paradifferential operator~$T_a$ by
\begin{equation}\label{eq.para}
\widehat{T_a u}(\xi)=(2\pi)^{-d}\int \chi(\xi-\eta,\eta)\widehat{a}(\xi-\eta,\eta)\Psi(\eta)\widehat{u}(\eta)
\, d\eta,
\end{equation}
where
$\widehat{a}(\theta,\xi)=\int e^{-ix\cdot\theta}a(x,\xi)\, dx$
is the Fourier transform of~$a$ with respect to the first variable; 
$\chi$ and~$\Psi$ are two fixed~$C^\infty$ functions such that:
\begin{equation}\label{cond.psi}
    \Psi(\eta)=0\quad \text{for } |\eta|\le \frac{1}{5},\qquad
    \Psi(\eta)=1\quad \text{for }|\eta|\geq \frac{1}{4},
\end{equation}
and~$\chi(\theta,\eta)$ 
satisfies, for~$0<\eps_1<\eps_2$ small enough,
\[
    \chi(\theta,\eta)=1 \quad \text{if}\quad |\theta|\le \eps_1| \eta|,\qquad
    \chi(\theta,\eta)=0 \quad \text{if}\quad |\theta|\geq \eps_2|\eta|,
\]
and such that
\[
    \forall (\theta,\eta), \qquad | \partial_\theta^\alpha \partial_\eta^\beta \chi(\theta,\eta)|\le 
C_{\alpha,\beta}(1+| \eta|)^{-|\alpha|-|\beta|}.
\]
\end{enumerate}
\end{defi}
\begin{rema}
The cut-off $\chi$ can be appropriately chosen so that when $a=a(x)$, the paradifferential operator $T_au$ becomes the usual paraproduct.
\end{rema}
\begin{defi}\label{defi:order}
Let~$m\in\Rr$.
An operator~$T$ is said to be of  order~$m$ if, for all~$\mu\in\Rr$,
it is bounded from~$H^{\mu}$ to~$H^{\mu-m}$.
\end{defi}
Symbolic calculus for paradifferential operators is summarized in the following theorem.
\begin{theo}[Symbolic calculus]\label{theo:sc}
Let~$m\in\Rr$ and~$\varrho\in [0,1]$.
\begin{enumerate}[(i)]
    \item If~$a \in \Gamma^m_0(\Rr^d)$, then~$T_a$ is of order~$m$. 
    Moreover, for all~$\mu\in\Rr$ there exists a constant~$K$ such that
    \begin{equation}\label{esti:quant1}
    \Vert T_a \Vert_{H^{\mu}\rightarrow H^{\mu-m}}\le K M_{0}^{m}(a).
    \end{equation}
    \item If~$a\in \Gamma^{m}_{\varrho}(\Rr^d), b\in \Gamma^{m'}_{\varrho}(\Rr^d)$ then 
    $T_a T_b -T_{ab}$ is of order~$m+m'-\varrho$. 
    Moreover, for all~$\mu\in\Rr$ there exists a constant~$K$ such that
    \begin{equation}\label{esti:quant2}
    \begin{aligned}
    \Vert T_a T_b  - T_{a  b} \Vert_{H^{\mu}\rightarrow H^{\mu-m-m'+\varrho}}
    &\le 
    K (M_{\varrho}^{m}(a)M_{0}^{m'}(b)+M_{0}^{m}(a)M_{\varrho}^{m'}(b)).
    \end{aligned}
    \end{equation}
    \item Let~$a\in \Gamma^{m}_{\varrho}(\Rr^d)$. Denote by 
    $(T_a)^*$ the adjoint operator of~$T_a$ and by~$\overline{a}$ the complex conjugate of~$a$. Then $(T_a)^* -T_{\overline{a}}$ is of order~$m-\varrho$.
    Moreover, for all~$\mu$ there exists a constant~$K$ such that
    \begin{equation}\label{esti:quant3}
    \Vert (T_a)^*   - T_{\overline{a}}   \Vert_{H^{\mu}\rightarrow H^{\mu-m+\varrho}}
    \le 
    K M_{\varrho}^{m}(a).
    \end{equation}
\end{enumerate}
\end{theo}
\begin{rema}\label{rema:low}
In the definition \eqref{eq.para} of paradifferential operators, the cut-off $\Psi$ removes the low frequency part of $u$. In particular, when $a\in \Gamma^m_0$  we have
\[
\Vert T_a u\Vert_{H^\sigma}\le CM_0^m(a)\Vert \nabla u\Vert_{H^{\sigma+m-1}}. 
\]
\end{rema}
To handle symbols of negative Zygmund regularity, we use  
\begin{prop}[\protect{\cite[Proposition 2.12]{ABZ}}]\label{prop:negOp}
Let  $m\in \Rr$ and $\varrho<0$. We denote by $\dot \Gamma^m_\varrho(\Rr^d)$ the class of symbols $a(x, \xi)$ that are homogeneous of order $m$ in $\xi$, smooth in $\xi\in \Rr^d\setminus \{0\}$ and such that  
\[
M_{\varrho}^{m}(a)= 
\sup_{|\alpha|\le 2(d+2) +\varrho ~}\sup_{|\xi| \ge \mez~}
\Vert |\xi|^{|\alpha|-m}\partial_\xi^\alpha a(\cdot,\xi)\Vert_{C^\varrho_*(\Rr^d)}<\infty.
\]
 If $a\in \dot \Gamma^m_\varrho$ then the operator $T_a$ defined by  \eqref{eq.para} is of order~$m-\varrho$.
\end{prop}

\section{Extension operator on Lipschitz domains} \label{sec.extension}
We first recall from  \cite{Stein} the definition of {\it minimally smooth} domains.
\begin{defi}\label{defi:mmdomain}
(i) An open set $\cU\subset \Rr^N$ is called a special Lipschitz domain with constant $M>0$ there exists a Lipschitz function $\varphi: \Rr^{N-1}\to \Rr$ with Lipschitz constant bounded by  $M$ such that, upon relabeling and reorienting the coordinate axes, 
\[
\cU=\{x=(x', x_N)\in \Rr^{N-1}\times \Rr: x_N<\varphi_i(x')\}.
\]
(ii) An open set $\cU \subset \Rr^N$ is called a minimally smooth with constants $(\iota, K, M)$ if there exists an $\iota>0$, an integer $K$, a real number $M>0$, and a family $\{V_i\}_{i=1}^\infty$ of open subsets of $\Rr^N$ such that:
\begin{itemize}
\item $\p \cU\subset \cup_{i=1}^\infty V_i$.
\item If $x\in \p \cU$, then $B(x, \iota)\subset V_i$ for some $i$.
\item No point of $\Rr^N$ is contained in more than $K$ of the $V_i$'s.
\item For each $i$ there exists a special Lipschitz domain $W_i$ with constant $M$ such that $\cU\cap V_i=W_i\cap V_i$.
\end{itemize}
\end{defi}
The preceding definition includes all  bounded Lipschitz domains and certain unbounded Lipschitz domains.
The unbounded fluid domain considered in this paper is either the half-space $\{(x,y)\in\mathbb{R}^d\times \mathbb{R}: y<f(x)\}$ or the  strip $\{(x,y)\in\mathbb{R}^d\times \mathbb{R}: b(x)<y<f(x)\}$, where  $b$ and $f$ are Lipschitz  functions such that $h:=\inf_{x\in \Rr^d}(\eta(x')-b(x'))>0$. In the former case, we can choose $(\iota, K, M)=(1, 1, \| \na f\|_{L^\infty(\Rr^d)})$. In the latter case, one can take $(\iota, K, M)=\left(\frac{h}{4} , 2,  \max\{\|\nabla b\|_{L^{\infty}}, \|\nabla f\|_{L^{\infty}}\}\right)$. 

Each minimally smooth domain admits a universal extension operator for Sobolev spaces:
\begin{theo}[Extension of Sobolev functions, \cite{Stein}, Theorem 5, p.181]\label{th.extension} Let $\cU \subset \mathbb{R}^N$ be a minimally smooth domain with contants $(\iota, K, M)$. There exists an operator $\mathfrak{C}$ 
mapping functions on $\cU$ to functions on $\mathbb{R}^N$, such that $\mathfrak{C}(u)\vert_{\cU}=u$ and $\mathfrak{C}$ is continuous from $W^{\sigma,p}(\cU)$ to  $W^{\sigma,p}(\mathbb{R}^N)$ for all $(\sigma, p)\in [0, \infty)\times [1, \infty]$.  
\end{theo}

In~\cite{Stein}, \cref{th.extension} is proven for $\sigma\in \Nn$. The case $\sigma\notin \Nn$ follows by linear interpolation  and the definition \eqref{def:Wsp}. 
By virtue of \cref{th.extension}, many of the known product rules in the whole space  remain valid for minimally smooth domains. We state some of them below. 
\begin{theo}(Product estimates in Sobolev spaces)\label{thm.sobolevProducts} Let $\cU \subset \mathbb{R}^N$ be a minimally smooth domain. Let $s\geqslant 0$, $p_1,q_1 \in (1,\infty]$, and $p, p_2,q_2 \in (1,\infty)$ such that $\frac{1}{p} = \frac{1}{p_1}+\frac{1}{p_2} = \frac{1}{q_1} + \frac{1}{q_2}$. If
\[
s\in\Nn, \quad\text{or}~ p=2,\quad\text{or}~q_1=q_2=\infty,
\]
then  there exists $C: \Rr_+\to \Rr_+$ such that for any $u \in W^{s,p_2}(\mathcal{U})$ and $v \in W^{s,q_2}(\mathcal{U})$, there holds 
\begin{equation}
    \label{eq.productSobol}
    \|uv\|_{W^{s, p}(\cU)} \le C(s, p_1, p_2, q_1, q_2, N, \iota, K, M)\left(\|u\|_{L^{p_1}(\cU)}\|v\|_{W^{s,p_2}(\cU)} + \|v\|_{L^{q_1}(\cU)}\|u\|_{W^{s,q_2}(\cU)}\right).
\end{equation}
\end{theo}
\begin{proof} 
By virtue of \cref{th.extension}, it suffices to consider the case $\cU=\mathbb{R}^N$. Let $s\geqslant 0$, $p_1,q_1 \in (1,\infty]$, and $p, p_2,q_2 \in (1,\infty)$ such that $\frac{1}{p} = \frac{1}{p_1}+\frac{1}{p_2} = \frac{1}{q_1} + \frac{1}{q_2}$.  We recall from \cite[Chapter 2, Proposition 1.1]{TaylorToolsPDE} that 
\bq\label{product:Fs}
\|uv\|_{H^{s, p}(\Rr^N)} \le C\left(\|u\|_{L^{p_1}(\Rr^N))}\|v\|_{H^{s,p_2}(\Rr^N)} + \|v\|_{L^{q_1}(\Rr^N))}\|u\|_{H^{s,q_2}(\Rr^N)}\right),
\eq
where 
\[
\| u\|_{H^{s, p}(\Rr^N)}:=\| \| \{2^{sj}\Delta_j u\}\|_{\ell^2}\|_{L^p(\Rr^N)},
\]
$u=\sum_{j=-1}^\infty \Delta_ju$ being a Littlewood–Paley decomposition of $u$. 

If $s\in \Nn$, then $H^{s, p}(\Rr^N)=W^{s, p}(\Rr^N)$ for all $p\in (1, \infty)$, and hence \eqref{eq.productSobol} follows from \eqref{product:Fs}.  See \cite[Section 2.5.6]{MR781540}.

Next, we consider  $s\notin \Nn$, so that $ \| u\|_{W^{s, r}(\Rr^N)}$  is equivalent to 
\[
 \| u\|_{B^s_{r, r}(\Rr^N)}:=\| \| \{2^{sj}\Delta_j u\}\|_{\ell^r}\|_{L^r(\Rr^N)}.
 \]
 Consequently, if $r\ge 2$, then 
\[
\| u\|_{H^{s, r}(\Rr^N)}\les \| u\|_{B^s_{r, r}(\Rr^N)} \les  \| u\|_{W^{s, r}(\Rr^N)}.
\]
Thus, when $p=2$,  \eqref{eq.productSobol} again follows from \eqref{product:Fs}. Finally, the case  $p_1=q_1=\infty$ is proven in \cite[Corollary 2.86]{BCD}. 
\end{proof}
By combining the product rule \eqref{pr} with \cref{th.extension}, we obtain the following. 
\begin{coro}\label{coro:productdomain}
Let $\cU \subset \mathbb{R}^N$ be a minimally smooth  domain. Consider $s_0$, $s_1$, $s_2$ nonnegative numbers satisfying $s_1+s_2>0$, $s_0\le \min\{s_1, s_2\}$, and $s_0<s_1+s_2- \frac{N}{2}$.  There exists $C: \Rr_+\to \Rr_+$ such that 
\bq\label{pr:domnain}
\forall (u, v)\in H^{s_1}(\cU)\times H^{s_2}(\cU),\quad \| uv\|_{H^{s_0}(\cU)}\le C(s_0, s_1, s_2, N, \iota, K, M)\| u\|_{H^{s_1}(\cU)}\| v\|_{H^{s_2}(\cU)}.
\eq
\end{coro}

Finally, we will also need the following Gagliardo-Nirenberg inequalities  for minimally smooth domains. 

\begin{theo}[Gagliardo--Nirenberg]\label{thm.gagliardoNirenberg}Let $\cU \subset \mathbb{R}^N$ be a minimally smooth domain with constants $(K, M)$. Let $\sigma > s >0$ and $p, r \in (1,\infty)$ such that there exists $\theta \in (0,1)$ satisfying
\begin{equation}
    \frac{1}{p}=\frac{\theta}{r}\quad \text{and}\quad s=\theta \sigma.
\end{equation}    
Then, for any $u \in L^{\infty}(\cU) \cap W^{\sigma,r}(\cU)$ there holds 
\begin{equation}
    \label{eq.variantGN}
    \|u\|_{W^{s,p}(\cU)} \le C(s, \sigma, p, r, N, \iota, K, M)\|u\|_{L^{\infty}(\cU)}^{1-\theta}\|u\|_{W^{\sigma,r}(\cU)}^{\theta}.
\end{equation}
\end{theo}

\begin{proof} By virtue of the extension operator in \cref{th.extension}, it suffices to consider the case $\cU = \mathbb{R}^d$, which is the content of \cite{BM2018}. 
\end{proof}

Another consequence of the existence of a continuous extension operator on $W^{s,p}(\Omega)$ for minimally smooth domains $\Omega$ is the equivalent norm characterization given by \eqref{eq.equivalent-norm-sobolev}. More precisely, we have 
\begin{prop}\label{lem.minimally-smooth-norm} Let $\Omega \subset \mathbb{R}^N$ be a minimally smooth domain with constants $(\iota, K,M)$. Let $s=m+\mu$ with $m\in \mathbb{N}$ and $\mu \in (0,1)$, and let $p\in [1,\infty]$. Then the norm $\|\cdot \|_{[W^{s,p}(\Omega)]}$ on $W^{s,p}(\Omega) = \left(W^{m,p}(\Omega),W^{m+1,p}(\Omega)\right)_{\mu,p}$ is equivalent to the norm given by 
\bq\label{equivalentnorm}
    \|u\|_{W^{s,p}(\Omega)} = \|u\|_{W^{m,p}(\Omega)} + \sum_{|\alpha |=m} |\partial^{\alpha}u|_{W^{\mu,p}(\Omega)},
\eq
where 
\[
|v|_{W^{\mu, p}(\Omega)}:=
    \begin{cases} 
        &\left(\iint_{\Omega \times \Omega} \frac{|v(x)-v(y)|^p}{|x-y|^{p\mu+N}} \,\mathrm{d}x\mathrm{d}y\right)^\frac{1}{p}\quad\text{if~} p<\infty,\\
        &\operatorname{essup}_{x, y\in \Omega, x\ne y}\frac{|v(x)-v(y)}{|x-y|^\mu}\quad\text{if~} p=\infty.
    \end{cases}
\]
\end{prop}
When $\Omega$ is $\Rr^N$ or a special Lipschitz domain or  a bounded Lipschitz domain, the norm equivalence in  \cref{lem.minimally-smooth-norm}  is proven in \cite[Chapters 35,  36]{tartar07}.  In order to prove it for minimally smooth domains,  we will make use of another  equivalent norm given in the following proposition.  
\begin{prop}\label{equivalence-norms-infimum} 
Let $\Omega \subset \mathbb{R}^N$ be a minimally smooth domain, and let  $s\ge 0$ and $p\in[1,\infty]$. Then the norm $\|\cdot \|_{[W^{s,p}(\Omega)]}$ is equivalent to the norm given by
\bq\label{normWsp:inf}
    \inf\left\{\|U\|_{W^{s,p}(\mathbb{R}^N)}: U\in W^{s, p}(\Rr^N),~U\vert_{\Omega} = u\right\}.
\eq
\end{prop}
\begin{proof} Write $s=m+\mu$ for some integer $m$ and $\mu \in [0,1)$. We follow the proof in \cite[Chapter 34]{tartar07}. Since the restriction operator $U \mapsto U\vert_{\Omega}$ is continuous as a map $W^{m,p}(\mathbb{R}^N) \to W^{m,p}(\Omega)$ and $W^{m+1,p}(\mathbb{R}^N) \to W^{m+1,p}(\Omega)$, it is then continuous 
\[
    W^{s,p}(\mathbb{R}^N) = (W^{m,p}(\mathbb{R}^N),W^{m+1,p}(\mathbb{R}^N))_{\mu,p} \longrightarrow (W^{m,p}(\Omega),W^{m+1,p}(\Omega))_{\mu,p} =W^{s,p}(\Omega).
\]
Consequently the norm $[W^{s,p}(\Omega)]$ is controlled by the norm \eqref{normWsp:inf}.   The reverse inequality can be proven similarly with the use of the continuous  extension operator $\mathfrak{C} : W^{s,p}(\Omega) \to W^{s,p}(\mathbb{R}^N)$ since $\Omega$ is minimally smooth. 
\end{proof}

We can now proceed to the proof of the equivalence norm. 

\begin{proof}[Proof of \cref{lem.minimally-smooth-norm}] We only provide the proof for the case $p< \infty$; the case $p=\infty$ can be obtained by obvious modifications. In the following we write $C$ for any constant which depends only on the parameters $(N,\iota, K, M,p,\mu)$. We recall that in the case $\Omega = \mathbb{R}^N$ or special Lipschitz domains, such a norm equivalence is well-known and we refer to \cite[Chapters 35, 36]{tartar07} for the details. Here we consider a minimally smooth domain $\Omega$ with constant $(\iota, K, M)$, open sets $\{V_i\}$ and special Lipschitz domains $\{W_i\}$ as in \cref{defi:mmdomain}. For simplicity we shall consider the case $s=\mu\in (0, 1)$, i.e.
\[
    \|u\|_{[W^{\mu,p}(\Omega)]} \sim \|u\|_{L^p(\Omega)} + |u|_{W^{\mu,p}(\Omega)} =: \|u\|_{W^{\mu,p}(\Omega)}.  
\]
We have
\[
    \|u\|_{W^{\mu,p}(\Omega)}= \|\mathfrak{C}f\|_{W^{\mu,p}(\Omega)} \leq \|\mathfrak{C}f\|_{W^{\mu,p}(\mathbb{R}^N)} \leq C \|\mathfrak{C}f\|_{[W^{\mu,p}(\mathbb{R}^N)]} \leq C \|f\|_{[W^{\mu,p}(\Omega)]}, 
\]
where we have used the norm equivalence on $W^{\mu,p}(\mathbb{R}^N)$ and the continuity of the extension operator $\mathfrak{C}$.  For the converse we will utilize the precise definition of the extension operator $\mathfrak{C}$  in  \cite[Chapter VI.3.3.1]{Stein}: there exist smooth functions $\{\lambda_j\}_{j\in \mathbb{N}}$ and  $\Lambda_{\pm}$ taking values in $[0, 1]$ such that
\begin{itemize}
    \item $\forall j\in \Nn,~\operatorname{supp}(\lambda_j)\subset V_j$;
    \item $\operatorname{supp}(\Lambda_-)\subset\Omega?$,  $\operatorname{supp}(\Lambda_+)\subset \{x \in \Rr^N: \operatorname{dist}(x,\p\Omega)<\iota\}$,
     $\Lambda_++\Lambda_-=1$ on $\overline{\Omega}$;
       \item $\forall x\in \operatorname{supp}(\Lambda_+),~\exists j\in \Nn,~\ld_j(x)=1$;    
    \item $\forall m\in \Nn,~\exists C_m>0, \sup_{j\in \Nn}\|\lambda_j\|_{W^{m,\infty}}\leq C_m$;
    \item No point in $\Rr^N$ is contained in more than $K$ of the $\supp \ld_i$'s.
\end{itemize}
Then the extension  $U \in W^{\mu,p}(\Rr^N)$ of  $u\in W^{\mu,p}(\Omega)$ is  given by 
\begin{equation}\label{stein.extension}
    U(x):= \Lambda_+(x) \frac{\sum_{i\in \Nn}\lambda_i(x)\mathfrak{C}_i(\lambda_iu)(x)}{\sum_{j\in \Nn} \lambda_j^2(x)} + \Lambda_-(x)u(x), 
\end{equation}
where $\mathfrak{C}_i : W^{\mu, p}(W_i) \to W^{\mu,p}(\mathbb{R})$ are the continuous extension operators from \cref{th.extension} in the case of special Lipschitz domains $W_i$; the norms of $\mathfrak{C_i}$ are uniformly bounded.  Here the product $\ld_i u$ is defined on $\Omega\cap V_i=W_i\cap V_i\subset W_i$. 

Let us make an important observation: if $\phi\in W^{1, \infty}(\Rr^N)$ and $v\in W^{\mu,p}(\mathbb{R}^N)$, then 
\[
    \|\phi v\|_{W^{\mu,p}(\mathbb{R}^N)} \leq C\|\phi\|_{W^{1,{\infty}}}\|v\|_{W^{\mu,p}(\mathbb{R}^N)}. 
\]
To see this, fix $(x,y)\in \mathbb{R}^N \times \mathbb{R}^N$ and observe 
\begin{multline*}
    |\phi(x)v(x)-\phi(y)v(y)|^p \leq C |\phi(x)-\phi(y)|^p|v(y)|^p+C |\phi(x)|^p|(v(x)-v(y))|^p \\
    \leq C\|\phi\|^p_{W^{1,\infty}(\Rr^N)} \left(\min\{1,|x-y|^p\}|v(y)|^p + |v(x)-v(y)|^p\right). 
\end{multline*}
The claim then follows from  the fact that the function $z \mapsto \frac{\min\{1,|z|^{p}\}}{|z|^{N+p\mu}}$ is integrable on $\mathbb{R}^N$. Moreover, if $\operatorname{supp}(\phi) \subset \Omega$, then $|\phi v|_{W^{\mu,p}(\mathbb{R}^N)} \leq C\|\phi\|_{W^{1,{\infty}}}\|v\|_{W^{\mu,p}(\Omega)}$. To see this it suffices to observe that if  $x\notin \operatorname{supp}(\phi)$ and $y\in \operatorname{supp}(\phi)$ then 
\[
    |\phi(x)v(x) - \phi(y)v(y)| = |\phi(y)v(y)| = |\phi(x)-\phi(y)||v(y)| \leq 2\| \phi\|_{W^{1, \infty}(\Rr^N)} \min\{1,|x-y|\}|v(y)|.
\]
We are now ready for the proof.  Applying  \cref{equivalence-norms-infimum} with the extension $U$ given by  \eqref{stein.extension}, we have
\[
    \|u\|_{[W^{\mu,p}(\Omega)]}^p \leq   \|U\|_{[W^{\mu,p}(\Omega)]}^p\le  C\left\|\Lambda_+(x) \frac{\sum_{i\in \Nn}\lambda_i(x)\mathfrak{C}_i(\lambda_iu)(x)}{\sum_{j\in \Nn } \Lambda_j^2(x)} + \Lambda_-(x)u(x)\right\|^p_{W^{\mu,p}(\mathbb{R}^N)},
\]
 Since  and $\Lambda_{-}\in W^{1, \infty}(\Rr^N)$ and $\operatorname{supp}(\Lambda_-) \subset \Omega$, we infer 
\bq\label{enormsum:1}
    \|u\|_{[W^{\mu,p}(\Omega)]}^p \leq C\left\|\sum_{i\in \Nn}\theta_i\mathfrak{C}_i(\lambda_iu)\right\|^p_{W^{\mu,p}(\mathbb{R}^N)} + C\|u\|_{W^{\mu,p}(\Omega)}, 
\eq
where $\theta_i(x)= \frac{\Lambda_+(x)\lambda_i(x)}{\sum_{j\in \Nn} \lambda_j^2(x)}$. We claim that
\bq\label{claim:enormsum}
    \left\|\sum_{i\in \Nn}\theta_i\mathfrak{C}_i(\lambda_iu)\right\|^p_{W^{\mu,p}(\mathbb{R}^N)} \leq C \sum_{i\in \Nn}\|\mathfrak{C}_i(\lambda_iu)\|_{W^{\mu,p}(\mathbb{R}^N)}^p + C\|u\|_{L^p(\Omega)}^p. 
\eq
The control of the  $L^p$ norm is a consequence of the fact that no point in $\Rr^N$ is contained in more than $K$ of the $\supp \tt_i$'s. To estimate the seminorm we write 
\begin{multline*}
    \left\vert\sum_{i\in \Nn}\theta_i(x)\mathfrak{C}_i(\lambda_iu)(x) - \theta_i(y)\mathfrak{C}_i(\lambda_iu)(y)\right\vert^p \leq C\left\vert\sum_{i\in \Nn}(\theta_i(x)-\theta_i(y))\mathfrak{C}_i(\lambda_iu)(y)\right\vert^p \\
    + C\left\vert\sum_{i\in \Nn}\theta_i(x)(\mathfrak{C}_i(\lambda_iu)(x) -\mathfrak{C}_i(\lambda_iu)(y))\right\vert^p.
\end{multline*}
On the preceding right-hand side, the first summation contains no more than $2K$ nonzero terms, while the second summation contains no more than $K$ nonzero terms. Using this and the uniform bound for $\| \tt_i\|_{W^{1, \infty}(\Rr^N)}$, we obtain
\bq\label{enormsum:3}
\begin{aligned}
    \left\vert\sum_{i\in \Nn}\theta_i(x)\mathfrak{C}_i(\lambda_iu)(x) - \theta_i(y)\mathfrak{C}_i(\lambda_iu)(y)\right\vert^p \leq C\sum_{i\in \Nn}\min\{1,|x-y|^p\}|\mathfrak{C}_i(\lambda_if)(y)|^p + \\
    +C\sum_{i\in \Nn} |\mathfrak{C}_i(\lambda_iu)(x) -\mathfrak{C}_i(\lambda_iu)(y)|^p.
\end{aligned}
\eq
After integration we conclude the proof of \eqref{claim:enormsum}.

For the special Lipschitz domain $W_i$, $\mathfrak{C}_i$ is defined by reflection and it was  proven  \cite[Lemma 36.1]{tartar07} that 
\bq\label{enormsum:2}
\|\mathfrak{C}_i(\lambda_iu)\|_{W^{\mu,p}(\mathbb{R}^N)} \leq C \|\lambda_iu\|_{W^{\mu,p}(\Omega)}.
\eq
It follows from \eqref{enormsum:1}, \eqref{claim:enormsum}, and \eqref{enormsum:2} that
\begin{align*}
    \|u\|_{[W^{\mu,p}(\Omega)]}^p & \leq C\|u\|_{L^p(\Omega)}^p + C \sum_{i\in \Nn} \|\lambda_iu\|_{W^{\mu,p}(\Omega)}^p \\
   & \leq C\|u\|^p_{L^p(\Omega)} + C\iint_{\Omega \times \Omega}\sum_{i\in \Nn} |\lambda_i(x)u(x)-\lambda_i(y)u(y)|^p\frac{\mathrm{d}x\mathrm{d}y}{|x-y|^{N+p\mu}}.  
\end{align*}
Then arguing as in the proof of \eqref{enormsum:3} we conclude that $ \|u\|_{[W^{\mu,p}(\Omega)]}^p \le  C \|u\|_{W^{\mu,p}(\Omega)}^p$.
\end{proof}

\section{Interpolation results}\label{sec.interpolation}

\subsection{Interpolation inequalities} 

First, we recall the following lemma. 

\begin{lemm}[Théorème 3.1, \cite{LM}]\label{lemm:interpolationLM} Let $I$ be an open interval and $s\in \Rr$. If $u(x, z): \Rr^d\times I\to \Rr$ satisfies $u\in L^2_z(I; H^{s+\mez}(\Rr^d))$ and $\p_zu\in L^2_z(I; H^{s-\mez}(\Rr^d))$, then $u\in C(\overline{I}; H^s(\Rr^d))$ and there exists an absolute constant $C>0$  such that 
\bq\label{interpolation:LM}
\| u\|_{C(\overline{I}; H^s)}\le C\left(\| u\|_{L^2(I; H^{s+\mez})}+ \| \p_zu\|_{L^2(I; H^{s-\mez})}\right).
\eq
\end{lemm}

We will also need to use the following interpolation inequality used in~\cite[Proposition 4.1]{WZZZ}, which we prove. 

\begin{lemm}\label{lemm.interpol} Consider either $J=(-\infty, 0)$ or $J=(-1, 0)$. Let $\mu \in [0,1)$ and $\psi(x, z)\in H^\mu(\Rr^d\times J)$. Then there exists $C(\mu, d)>0$ such that
 \begin{equation}
    \label{eq.interpolationIneq}
    \|\psi\|_{H^{\mu}(\Rr^d\times J)} \le C(d,\mu)\left(\|\psi\|_{L^2(J; H^{\mu}(\Rr^d))} + \|\partial_z \psi\|_{L^2(J; H^{\mu-1}(\Rr^d))}\right).
\end{equation}
\end{lemm}

\begin{proof} If $\psi\in H^\mu(\mathbb{R}^d\times \Rr)$, by using the Fourier transform $\hat\psi(\xi,\eta)$, \eqref{eq.interpolationIneq} follows from the following inequality: that for any $\eta \in \Rr$ and $\xi \in \Rr^d$, 
\bq\label{ineq:etaxi}
    (1+\eta^2 + |\xi|^2)^{\mu} \lesssim (1+|\xi|^2)^{\mu} + |\eta|^2(1+|\xi|^2)^{\mu-1},\quad \mu\in (0, 1].
\eq
The left-hand side of \eqref{ineq:etaxi} is indeed controlled by the first term on the right-hand side when $\eta^2\le 1+|\xi|^2$ and by the second term otherwise. 

For the half-space $\Rr^d\times \Rr_-$, we extend $\psi$ to $\tilde{\psi}: \Rr^d\times \Rr\to \Rr$ by even reflection. By using the double-integral characterization of $H^\mu$ given by \eqref{eq.equivalent-norm-sobolev}, it can be checked that 
\[
\| \tilde{w}\|_{H^\mu(\Rr^d\times \Rr)}\le 3\| w\|_{H^\mu(\Rr^d\times \Rr_-)}.
\]
Moreover, we have $\|\p_z\tilde{w}\|_{L^2(\Rr, \Rr^d)}\le 2\| w\|_{L^2(\Rr_-; H^{\mu-1})}$. The inequality \eqref{eq.interpolationIneq} for the half-space then follows from that for the whole space. 

Next, we consider the strip $\Rr^d\times (-1, 0)$. We extend $w$  to $\tilde{w}: \Rr\times (-2, 0)\to \Rr$ defined  by 
\[
\tilde{w}(x, z)=\begin{cases}
w(x, z)\quad\text{if~} z\in (-1, 0),\\
w(x, -z-2)\quad\text{if~} z\in (-2, 0).
\end{cases}
\]
It can be checked that 
\[
\| \tilde{w}\|_{H^\mu(\Rr^d\times (-2, 0))}\les \| w\|_{H^\mu(\Rr^d\times (-1, 0))},\quad \|\p_z\tilde{w}\|_{L^2((-2, 0), \Rr^d)}\les \| w\|_{L^2((-1, 0); H^{\mu-1})}.
\]
Let $\chi(z)$ be a cut-off function satisfying $\chi \equiv 1$ on  $(-\tdm, \mez)$. The preceding inequalities hold for $\tilde{w}\chi(z)$ in place of $\tilde{w}$ with $\Rr^d\times (-2, 0)$ replaced by $\Rr^d\times \Rr_-$. Thus we can conclude by using the half-space case and the fact that $w\equiv \tilde{w}\chi(z)$ on $\Rr^d\times (-1, 0)$. 
\end{proof}

\subsection{Interpolation of linear operators}\label{sec.interpol-def}  

We recall the following interpolation result.

\begin{lemm}[\cite{BL}, Theorem 4.1.2, and \cite{BS}, Theorem 1.12]\label{lemm.interpolation} Let $X_0, X_1, Y_0, Y_1$ be Banach spaces and $T : X_0 + X_1 \mapsto Y_0 + Y_1$ be a linear operator bounded as a map $X_i \to Y_i$ with norm $C_i$. Let $\theta \in (0,1)$ and let $X_{\theta}=(X_0,X_1)_{\theta,q}$, the interpolation  space by the real method, or $X_{\theta}=[X_0,X_1]_{\theta}$, the interpolation space by the complex method, and similarly for $Y_{\theta}$. Then $T : X_{\theta} \to Y_{\theta}$ is continuous and $\|T\|_{X_{\theta} \to Y_{\theta}} \lesssim C_0^{1-\theta}C_1^{\theta}$. 
\end{lemm}
We will also need to interpolate Banach spaces intersected with subspaces. In general it is a difficult task to characterize these interpolated spaces. See \cite[Section 1.17.1]{Triebel}.  In this paper, we only need the following special case.
\begin{lemm}\label{lemm.interpol-def} Let $(X_0, X_1)$ be an interpolation couple of Banach spaces. Let $\theta \in (0,1)$ and let $X_{\theta}$ denote the real interpolation space $(X_0,X_1)_{\theta,2}$. 
Let us also assume that there exists a linear map $T : X_0 + X_1 \to X_0 + X_1$ such that $T^2=T$ on $X_0 + X_1$ and $T$ acts continuously as maps $X_0 \to X_0$ and $X_1\to X_1$. Let $F=\ker(T-\operatorname{id})$. Then $(X_0 \cap F, X_1 \cap F)_{\theta, 2} = X_{\theta} \cap F$.   
\end{lemm}

\begin{proof}
Let us show the equivalence of the interpolation norms $\|\cdot \|_{\theta, F}$ on $(X_0 \cap F, X_1 \cap F)_{\theta, 2}$ and $\|\cdot \|_{\theta}$ on $X_{\theta} \cap F$. 
Let $u \in (X_0 \cap F, X_1 \cap F)_{\theta, 2}$, and recall that by definition of the interpolation norm we have 
\[
    \|u\|_{\theta}^2 = \int_0^{\infty}t^{-(2\theta+1)}K(t,u)^2\,\mathrm{d}t,    
\]
where 
\begin{align*}
    K(t,u)&:=\inf\{\|u_0\|_{X_0} + t\|u_1\|_{X_1}, u=u_0+u_1, (u_0,u_1)\in X_0 \times X_1\}\\
    &\le \inf\{\|u_0\|_{X_0} + t\|u_1\|_{X_1}, u=u_0+u_1, (u_0,u_1)\in (X_0\cap F) \times (X_1\cap F)\} \\
    & =: K_F(t,u)
\end{align*}
so that $(X_0 \cap F, X_1 \cap F)_{\theta, 2} \hookrightarrow X_{\theta} \cap F$.

Conversely, let  $u \in X_{\theta}\cap F$ and for all $t>0$ choose some $u_0(t) \in X_0$ and $u_1(t)\in X_1$ such that $u=u_0(t) + u_1(t)$ and $\|u_0(t)\|_{X_0}+t\|u_1(t)\|_{X_1} \le 2 K(t,u)$, which is possible by definition of $K$.  
Since $u\in F$ there holds $u=T(u)=T(u_0(t)) + T(u_1(t))$ with $\|T(u_i(t))\|_{X_i} \le C \|u_i(t)\|_{X_i}$. Since $T^2=T$ on $X_0\cup X_1$, $Tv\in F$ for all $v\in X_0\cup X_1$.  It follows that $u=T(u_0(t))+T(u_1(t)) \in (X_0 \cap F) + (X_1\cap F)$ and 
\[
    K_F(t,u) \le \|T(u_0(t))\|_{X_0} + t\|T(u_1(t))\|_{X_1} \le C(\|u_0(t)\|_{X_0} + t\|u_1(t)\|_{X_1}) \le 2C K(t,u),
\]
from which the converse embedding follows.
\end{proof}

\section{Proof of \cref{lemm.FaaDiBruno} and \cref{lemm.diffeoSobolev}}\label{sec.compositionSobolev} 

The proof of \cref{lemm.FaaDiBruno} and \cref{lemm.diffeoSobolev} relies on the following version of the Fa\`a di Bruno formula. 

Let $\alpha \in \mathbb{N}^{N}$. We say that $(k,\beta,\gamma) \in \mathcal{D}(\alpha)$ if there exists $\gamma_1, \dots, \gamma_k \in \mathbb{N}^{N}$ and $\beta_1, \dots \beta_k \in \mathbb{N}^{N}$ such that $\alpha = |\beta_1|\gamma_1 + \cdots |\beta_k|\gamma_k$ and that all the $\gamma_i$ are distincs. We write $\beta =\beta_1 + \dots +\beta_k$. 

The Fa\`a di Bruno formula that we shall use reads as follows.  

\begin{lemm}[\cite{Ma09}] Let $N\geqslant 1$, and let also $\varphi : \mathbb{R}^N \to \mathbb{R}^N$ and $F : \mathbb{R}^N \to \mathbb{R}$ be two functions being $n$ times differentiable. Then for any multi-index $\alpha =(\alpha_1, \dots, \alpha_N)$ such that $|\alpha|\leq n$ there holds 
\begin{equation}
    \label{eq.faaDiBruno}
    \partial^{\alpha}(F \circ \varphi)= \alpha ! \sum_{(k,\beta,\gamma)\in\mathcal{D}(\alpha)}\partial^{\beta}F \circ \varphi \prod_{j=1}^k \frac{1}{\beta_j!}\left(\frac{1}{\gamma_j!}\partial^{\gamma_j} \varphi \right)^{\beta_j}\,,
\end{equation}
where we recall that $\prod_{j=1}^k \frac{1}{\beta_j!}\left(\frac{1}{\gamma_j!}\partial^{\gamma_j} \varphi \right)^{\beta_j}$ is a shorthand for 
\[
    \prod_{j=1}^k\prod_{i=1}^N \frac{1}{\beta_{j,i}!}\left(\frac{1}{\gamma_{j,1}! \cdots \gamma_{j,N}!}\frac{\p^{\gamma_{j,1} + \cdots + \gamma_{j,N}}\varphi_i}{\p x_{\gamma_{j,1}} \cdots \p x_{\gamma{j,N}}}  \right)^{\beta_{j,i}}.    
\]
\end{lemm}

When we use~\eqref{eq.faaDiBruno}, we will only do product estimates in Sobolev or Lebesgue spaces, therefore we will only need to keep track of the number of derivatives taken on $F$ and $\varphi$. Also, we can discard the constant appearing in the formula so that we will abusively write~\eqref{eq.faaDiBruno} as
\begin{equation}
    \label{eq.faaDiBruno2}
    \partial^{n}(F \circ \varphi) = \sum_{\beta_1 + 2\beta_2 + \dots + n \beta_n= n}c_{n,\beta}\left(\partial^{\beta_1 + \cdots + \beta_n}F \circ \varphi \right)\prod_{j=1}^n (\partial^{j} \varphi)^{\beta_j}\,.
\end{equation}
Note that in order to obtain~\eqref{eq.faaDiBruno2} we regrouped derivatives on $\varphi$ by their number. The notation $(\partial^j \varphi)^{\beta_j}$ therefore stands for \textit{any} possible product of $\beta_j$ derivatives of order $j$ of the vector-valued function $\varphi$, for example $\partial_1 \varphi_2 \partial_2 \varphi_1$ is a $(\partial^1 \varphi)^2$ term, therefore the $\beta_{j}$ (which are multi-indicies) can be taken as integers as we regroup all $(\partial^j \varphi)^{\beta_j}$ terms (this changes the constants $c_{n,\beta}$, which are not of importance in our analysis). The summation condition on the $\beta_j \in \mathbb{N}$ is therefore $\sum_{j=1}^n j\beta_j = n$.
\subsection{Proof of \cref{lemm.FaaDiBruno}}We split the proof between integer and non-integer cases. 
\medskip 

\textit{Step 1: Integer $\sigma = n \geqslant 0$}. For $n=0$,  we make the change of variables $y=\varphi (x)$ and use the fact that $|\det(\na \varphi^{-1})|=|\det \na \varphi|^{-1}\le c_0^{-1}$ to have
\bq\label{composition:L2}
\begin{aligned}
  \forall p\in [1, \infty),~  \|F \circ \varphi \|_{L^p(\mathcal{U})}^p &= \int_{\mathcal{U}} |F(\varphi(x))|^p \,\mathrm{d}x \\
    & = \int_{\Omega} |F(y)|^p |\det (\nabla \varphi^{-1})(y)| \,\mathrm{d}y \le c_0^{-1}\|F\|^p_{L^p(\Omega)}.
\end{aligned}
\eq

For $n\ge 1$, we shall estimate the  derivative $\partial^n(F \circ \varphi)$ in $L^2(\mathcal{U})$.  We use the Fa\`a-di Bruno formula \eqref{eq.faaDiBruno2}, H\"older's inequality, and \eqref{composition:L2} to have
\begin{align*}
    \|\partial ^n (F \circ \varphi)\|_{L^2(\mathcal{U})} &\lesssim \sum_{\substack{\beta_1 + 2\beta_2 + \cdots +n\beta_n=n \\\beta_ 1 + \cdots + \beta_n = \beta}} \left\|(\partial^{\beta}F \circ \varphi) \prod_{j=1}^n (\partial^j \varphi)^{\beta_j}\right\|_{L^{2}(\mathcal{U})} \\ 
    & \lesssim \sum_{\substack{\beta_1 + 2\beta_2 + \cdots +n\beta_n = n \\\beta_ 1 + \cdots + \beta_n = \beta}} \|\partial^{\beta}F \circ \varphi\|_{L^p(\mathcal{U})} \left\|\prod_{j=1}^n (\partial^j \varphi)^{\beta_j}\right\|_{L^q(\mathcal{U})} \\ 
    & \le C\sum_{\substack{\beta_1 + 2\beta_2 + \cdots +n\beta_n =n\\\beta_ 1 + \cdots + \beta_n = \beta}} \|\partial^{\beta}F\|_{L^p(\Omega)} \prod_{j=1}^n \|\partial^j \varphi\|^{\beta_j}_{L^{q_j\beta_j}(\mathcal{U})},
\end{align*}
where the to-be-chosen exponents $p$, $q$, and $q_j$ satisfy 
\[
    \frac{1}{p} + \frac{1}{q} = \frac{1}{p} + \sum_{j=1}^n \frac{1}{q_j} = \frac{1}{2},
\]
and $C$ depends only on $(c_0, m, N)$.     We will estimate the right-hand side using the Gagliardo-Nirenberg inequality in bounded domains, which is recalled in \cref{thm.gagliardoNirenberg}. 
We choose $q_1=\infty$. For  $j\ge 2$, we choose $\frac{1}{q_j\beta_j}=\frac{\tt_j}{2}=\frac{j-1}{2(n-1)}$ and note that $\p^j\varphi=\p^j\psi$, so that 
\[
    \|\partial ^j \varphi\|_{L^{q_j\beta_j}(\mathcal{U})}= \|\partial ^j \psi\|_{L^{q_j\beta_j}(\mathcal{U})}\lesssim \|\p\psi\|_{H^{n-1}(\mathcal{U})}^{\theta_j} \|\na \psi\|_{L^{\infty}(\mathcal{U})}^{1-\theta_j}\lesssim  \|\psi\|_{H^{n}(\mathcal{U})}^{\theta_j} (1+\|\na \varphi\|_{L^{\infty}(\mathcal{U})})^{1-\theta_j}.
\]
This implies 
\[
\prod_{j=2}^n \|\partial^j \varphi\|^{\beta_j}_{L^{q_j\beta_j}(\mathcal{U})}\le C(\|\nabla \varphi\|_{L^{\infty}(\mathcal{U})})\| \psi\|_{H^n(\Omega)}^\nu,
\]
where
\[
\nu=\sum_{j=2}^n \tt_j\beta_j=\sum_{j=2}^n\frac{\tt_j(j-1)}{(n-1)}=\frac{(n-\beta_1)-(\beta-\beta_1)}{n-1}=\frac{n-\beta}{n-1}\in [0, 1].
\]
On the other hand, we have
\[
\frac{1}{p}=\mez-\sum_{j=1}^n \frac{1}{q_j} =\mez -\sum_{j=2}^n \frac{\tt_j\beta_j}{2}=\frac{1-\nu}{2}=\frac{\beta-1}{2(n-1)}.
\]
Using Gagliardo-Nirenberg's inequality  again gives
\[
    \|\partial^{\beta}F\|_{L^p(\Omega)} \lesssim \|\nabla F\|_{L^{\infty}(\Omega)}^{1-\theta_0}\|F\|_{H^n(\Omega)}^{\theta_0}\,,    
\]
where  $\displaystyle\frac{1}{p} = \frac{\tt_0}{2}=\frac{\beta-1}{2(n-1)}$, so $\tt_0=\frac{2}{p}=1-\nu$. 

Putting the above estimates together, we find 
\begin{align*}
    \|\partial^{\beta}F\|_{L^p(\Omega)} \prod_{j=1}^n \|\partial^j \varphi\|^{\beta_j}_{L^{q_j\beta_j}(\mathcal{U})}&= \|\partial^{\beta}F\|_{L^p(\Omega)}\|\p \varphi\|_{L^\infty}^{\beta_1} \prod_{j=2}^n \|\partial^j \varphi\|^{\beta_j}_{L^{q_j\beta_j}(\mathcal{U})} \\
    &\le C(\|\nabla \varphi\|_{L^{\infty}(\mathcal{U})})\|F\|_{H^n(\Omega)}^{1-\nu}(\|\nabla F\|_{L^{\infty}(\Omega)}\|\psi\|_{H^n(\mathcal{U})})^{\nu}.
\end{align*}
The desired estimate \eqref{eq.compositionEstimate} then follows by applying Young's inequality.
\medskip 

\textit{Step 2: Non-integer $\sigma$.} We write  $\sigma = n + \mu$ with  $n= \lfloor \sigma \rfloor \geqslant 0$ and $\mu\in (0,1)$.
    
We first consider the case $n=0$. Using the change of variables $\tilde{x}=\varphi(x)$, $\tilde{y}=\varphi(y)$ as in \eqref{composition:L2}, we find 
\begin{align*}
  |F\circ \varphi|_{H^{\mu}}^2&=  \int_{\mathcal{U}}\int_\mathcal{U} \frac{|F(\varphi(x))-F(\varphi(y))|^2}{|x-y|^{N+2\mu}}dxdy\\
  &\le c_0^{-1}\int_{\Omega}\int_\Omega \frac{|F(\tilde{x})-F(\tilde{y})|^2}{|\varphi^{-1}(\tilde{x})-\varphi^{-1}(\tilde{y})|^{N+2\mu}}\,\mathrm{d}\tilde{x}\mathrm{d}\tilde{y} \\
    &=c_0^{-1} \int_{\Omega}\int_\Omega \frac{|F(\tilde{x})-F(\tilde{y})|^2}{|\tilde{x}-\tilde{y}|^{N+2\mu}}\frac{|\tilde x-\tilde y|^{N+2\mu}}{|\varphi^{-1}(\tilde{x})-\varphi^{-1}(\tilde{y})|^{N+2\mu}}\,\mathrm{d}\tilde{x}\mathrm{d}\tilde{y}\\
    &\le c_0^{-1}\| \na \varphi\|_{L^\infty(\cU)}^{N+2\mu} \int_{\Omega}\int_\Omega \frac{|F(\tilde{x})-F(\tilde{y})|^2}{|\tilde{x}-\tilde{y}|^{N+2\mu}}\,\mathrm{d}\tilde{x}\mathrm{d}\tilde{y}. 
\end{align*}
This and \eqref{composition:L2} imply
\bq
      \|F \circ \varphi\|_{H^{\sigma}(\cU)} \les (1+\|\nabla \varphi\|_{L^{\infty}}^{\frac{N}{2}+\mu})\|F\|_{H^\sigma(\Omega)}.
\eq
Next, we consider $n\geqslant 1$. Since 
\[
    \|F \circ \varphi\|_{H^{\sigma}(\cU)} \les \|F \circ \varphi\|_{L^2(\cU)} + \|\partial^n(F \circ \varphi)\|_{H^{\mu}(\cU)},
\]
it remains to estimate $ \|\partial^n(F \circ \varphi)\|_{H^{\mu}(\cU)}$. In order to do so we use again~\eqref{eq.faaDiBruno2}.  Note however that since $\phi=\operatorname{id} +\psi$ and $\psi\in L^2$, $\p \varphi$ does not belong to any $L^p(\cU)$ with $p\in [1, \infty)$ if $\cU$ is unbounded. To address this, assuming $N=1$ for notational simplicity, we substitute $\p \varphi=1+\p\psi$ in \eqref{eq.faaDiBruno2} to have
\[
\partial ^n (F \circ \varphi)=\sum_{\substack{\beta_1 + 2\beta_2 + \cdots +n\beta_n=n \\\beta_ 1 + \cdots + \beta_n  = \beta\\ 0\le  k\le  \beta_1}} c_{j, k}(\partial^{\beta}F \circ \varphi)(\p\psi)^k\prod_{j=2}^n (\partial^j \varphi)^{\beta_j}
\]

and \cref{thm.sobolevProducts} for product estimates in Sobolev spaces:
\begin{align*}
        \|\partial ^n (F \circ \varphi)\|_{H^{\mu}(\cU)} &\lesssim \sum_{\substack{\beta_1 + 2\beta_2 + \cdots +n\beta_n=n \\\beta_ 1 + \cdots + \beta_n = \beta \\ 0\le  k\le  \beta_1}} \left\|(\partial^{\beta}F \circ \varphi)(\p\psi)^k\prod_{j=2}^n (\partial^j \varphi)^{\beta_j}\right\|_{H^{\mu}(\cU)} \\ 
        & \lesssim \sum_{\substack{\beta_1 + 2\beta_2 + \cdots +n\beta_n=n \\\beta_ 1 + \cdots + \beta_n = \beta\\ 0\le  k\le  \beta_1}} \|\partial^{\beta}F \|_{L^{p}(\Omega)} \left\|(\p\psi)^k\prod_{j=2}^n (\partial^j \varphi)^{\beta_j}\right\|_{W^{\mu,q}(\cU)} \\ 
        & \qquad+ \sum_{\substack{\beta_1 + 2\beta_2 + \cdots +n\beta_n=n \\\beta_ 1 + \cdots + \beta_n = \beta\\ 0\le  k\le  \beta_1}} \|\partial^{\beta}F \circ \varphi\|_{W^{\mu,\tilde{p}}(\cU)} \left\|(\p\psi)^k\prod_{j=2}^n(\partial^j \varphi)^{\beta_j}\right\|_{L^{\tilde{q}}(\cU)}=: I+II,
\end{align*}
    where $\frac{1}{p} + \frac{1}{q} = \frac{1}{\tilde{p}} + \frac{1}{\tilde{q}} =\frac{1}{2}$ and $q$, $\tilde{p} \in (1, \infty)$, all to be chosen.  Note that we have used \eqref{composition:L2} to have  $\|\partial^{\beta}F\circ \varphi \|_{L^{p}(\Omega)}\les  \|\partial^{\beta}F \|_{L^{p}(\Omega)}$.
    
    Next, for $\tilde{p} \in (1,\infty)$ we claim that for any $\mu \in (0,1)$ and $G\in W^{\mu, \tilde{p}}(\Omega)$ there holds  
    \begin{equation}
        \label{eq.interpolationComp}
        \|G \circ \varphi\|_{W^{\mu,\tilde{p}}(\cU)} \le C(\|\nabla \varphi\|_{L^{\infty}(\cU)}) \|G\|_{W^{\mu, \tilde{p}}(\Omega)}\,.
    \end{equation}
    Remark that indeed~\eqref{eq.interpolationComp} holds for $\mu \in \{0,1\}$ by direct computations, therefore the linear operator $\Phi : G \mapsto G \circ \varphi$ is continuous $W^{0,\tilde{p}}(\mathcal{U}) \to W^{0,\tilde{p}}(\Omega)$ and $W^{1,\tilde{p}}(\Omega) \to W^{1,\tilde{p}}(\mathcal{U})$ with constant $C(\|\nabla \varphi\|_{L^{\infty}(\cU)})$. Therefore using \cref{lemm.interpolation}, $\Phi$ is continuous on $[W^{0,{\tilde{p}}},W^{1,\tilde{p}}]_{\mu}=W^{\mu,\tilde{p}}$, by complex interpolation (see~\cite{Lunardi,BL}). 
    
Using \eqref{eq.interpolationComp} we bound
    \begin{align*}
        \|\partial^{\beta}F \circ \varphi\|_{W^{\mu,\tilde{p}}(\cU)} \left\|(\p\psi)^k\prod_{j=2}^n(\partial^j \varphi)^{\beta_j}\right\|_{L^{\tilde{q}}(\cU)} 
        & \le C(\|\nabla \varphi\|_{L^{\infty}(\cU)})  \|\partial ^{\beta}F\|_{W^{\mu,\tilde{p}}(\Omega)}\|\p\psi\|^k_{L^\infty}\prod_{j=2}^n \|\partial^j \varphi\|^{\beta_j}_{L^{\tilde{q}_j\beta_j}(\cU)},
    \end{align*}
    where $\displaystyle  \frac{1}{\tilde{p}}+\frac{1}{\tilde{q}}= \frac{1}{\tilde{p}} + \sum_{j=1}^n\frac{1}{\tilde{q}_j}= \frac{1}{2}$.  Clearly $\| \p\psi\|_{L^\infty}\les 1+\|\na \varphi\|_{L^\infty}$. For $j\ge 2$ and  $ \frac{1}{\tilde{q}_j\beta_j} =\frac{\theta_j}{2}=\frac{j-1}{2(\sigma -1)}$,  Gagliardo-Nirenberg's inequality yields
   \[
        \|\partial ^j \varphi\|_{L^{\tilde{q}_j\beta_j}(\cU)}= \|\partial ^j \psi\|_{L^{\tilde{q}_j\beta_j}(\cU)}= \|\partial ^{j-1}\p \psi\|_{L^{\tilde{q}_j\beta_j}(\cU)}  \lesssim \|\psi\|_{H^{\sigma}(\cU)}^{\theta_j}(1+\|\nabla \varphi\|_{L^{\infty}(\cU)})^{1-\theta_j}\,.
    \]
  Since 
      \[
    \frac{1}{\tilde{p}}=\mez -\sum_{j=2}^n\frac{1}{\tilde{q}_j}= \frac{\beta + \mu -1}{2(\sigma -1)},
    \]
 Gagliardo-Nirenberg's inequality again yields
   \[
           \|\partial^{\beta}F\|_{W^{\mu, \tilde{p}}(\Omega)} \lesssim \|\nabla F\|_{L^{\infty}(\Omega)}^{1-\theta_0}\|F\|_{H^{\sigma}(\Omega)}^{\theta_0},
    \]
    where $\theta_0 = \frac{\beta + \mu -1}{\sigma -1}=\frac{2}{\tilde{p}}$. It follows that 
    \[
      II \le C(\|\nabla \varphi\|_{L^{\infty}(\cU)}) \|F\|_{H^{\sigma}(\Omega)}^{\tt_0}(\|\nabla F\|_{L^{\infty}(\Omega)}\|\psi\|_{H^{\sigma}(\cU)})^{1-\tt_0},
    \]
    since $\sum_{j=2}^n \theta_j\beta_j=1-\tt_0$.  Therefore, $II$ is bounded by the right-hand side of \eqref{eq.compositionEstimate}   after an application of Young's inequality. 
   
Next, we estimate $I$. Since $\p^j\varphi=\p^j\psi$ for $j\ge 2$, we have
\[
I= \|\partial^{\beta}F \|_{L^{p}(\Omega)} \left\| \prod_{j=1}^n (\partial^j \psi)^{\gamma_j}\right\|_{W^{\mu,q}(\cU)},
\]
where $\gamma_j=\beta_j$ for $j\ge 2$, and $\gamma_1=k\in\{0, \dots, \beta_1\}$. We first use Sobolev product estimates and H\"older's inequality to obtain 
  \begin{align*}
       I  &\lesssim  \|\partial^{\beta}F \|_{L^{p}(\Omega)}\sum_{i\in \{1, \dots, n\}: \gamma_i\ge 1} \|\partial^i \psi \|_{W^{\mu, \tilde{q}_i}}a_i \prod_{j\neq i} \|\partial^j \psi\|^{\gamma_j}_{L^{q_j\gamma_j}},
    \end{align*}
    where  
    \[
    a_i=\begin{cases} \|\partial^i\psi \|_{L^{\bar{q}_i (\gamma_i-1)}}^{\gamma_i-1}\quad\text{if~}\gamma_i\ge 2,\\
    1\quad\text{if~}\gamma_i=1,
    \end{cases}
    \]
$\tilde{q_i}\in (1, \infty)$, and  
    \[
    \begin{cases}
        \frac{1}{p}  +\frac{1}{\bar{q}_i} + \frac{1}{\tilde{q}_i} + \sum_{j\neq i} \frac{1}{q_j} = \frac{1}{2} \quad \text{if~} \gamma_i\ge 2,\\
         \frac{1}{p} + \frac{1}{\tilde{q}_i} + \sum_{j\neq i} \frac{1}{q_j} = \frac{1}{2} \quad \text{if~} \gamma_i=1.
        \end{cases}
    \]
    By \cref{thm.gagliardoNirenberg}, 
    \begin{align*}
  &\|\partial^i \psi \|_{W^{\mu, \tilde{q}_i}}\les \| \na \psi\|_{L^\infty}^{1-\tilde{\tt}_i}\| \psi\|_{H^\sigma}^{\tilde{\tt}_i},\quad \frac{1}{\tilde{q}_i}=\frac{\tilde{\tt}_i}{2}=\frac{i+\mu-1}{2(\sigma-1)},\\
      &\|\partial^j \psi\|_{L^{q_j\gamma_j}}\les  \| \na \psi\|_{L^\infty}^{1-\tt_j}\| \psi\|_{H^\sigma}^{\tt_j},\quad\frac{1}{q_j\gamma_j}=\frac{\tt_j}{2}=\frac{j-1}{2(\sigma-1)},\\
  &\|\partial^i\psi \|_{L^{\bar{q}_i (\gamma_i-1)}}\les \| \na \psi\|_{L^\infty}^{1-\bar{\tt}_i}\| \psi\|_{H^\sigma}^{\bar{\tt}_i},\quad \frac{1}{\bar{q}_i(\gamma_i-1)}=\frac{\bar{\tt}_i}{2}=\frac{i-1}{2(\sigma-1)}
    \end{align*}
\medskip 

\underline{Case $\gamma_1=k=\beta_1$.}  Then we have $\gamma_j=\beta_j$ for all $j$. Fix any $i$ such that $\beta_i\ge 1$. We have
\bq\label{def:nui}
\nu_i:=\tilde{\tt}_i+\sum_{j\ne i} \tt_j\gamma_j=\frac{(n-i\beta_i)-(\beta-\beta_i)+i+\mu-1}{\sigma-1}=\frac{\sigma-1+i(1-\beta_i)-(\beta-\beta_i)}{\sigma-1}\le 1
\eq
since $\beta\ge \beta_i\ge 1$.

We first consider the case $\beta_i=1$, i.e. $\gamma_i=1$. Then $a_i=1$ and 
\[
\frac1p=\mez-\left(\frac{1}{\tilde{q}_i} +\sum_{j\neq i} \frac{1}{q_j}\right)=\mez- \mez\left(\tilde{\tt}_i+\sum_{j\ne i} \tt_j\gamma_j\right)=\frac{1-\nu_i}{2}.
\]
Using Gagliardo-Nirenberg's inequality  again gives
\bq\label{GN:dbetaF:Hr}
    \|\partial^{\beta}F\|_{L^p(\Omega)} \lesssim \|\nabla F\|_{L^{\infty}(\Omega)}^{1-\theta_0}\|F\|_{H^r(\Omega)}^{\theta_0}\,,    
\eq
where  $\displaystyle\frac{1}{p} = \frac{\tt_0}{2}=\frac{\beta-1}{2(r-1)}$, so $\tt_0=\frac{2}{p}=1-\nu_i$. It follows that 
 \[
 \|\partial^{\beta}F \|_{L^{p}(\Omega)} \|\partial^i \psi \|_{W^{\mu, \tilde{q}_i}}a_i \prod_{j\neq i} \|\partial^j \psi\|^{\gamma_j}_{L^{q_j\gamma_j}}\les \|F\|_{H^r(\Omega)}^{1-\nu_i}\left( \|\nabla F\|_{L^{\infty}(\Omega)}\| \psi\|_{H^\sigma}\right)^{\nu_i}
 \]
which is controlled by the right-hand side of  \eqref{eq.compositionEstimate} via Young's inequality provided $r\le \sigma$. Indeed, since $\beta_i=1$, we have  $\nu_i=\frac{\sigma-1-(\beta-1)}{\sigma-1}$ and 
\[
\frac{1}{r-1}-\frac{1}{\sigma-1}=\frac{2}{p(\beta-1)}-\frac{1}{\sigma-1}=\frac{1-\nu_i}{\beta-1}-\frac{1}{\sigma-1}=\frac{(1-\nu_i)(\sigma-1)-(\beta-1)}{(\beta-1)(\sigma-1)}=0,
\]
whence $r=\sigma$. 

Next, we consider the case $\beta_i\ge 2$. Then the power of $\| \psi\|_{H^\sigma}$ will be 
\[
\tilde{\nu}_i:=\nu_i+\bar{\tt}_i(\gamma_i-1)=\nu_i+\tt_i(\bar{\beta}_i-1)=\frac{\sigma-1-(\beta-1)}{\sigma-1}\le 1,
\]
where $\nu_i$ is given by \eqref{def:nui}. Hence, 
\[
 \frac{1}{p} =\mez-\left( \frac{1}{\bar{q}_i} +\frac{1}{\tilde{q}_i} + \sum_{j\neq i} \frac{1}{q_j}\right) = \frac{1-\nu_i}{2}-\frac{1}{\bar{q}_i}=\frac{1-\nu_i}{2}-\frac{\bar{\tt}_i(\gamma_i-1)}{2}=\frac{1-\tilde{\nu}_i}{2}
 \]
and \eqref{GN:dbetaF:Hr} holds with $r$ replaced by $\tilde{r}$, where $\displaystyle\frac{1}{p} = \frac{\tt_0}{2}=\frac{\beta-1}{2(\tilde{r}-1)}$, so $\tt_0=\frac{2}{p}=1-\tilde{\nu}_i$.  Consequently, as in the case $\beta_i=1$ above, it suffices to check if $\tilde{r}\le \sigma$. Indeed,  a direct calculation reveals that $\tilde{r}=\sigma$.

We have proven that $I$ is controllable when $\gamma_1=k=\beta_1$.
\medskip 

\underline{Case $\gamma_1=k<\beta_1$.} In this case the exponent $\nu_i$ in \eqref{def:nui} decreases, and  so are $\tilde{\nu}_i$ and $p$ in view of  the relations between these exponents, which were established in Case 1. As the exponent $p$ in $\| \p^\beta F\|_{L^p}$ decreases, the exponent $r$ in the interpolation inequality \eqref{GN:dbetaF:Hr} also decreases, and hence we can conclude as in Case 1.

\subsection{Proof of \cref{lemm.diffeoSobolev}} We divide the proof into the integer case and the non-integer case. 

\textit{Step 1: Integer $\sigma = n \geqslant 0$.}     By the change of variables $x= \varphi(y)$ we obtain 
\[
\| \varphi^{-1}-\operatorname{id}\|_{L^2(\Omega)}^2=\int_{\cU} |y-\varphi(y)|^2|\det \na \varphi(y)|dy\le C(\| \na \varphi \|_{L^\infty(\cU)})\| \psi\|^2_{L^2(\cU)}.
\]
Differentiating $\varphi \circ \varphi^{-1} = \operatorname{id}$ gives 
$\na \varphi^{-1}=(\na \varphi \circ \varphi^{-1})^{-1}$ which combined with the cofactor formula yields
\[
    \na \varphi^{-1} = \frac{1}{\det ((\na \varphi) \circ \varphi^{-1})} \operatorname{Cof}((\na \varphi) \circ \varphi^{-1})^T. 
\]
Since  $|\det \na \varphi| \geq c_0$ and $\na \varphi \in L^{\infty}$, it follows that  
\begin{equation}
    \label{eq.Linfty-bound}
    \|\na \varphi^{-1} \|_{L^{\infty}(\Omega)}\le \frac{C_N}{c_0} \|\na \varphi\|_{L^{\infty}(\mathcal{U})}^{N}\le C(\|\na \psi\|_{L^{\infty}(\mathcal{U})}).
\end{equation} 
This basic estimate will be used through the proof. 

Differentiating the identity   $\varphi^{-1}+\psi\circ \varphi^{-1}=\operatorname{id}$ gives
\bq\label{grad:varphiinv}
    \nabla\varphi^{-1} - I_N = -(\nabla \psi \circ \varphi^{-1}) \nabla \varphi^{-1}. 
\eq
Then invoking \eqref{eq.Linfty-bound} and the change of variables $x\mapsto \varphi(x)$  using  $|\det \na \varphi| \geq c_0$, we obtain
\[
    \|\nabla \varphi^{-1} - I_N\|_{L^2(\Omega)} \le C(\|\nabla \psi\|_{L^{\infty}(\mathcal{U})})\|\nabla \psi\|_{L^2(\mathcal{U})}.
\]
This yields the desired $H^1$ estimate. Next, we claim that for all $n\ge 1$,

\begin{equation}
    \label{eq.iterativeFormula}
    \begin{aligned}
    \partial^n (\varphi^{-1} - \operatorname{id}) &= \frac{1}{(\partial \varphi \circ \varphi^{-1})^{2n-1}}\sum_{ \substack{\beta_1 + 2\beta_2 + \cdots + n\beta_n = 2n-1 \\ \beta_1 + \cdots + \beta_n = n \\ \beta_1'+\beta_1''=\beta_1}}c_{\beta,n} (\partial \varphi \circ \varphi^{-1})^{\beta_1'}(\partial \psi \circ \varphi^{-1})^{\beta_1''} \prod_{j=2}^n (\partial^j \psi \circ \varphi^{-1})^{\beta_j},
\end{aligned} 
\end{equation}
where $\beta=(\beta_1', \beta_1'', \beta_2, \dots, \beta_n)$. \eqref{eq.iterativeFormula} is true for $n=1$ in view of \eqref{grad:varphiinv} and the formula $\p\varphi^{-1}=\frac{1}{\p \varphi \circ \varphi^{-1}}$. Assume that \eqref{eq.iterativeFormula} is true for some $n\ge 1$. We use the product rule to differentiate \eqref{eq.iterativeFormula}. Since  
\[
 \p\frac{1}{(\partial \varphi \circ \varphi^{-1})^m}=-m \frac{\p (\partial \varphi \circ \varphi^{-1})}{(\partial \varphi \circ \varphi^{-1})^{m+1}}=-m\frac{\p^2\psi \circ \varphi^{-1}}{(\partial \varphi \circ \varphi^{-1})^{m+2}},
 \]
 differentiating $\frac{1}{(\partial \varphi \circ \varphi^{-1})^{2n-1}}$  increases the power of its denominator by $2$, the sum of $\beta_j$ by $1$ and the sum of $j\beta_j$ by $2$ due to the new term $\p^2\psi \circ \varphi^{-1}$.
On the other hand, if $m \ge 1$ then for $u\in\{\varphi, \psi\}$, 
 \[
 \p (\partial^j u \circ \varphi^{-1})^m=m \frac{(\partial^j u \circ \varphi^{-1})^{m-1}  (\partial^{j+1} \psi \circ \varphi^{-1})}{\p \varphi \circ \varphi^{-1}}=m \frac{(\p \varphi \circ \varphi^{-1})(\partial^j u \circ \varphi^{-1})^{m-1}  (\partial^{j+1} \psi \circ \varphi^{-1})}{(\p \varphi \circ \varphi^{-1})^2}.
 \]
 This  increases the power of the denominator in $\frac{1}{(\partial \varphi \circ \varphi^{-1})^{2n-1}}$ by $2$, the sum of $\beta_j$ by $1$ (due to  $\p \varphi \circ \varphi^{-1}$ in the numerator) and the sum $j\beta_j$ by $1.1+j.(-1)+(j+1).1=2$.
 
 Next, we generalize \eqref{eq.iterativeFormula} to dimension $N\geq 2$. To this end we observe that one can prove by the same induction argument that it holds: 
\begin{equation}
    \label{eq.iterativeFormula-bis}
    \partial^n (\varphi^{-1} - \operatorname{id}) =\sum_{ \substack{\beta_1 + 2\beta_2 + \cdots + n\beta_n = 2n-1 \\ \beta_1 + \cdots + \beta_n = n \\ \beta_1'+\beta_1''=\beta_1}}c_{\beta,n} (\nabla \varphi \circ \varphi^{-1})^{-(2n-1)}(\partial \varphi \circ \varphi^{-1})^{\beta_1'}(\partial \psi \circ \varphi^{-1})^{\beta_1''} \prod_{j=2}^n (\partial^j \psi \circ \varphi^{-1})^{\beta_j},
\end{equation}
where:
\begin{itemize}
    \item $\partial^n (\varphi^{-1} - \operatorname{id})$ should be understood as any possible $\partial_1^{n_1} \cdots \partial_N^{n_N}(\varphi^{-1}-\operatorname{id})_{k}$ for some $1 \leq k \leq N$ and $n_1 + \dots n_N = n$;
    \item $(\p^j \psi \circ \varphi^{-1})^k$ should be understood as the product of $k$ terms of the form  $\p_1^{j_1} \cdots \p_N^{j_N} \psi_\ell \circ \varphi^{-1}$ for some $\ell \le N$ and $j_1+\dots+j_N=j$;
    \item $(\p \varphi \circ \varphi^{-1})$ should be understood as $(\na \varphi \circ \varphi^{-1})_{ik}$, i.e. a matrix element of the gradient. 
    \item $(\nabla \varphi \circ \varphi^{-1})^{-(2n-1)}$ should be understood as $\prod_{p=1}^{2n-1} [(\na \varphi\circ\varphi^{-1})^{-1}]_{i_pk_p}$
\end{itemize}
Next, we expand 
\[
(\na\varphi\circ\varphi^{-1})_{ij}^{\beta'_1}=(\na\varphi\circ\varphi^{-1}_{ij}- \delta_{ij} +\delta_{ij})^{\beta'_1} = \sum_{0\leq k\leq \beta'_1}c'_{k,\beta'_1}((\na\varphi\circ \varphi^{-1})_{ij}-\delta_{ij})^k
\]
 and relabel $k$ as $\beta'_1$. Similarly we can write 
\[
    (\na \varphi\circ\varphi^{-1})^{-(2n+1)} = ((\na \varphi\circ\varphi^{-1})^{-1} - I_N + I_N)^{2n+1} = \sum_{0\leq k \leq 2n+1} ((\na \varphi\circ\varphi^{-1})^{-1}-I_N)^k.
\]
Combining these remarks we arrive at our final formula: 
\begin{multline}
    \label{eq.iterativeFormula-ter}
    \partial^n (\varphi^{-1} - \operatorname{id}) =\sum_{ \substack{\beta_1 + 2\beta_2 + \cdots + n\beta_n = 2n-1 \\ \beta_1 + \cdots + \beta_n = n \\ \beta_1'+\beta_1''\leq \beta_1 \\ 0\leq k \leq 2n+1}}c_{\beta,n,k} ((\nabla \varphi \circ \varphi^{-1})^{-1}-I_N)^{k}(\na \varphi \circ \varphi^{-1} - I_N)^{\beta_1'}(\partial \psi \circ \varphi^{-1})^{\beta_1''} \\ 
    \times \prod_{j=2}^n (\partial^j \psi \circ \varphi^{-1})^{\beta_j}.
\end{multline} 

Consider $n\ge 2$. Applying  Hölder's inequality to \eqref{eq.iterativeFormula-ter} yields
\[
\begin{aligned}
    \|\partial ^n (\varphi^{-1} - \operatorname{id})\|_{L^2(\Omega)} &\le C(\|\nabla \psi\|_{L^{\infty}(\mathcal{U})}) \sum_{\substack{\beta_1 + 2\beta_2 + \cdots + n\beta_n = 2n-1\\\beta_1 + \cdots + \beta_n = n}} \left\|\prod_{j=2}^n (\partial^j \psi \circ \varphi^{-1})^{\beta_j}\right\|_{L^2(\Omega)}\\
    &\le C(\|\nabla \psi\|_{L^{\infty}(\mathcal{U})}) \sum_{\substack{\beta_1 + 2\beta_2 + \cdots + n\beta_n = 2n-1\\\beta_1 + \cdots + \beta_n = n}} \left\|\prod_{j=2}^n (\partial^j \psi \circ \varphi^{-1})^{\beta_j}\right\|_{L^{\beta_jq_j}(\Omega)},
\end{aligned}
\]
where $\sum_{j=2}^n\frac{1}{q_j}=\mez$. 

For  $j\geq 2$ we write $\p^j \psi = \p^{j-1}(\p \psi)$ and apply \cref{thm.gagliardoNirenberg} to $\p \psi$ to have 
        \[
            \|\partial^j \psi\|_{L^{q_j\beta_j}(\mathcal{U})} \lesssim \|\nabla \psi\|^{1-\theta_j}_{L^{\infty}(\mathcal{U})}\|\psi\|_{H^n(\mathcal{U})}^{\theta_j},   
        \]
        where   $\frac{1}{q_j\beta_j}=\frac{\tt_j}{2}= \frac{j-1}{2(n-1)}$, which sets the choice of $q_j$. We check that  
        \[
            \sum_{j=2}^n \frac{1}{q_j} = \sum_{j=2}^n \frac{(j-1)\beta_j}{2(n-1)} = \frac{1}{2(n-1)}\left(\sum_{j=1}^n j\beta_j- \sum_{j=1}^n \beta_j\right) =  \frac{(2n-1)-n}{2(n-1)}= \frac{1}{2}.    
        \]
        Combining the above estimates, we deduce
        \[
            \|\partial ^n (\varphi^{-1}-\operatorname{id})\|_{L^2(\Omega)} \le C(\|\nabla \psi\|_{L^{\infty}(\mathcal{U})})\sum_{\substack{\beta_1 + 2\beta_2 + \cdots + n\beta_n = 2n-1\\\beta_1 + \cdots + \beta_n = n}} \|\psi\|_{H^n(\mathcal{U})}^{\theta(\beta)},
        \]
        where
        \[
            \theta (\beta) = \sum_{j=2}^n \beta_j\theta_j = 2\sum_{j=2}^n \frac{1}{q_j} = 1, 
        \]
        which yields the desired $H^n$ estimate. 
        \medskip 
        
\textit{Step 2: Non-integer  $\sigma> 0$.} Let $n = \lfloor \sigma \rfloor \geqslant 0$ and  $\mu =\sigma -n \in (0,1)$. 
We use the following estimates. 

\begin{lemm} Let $\psi$ and $\varphi$ as in \cref{lemm.diffeoSobolev}. Let $\mu \in (0,1)$ and $q\in(1,\infty)$ and $i\geq 0$ be an integer. Then the following estimates hold: 
    \begin{equation}
        \label{eq.Wo-finite-assump}
        \left\|(\na \varphi \circ \varphi^{-1})^{-1}-I_N\right\|_{W^{\mu,q}(\Omega)} \le C(\|\na \psi\|_{L^{\infty}(\mathcal{U})})\|\na \psi\|_{W^{\mu,q}(\cU)} 
    \end{equation}
    and
     \begin{equation}
        \label{eq.psi-sob}
        \|G \circ \varphi^{-1}\|_{W^{\mu,q}(\Omega)} \le C(\|\nabla \psi\|_{L^{\infty}(\mathcal{U})})\|G\|_{W^{\mu,q}(\mathcal{U})}
    \end{equation}
    for $G\in W^{\mu,q}(\mathcal{U})$. 
\end{lemm}

\begin{proof} To prove  \eqref{eq.psi-sob}, we note that it holds with $\mu=0$ and $\mu=1$ with the aid of \eqref{eq.Linfty-bound}. Then  \eqref{eq.psi-sob} follows by  a similar interpolation argument as for \eqref{eq.interpolationComp}. 
    
In order to prove \eqref{eq.Wo-finite-assump}, we denote $w = (\na \varphi \circ \varphi^{-1})^{-1}-I_N\equiv \na \varphi^{-1}-I_N$ and use the following two expressions for $w$:  
\begin{align}
    w &= (\na\psi \circ \varphi^{-1}+I_N)^{-1}-I_N\label{eq.Wmu-version}\\
    &= -\na \psi \circ \varphi^{-1} \na \varphi^{-1}. \label{eq.Lq-version}
\end{align}
Let us recall the equivalent norm  \eqref{eq.equivalent-norm-sobolev} of  $W^{\mu,q}(\Omega)$: 
\begin{equation}
    \label{Fmuq2-norm}
    \|u\|_{W^{\mu,q}(\Omega)} = \|u\|_{L^q(\Omega)} +|u|_{W^{\mu, q}(\Omega)},\quad |u|_{W^{\mu, q}(\Omega)}= \left(\iint_{x\in \Omega \times \Omega} \frac{|u(x')-u(x)|^q}{|x-x'|^{N+q\mu}}\,\mathrm{d}x'\mathrm{d}x\right)^{\frac{1}{q}}. 
\end{equation}
Let $w_{ij}$ denote the $(i, j)$-entry of $w$. Since   \eqref{eq.Lq-version} and \eqref{eq.Linfty-bound} imply 
\[ 
    \|w_{ij}\|_{L^q} \le C(\|\na \psi\|_{L^{\infty}(\cU)})\|\na \psi\|_{L^q}\leq C(\|\na \psi\|_{L^{\infty}(\cU)})\|\na \psi\|_{W^{\mu,q}},
\]
it remains to bound the integral term in  \eqref{Fmuq2-norm}. To this end we observe that for any invertible matrices $A$ and $B$ we have $A^{-1}-B^{-1} = A^{-1}(B-A)B^{-1}$, and therefore using \eqref{eq.Wmu-version} we can write 
\[
    w_{ij}(x)-w_{ij}(x')=\sum_{1\leq k, \ell \leq N}(\na\psi \circ \varphi^{-1}(x)+I_N)^{-1}_{ik}(\na\psi \circ \varphi^{-1}(x')-\na\psi \circ \varphi^{-1}(x))_{k\ell} (\na\psi \circ \varphi^{-1}(x')+I_N)^{-1}_{\ell j}.
\]
Since  $\| (\na\psi \circ \varphi^{-1}+I_N)^{-1} \|_{L^\infty(\Omega)}\le C(\|\na \psi\|_{L^{\infty}(\cU)})$, it follows that 
\begin{equation}
    \label{eq.rewriting}
    |w_{ij}(x)-w_{ij}(x')|^q \le C(\|\na \psi\|_{L^{\infty}(\cU)}) \sum_{1\leq k,\ell \leq N}|(\na \psi \circ \varphi^{-1}(x)-\na\psi \circ \varphi^{-1}(x'))_{k\ell}|^q, 
\end{equation}
whence 
\[
   |w_{ij}(x')|_{W^{\mu, q}(\Omega)} \le C(\|\na \psi\|_{L^{\infty}(\cU)}) \sum_{1\leq k,\ell \leq N} |(\na \psi \circ \varphi^{-1})_{k\ell}|_{W^{\mu,q}(\Omega)}.
\] 
Finally, applying \eqref{eq.psi-sob} with $G=\na \psi$  yields  \eqref{eq.Wo-finite-assump}. 
\end{proof}
Now we return to the proof of \eqref{eq.diffoSobolev} for $\sigma=n+\mu$. The case $n=0$ follows from \eqref{eq.psi-sob}:
\[
    \|\varphi^{-1}-\operatorname{id}\|_{H^{\mu}} = \|-\psi \circ \varphi^{-1}\|_{H^{\mu}} \le  C(\|\nabla \psi\|_{L^{\infty}(\mathcal{U})})\|\psi\|_{H^{\mu}(\mathcal{U})}.  
\]
The case $n=1$ is a consequence of \eqref{eq.Wo-finite-assump} with $q=2$. 

Now we can assume $n\geq 2$. 
Since the $L^{\infty}$ norms of $(\na \varphi \circ \varphi^{-1})^{-1}-I_N$, $\na \varphi \circ \varphi^{-1}$, and $\p \psi \circ \varphi^{-1}$ are bounded by $C(\|\nabla \psi\|_{L^{\infty}})$, applying the product rule in Sobolev spaces given by \cref{thm.sobolevProducts} to \eqref{eq.iterativeFormula-ter}, we obtain
\begin{multline*}
    \|\partial^n(\varphi^{-1} - \operatorname{id})\|_{H^{\mu}(\Omega)}\le C(\|\nabla \psi\|_{L^{\infty}})\sum_{\substack{\beta_1 + 2\beta_2 + \cdots + n\beta_n = 2n-1 \\ \beta_1 + \cdots + \beta_n = n}} A(\beta) \\ 
    +  C(\|\nabla \psi\|_{L^{\infty}})\sum_{\substack{\beta_1 + 2\beta_2 + \cdots + n\beta_n = 2n-1 \\ \beta_1 + \cdots + \beta_n = n \\ \beta_1' \geq 1}} B'(\beta) +  C(\|\nabla \psi\|_{L^{\infty}})\sum_{\substack{\beta_1 + 2\beta_2 + \cdots + n\beta_n = 2n-1 \\ \beta_1 + \cdots + \beta_n = n \\ \beta_1'' \geq 1}}B''(\beta) \\
    + C(\|\nabla \psi\|_{L^{\infty}})\sum_{i=2}^n\sum_{\substack{\beta_1 + 2\beta_2 + \cdots + n\beta_n = 2n-1 \\ \beta_1 + \cdots + \beta_n = n \\ \beta_i \geq 1}} B_i(\beta),   
\end{multline*}
where the terms $A(\beta)$, $C'(\beta)$, $C''(\beta)$ and $C_i(\beta)$ are defined by 
\begin{align*}
    A(\beta) = \left\|(\na \varphi \circ \varphi^{-1})^{-1}-I_N\right\|_{W^{\mu,q_1}(\Omega)}  
    \prod_{j=2}^n\|\partial^j \psi \circ \varphi^{-1}\|_{L^{q_j\beta_j}(\Omega)}^{\beta_j}.
\end{align*}
\[
    B'(\beta) = \left\|(\na \varphi \circ \varphi^{-1} - I_N)^{\beta_1'}\right\|_{W^{\mu,q_1}(\Omega)}  
    \prod_{j=2}^n\|\partial^j \psi \circ \varphi^{-1}\|_{L^{q_j\beta_j}(\Omega)}^{\beta_j}, 
\]
\[
        B''(\beta) = \left\|(\p \psi \circ \varphi^{-1})^{\beta_1''}\right\|_{W^{\mu,q_1}(\Omega)}  
        \prod_{j=2}^n\|\partial^j \psi \circ \varphi^{-1}\|_{L^{q_j\beta_j}(\Omega)}^{\beta_j}, 
\]
\[
    C_i(\beta) = \|(\partial^i \psi \circ \varphi^{-1})^{\beta_i}\|_{W^{\mu,q_i'}}
    \prod_{\substack{2\leq j \leq n \\ j\neq i}}\|\partial^j \psi \circ \varphi^{-1}\|_{L^{\tilde{q}_j\beta_j}(\Omega)}^{\beta_j}, 
\]
with $\displaystyle\sum_{j=1}^n \frac{1}{q_j}=\frac{1}{\tilde{q}_i}+\displaystyle\sum_{\substack{1\leq j \leq n\\j\neq i}} \frac{1}{\tilde{q}_j}=\frac{1}{2}$, and $q_1, \tilde{q}_i \in [2,\infty)$ to be chosen later.
We interpolate using \cref{thm.gagliardoNirenberg} on $\partial \psi$, which yields for any $j\geq 2$, 
\bq\label{interpolate:djpsi}
    \|\partial^j \psi \circ \varphi^{-1}\|_{L^{q_j\beta_j}(\Omega)}\le C(\|\nabla \psi\|_{L^{\infty}}) \|\partial^j \psi\|_{L^{q_j\beta_j}(\cU)}\le C(\|\nabla \psi\|_{L^{\infty}})\|\psi\|^{\theta_j}_{H^{\sigma}(\Omega)}
\eq
with $\frac{1}{\beta_jq_j}=\frac{\theta_j}{2}=\frac{j-1}{2(\sigma - 1)}$. It follows that 
\[
    \sum_{j=2}^n \frac{1}{q_j} = \sum_{j=2}^n \frac{(j-1)\beta_j}{2(\sigma - 1)} = \frac{n-1}{2(\sigma - 1)}, \quad \sum_{j=2}^n\beta_j\theta_j = 2 \sum_{j=2}^n \frac{1}{q_j} =2\left(\mez-\frac{1}{q_1}\right)= 1-\frac{2}{q_1},
\] 
and hence $\frac{1}{q_1}=\frac{\mu}{2(\sigma-1)}$. Applying \cref{thm.gagliardoNirenberg}  again yields 
\bq\label{interpolate:napsWq1}
\| \na \psi\|_{W^{\mu, q_1}} \le C(\|\na \psi\|_{L^{\infty}(\cU)})\| \na \psi\|_{H^{\sigma-1}}^\tt,\quad \frac{1}{q_1}=\frac{\tt}{2}=\frac{\mu}{2(\sigma-1)}.
\eq
Combining \eqref{interpolate:djpsi}, \eqref{eq.Wo-finite-assump}, and \eqref{interpolate:napsWq1}, we deduce 
\begin{equation}
    \label{eq.estAbeta}
    A(\beta) \le C(\|\na \psi\|_{L^{\infty}(\cU)})\| \na \psi\|_{W^{\mu, q_1}} \|\psi\|^{1-\frac{2}{q_1}}_{H^{\sigma}(\cU)}\le C(\|\na \psi\|_{L^{\infty}(\cU)})\|\psi\|_{H^{\sigma}(\cU)}.
\end{equation}
We now move to estimating $B''(\beta)$. In view of \eqref{eq.psi-sob} and by interpolation using \cref{thm.gagliardoNirenberg} on $\p \psi$, it holds 
\begin{equation}
    \label{eq.estimate-Bpp}
    \|\p \psi \circ \varphi^{-1}\|_{W^{\mu,q_1}(\Omega)} \le C(\|\nabla \psi\|_{L^{\infty}})\|\p\psi\|_{W^{\mu,q_1}(\cU)} \le C(\|\nabla \psi\|_{L^{\infty}})\|\psi\|_{H^{\sigma}(\cU)}^{\theta_1}, 
\end{equation}
with $\frac{1}{q_1}=\frac{\theta_1}{2}=\frac{\mu}{2(\sigma - 1)}$ so that in case $\beta''_1=1$ we obtain 
\begin{equation}
    \label{eq.estBprime}
    B''(\beta) \le C(\|\na \psi\|_{L^{\infty}(\cU)})\|\psi\|_{H^{\sigma}(\cU)}.
\end{equation}
Let us now assume that $\beta''_1 \geq 2$. The product estimate \eqref{eq.productSobol} with $p_1=q_1=\infty$ yields 
\[
    \|(\p \psi \circ \varphi^{-1})^{\beta_1}\|_{W^{\mu,q_1}(\Omega)} \les \|\p\psi \circ \varphi^{-1}\|_{W^{\mu,q_1}(\Omega)} \|\p \psi \circ \varphi^{-1}\|_{L^{\infty}(\Omega)}^{\beta_1-1}
\] 
Using change of variables, estimate \eqref{eq.psi-sob} and \cref{thm.gagliardoNirenberg} on $\p\psi$, we get
\begin{equation*}
    \|(\p \psi \circ \varphi^{-1})^{\beta_1}\|_{W^{\mu,q_1}(\Omega)} \le C(\|\nabla \psi\|_{L^{\infty}}(\cU)) \|\p\psi\|_{W^{\mu,q'_1}(\cU)} \le C(\|\nabla \psi\|_{L^{\infty}(\cU)}) \|\psi\|_{H^{\sigma}(\cU)}^{\theta_1},
\end{equation*}
where $\frac{1}{q_1}=\frac{\theta_1}{2}=\frac{\mu}{2(\sigma - 1)}$. It follows that $B''(\beta)$ is controlled by the right-hand side of \eqref{eq.estBprime}. 

The term $B'(\beta)$ is estimated similarly, the only difference being that instead of \eqref{eq.estimate-Bpp} we use 
\begin{align*}
    \left\|\na \varphi \circ \varphi^{-1} - I_N\right\|_{W^{\mu,q_1}(\Omega)} &\le C(\|\na \psi\|_{L^{\infty}(\cU)})\|\na \varphi - I_N\|_{W^{\mu,q_1}(\cU)} \\
    &\le C(\|\na \psi\|_{L^{\infty}(\cU)})\|\na \psi\|_{W^{\mu,q_1}(\cU)}\\ 
    &\le C(\|\na \psi\|_{L^{\infty}(\cU)})\|\psi\|_{H^{\sigma}}^{\theta_1},
\end{align*}
which follows from \eqref{eq.psi-sob}, $\na \varphi = \na \psi + I_N$ and interpolation. Therefore, we arrive at 
\begin{equation}
    \label{eq.estBprimeprime}
    B'(\beta) \le C(\|\nabla \psi\|_{L^{\infty}(\cU)})\|\psi\|_{H^{\sigma}(\cU)}.
\end{equation}

We turn to estimating $C_i(\beta)$. First, we choose $\frac{1}{\tilde{q}_j}=\frac{1}{q_j}=\frac{\theta_j}{2}=\frac{j-1}{2(\sigma-1)}$, so that interpolating using \cref{thm.gagliardoNirenberg} we have  
\[
    \prod_{\substack{2\leq j \leq n \\ j\neq i}}\|\partial^j \psi \circ \varphi^{-1}\|_{L^{\tilde{q}_j\beta_j}(\Omega)}^{\beta_j} \le \|\psi\|_{H^{\sigma}(\cU)}^{\theta}, 
\]
where 
\[
    \theta= \sum_{\substack{2\leq j \leq n \\j \neq i}} \beta_j\theta_j=\sum_{\substack{2\leq j \leq n \\j \neq i}} \beta_j\frac{j-1}{\sigma - 1} = \frac{n-1}{\sigma - 1} - \frac{\beta_i(i - 1)}{\sigma-1}. 
\]
In case $\beta_i=1$ an application of \eqref{eq.psi-sob} and interpolation on $\p \psi$ using \cref{thm.gagliardoNirenberg} gives 
\[
    \|\partial^i \psi \circ \varphi^{-1}\|_{W^{\mu,\tilde{q}_i}(\Omega)} \le C(\|\na \psi\|_{L^{\infty}(\cU)}) \|\p^i\psi\|_{W^{\mu,\tilde{q}_i}(\cU)} \le C(\|\na \psi\|_{L^{\infty}(\cU)}) \|\psi\|_{H^{\sigma}}^{\tilde{\theta}_i}, 
\]
where $\frac{1}{\tilde{q}_i}=\frac{\tilde{\theta}_i}{2}=\frac{\mu + i - 1}{2(\sigma - 1)}$. Observe that $\tilde{\theta}_i = 1 - \theta$ and that $\frac{1}{\tilde{q}_i} + \sum_{j\neq i} \frac{1}{\tilde{q}_j} = \frac{1}{2}$. Finally we have obtained that 
\begin{equation}
    \label{eq.estCi}
    C_i(\beta) \le C(\|\na \psi\|_{L^{\infty}(\cU)})\|\psi\|_{H^{\sigma}(\cU)}. 
\end{equation}
In the case $\beta'_i \geq 2$, observe that we have $\frac{1}{\tilde{q}_i}=\frac{\mu + \beta_i(i-1)}{2(\sigma-1)}$. 
We use the product estimate of \cref{thm.sobolevProducts}, \eqref{eq.psi-sob} and interpolation inequalities of \cref{thm.gagliardoNirenberg}: 
\begin{align*}
    \|(\partial^i \psi \circ \varphi^{-1})^{\beta_i}\|_{W^{\mu,\tilde{q}_i}(\Omega)} & \le \|\partial^i \psi \circ \varphi^{-1}\|^{\beta_i-1}_{L^{(\beta_i-1)\hat{q}_i}(\cU)} \|\partial^i \psi \circ \varphi^{-1}\|_{W^{\mu,\bar{q}_i}(\Omega)}\\
    & \le C(\|\na \psi\|_{L^{\infty}(\cU)}) \|\partial^i \psi \|^{\beta_i-1}_{L^{(\beta_i-1)\hat{q}_i}(\cU)} \|\partial^i \psi\|_{W^{\mu,\bar{q}_i}(\cU)} \\
    & \le C(\|\na \psi\|_{L^{\infty}(\cU)}) \|\psi\|_{H^{\sigma}(\cU)}^{(\beta_i-1)\hat{\theta}_i}\|\psi\|_{H^{\sigma}(\cU)}^{\bar{\theta_i}}, 
\end{align*}
with $\frac{1}{(\beta_i-1)\hat{q}_i}=\frac{\hat{\theta}_i}{2}=\frac{i-1}{2(\sigma -1)}$ and $\frac{1}{\bar{q}_i}=\frac{\bar{\theta}_i}{2}=\frac{\mu + i-1 }{2(\sigma - 1)}$. Observe that such a choice of $(\hat{q}_i,\bar{q}_i)$ is justified since $\frac{1}{\hat{q}_i}+\frac{1}{\bar{q}_i}=\frac{\mu + \beta_i(i-1)}{2(\sigma-1)}=\frac{1}{\tilde{q}_i}$. This is enough to control $C_i(\beta)$ by the right-hand side of \eqref{eq.estCi} as $\theta + \bar{\theta}_i + (\beta_i-1)\hat{\theta}_i=1$. 

Gathering \eqref{eq.estAbeta}, \eqref{eq.estBprime}, \eqref{eq.estBprimeprime} and \eqref{eq.estCi} we obtain the claimed estimate, which ends the proof. 
        
\section{Sufficient conditions for $\mathfrak{G}$}\label{appendixE}

\begin{prop}\label{prop:gamma} Let $F : \mathbb{R} \to \mathbb{R}$, and let $J=(-1,0)$ in the finite depth case or $J=(-\infty,0)$ in the infinite depth case. Then for all $a >0$ and $\sigma\ge 0$, there exists a function $C_{F, a, \sigma}(\cdot):\Rr_+\to \Rr_+$ such that the following assertions hold.
\begin{enumerate}[(i)]
    \item If $F\in C^{\lceil \sigma \rceil}_b(\Rr)$, then for all $h, v \in H^{\sigma}(\Rr^d \times J)$, there holds 
    \begin{equation}
        \label{eq:est-unbounded}
        \|F(h(x,z)+az)v\|_{H_{x,z}^{\sigma}(\Rr^d \times J)} \le C_{F, a, \sigma}(\|h\|_{L^{\infty} \cap H^{\sigma}})\|v\|_{L^{\infty} \cap H^{\sigma}(\Rr^d \times J)}.
    \end{equation}
    \item If $F'\in C^{\lceil \sigma \rceil}_b(\Rr)$, then for all $h_1, h_2, w \in H^{\sigma}(\Rr^d \times J)$, there holds 
    \begin{equation}
        \label{eq:diff-est-unbounded}
        \|(F(h_1+az)-F(h_2+az))w\|_{H_{x,z}^{\sigma}(\Rr^d \times J)} \le C_{F, a, \sigma}(\|(h_1, h_2)\|_{L^{\infty} \cap H^{\sigma}})\|(h_1-h_2)w\|_{L^{\infty}\cap H^{\sigma}}.
    \end{equation}
\end{enumerate}
\end{prop}

\begin{proof} We first note that (ii) is a consequence of (i) since we can write 
\[
    (F(h_1+az)-F(h_2+az))w=\int_0^1 F'(az+(1-\tau)h_2+\tau h_1)\,\mathrm{d}\tau (h_1-h_2)w. 
\]
Then \eqref{eq:diff-est-unbounded} follows by applying \eqref{eq:est-unbounded} with $F$ replaced by  $F'$, $h=(1-\tau)h_2+\tau h_1$, and $v=(h_1-h_2)$. Without loss of generality, we  shall prove (i) for $a=1$.

We start by assuming that $\sigma = n$ is an integer, so that it reduces  to estimating the terms 
\[
    A_m=\|\p^{m}(F(h+z))\p^{n-m}v\|_{L^2},\quad 0 \le m \le n,
\] 
where we recall that $\p^\ell$ denotes any partial derivative of order $\ell$ in $(x, z)$. Clearly $A_0\le \| F\|_{L^\infty}\| v\|_{H^\sigma}$. Now consider $1\le m\le n$. Using the Fa\`a di Bruno formula \eqref{eq.faaDiBruno2}, we further reduce the task to estimating the terms
\bq\label{def:Ambeta}
    A_{m,\beta} = \left\|F^{(\beta)} (h+z) \p^{n-m}v\prod_{j=1}^m (\p^j(h+z))^{\beta_j}\right\|_{L^2},    
\eq
where $\beta_1 + \cdots + \beta_m = \beta$ and $\beta_1 + 2\beta_2 + \cdots + m\beta_m = m$. Note that $\p^j(h+z)=\p_zh + 1$ if $j=1$ and $\p=\p_z$, and  $\p^j(h+z)=\p^jh$ otherwise. Hence, 
\[
     A_{m,\beta}  = \begin{cases}
   \left\|F^{(\beta)} (h+z) \p^{n-m}v(\p_zh+1)^{\beta_1}\prod_{j=2}^m  (\p^jh)^{\beta_j}\right\|_{L^2}    \quad\text{if~} \p^1(h+z)\equiv \p_z(h+z),\\
      \left\| F^{(\beta)} (h+z) \p^{n-m}v\prod_{j=1}^m (\p^jh)^{\beta_j}\right\|_{L^2}    \quad\text{otherwise}.
    \end{cases}
\]
As $|\p_z h +1|^{\beta_1} \lesssim |\p_z h|^{\beta_1} +1$ and $F^{(\beta)}$ is bounded for $\beta\le n$, we use Hölder's inequality with 
\[
    \sum_{j=1}^m\frac{1}{q_j} + \frac{1}{q} = \sum_{j=2}^m \frac{1}{q_j} + \frac{1}{\tilde{q}} = \frac{1}{2},
\] 
where $q$, $\tilde{q}$, $q_j$ will be chosen later,  to have
\[
    A_{m,\beta} \le \|F^{(\beta)}\|_{L^{\infty}}\prod_{j=1}^m \|\p^jh\|_{L^{q_j\beta_j}}^{\beta_j}\|\p^{n-m}v\|_{L^{q}} + \|F^{(\beta)}\|_{L^{\infty}}\prod_{j=2}^m \|\p^jh\|_{L^{q_j\beta_j}}^{\beta_j}\|\p^{n-m}v\|_{L^{\tilde{q}}}.
\]
Then, we use the Gagliardo-Nirenberg inequality from \cref{thm.gagliardoNirenberg} (with $(r,\sigma)=(2,n)$) to bound 
\bq\label{est:Ambeta:2} 
\begin{aligned}
    A_{m,\beta} & \le \|F^{(\beta)}\|_{L^{\infty}}\prod_{j=1}^m \|h\|_{L^{\infty}}^{\beta_j(1-\theta_j)}\|h\|_{H^n}^{\beta_j\theta_j} \|v\|_{L^{\infty}}^{1-\theta}\|v\|_{H^n}^{\theta}\\
    &\qquad +\|F^{(\beta)}\|_{L^{\infty}}\prod_{j=2}^m \|h\|_{L^{\infty}}^{\beta_j(1-\theta_j)}\|h\|_{H^n}^{\beta_j\theta_j} \|v\|_{L^{\infty}}^{1-\tilde\theta}\|v\|_{H^{n\frac{n-m}{n-m+\beta_1}}}^{\tilde\theta},
\end{aligned}
\eq
where we have chosen $\frac{1}{q}=\frac{\theta}{2}= \frac{n-m}{2n}$, $\frac{1}{q_j\beta_j}=\frac{\theta_j}{2}=\frac{j}{2n}$, and $\frac{1}{\tilde{q}}=\frac{\tilde{\theta}}{2}= \frac{n-m+\beta_1}{2n}$. We check that
\[
    \sum_{j=1}^m \frac{1}{q_j} + \frac{1}{q} = \sum_{j=1}^m \frac{j\beta_j}{2n} + \frac{1}{2} - \frac{m}{2n} = \mez = \sum_{j=2}^m \frac{1}{q_j} + \frac{1}{\tilde{q}}.   
\] 
The right-hand side of \eqref{est:Ambeta:2} is now controlled by that of \eqref{eq:est-unbounded} since $n\frac{n-m}{n-m+\beta_1}\leq n$.

For the non-integer case  $\sigma = n+ \mu$, $\mu \in (0,1)$, it suffices to bounded $\|\p^n(F(h+z)v)\|_{H^{\mu}}$. This in turn  reduces to bounding the terms 
\[
    A_{m,\beta} = \left\|F^{(\beta)} (h+z) \p^{n-m}v\prod_{j=1}^m (\p^jh)^{\beta_j}\right\|_{H^{\mu}},
\]
where $j$, $\beta$, and the $\beta_j$ are as in \eqref{def:Ambeta}. Note that the terms $F^{(\beta)} (h+z) \p^{n-m}v (\p_zh+1)\prod_{j=2}^m (\p^jh)^{\beta_j}$, appearing when $\p(h+z)\equiv \p_z(h+z)$, can be controlled be similarly, therefore omitted. Denoting $w=F^{(\beta)} (h+z) \p^{n-m}v$, the product estimate \eqref{eq.productSobol} yields 
\bq\label{est:Ambeta:4}
\begin{aligned}
    A_{m,\beta}&\lesssim  \left\|\prod_{j=1}^m (\p^jh)^{\beta_j}\right\|_{L^{q'}}\|w\|_{W^{\mu,q}} + \left\|\prod_{j=1}^m (\p^jh)^{\beta_j}\right\|_{W^{\mu, p'}}\|w\|_{L^{p}} \\ 
    & \lesssim \prod_{j=1}^m \|\p^j h\|_{L^{q_j\beta_j}}^{\beta_j}\|w\|_{W^{\mu,q}}  \\ 
    &\qquad+ \|w\|_{L^p} \sum_{i\in\{1,\dots, m\}:~\beta_i\ge 1} \|\p^i h\|_{W^{\mu,\bar{p}_i}} a_i\prod_{j\neq i} \|\p^jh\|_{L^{p_j\beta_j}}^{\beta_j}=:I+II,
\end{aligned}
\eq
where 
\[
a_i=\begin{cases}
1\quad\text{if~}\beta_i=1,\\
 \|\p^i h\|_{L^{\tilde{p}_i(\beta_i-1)}}^{\beta_i-1}\quad\text{if~}\beta_i\ge 2,
\end{cases}
\]
\[
\frac{1}{q}+\frac{1}{q'}=\frac{1}{q}+\sum_{j=1}^m\frac{1}{q_j}=\mez,\quad q \in (1,  \infty),
\]
and
\[
\begin{cases}
\frac{1}{\bar{p}_i} + \frac{1}{\tilde{p}_i} + \sum_{j\neq i} \frac{1}{p_j} + \frac{1}{p}=\mez,\quad \bar{p}_i\in (1, \infty)\quad\text{if~}\beta_i\ge 2,\\
\frac{1}{\bar{p}_i}  + \sum_{j\neq i} \frac{1}{p_j} + \frac{1}{p}=\mez,\quad \bar{p}_i\in (1, \infty)\quad\text{if~}\beta_i=1.
\end{cases}
\]
Regarding $II$, we appeal to \cref{thm.gagliardoNirenberg}  and bound 
\[
    \|\p^i h\|_{W^{\mu,\bar{p}_i}} \le \|h\|_{L^{\infty}}^{1-\bar{\eta}_i}\|h\|_{H^{\sigma}}^{\bar{\eta}_i} \text{ with } \frac{1}{\bar{p}_i} = \frac{\bar{\eta}_i}{2} = \frac{i+\mu}{2\sigma},
\] 
\[
    \|\p^i h\|_{L^{\tilde{p}_i(\beta_i-1)}} \le \|h\|_{L^{\infty}}^{1-\tilde{\eta}_i}\|h\|_{H^{\sigma}}^{\tilde{\eta}_i} \text{ with } \frac{1}{\tilde{p}_i(\beta_i -1)} = \frac{\tilde{\eta}_i}{2} = \frac{i}{2\sigma}\quad\text{if~}\beta_i\ge 2,
\]
\[
    \|\p^j h\|_{L^{p_j\beta_j}} \le \|h\|_{L^{\infty}}^{1-\eta_j}\|h\|_{H^{\sigma}}^{\eta_j} \text{ with } \frac{1}{p_j\beta_j} = \frac{\eta_j}{2} = \frac{j}{2\sigma}.
\]
We have  $\frac{1}{\bar{p}_i} + \frac{1}{\tilde{p}_i} + \sum_{j\neq i} \frac{1}{p_j} = \frac{m+\mu}{2\sigma}$ if $\beta_i\ge 2$, and $\frac{1}{\bar{p}_i} + \sum_{j\neq i} \frac{1}{p_j} = \frac{m+\mu}{2\sigma}$ if $\beta_i=1$. Consequently,  $\frac{1}{p}=\frac{n-m}{2\sigma}$ for $\beta_i\ge 1$, and hence
\[
    \|w\|_{L^p}\le \|F^{(\beta)}\|_{L^{\infty}}\|\p^{n-m}v\|_{L^p} \le     \|F^{(\beta)}\|_{L^{\infty}} \|v\|_{L^{\infty}}^{1-\eta}\|v\|_{H^{\sigma}}^{\eta}, 
\]
where $\frac{1}{p}=\frac{n-m}{2\sigma} = \frac{\eta}{2}$. Thus $II$ is bounded by the right-hand side of \eqref{eq:est-unbounded}.

As for $I$, we apply  \cref{thm.gagliardoNirenberg} again to have
\[
    \|\p^j h\|_{L^{q_j\beta_j}} \le \|h\|_{L^{\infty}}^{1-\theta_j}\|h\|_{H^{\sigma}}^{\theta_j},
\]
where $\frac{1}{q_j\beta_j}=\frac{\theta_j}{2}=\frac{j}{2\sigma}$. Then we have $\sum_{j=1}^m \frac{1}{q_j} = \frac{m}{2\sigma}$ and $\frac{1}{q}=\frac{\sigma - m}{2\sigma}$. Hence, 
\bq\label{AppendixE:I:0}
I\le  C(\|h\|_{L^{\infty}\cap H^\sigma})\|w\|_{W^{\mu,q}}.
\eq
We recall from \eqref{eq.equivalent-norm-sobolev} that 
\bq\label{intchar}
    \|w\|_{W^{\mu, q}}= \|w\|_{L^q} + \| E[w]\|_{L^q},    
\eq
where $N=d+1$ and 
\[
E[w](x,z):=\displaystyle \int_{z'\in J}\int_{x'\in\Rr^d} \frac{|w(x,z)-w(x',z')|^q}{(|x-x'|^2+|z-z'|^2)^{N/2+q\mu/2}} \,\mathrm{d}x'\mathrm{d}z'.
\]
Let $(x,z)\in\Rr^d \times J$ and $(x',z')\in\Rr^d \times J$. We have 
\begin{multline}
    \label{eq:split-w}
    |w(x', z') - w(x,z)| \lesssim |F^{(\beta)}(h(x',z')+z')||\p^{n-m}v(x,z)-\p^{n-m}v(x',z')| \\
    + |\p^{n-m}v(x,z)||F^{(\beta)}(h(x,z)+z)-F^{(\beta)}(h(x',z')+z)| \\
    + |\p^{n-m}v(x,z)||F^{(\beta)}(h(x',z')+z)-F^{(\beta)}(h(x',z')+z')|. 
\end{multline}
In the first line of \eqref{eq:split-w}, we bound $|F^{(\beta)}(h(x',z')+z')| \le \|F^{(\beta)}\|_{L^{\infty}}$. In the second line, since $F^{(\beta+1)}$ is bounded for $\beta+1\le n+1\le \lceil \sigma\rceil$, we can use the Lipschitz continuity of $F^{(\beta)}$ to have
\[
    |F^{(\beta)}(h(x,z)+z)-F^{(\beta)}(h(x',z')+z)| \le \|F^{(\beta +1)}\|_{L^{\infty}}|h(x,z)-h(x,z')|.
\]
In the last line, we combine  the Lipschitz continuity and the boundedness of $F^{(\beta)}$ to obtain
\[
    |F^{(\beta)}(h(x',z')+z)-F^{(\beta)}(h(x',z')+z')| \lesssim (\|F^{(\beta)}\|_{L^\infty}+\|F^{(\beta +1)}\|_{L^\infty}) \frac{|z-z'|}{1+|z-z'|}. 
\]
It follows that
\begin{multline*}
    E[w](x,z) \lesssim E[\p^{n-m}v](x,z) +  |\p^{n-m}v(x,z)|^qE[h](x,z) \\ 
     + |\p^{n-m}v(x,z)|^q \underbrace{\int_{(x',z') \in \Rr^d\times J} \frac{|z-z'|^q}{(1+|z-z'|^q)(|x-x'|^2+|z-z'|^2)^{N/2+q\mu/2}} \mathrm{d}z' \mathrm{d}x'}_{I}. 
\end{multline*}
To bound $I$ we make the changes of variables $(x', z')\mapsto (x'-x, z'-z)$ and $x'\mapsto |z'|y$, so that
\bq\label{AppE:100}
\begin{aligned}
    I &= \int_{(x',z') \in \Rr^d\times (J-z)} \frac{|z'|^q}{(1+|z'|^q)(|x'|^2+|z'|^2)^{N/2+q\mu/2}} \mathrm{d}z' \mathrm{d}x' \\
    &= \int_{z'\in J-z} \frac{|z'|^{q-N-q\mu}}{(1+|z'|)^q} \int_{x' \in \Rr^d}\frac{\mathrm{d}x'}{(1+\frac{|x'|^2}{|z'|^2})^{N/2+q\mu/2}}\,\mathrm{d}z' \\
   & \le \int_{z'\in \Rr} \frac{|z'|^{q-N-q\mu+d}}{(1+|z'|)^q} \int_{y\in \Rr^d}\frac{\mathrm{d}y}{(1+|y|^2)^{N/2+q\mu/2}}\,\mathrm{d}z'. 
\end{aligned}
\eq
The $y$-integral converges   because $N+2\mu>d=N-1$. The integral in $z'$ converges since $\frac{|z'|^{q-N-q\mu+d}}{(1+|z'|)^q} \underset{|z'| \to 0}{\sim} |z'|^{q(1-\mu)-1}$ and $\frac{|z'|^{q-N-q\mu+d}}{(1+|z'|)^q} \underset{|z'| \to \infty}{\sim} |z'|^{-1-q\mu}$. Therefore, $I$ is bounded by a constant depending only on $(d, \mu)$. 
It follows that 
\[
    E[w](x,z) \lesssim E[\p^{n-m}v](x,z) +  |\p^{n-m}v(x,z)|^qE[h](x,z) +  |\p^{n-m}v(x,z)|^q. 
\]  
Taking $1/q$-th power and then taking the $L^q$ norm in $(x,z)$ and applying H\"older inequality to the second term, we obtain 
\begin{multline*}
    \|w\|_{W^{\mu,q}} \lesssim   \|\p^{n-m}v\|_{W^{\mu, q}}+\| \p^{n-m}v\|_{L^{r_1}}\|(E[h])^{\frac{1}{q}}\|_{L^{r_2}} + \| \p^{n-m}v\|_{L^{q}} \\
     \lesssim  \|\p^{n-m}v\|_{W^{\mu, q}}+\|\p^{n-m}v\|_{L^{r_1}}\|h\|_{W^{\mu,r_2}} + \|\p^{n-m}v\|_{L^{q}},
\end{multline*}
where $r_1=\frac{2\sigma}{n-m}$ and $r_2= \frac{2\sigma}{\mu}$, so that $\frac{1}{q}=\frac{1}{r_1}+\frac{1}{r_2}$.  Applying \cref{thm.gagliardoNirenberg} gives
\begin{align*}
&\|\p^{n-m}v\|_{L^{r_1}}\les \| v\|_{L^\infty}^{1-s_0}\| v\|_{H^\sigma}^{s_0},\quad \frac{1}{r_1}=\frac{s_0}{2}=\frac{n-m}{2\sigma},\\
&\|h\|_{W^{\mu,r_2}}\les \| h\|_{L^\infty}^{1-s_1}\| h\|_{H^\sigma}^{s_1},\quad \frac{1}{r_2}=\frac{s_1}{2}=\frac{\mu}{2\sigma},\\
&  \|\p^{n-m}v\|_{W^{\mu, q}}\les \| v\|_{L^\infty}^{1-s_2}\| v\|_{H^\sigma}^{s_2},\quad \frac{1}{q}=\frac{s_2}{2}=\frac{n-m+\mu}{2\sigma}\equiv \frac{\sigma-m}{2\sigma},\\
&\|\p^{n-m}v\|_{L^{q}}\les \| v\|_{L^\infty}^{1-s_3}\| v\|_{H^r}^{s_3},\quad \quad\frac{1}{q}=\frac{s_3}{2}=\frac{n-m}{2r},~r=\frac{\sigma(n-m)}{\sigma-m}< \sigma.
\end{align*}
The above estimates yield 
\[
  \|w\|_{W^{\mu,q}} \les (1+\| h\|_{L^\infty\cap H^\sigma})\| v\|_{L^\infty\cap H^\sigma}
  \]
which in conjunction with \eqref{AppendixE:I:0} shows that $I$ is controlled by the right-hand side of \eqref{eq:est-unbounded}. 
\end{proof}

\newcommand{\etalchar}[1]{$^{#1}$}

\end{document}